\documentclass[11pt]{amsart}

\usepackage{amsthm,amsmath,amsfonts,amssymb,amsthm,graphicx,epsf,epstopdf,
color,pifont,subfigure,flafter,bm,amscd}
\allowdisplaybreaks[4]

\def\epsilon{\varepsilon}

\newtheorem {theorem} {Theorem}
\newtheorem {proposition} [theorem] {Proposition}

\newtheorem {lemma} [theorem] {Lemma}

\begin{document}

\title[A quintic $\mathbb{Z}_2$-equivariant Li\'enard system : (II)]
{	A quintic $\mathbb{Z}_2$-equivariant Li\'enard system arising from the complex Ginzburg-Landau equation: (II)}

\author[H. Chen et.al.]
{Hebai Chen$^1$, Xingwu Chen$^2$, Man Jia$^1$ and Yilei Tang$^3$  }

\address{$^1$ School of Mathematics and Statistics, HNP-LAMA, Central South University,
	Changsha, Hunan 410083,   China}
\email{chen\_hebai@csu.edu.cn  }
\email{mathmj@csu.edu.cn,  corresponding author}

\address{$^2$ School of Mathematics, Sichuan University,
	Chengdu, Sichuan 610064,  China}
\email{scuxchen@scu.edu.cn, xingwu.chen@hotmail.com, corresponding author}

\address{$^3$    School of Mathematical Sciences, CMA-Shanghai, Shanghai
	Jiao Tong University,
	Shanghai, 200240,   China}
\email{mathtyl@sjtu.edu.cn}

\begin{abstract}
We continue to study a quintic $\mathbb{Z}_2$-equivariant Li\'enard system $\dot x=y,
~ \dot y=-(a_0x+a_1x^3+a_2x^5)-(b_0+b_1x^2)y$ with $a_2b_1\ne 0$,
arising from the complex Ginzburg-Landau equation.
Global dynamics of the system have been studied in [{\it SIAM J. Math. Anal.}, {\bf 55}(2023) 5993-6038] when the sum of the indices of all equilibria is $-1$, i.e., $a_2<0$.
The aim of this paper is to study the global dynamics of this quintic Li\'enard system when the sum of the indices of all equilibria is $1$, i.e., $a_2>0$.
\end{abstract}


\date{}
\dedicatory{}

\maketitle

\section{Introduction and  main results}
With the aid of appropriate traveling wave transformations, numerous nonlinear partial differential equations (PDEs) can be transformed into manageable nonlinear ordinary differential equations (ODEs). Among these transformed equations, Li\'enard equations  occupy a significant position. Named after  the  French mathematician   Li\'enard, these ODEs are characterized by the  form:
\[
\ddot{ x}+f(x)\dot{x}+g(x)=0
\]%
where
$x$ represents the dependent variable,
$\dot{x}$  denotes its derivative with respect to an independent variable (often time), and
$f(x) $ and
$g(x)$ are nonlinear functions of $x$.
Li\'enard equations are a significant class of ODEs, extensively utilized in various fields such as electrical mechanics, mechanical engineering, physics, finance systems and biomedical systems, providing crucial insights into system behaviors. These equations are particularly valuable because many complex mathematical models can be transformed into a Li\'enard-type system, allowing researchers to study their dynamic behaviors more effectively. This transformation facilitates a deeper understanding of the underlying principles governing these systems, making the Li\'enard equations a powerful tool in theoretical and applied research. For further insights and detailed studies on the applications and transformations involving Li\'enard equations, one can refer to sources such as \cite{CHLWP,CCJT, CJT, CLT,HCCW2012,LS, Ve, Zh} and the references therein.

In this paper, we continue our study  of  global dynamics for the quintic $\mathbb{Z}_2$-equivariant Li\'enard system
 \begin{eqnarray}
 \begin{array}{ll}
 \dot x=y,
 \\
 \dot y=-(a_0 x+a_1 x^3+ a_2 x^5)-(b_0+b_1 x^2)y
 \end{array}
 \label{initial0}
 \end{eqnarray}
 where $(a_0, a_1, a_2, b_0, b_1)\in \mathbb{R}^5$ and $a_2 b_1\ne 0$.
Feng \cite{Feng1, Feng2} proved that certain  uniformly translating solutions of the complex Ginzburg-Landau equation
\begin{eqnarray*}
\begin{aligned}
u_t=&\alpha u+(b_1+ic_1) u_{xx}-(b_2-ic_2)|u|^2u-(b_3-ic_3)|u|^4u
\\
&+(b_4+ic_4)(|u|^2u)_x+(b_5+ic_5)(|u|^2)_xu
\end{aligned}
\end{eqnarray*}
can  converted to solutions of  the  Li\'enard system \eqref{initial0}.
The  Ginzburg-Landau equation,   a  classic nonlinear  PDE,  has
   garnered significant attention from researchers, see \cite{CG, GR, GW, KKP, LW, LP,OY,Wang} and the references therein.
As demonstrated in \cite{CCJT},
system (\ref{initial0}) is also a versal unfolding of the following degenerate system
$\dot x=y,
\dot y=-a_2 x^5-b_1 x^2 y
$
within the $\mathbb{Z}_2$-equivariant class for sufficiently small $|a_0|, |a_1|$ and $|b_0|$.

For system \eqref{initial0}, we  have  completely     investigated    its  global dynamics     in \cite{CCJT} for the case $a_2<0$ (saddle case), i.e, the sum of the indices of all equilibria is $-1$.
  In this paper, our focus shifts to
  the case $a_2>0$ focus case), i.e., the sum of the indices of all equilibria is $+1$.
With a linear  transformation
$
(x,~y)\to ( a_2 ^{-1/4}x,~ a_2 ^{-1/4}y),
$
system \eqref{initial0} is rewritten as
\begin{eqnarray}
\begin{array}{ll}
\dot x=y,
\\
\dot y=-(\mu_1x+\mu_2x^3+x^5)-(\mu_3 +   bx^2)y=:-g(x)-f(x)y,
\end{array}
\label{initial}
\end{eqnarray}
where $(\mu_1 , \mu_2, \mu_3,    b)=(a_0,  a_1   a_2 ^{-1/2},  b_0,   b_1a_2 ^{-1/2})\in \mathbb{R}^4$ and $b\neq0$.
Since system \eqref{initial} is invariant under the transformation $(y,t,\mu_3,b)\to(-y,-t,-\mu_3,-b)$,
we only need to study the case $b>0$.

For system \eqref{initial},
Dangelmayor et al.  \cite{DAN}
gave some cross sections of local bifurcation diagram in a small neighborhood of $(\mu_1,\mu_2,\mu_3)=(0,0,0)$
and some local phase portraits near the origin,
without a detailed quantitative analysis.
Actually, one can find that the quantitative proof is indeed non-trivial
in a small neighborhood of $(\mu_1,\mu_2,\mu_3)=(0,0,0)$ as shown in this paper.
For the convenience to read, we present the main results of this paper here.

Clearly, system  \eqref{initial}  has a unique equilibrium as $(\mu_1,\mu_2)$ belongs to the region
$$
\mathcal{G}_1:=\{(\mu_1,\mu_2)\in\mathbb{R}^2:  {\mu_2}^2-4\mu_1<0, ~\mu_2<0 \}\cup\{(\mu_1,\mu_2)\in\mathbb{R}^2: \mu_1\geq0, ~\mu_2\geq0\} 
$$
and at least three equilibria as $(\mu_1,\mu_2)$ belongs to the region
$$
\mathcal{G}_2:=\{(\mu_1,\mu_2)\in\mathbb{R}^2: {\mu_2}^2-4\mu_1\geq0, ~\mu_2<0\}\cup\{(\mu_1,\mu_2)\in\mathbb{R}^2: \mu_1<0, ~\mu_2\geq0\}.
$$

\begin{theorem}

 System \eqref{initial} has a unique equilibrium if and only if $(\mu_1,\mu_2)\in\mathcal{G}_1$.
Furthermore,
 system \eqref{initial} has a unique limit cycle as $\mu_3<0$ and no limit cycles as $\mu_3\geq0$,
 and all global phase portraits  in  the Poincar\'e disc  are shown in {\rm Figure \ref{gpp1}}.
\label{Result1}
\end{theorem}

\begin{figure}[h!]
	\centering
	\subfigure[$\mu_3<0$, $0<b<2\sqrt{3}$ ]{
		\scalebox{0.31}[0.31]{
			\includegraphics{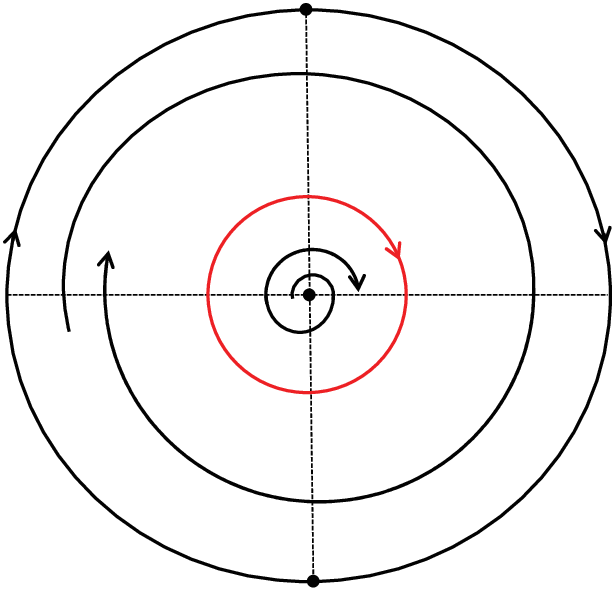}}}
	\subfigure[$\mu_3<0$, $b\geq2\sqrt{3}$ ]{
	\scalebox{0.31}[0.31]{
			\includegraphics{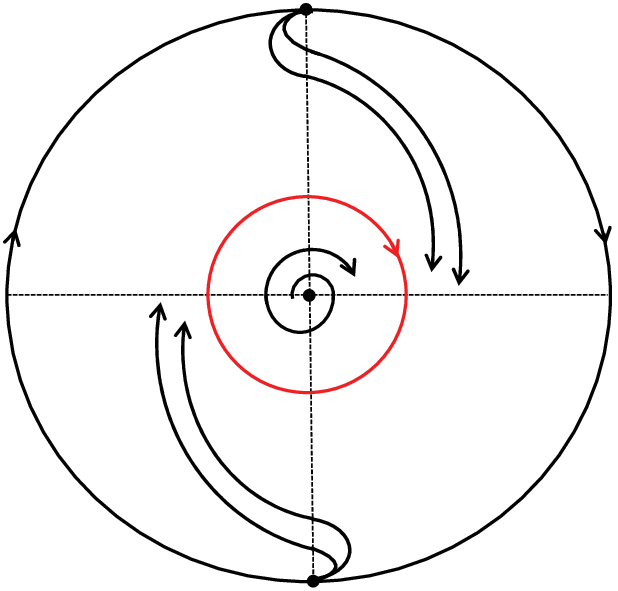}}}
			\subfigure[$\mu_3\geq0$,    $0<b<2\sqrt{3}$ ]{
		\scalebox{0.31}[0.31]{
				\includegraphics{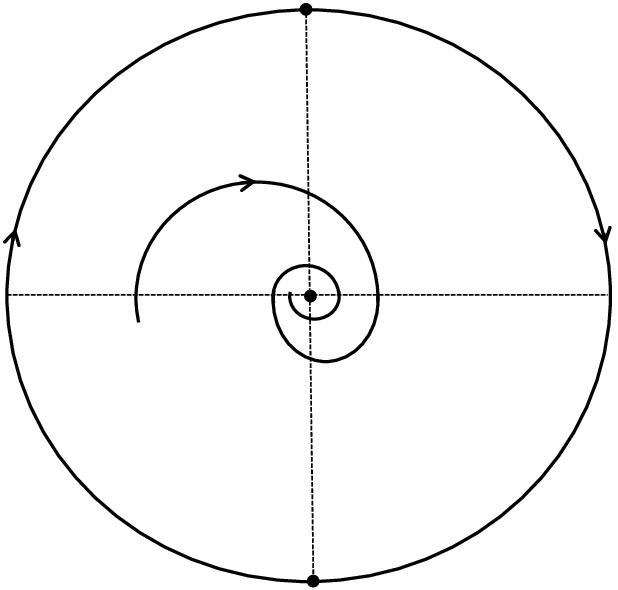}}}
		\subfigure[$\mu_3\geq0$, $b\geq2\sqrt{3}$ ]{
			\scalebox{0.31}[0.31]{
				\includegraphics{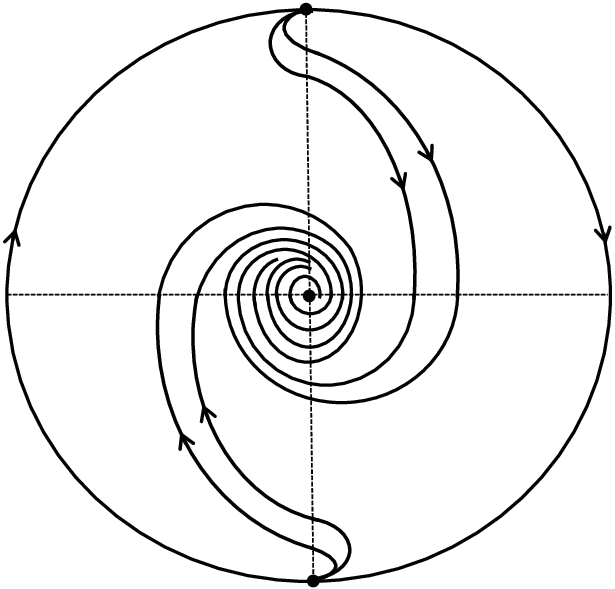}}}
	\caption{Global phase portraits    of    \eqref{initial} in $\mathcal{G}_1$.}
	\label{gpp1}
\end{figure}

Consider $(\mu_1,\mu_2)\in\mathcal{G}_2$.
By the following scaling transformation
\begin{eqnarray}
(x, y, t)\to \left(sx, s^3y,  \frac{1}{s^2} t\right),
\label{sctra}
\end{eqnarray}
system \eqref{initial}  is changed into
\begin{eqnarray}
\begin{array}{ll}
\dot x=y,
\\
\dot y=-x(a_1+x^2)(-1+x^2)-\delta(a_2+x^2)y=:-\hat g(x)-\hat f(x)y,
\end{array}
\label{initial1}
\end{eqnarray}
where
\[
(a_1,a_2,\delta):=\left(\frac{\mu_2}{s^2}+1, \frac{\mu_3}{bs^2},b\right)\in [-1,+\infty) \times(-\infty,+\infty)\times(0,+\infty)=:\Omega
\]
and
\[
s:=\sqrt{\frac{-\mu_2+\sqrt{{\mu_2}^2-4\mu_1}}{2}}\ne 0.
\]
 It is clear that system \eqref{initial1} has three equilibria
$\hat E_{l2}:=(-1,0)$,   $\hat E_0:=(0, 0)$,    $\hat E_{r2}:=(1,0)$,
and two additional equilibria $\hat E_{l1}:=(-\sqrt{-a_1},0)$,    $\hat E_{r1}:=(\sqrt{-a_1},0)$ if $-1<a_1<0$.

\begin{figure}[h!]
	\centering
	\subfigure[ when $0<\delta_0<2\sqrt{3}$   ]{
		\scalebox{0.72}[0.72]{
			\includegraphics{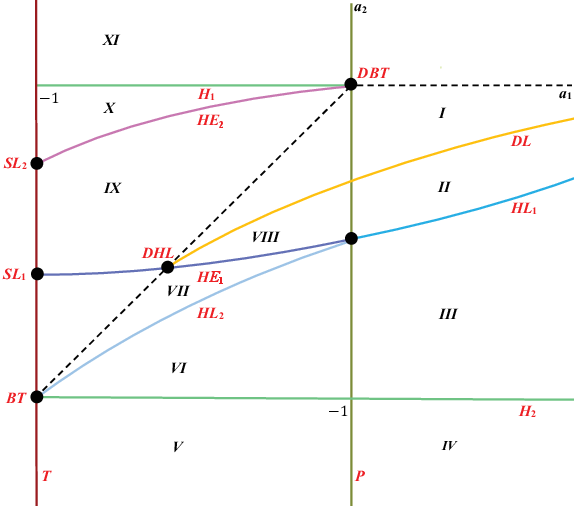}}}
	\subfigure[ when $\delta_0\geq2\sqrt{3}$  ]{
		\scalebox{0.72}[0.72]{
			\includegraphics{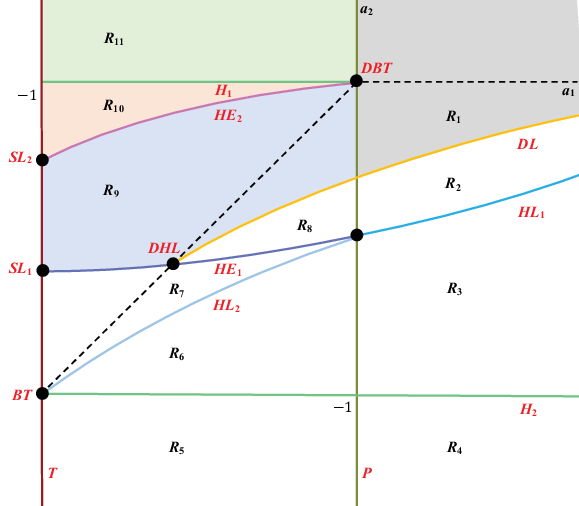}}}
	\caption{The slice $\delta=\delta_0$ of bifurcation diagram  of    \eqref{initial1}.  }
	\label{bf1}
\end{figure}

\begin{theorem}
Given  $\delta=\delta_0>0$,  the  cross-section of the global bifurcation diagram
	of  system  \eqref{initial1} is given in {\rm Figure~\ref{bf1}}.
	Bifurcation diagram of system \eqref{initial1} consists of the following bifurcation sets:
	\begin{description}
		\item[(i)] There is a pitchfork bifurcation surface  
		\[
		P:=\{(a_1,a_2,\delta)\in \Omega ~|~a_1=0\}.
		\]

		\item[(ii)] There is a transcritical bifurcation surface
		\[
		T:=\{(a_1,a_2,\delta)\in \Omega ~|~a_1=-1\}
		\]
		for $\hat E_{l2}$ and $\hat E_{r2}$.
		
		\item[(iii)] There are a supcritical Hopf bifurcation surface
		\[
		H_1 :=\{(a_1,a_2,\delta)\in \Omega ~|~-1\leq a_1<0~{\rm and}~a_2=0\}
		\]
		for $\hat E_0$
		and
		 a subcritical Hopf bifurcation surface
		\[
		H_2:=\{(a_1,a_2,\delta)\in \Omega ~|~a_1>-1~{\rm and}~a_2=-1\}
		\]
		for $\hat E_{l2}$ and $\hat E_{r2}$.

		\item[(iv)] There is a degenerate Bogdanov-Takens bifurcation curve
		\[
		DBT:=\{(a_1,a_2,\delta)\in \Omega ~|~a_1=a_2=0\}
		\]
		for $\hat E_0$, which is the intersection of $H_1$ and $P$.
		
		\item[(v)]
		There is a Bogdanov-Takens bifurcation curve
		\[
		BT:=\{(a_1,a_2,\delta)\in \Omega ~|~a_1=a_2=-1\}
		\]
		for $\hat E_{l2}$ and $\hat E_{r2}$, which is the intersection of $H_2$ and $T$.
		
				\item[(vi)]
		There is a degenerate equilibrium bifurcation surface
		\[
		DE:=\{(a_1,a_2,\delta)\in \Omega ~|~\delta=2\sqrt{3}\}
		\]
		for equilibria at infinity.
		
		\item[(vii)] There are two homoclinic bifurcation surfaces
		\begin{eqnarray*}
		HL_1&:=&\{(a_1,a_2,\delta)\in \Omega ~|~a_2=\varphi_1(a_1,\delta), ~ a_1\geq0\},
		\\
				HL_2&:=&\{(a_1,a_2,\delta)\in \Omega ~|~a_2=\varphi_2(a_1,\delta), ~ -1<a_1<0\},
		\end{eqnarray*}
		 where $\varphi_1,\varphi_2$ are continuous, $\varphi_1(a_1,\delta)\in(-1, -1/3)$ and $\varphi_2(a_1,\delta)\in(-1, \min\{-1/3, a_1\})$.
		
		\item[(viii)] 	 There are two $2$-saddle loop bifurcation surfaces
				\begin{eqnarray*}
								HE_1&:=&\{(a_1,a_2,\delta)\in \Omega ~|~a_2=\varphi_3(a_1,\delta), ~ -1<a_1<0\},
								\\
		HE_2&:=&\{(a_1,a_2,\delta)\in \Omega ~|~a_2=\varphi_4(a_1,\delta), ~-1< a_1<0\},
			\end{eqnarray*}
		where  $\varphi_3(a_1,\delta)\in(\varphi_2(a_1,\delta), 0)$ and
		$\varphi_4(a_1,\delta)\in (\max\{\varphi_3(a_1,\delta), -1/3\}, 0)$ are continuous.
		
		\item[(ix)] There is a double limit cycle bifurcation surface
		\[
		DL:=\{(a_1,a_2,\delta)\in \Omega ~|~a_2=\varphi_5(a_1,\delta), ~ a_1>a^* \},
		\]
		where $\varphi_5(a^*,\delta)=\varphi_3(a^*,\delta)=a^*<-1/3$,
		$\varphi_5(a_1,\delta)\in(\varphi_3(a_3,\delta),\min\{-1/3, a_1\})$ for $a^*<a_1<0$
		and $\varphi_5(a_1,\delta)\in(\varphi_1(a_3,\delta),-1/3)$ for $a_1\geq0$.
		
		\item[(x)] There are  infinitely many heteroclinic bifurcation surfaces 	$SC_{1,i}$, $SC_{2,i}$
		for $\hat E_0$ and equilibria at infinity, given by
		\[
		SC_{k,i}:=\{(a_1,a_2,\delta)\in \Omega ~|~a_2=\psi_{k,i}(a_1,\delta), ~ a_1\geq0, ~ \delta\geq2\sqrt{3}\},
		\]
	for $k=1,2$ and $i\in\mathbb{N}$ respectively, where $\psi_{1,i}, \psi_{2,i}$ {\rm(}$i\!\in\! \mathbb{N}${\rm)}
are continuous,
$\psi_{1, i}(a_1,\delta)>\psi_{2,i}(a_1,\delta)>\psi_{1,i+1}(a_1,\delta)>\psi_{2,i+1}(a_1,\delta)>\varphi_5(a_1,\delta)$.
		
		\item[(xi)] There are infinitely many heteroclinic bifurcation surfaces
		$SC_{3,i}$, $SC_{4,i}$, $SC_{5,i}$, $SC_{6,i}$ for the stable manifolds $\hat E_{r1}$, $\hat E_{l1}$
			and equilibria at infinity, given by
		\[
		SC_{k,i}:=\{(a_1,a_2,\delta)\in \Omega ~|~a_2=\psi_{k,i}(a_1,\delta),  ~ -1 \leq a_1<0, ~\delta\geq2\sqrt{3}\}
		\]
		for $k=3,4,5,6$ and $i\in\mathbb{N}$ respectively, where $\psi_{3,i}, \psi_{4,i}, \psi_{5,i}, \psi_{6,i}$ {\rm(}$i\!\in\! \mathbb{N}^*${\rm)}
		are continuous,
	$\psi_{3, i}(a_1,\delta)>\psi_{4,i}(a_1,\delta)>\psi_{5, i}(a_1,\delta)>\psi_{6,i}(a_1,\delta)>\psi_{3, i+1}(a_1,\delta)>\psi_{4,i+1}(a_1,\delta)>\psi_{5, i+1}(a_1,\delta)>\psi_{6,i+1}(a_1,\delta)$.
	\end{description}
	\label{mr1}
\end{theorem}

\begin{figure}[h!]
	\centering
		\subfigure[in $T_1$]{
			\scalebox{0.31}[0.31]{
				\includegraphics{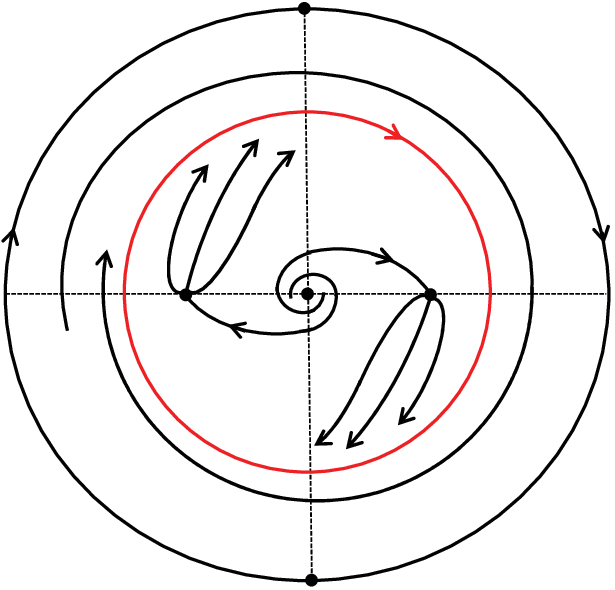}}}
\subfigure[in $BT$ ]{
			\scalebox{0.31}[0.31]{
				\includegraphics{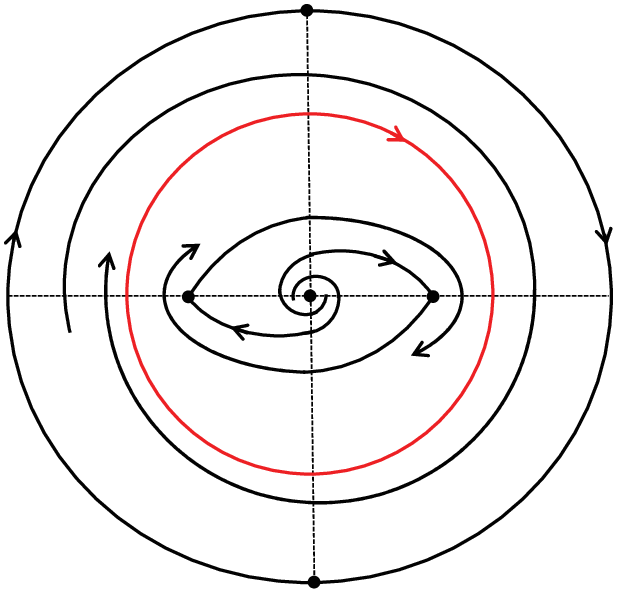}}}
\subfigure[in $T_2$]{
			\scalebox{0.31}[0.31]{
				\includegraphics{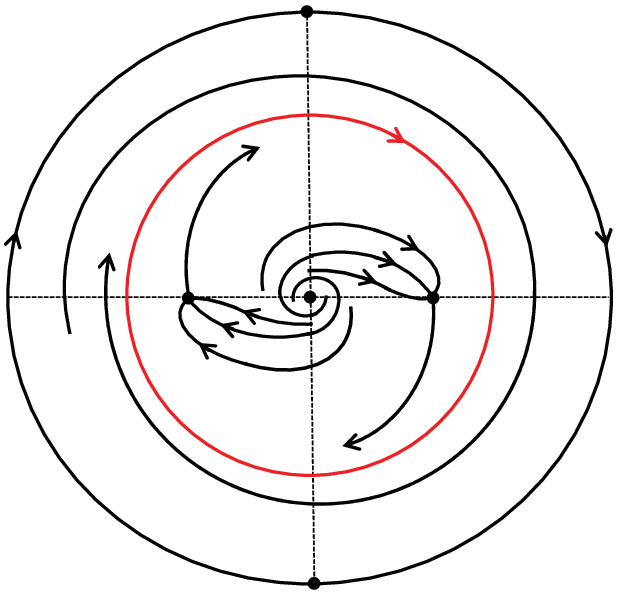}}}
			\subfigure[ in $SL_1$]{
			\scalebox{0.31}[0.31]{
				\includegraphics{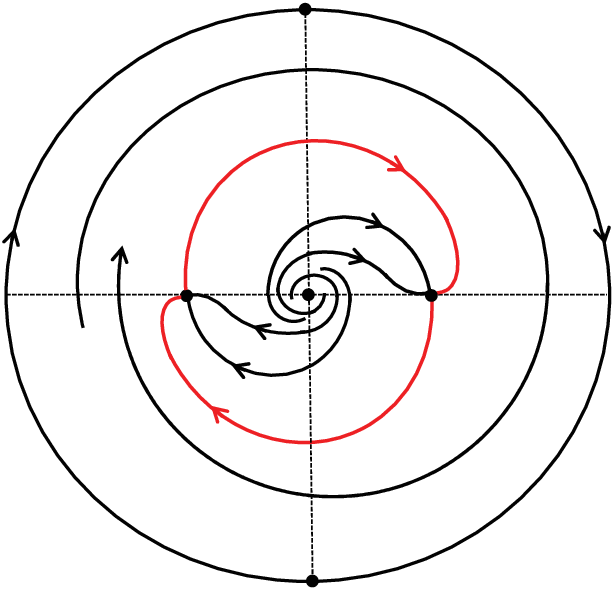}}}
\subfigure[ in $T_3$]{
			\scalebox{0.31}[0.31]{
			\includegraphics{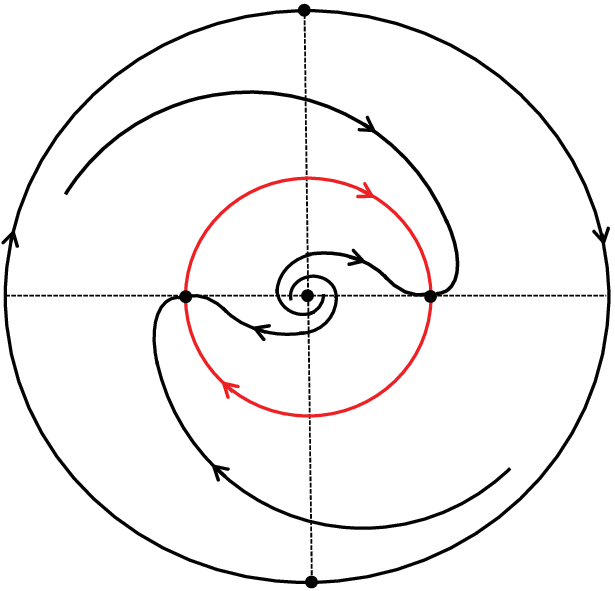}}}
\subfigure[ in $SL_2$]{
			\scalebox{0.31}[0.31]{
			\includegraphics{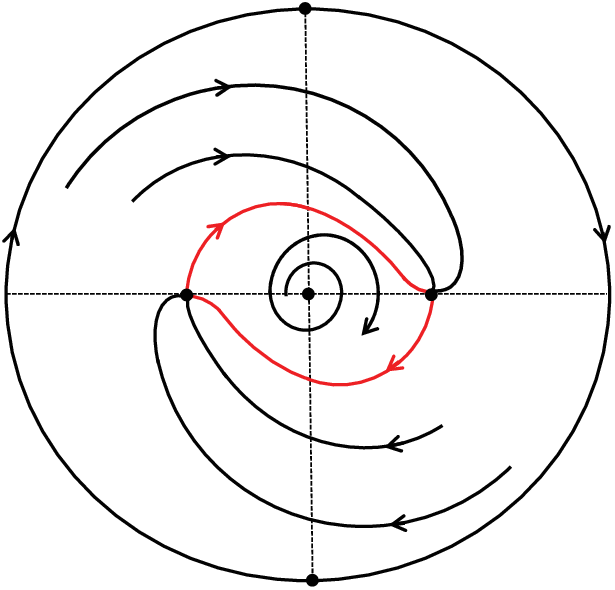}}}
\subfigure[ in $T_4$]{
			\scalebox{0.31}[0.31]{
			\includegraphics{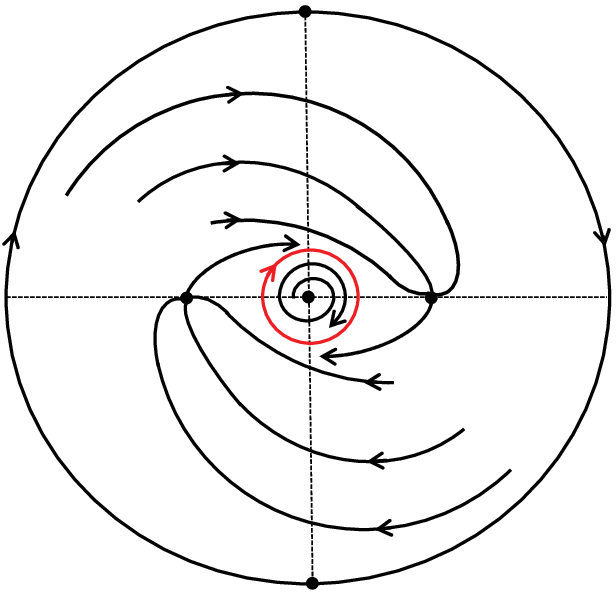}}}
	\subfigure[ in $T_5$]{
			\scalebox{0.31}[0.31]{
			\includegraphics{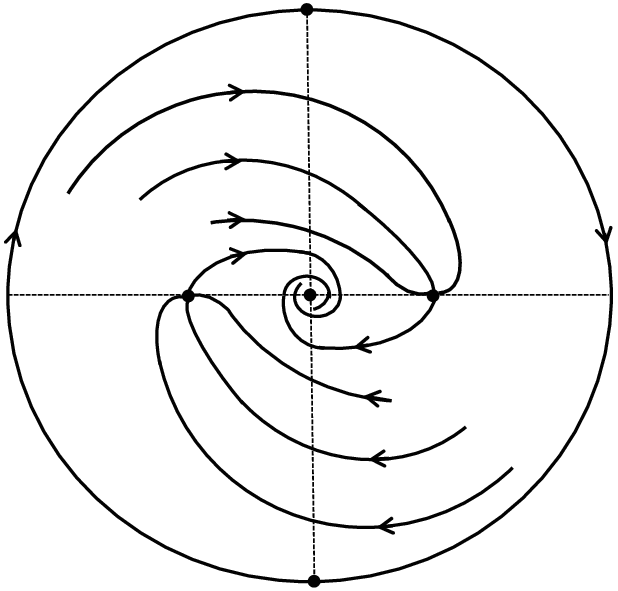}}}

	\caption{Global phase portraits    of    \eqref{initial1} for
 $0<\delta_0<2\sqrt{3}$ and   $a_1=-1$.}
	\label{gpp2}
\end{figure}
\begin{figure}
	\centering
	\subfigure[ in $T_1$]{
				\scalebox{0.31}[0.31]{
			\includegraphics{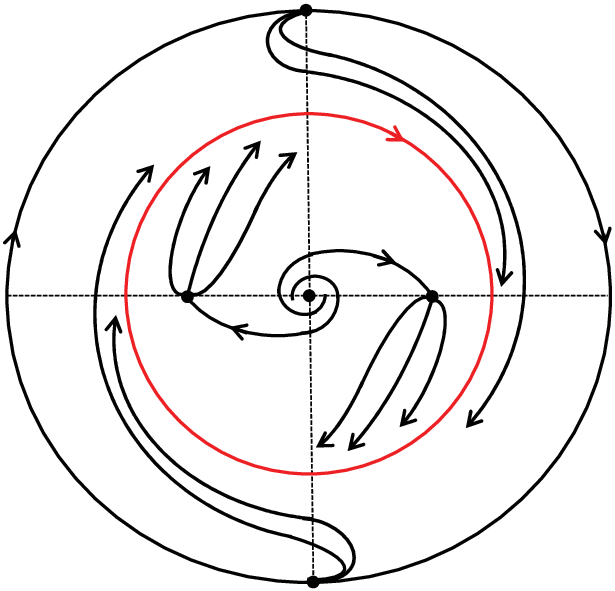}}}
	\subfigure[ in $BT$]{
				\scalebox{0.31}[0.31]{
			\includegraphics{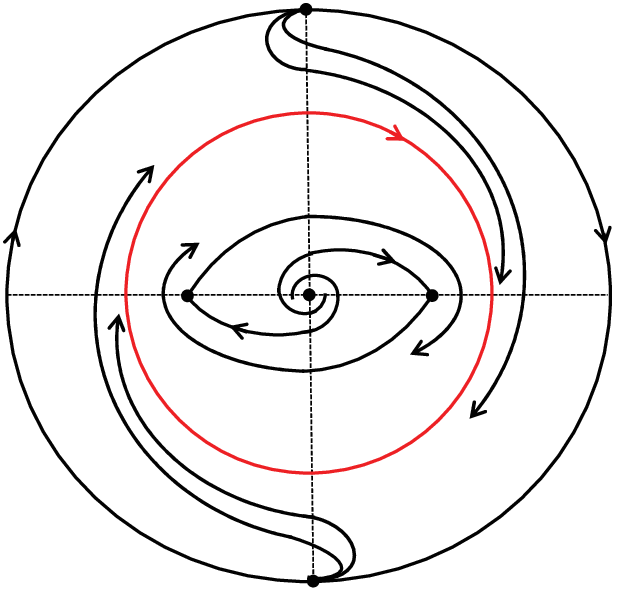}}}
	\subfigure[ in $T_2$ ]{
				\scalebox{0.31}[0.31]{
			\includegraphics{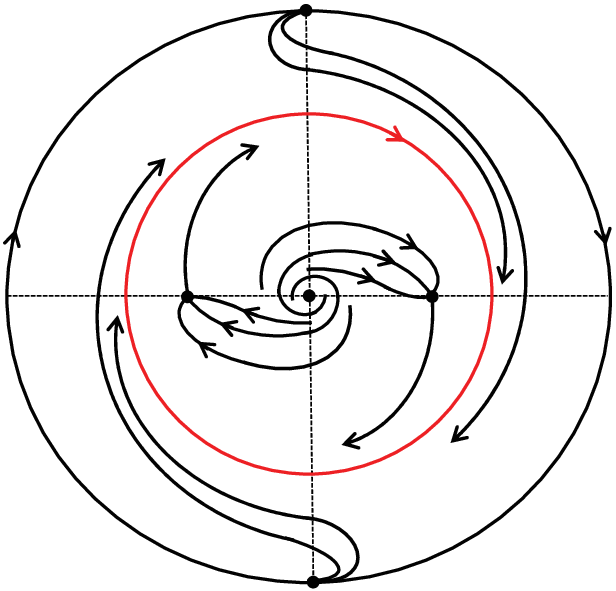}}}
	\subfigure[ in $SL_1$]{
				\scalebox{0.31}[0.31]{
			\includegraphics{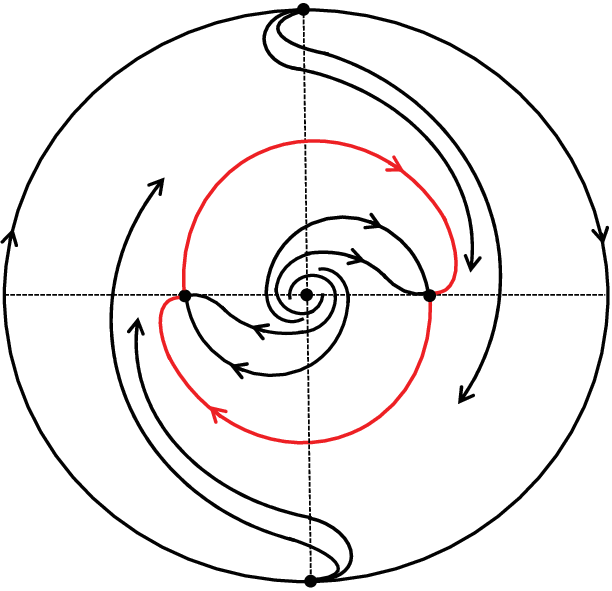}}}
		\subfigure[  in $T_{31i}$]{
			\scalebox{0.31}[0.31]{
				\includegraphics{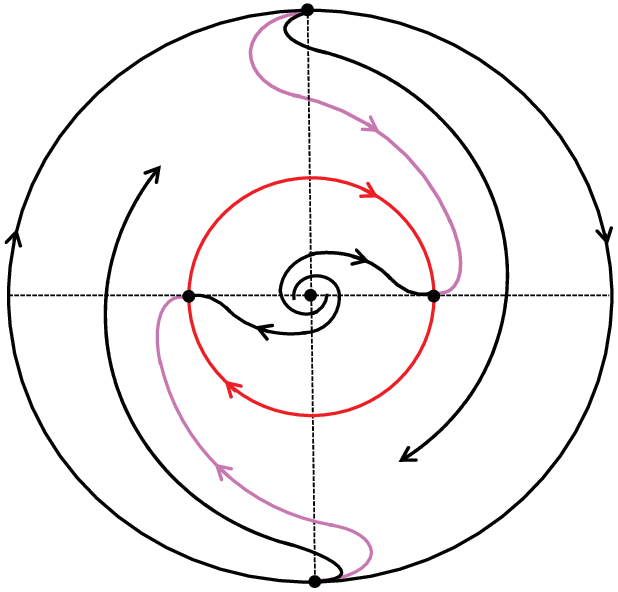}}}
		\subfigure[  in $T_{32i}$]{
			\scalebox{0.31}[0.31]{
				\includegraphics{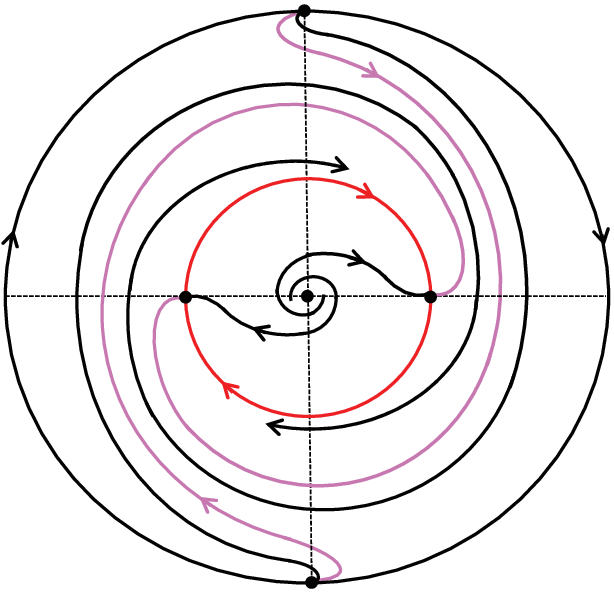}}}
		\subfigure[  in $T_{33i}$]{
			\scalebox{0.31}[0.31]{
				\includegraphics{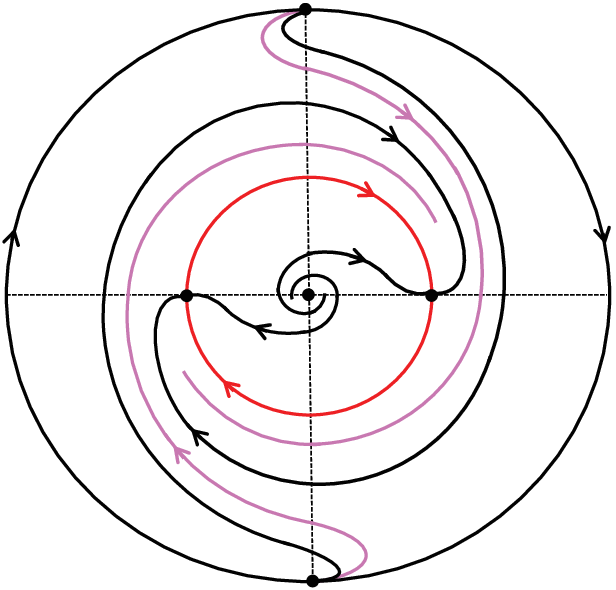}}}
	\subfigure[  in $T_{34i}$]{
			\scalebox{0.31}[0.31]{
				\includegraphics{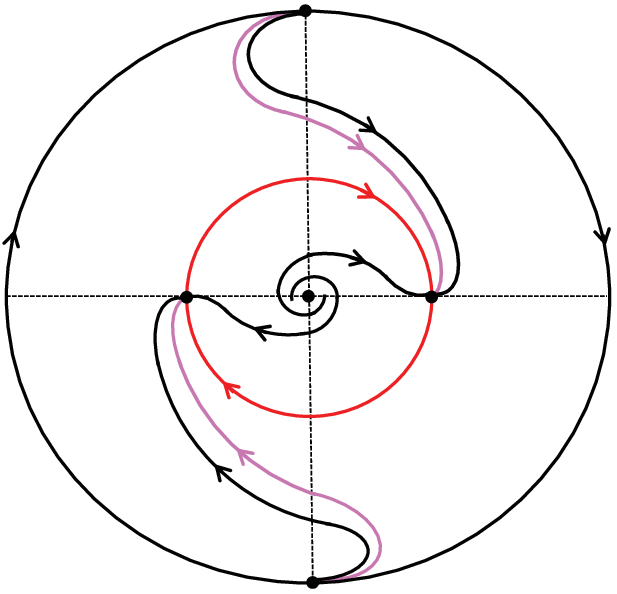}}}
\subfigure[ in $SL_{21i}$]{
			\scalebox{0.31}[0.31]{
			\includegraphics{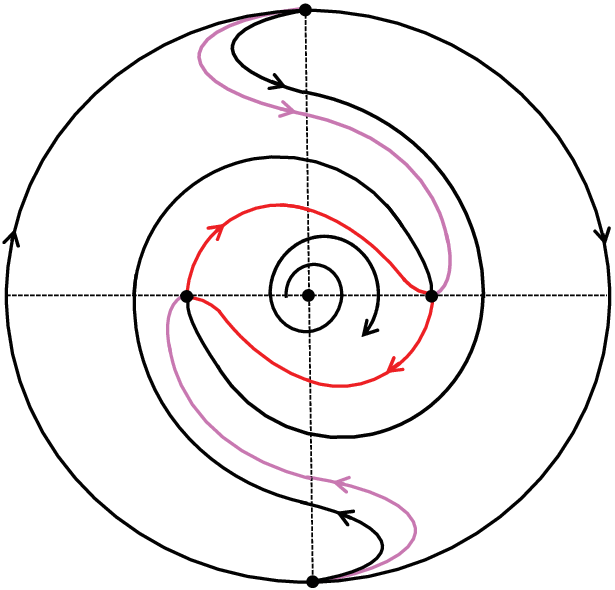}}}
\subfigure[ in $SL_{22i}$]{
			\scalebox{0.31}[0.31]{
			\includegraphics{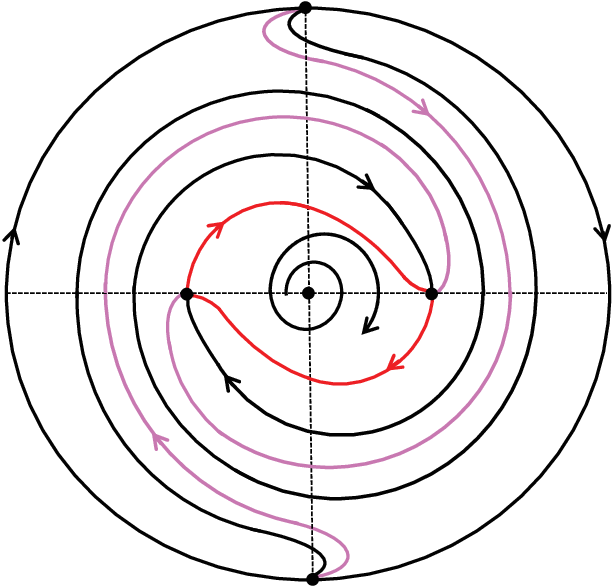}}}
	\subfigure[  in $SL_{23i}$]{
			\scalebox{0.31}[0.31]{
				\includegraphics{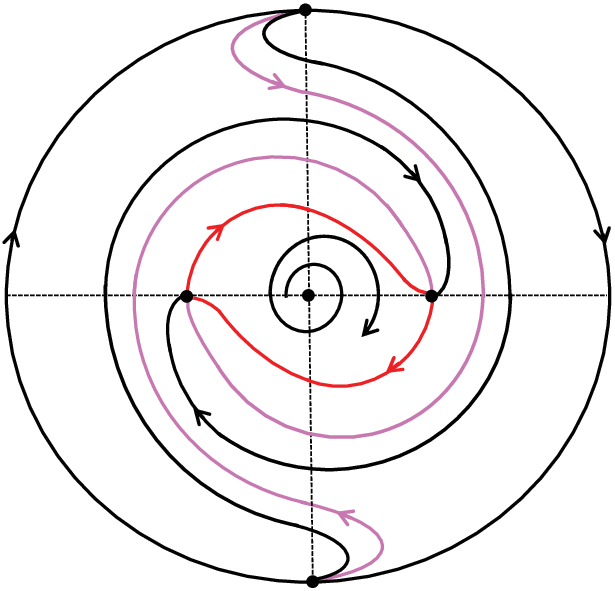}}}
\subfigure[  in $SL_{24i}$]{
			\scalebox{0.31}[0.31]{
			\includegraphics{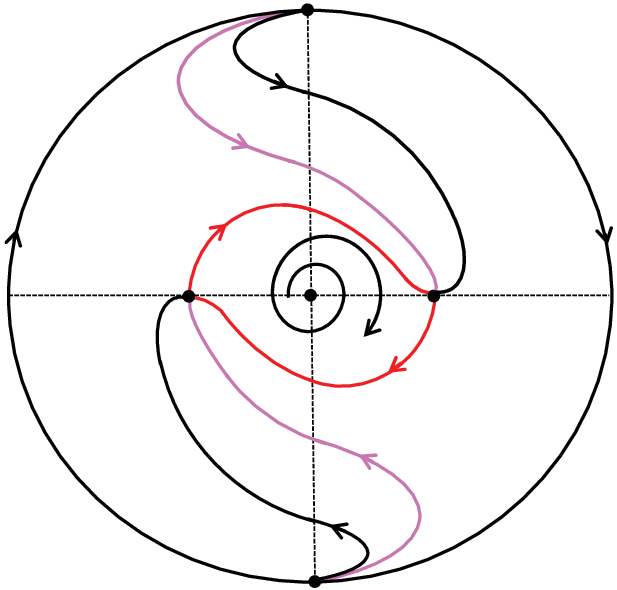}}}
\subfigure[ in $T_{41i}$]{
			\scalebox{0.31}[0.31]{
			\includegraphics{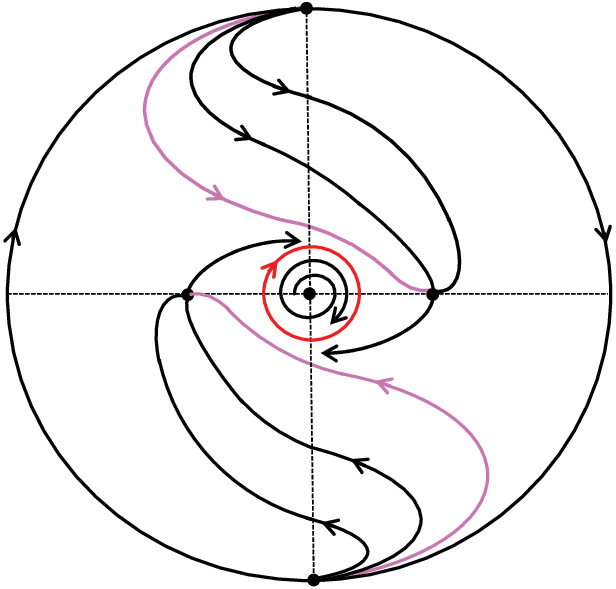}}}
\subfigure[ in $T_{42i}$]{
			\scalebox{0.31}[0.31]{
						\includegraphics{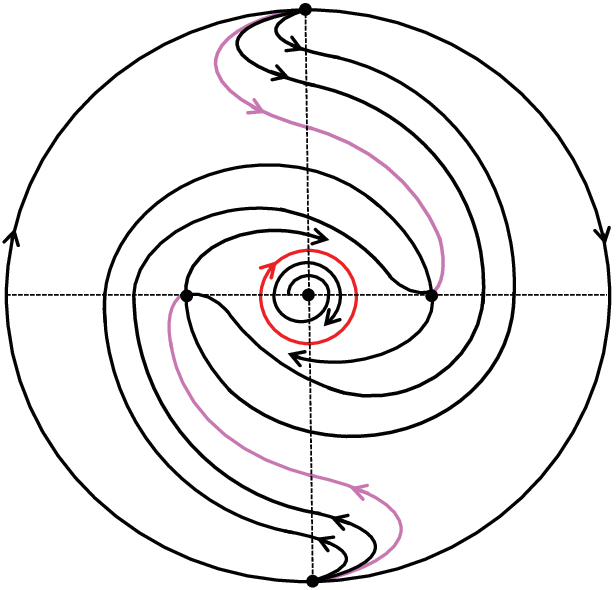}}}
						\subfigure[ in $T_{43i}$]{
				\scalebox{0.31}[0.31]{
							\includegraphics{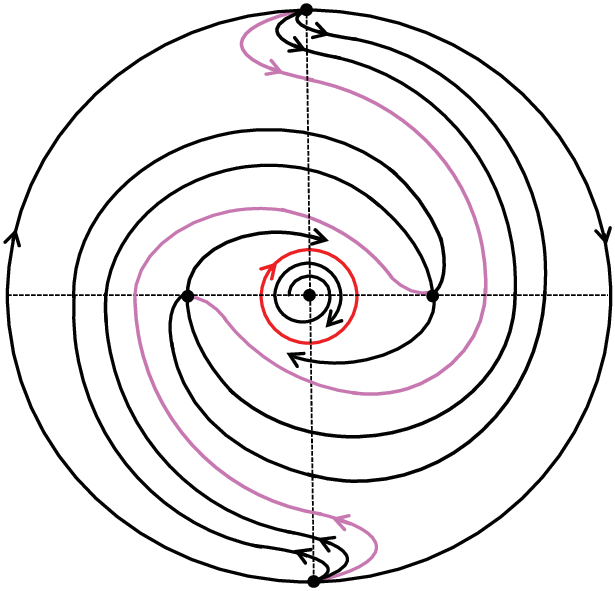}}}
		\subfigure[ in $T_{44i}$]{
				\scalebox{0.31}[0.31]{
					\includegraphics{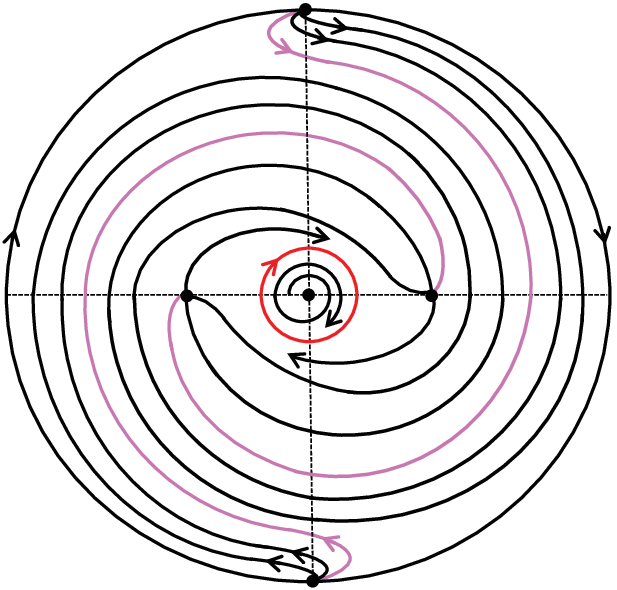}}}
			\subfigure[ in $T_{45i}$]{
			\scalebox{0.31}[0.31]{
					\includegraphics{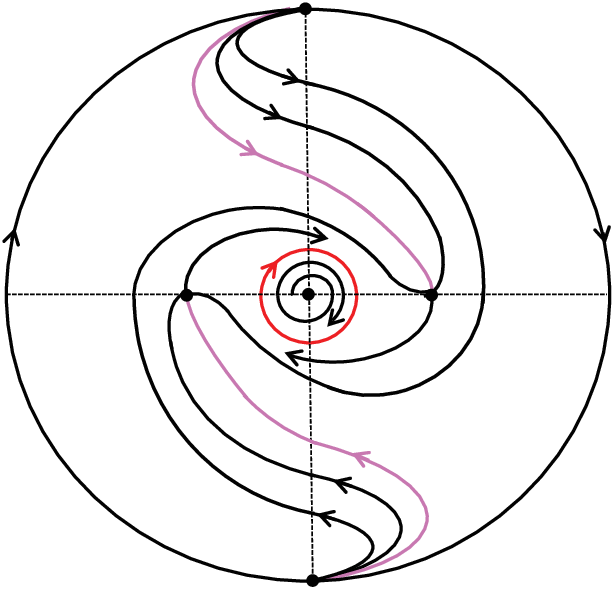}}}
							\subfigure[ in $T_{46i}$]{
			\scalebox{0.31}[0.31]{
								\includegraphics{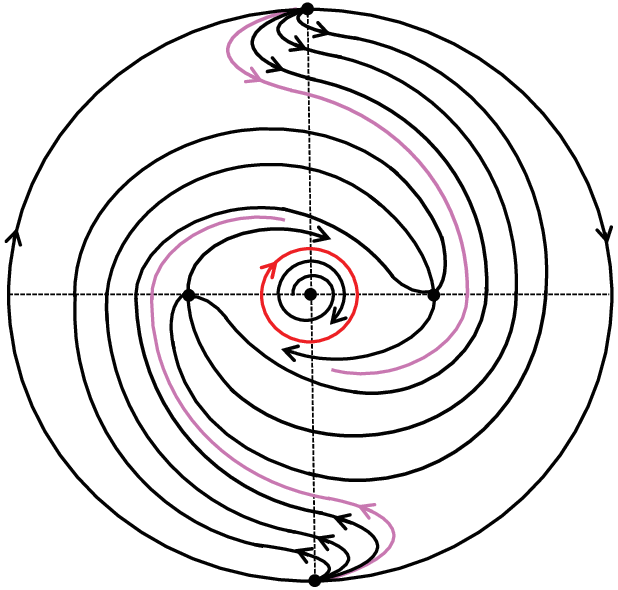}}}
				\subfigure[ in $T_{47i}$]{
			\scalebox{0.31}[0.31]{
					\includegraphics{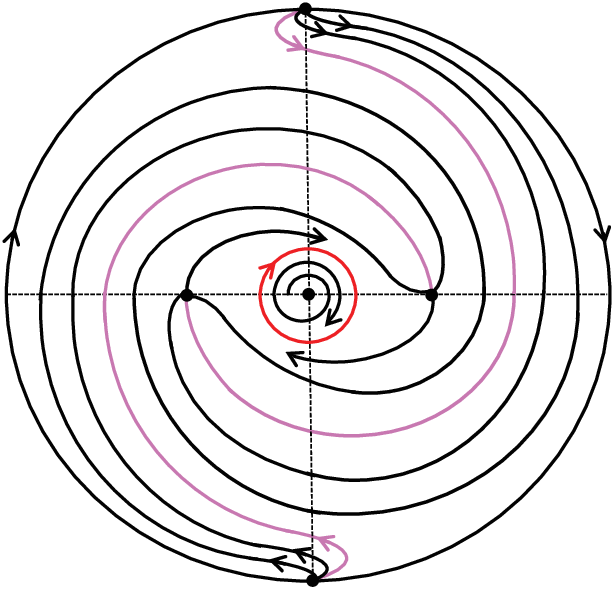}}}
		\subfigure[in $T_{48i}$]{
			\scalebox{0.31}[0.31]{
				\includegraphics{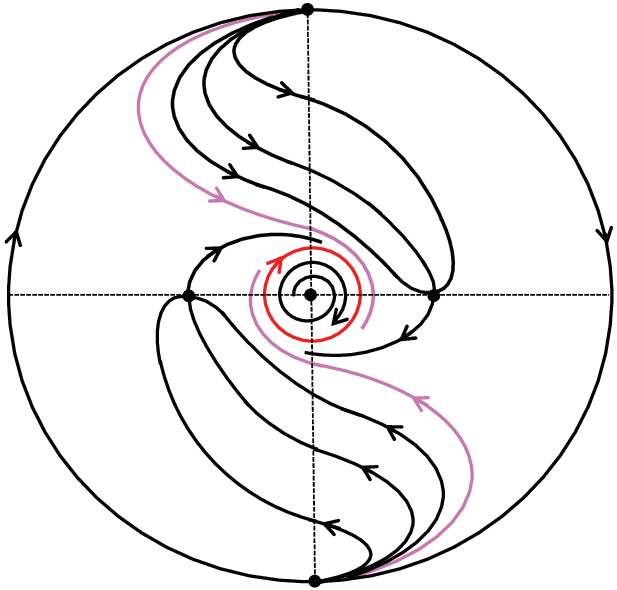}}}
			\subfigure[ in $T_{51i}$]{
			\scalebox{0.31}[0.31]{
				\includegraphics{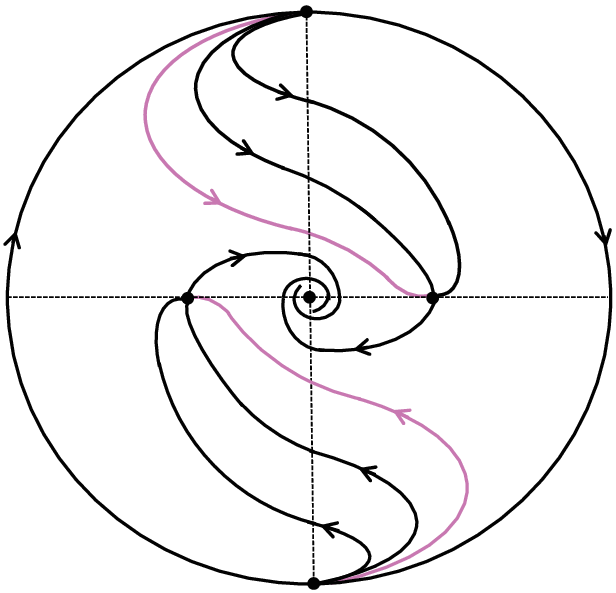}}}
					\subfigure[ in $T_{52i}$]{
			\scalebox{0.31}[0.31]{
						\includegraphics{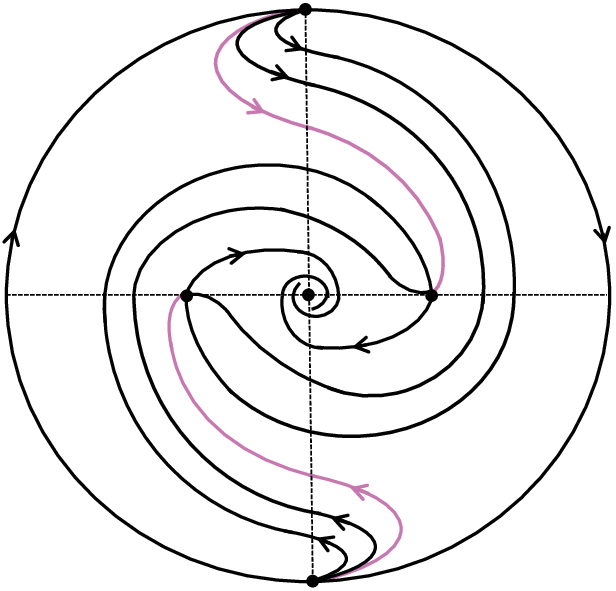}}}
						\subfigure[ in $T_{53i}$]{
			\scalebox{0.31}[0.31]{
							\includegraphics{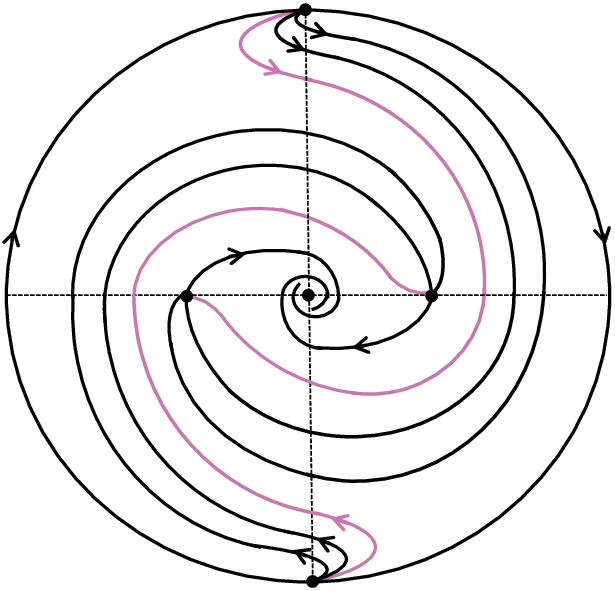}}}
\subfigure[  in  $T_{54i}$]{
			\scalebox{0.31}[0.31]{
				\includegraphics{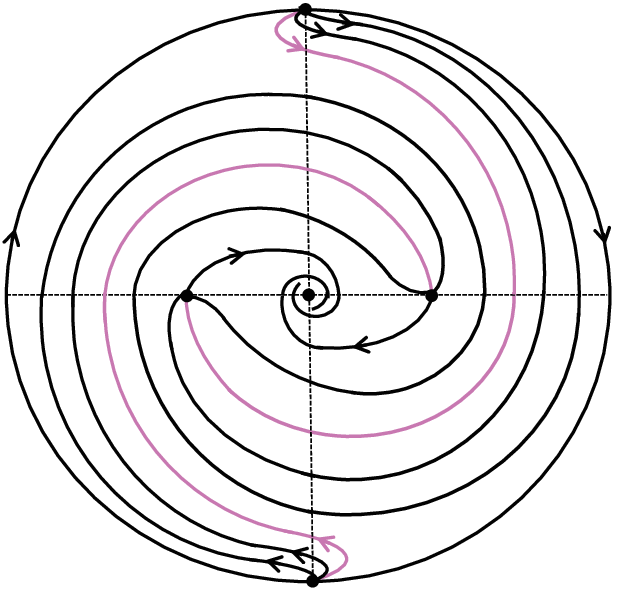}}}
	\caption{Global phase portraits    of    \eqref{initial1}  for $\delta_0\geq2\sqrt{3}$  and  $a_1=-1$.  }
	\label{gpp3}
\end{figure}

\addtocounter{figure}{-1}
\begin{figure}
	\centering
\subfigure[ in $T_{55i}$]{
			\scalebox{0.31}[0.31]{
					\includegraphics{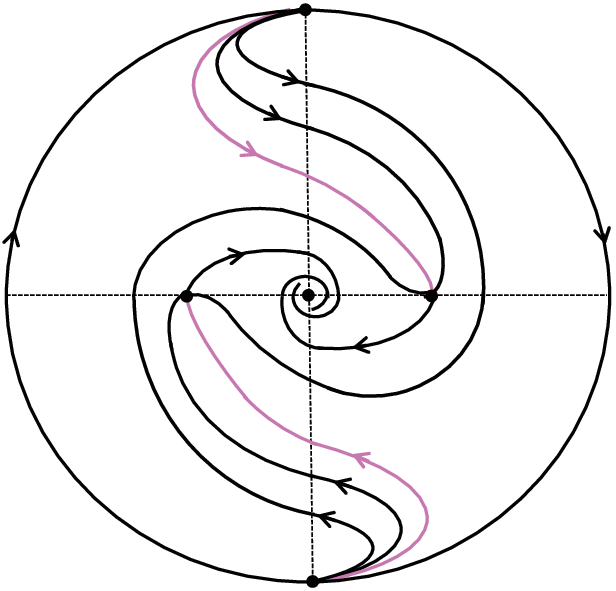}}}
							\subfigure[ in $T_{56i}$]{
			\scalebox{0.31}[0.31]{
								\includegraphics{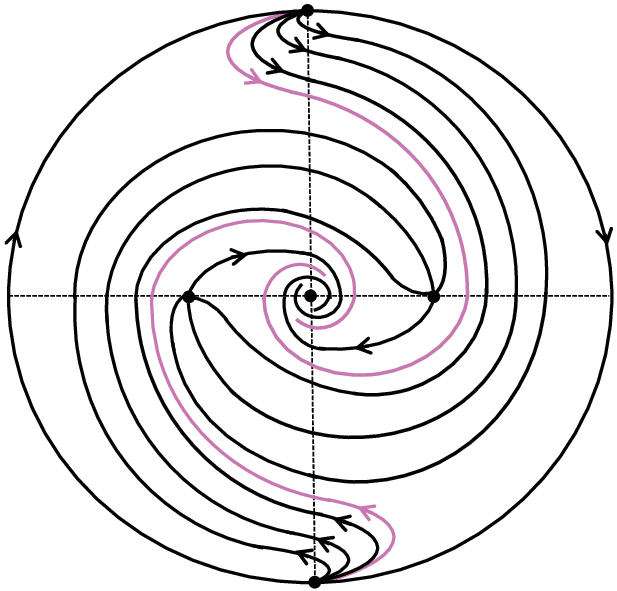}}}
		\subfigure[  in  $T_{57i}$]{
			\scalebox{0.31}[0.31]{
				\includegraphics{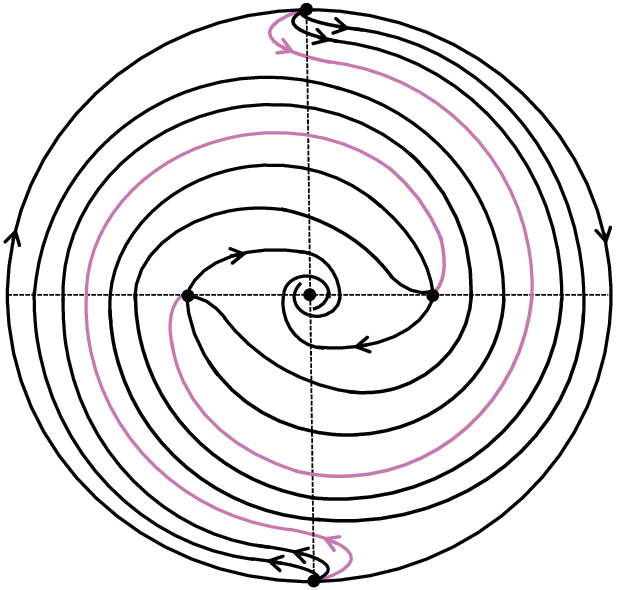}}}
		\subfigure[ in $T_{58i}$]{
			\scalebox{0.31}[0.31]{
			\includegraphics{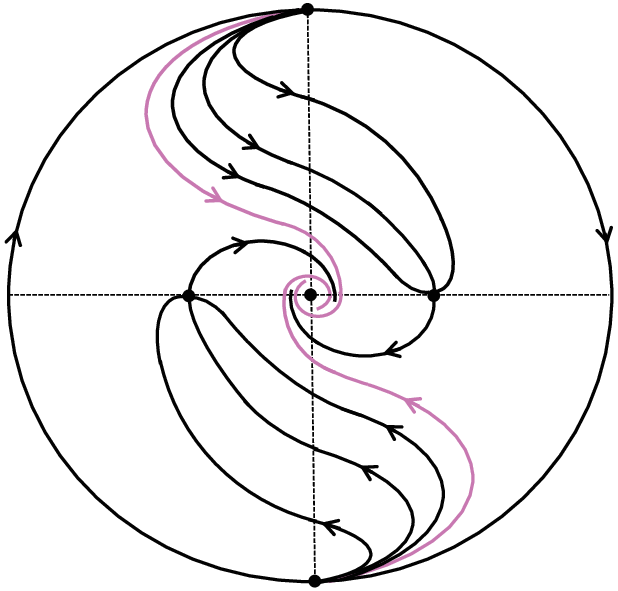}}}
	\caption{Continued.}
	\label{gpp71}
\end{figure}

\begin{figure}[h!]
	\centering
	\subfigure[ in $I$]{
			\scalebox{0.31}[0.31]{
			\includegraphics{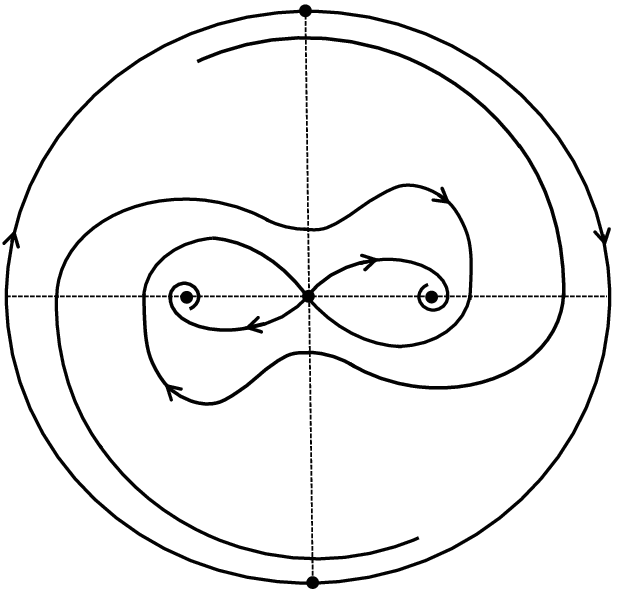}}}
		\subfigure[ in $DL_1$]{
			\scalebox{0.31}[0.31]{
				\includegraphics{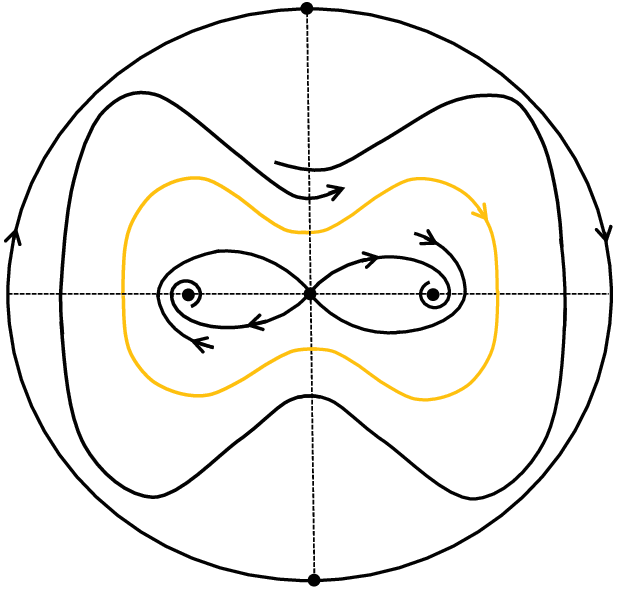}}}
	\subfigure[ in $II$ ]{
			\scalebox{0.31}[0.31]{
			\includegraphics{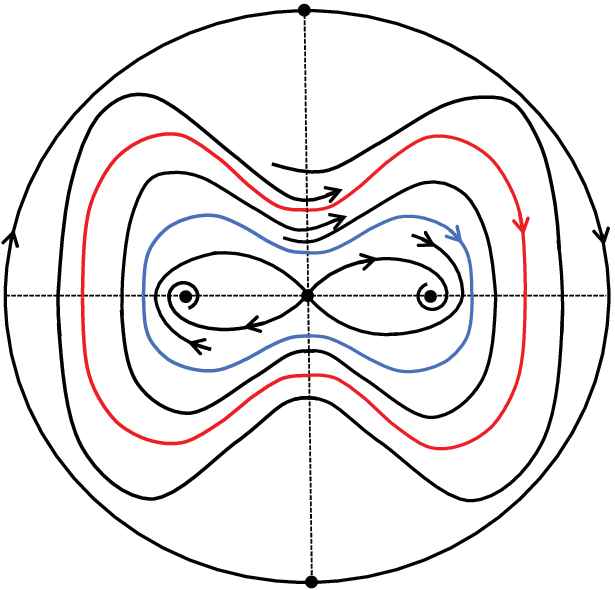}}}
					\subfigure[ in $HL_1$]{
			\scalebox{0.31}[0.31]{
					\includegraphics{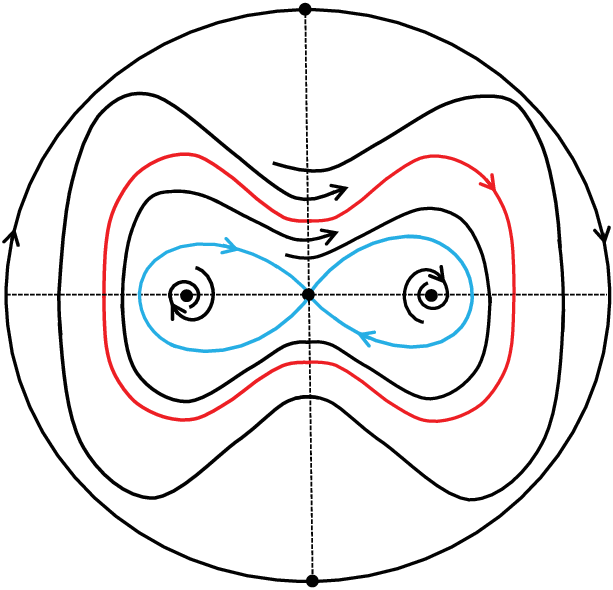}}}
			\subfigure[in $III$]{
			\scalebox{0.31}[0.31]{
				\includegraphics{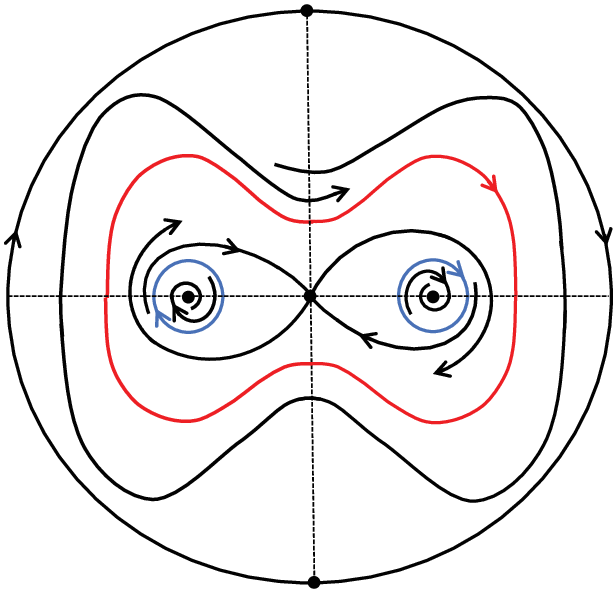}}}
		\subfigure[ in $IV$]{
			\scalebox{0.31}[0.31]{
				\includegraphics{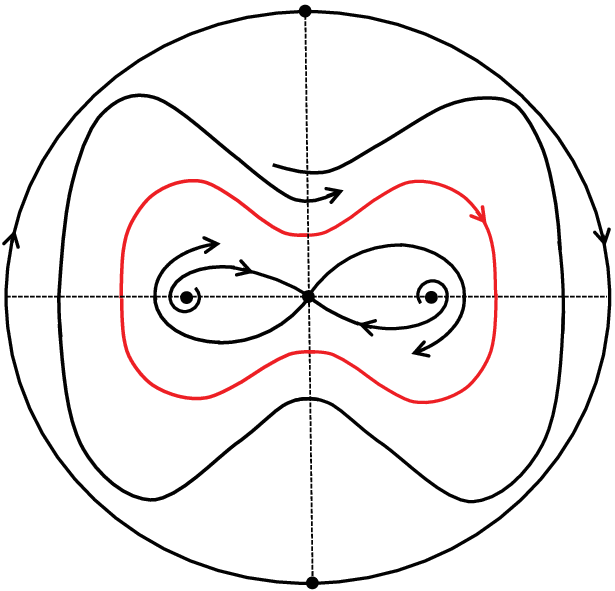}}}
	\caption{Global phase portraits   of    \eqref{initial1}  for $0< \delta_0<2\sqrt{3}$    and $a_1\geq0$.}
	\label{gpp4}
\end{figure}
\begin{figure}[h!]
	\centering
	\subfigure[ in $R_{11i}$]{
			\scalebox{0.31}[0.31]{
			\includegraphics{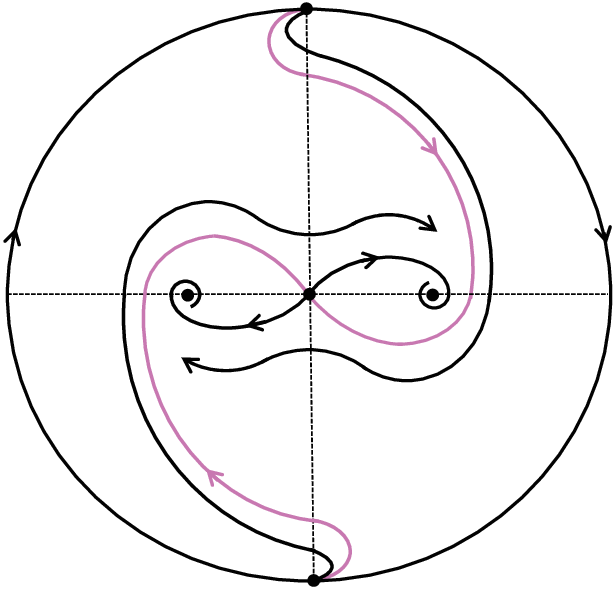}}}
	\subfigure[ in $R_{12i}$]{
			\scalebox{0.31}[0.31]{
			\includegraphics{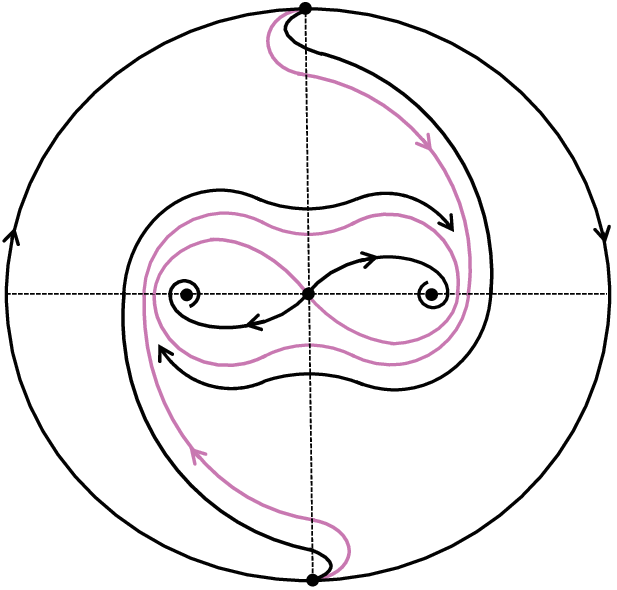}}}
	\subfigure[ in $R_{13i}$ ]{
			\scalebox{0.31}[0.31]{
			\includegraphics{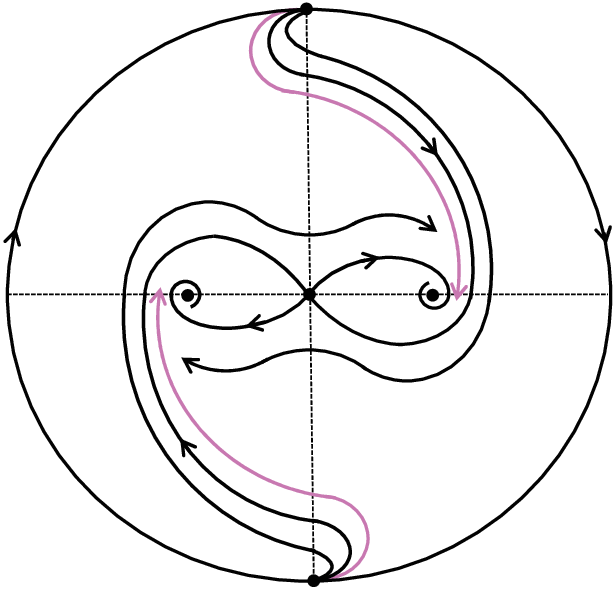}}}
	\subfigure[ in $R_{14i}$]{
			\scalebox{0.31}[0.31]{
			\includegraphics{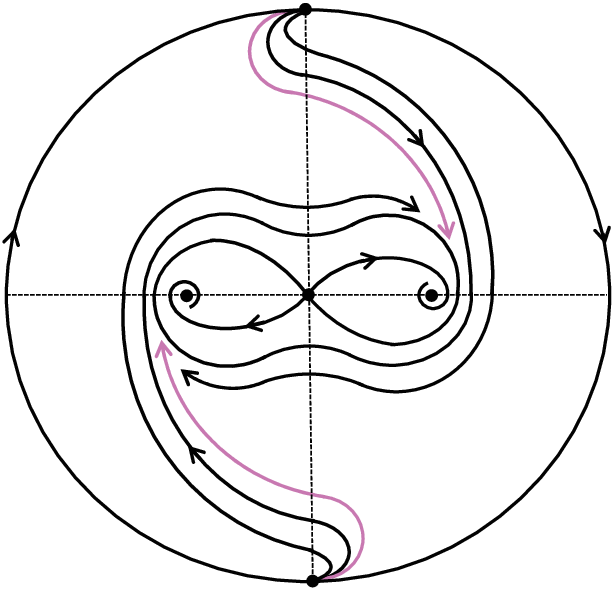}}}
	\subfigure[ in \textcolor{black}{$\widehat{DL_1}$}]{
			\scalebox{0.31}[0.31]{
			\includegraphics{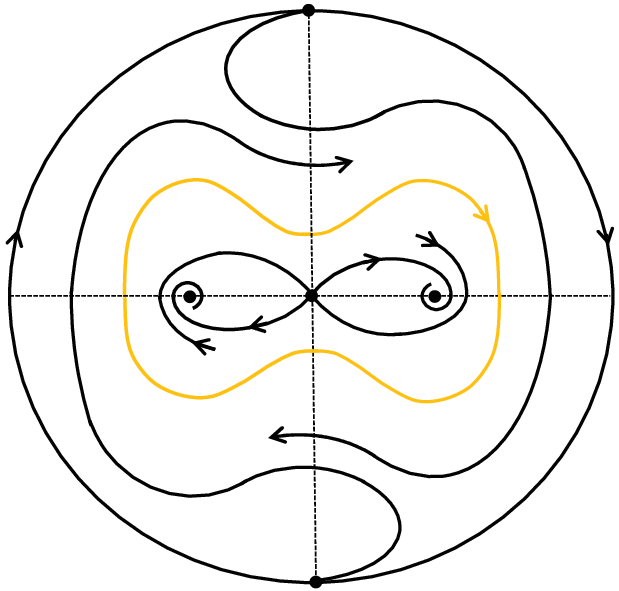}}}
	\subfigure[ in $R_2$]{
			\scalebox{0.31}[0.31]{
			\includegraphics{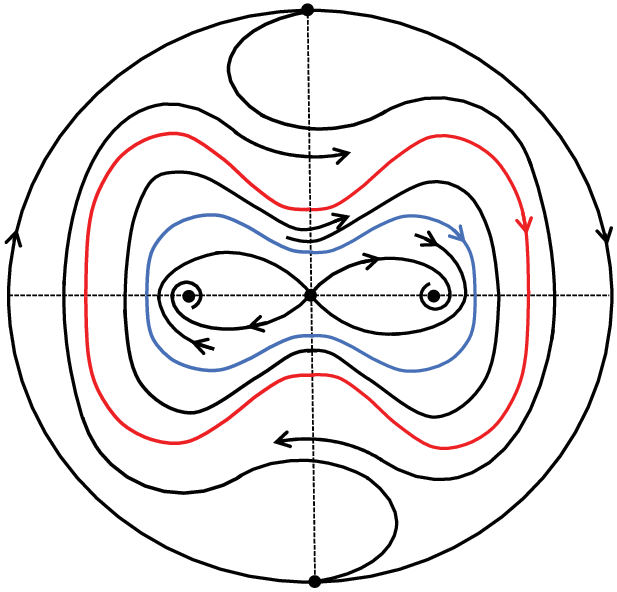}}}
	\subfigure[in \textcolor{black}{$HL_1$}]{
			\scalebox{0.31}[0.31]{
			\includegraphics{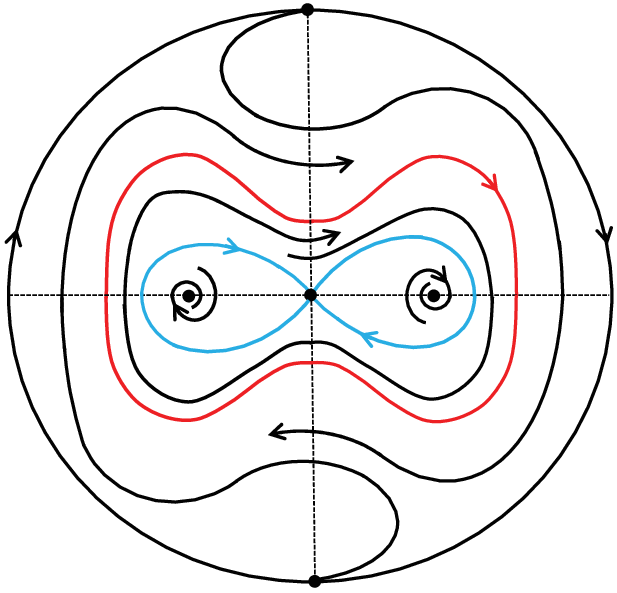}}}
	\subfigure[ in $R_3$]{
			\scalebox{0.31}[0.31]{
			\includegraphics{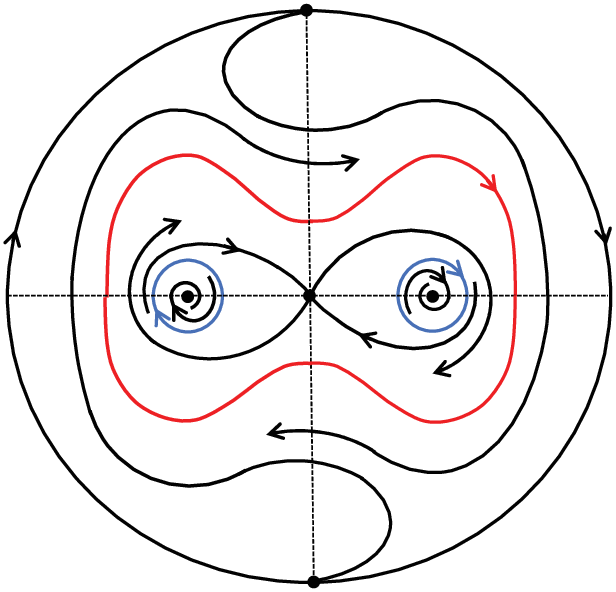}}}
	\subfigure[ in $R_4$]{
			\scalebox{0.31}[0.31]{
			\includegraphics{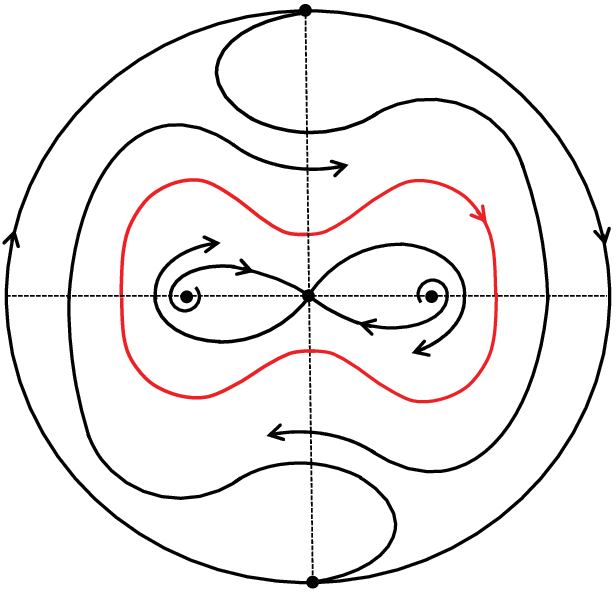}}}

	\caption{Global phase portraits    of    \eqref{initial1}  for $\delta_0\geq2\sqrt{3}$   and $a_1\geq0$.}
	\label{gpp5}
\end{figure}
\begin{figure}[h!]
	\centering
	\subfigure[ in $V$]{
			\scalebox{0.31}[0.31]{
			\includegraphics{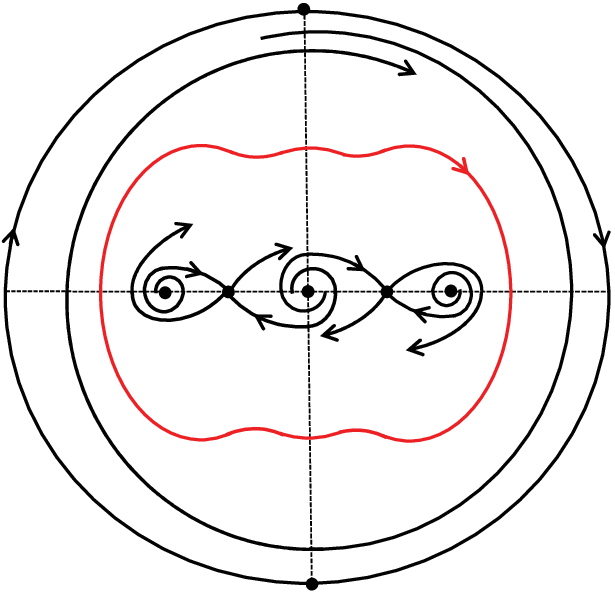}}}
	\subfigure[ in $VI$]{
			\scalebox{0.31}[0.31]{
			\includegraphics{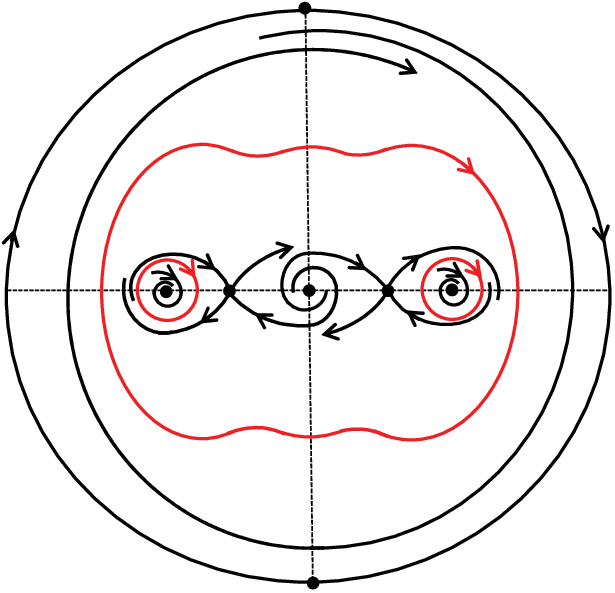}}}
	\subfigure[ in $HL_2$]{
			\scalebox{0.31}[0.31]{
			\includegraphics{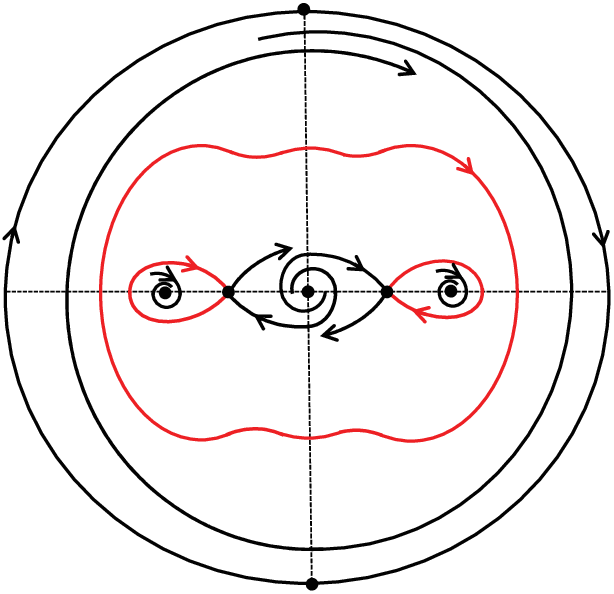}}}
			\subfigure[ in $VII$]{
			\scalebox{0.31}[0.31]{
				\includegraphics{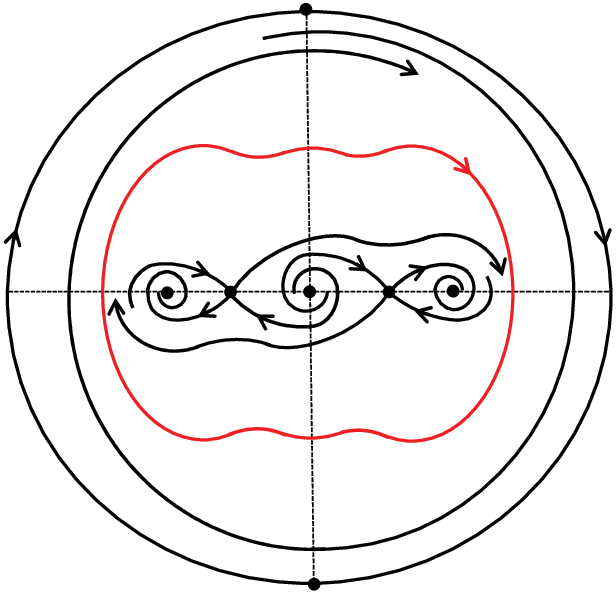}}}
		\subfigure[in $HE_{11}$]{
			\scalebox{0.31}[0.31]{
				\includegraphics{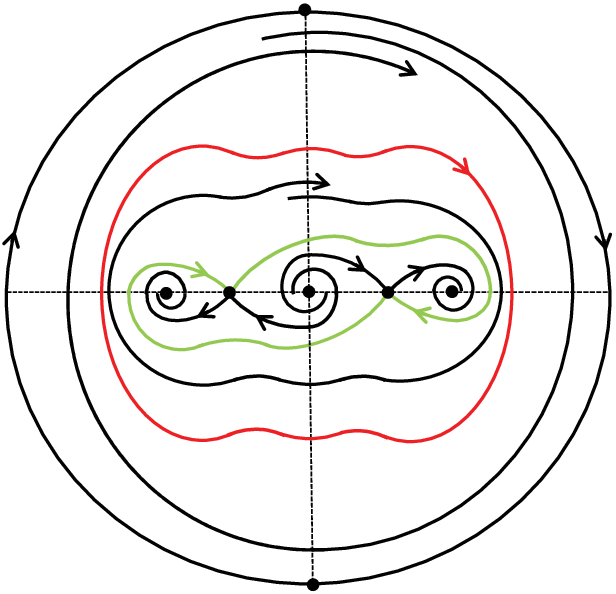}}}
		\subfigure[in $HE_{12}$]{
			\scalebox{0.31}[0.31]{
				\includegraphics{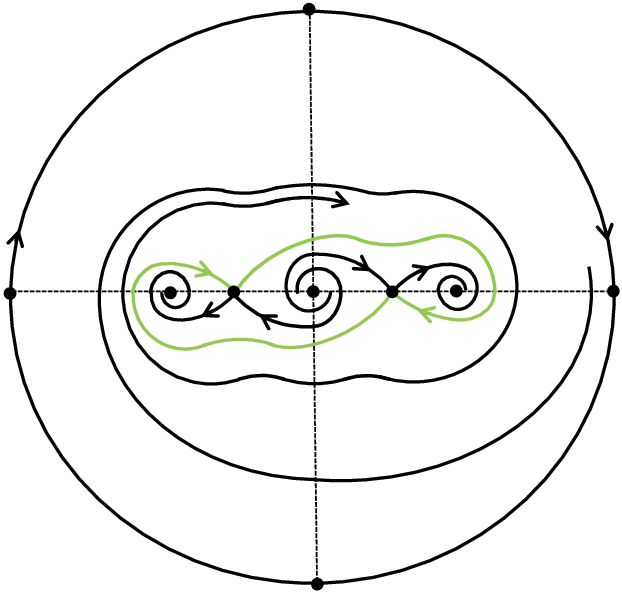}}}
\subfigure[ in $VIII$]{
			\scalebox{0.31}[0.31]{
					\includegraphics{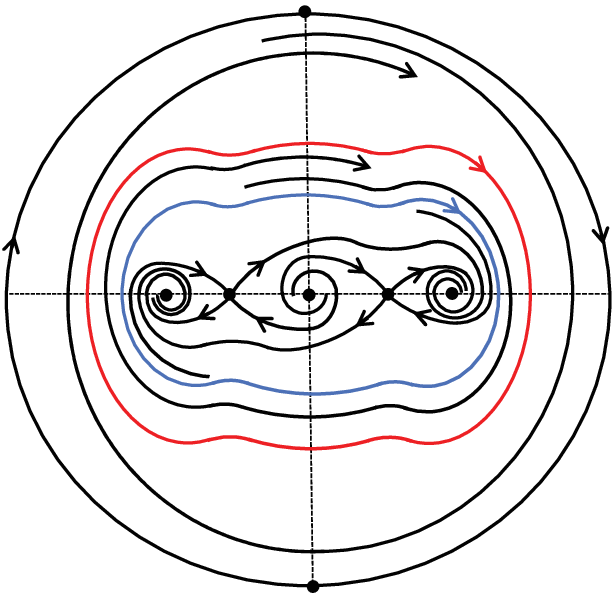}}}
			\subfigure[in $DL_2$]{
			\scalebox{0.31}[0.31]{
					\includegraphics{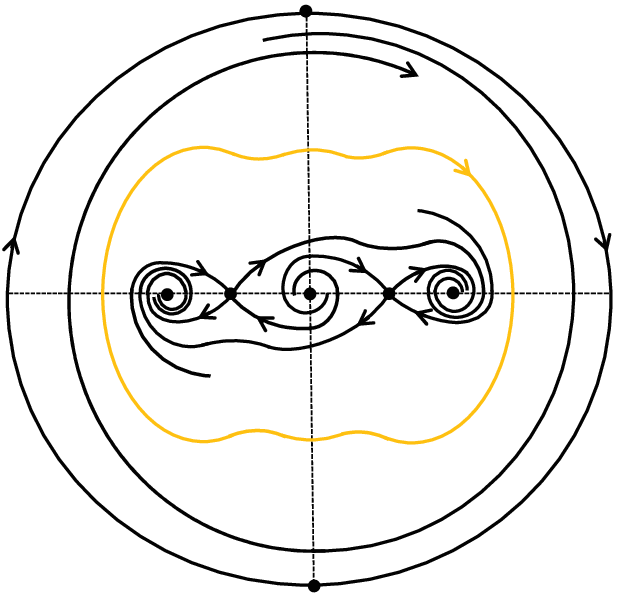}}}
\subfigure[ in $IX$]{
			\scalebox{0.31}[0.31]{
					\includegraphics{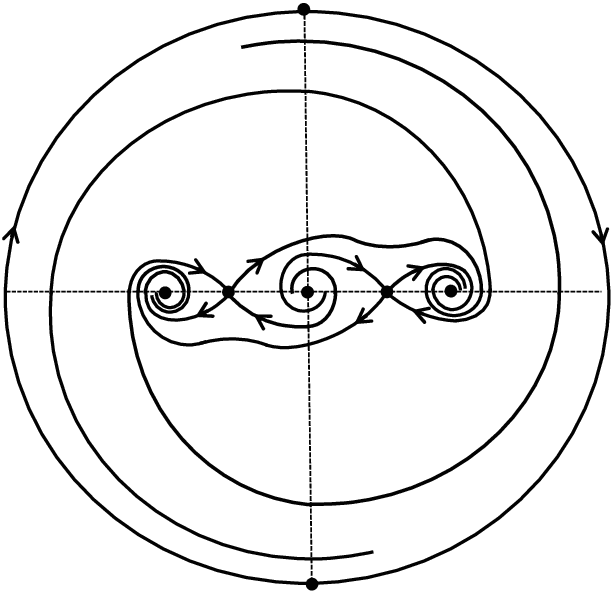}}}
					\subfigure[ in $HE_2$]{
			\scalebox{0.31}[0.31]{
						\includegraphics{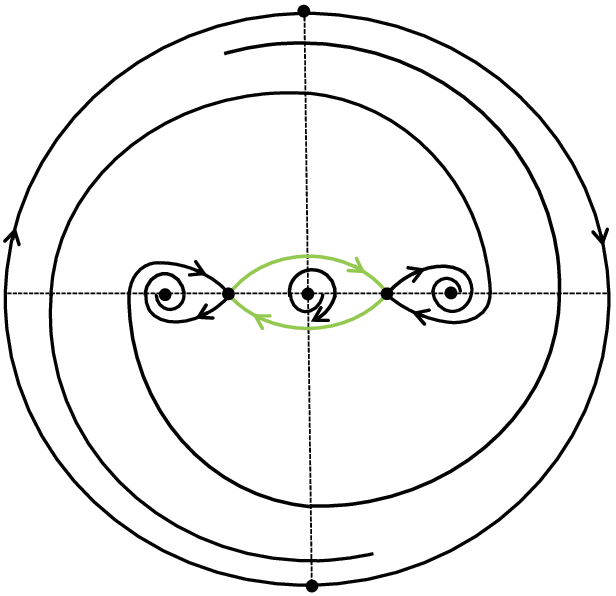}}}
				\subfigure[ in $X$]{
			\scalebox{0.31}[0.31]{
						\includegraphics{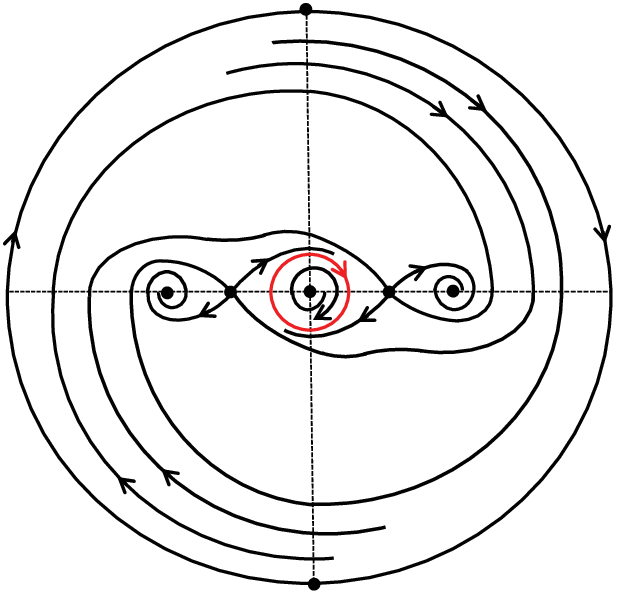}}}
									\subfigure[ in $XI$]{
			\scalebox{0.31}[0.31]{
							\includegraphics{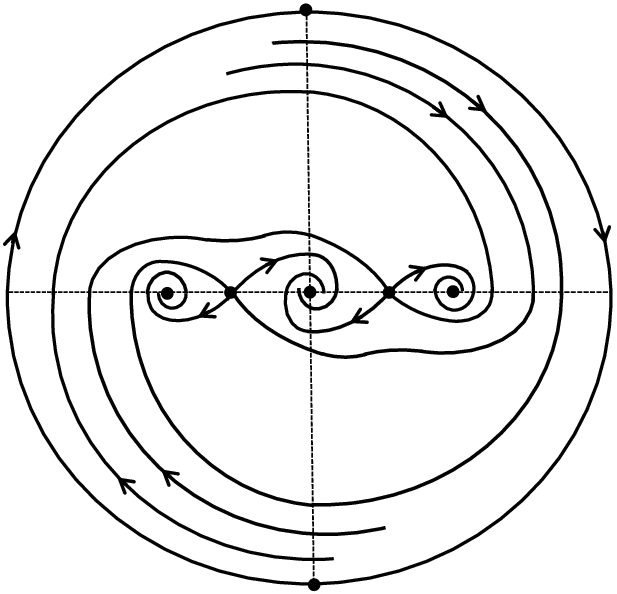}}}
	\caption{Global phase portraits    of    \eqref{initial1}  for $0<\delta_0<2\sqrt{3}$ and $-1<a_1<0$.}
	\label{gpp6}
\end{figure}
\begin{figure}
	\centering
	\subfigure[ in $R_5$]{
			\scalebox{0.31}[0.31]{
			\includegraphics{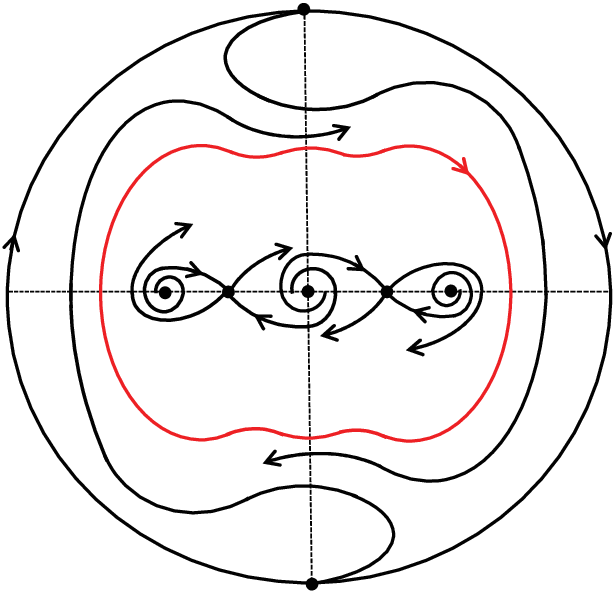}}}
	\subfigure[ in $R_6$]{
			\scalebox{0.31}[0.31]{
			\includegraphics{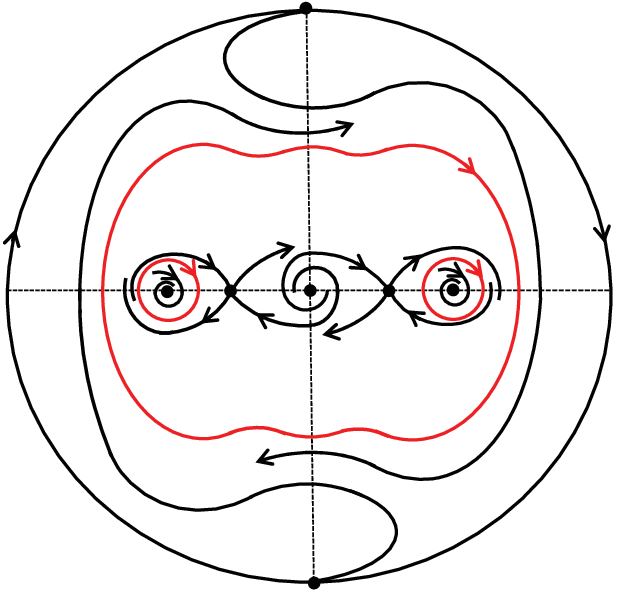}}}
	\subfigure[ in $HL_2$]{
			\scalebox{0.31}[0.31]{
			\includegraphics{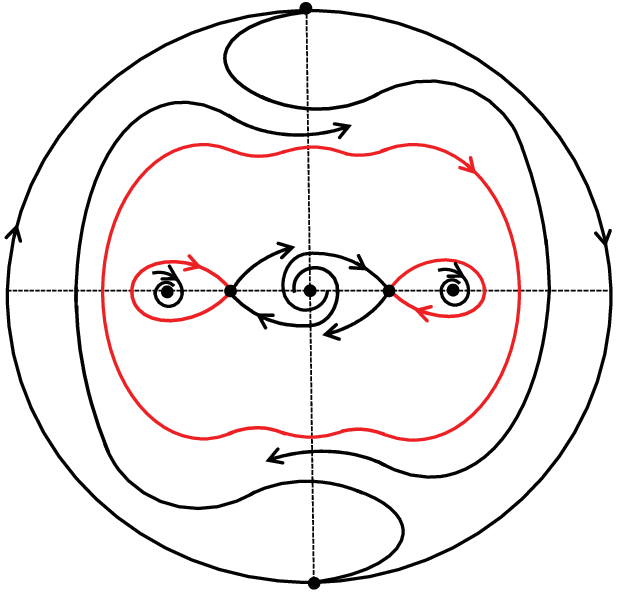}}}
	\subfigure[ in $R_7$]{
			\scalebox{0.31}[0.31]{
			\includegraphics{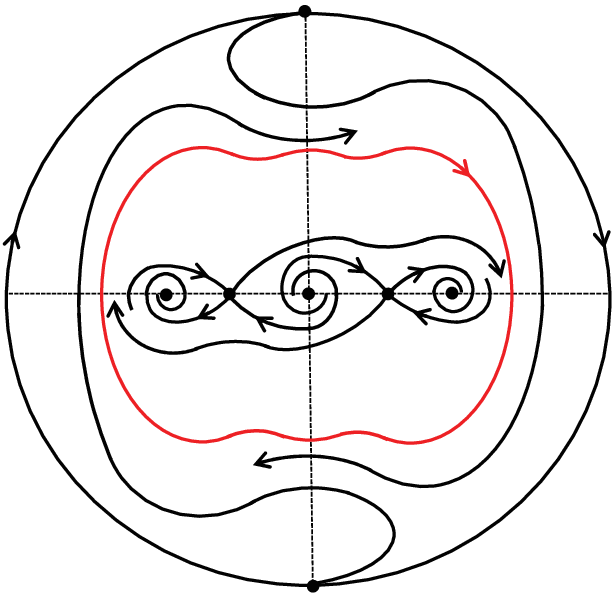}}}
	\subfigure[ in $\widehat {HE}_{11}$]{
			\scalebox{0.31}[0.31]{
			\includegraphics{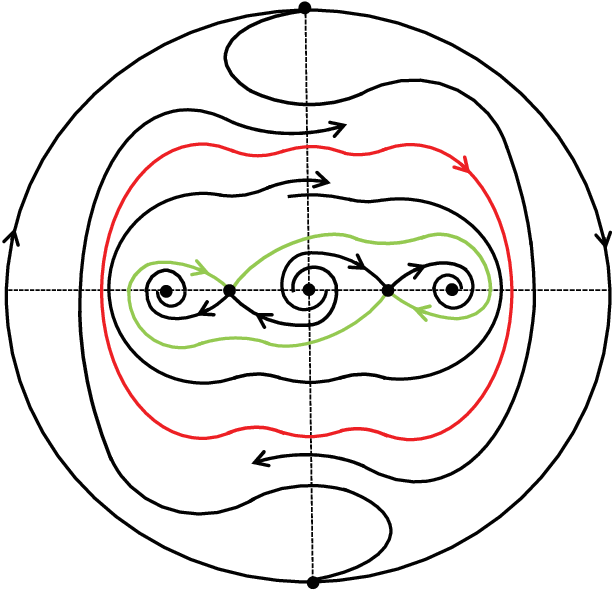}}}
	\subfigure[ in $\widehat {HE}_{12}$]{
			\scalebox{0.31}[0.31]{
			\includegraphics{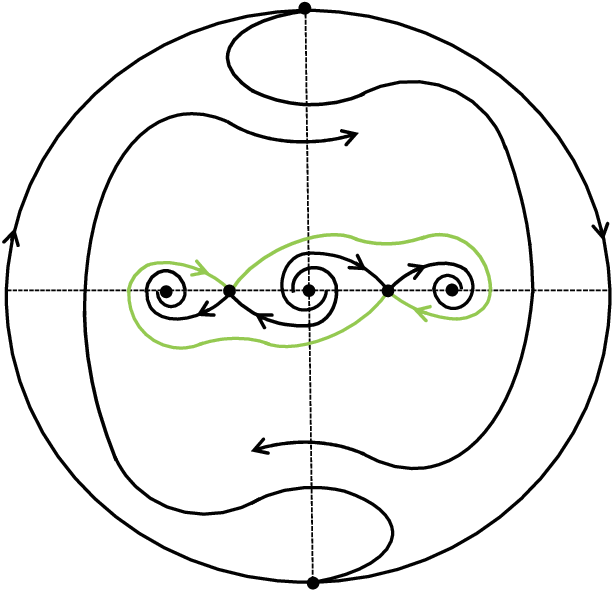}}}
	\subfigure[ in $R_8$]{
			\scalebox{0.31}[0.31]{
			\includegraphics{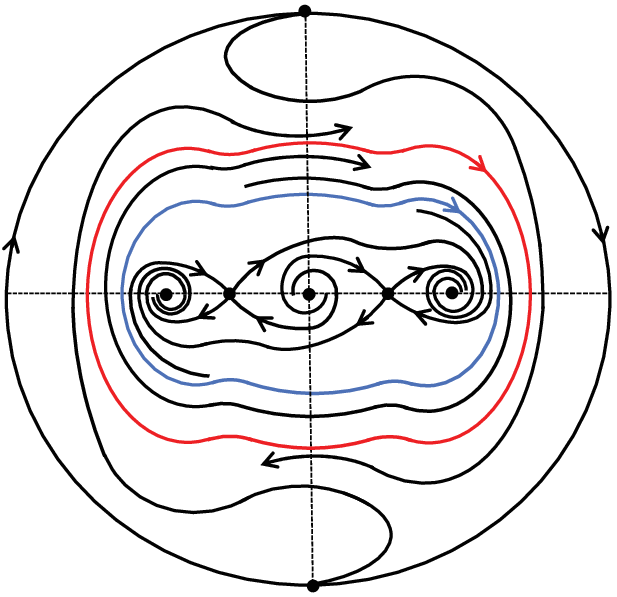}}}
	\subfigure[in $	\widehat{DL_2}$]{
			\scalebox{0.31}[0.31]{
			\includegraphics{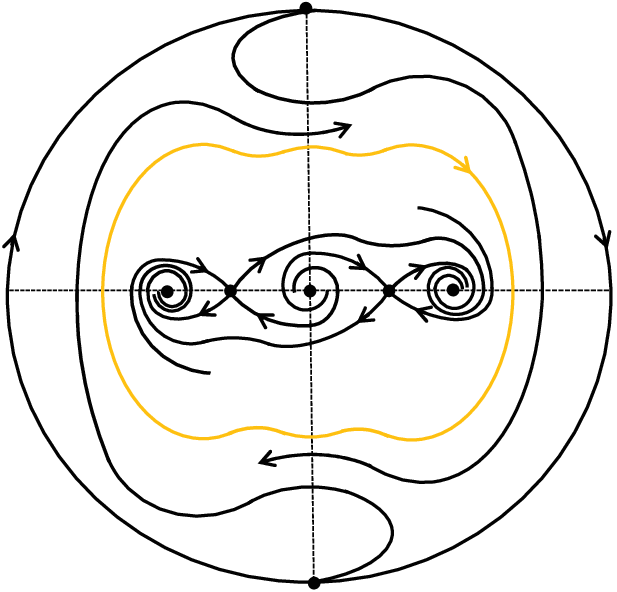}}}
	\subfigure[ in $R_{91i}$]{
			\scalebox{0.31}[0.31]{
			\includegraphics{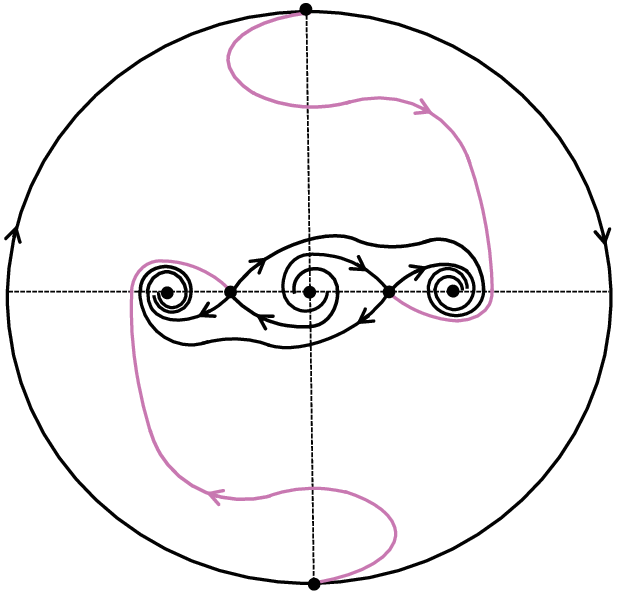}}}
\subfigure[ in $R_{92i}$]{
			\scalebox{0.31}[0.31]{
			\includegraphics{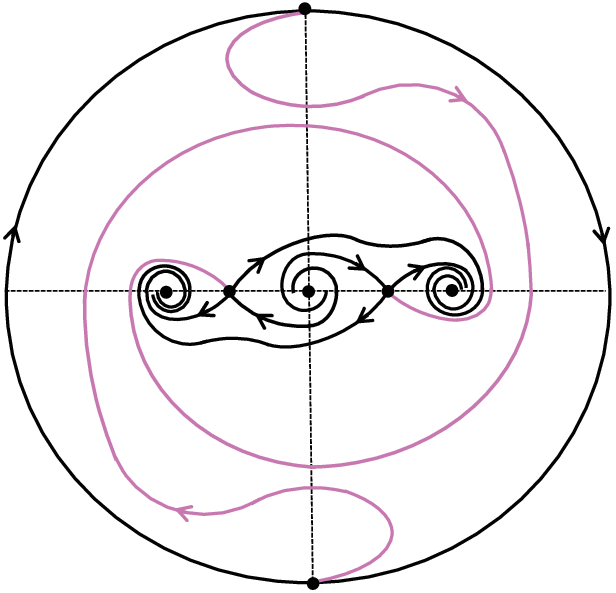}}}
			\subfigure[ in $R_{93i}$]{
			\scalebox{0.31}[0.31]{
				\includegraphics{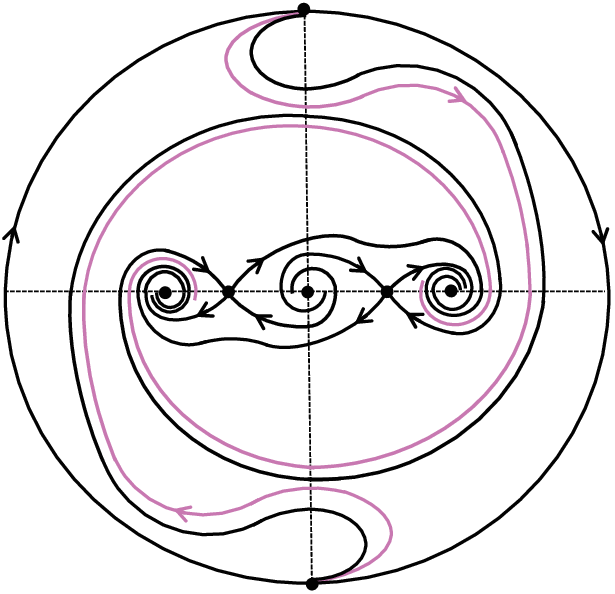}}}
	\subfigure[ in $R_{94i}$]{
			\scalebox{0.31}[0.31]{
			\includegraphics{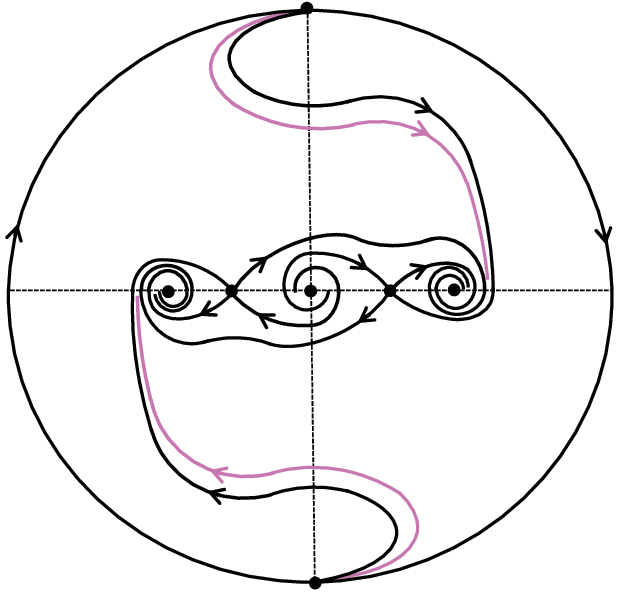}}}
		\subfigure[  in $HE_{21i}$]{
			\scalebox{0.31}[0.31]{
				\includegraphics{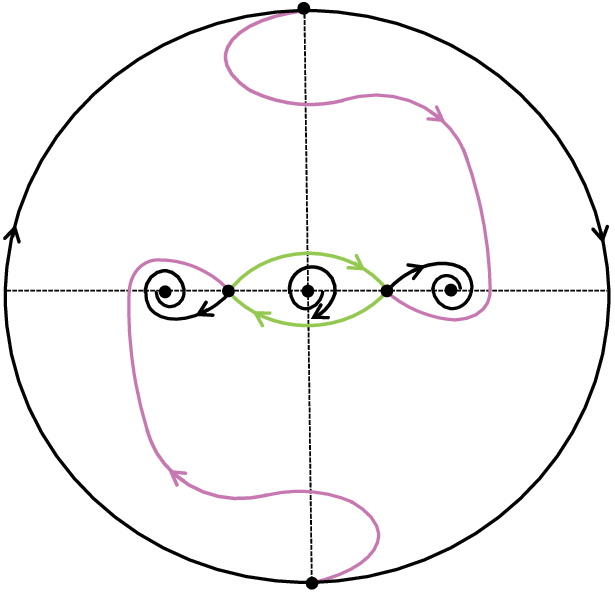}}}
\subfigure[ in $HE_{22i}$]{
			\scalebox{0.31}[0.31]{
				\includegraphics{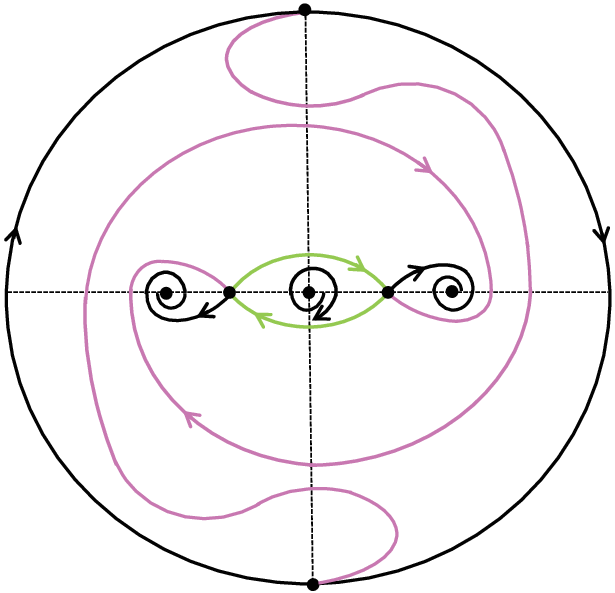}}}
		\subfigure[  in $HE_{23i}$]{
			\scalebox{0.31}[0.31]{
				\includegraphics{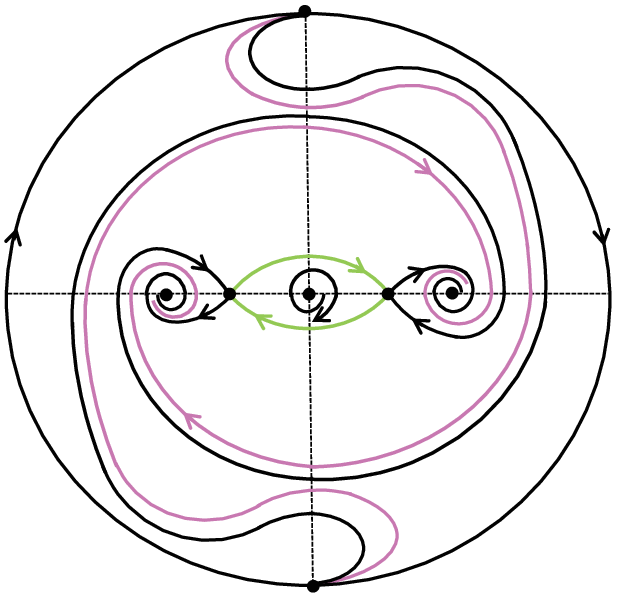}}}
		\subfigure[  in $HE_{24i}$]{
			\scalebox{0.31}[0.31]{
				\includegraphics{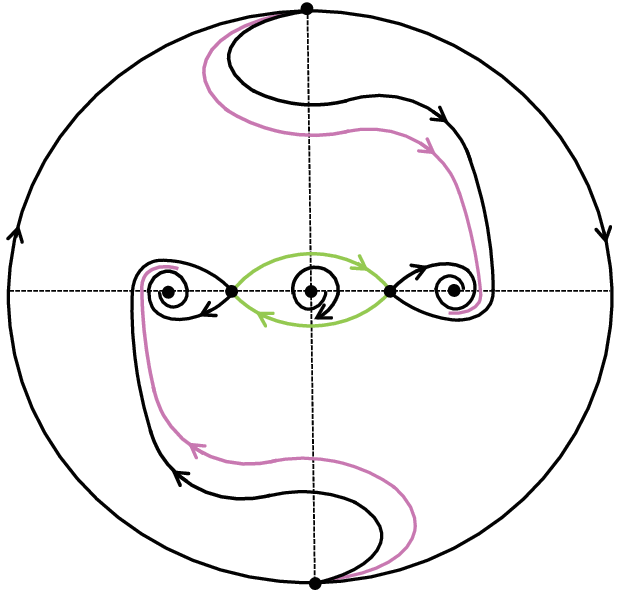}}}
		\subfigure[ in $R_{101i}$]{
			\scalebox{0.31}[0.31]{
				\includegraphics{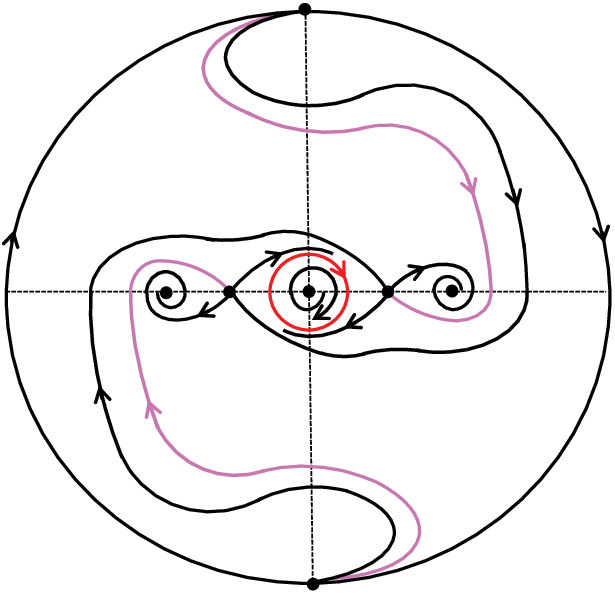}}}
		\subfigure[  in $R_{102i}$]{
			\scalebox{0.31}[0.31]{
				\includegraphics{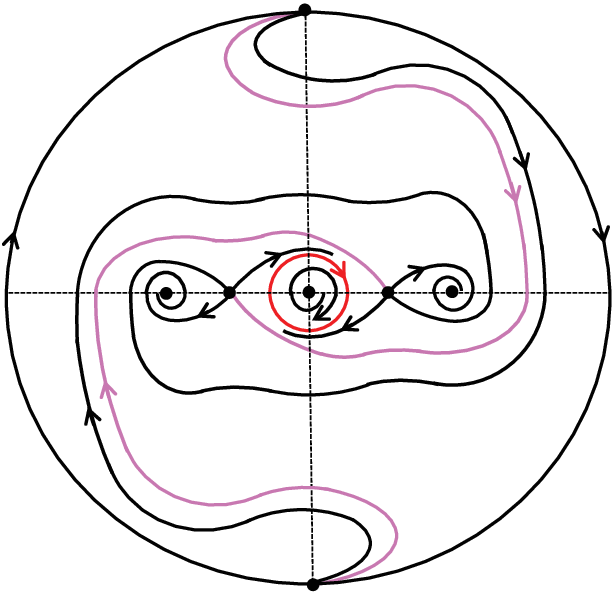}}}
\subfigure[  in $R_{103i}$]{
			\scalebox{0.31}[0.31]{
				\includegraphics{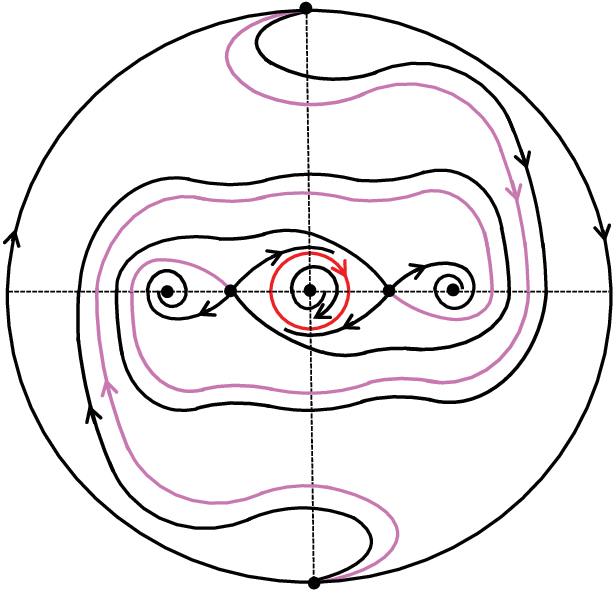}}}
		\subfigure[  in $R_{104i}$]{
			\scalebox{0.31}[0.31]{
				\includegraphics{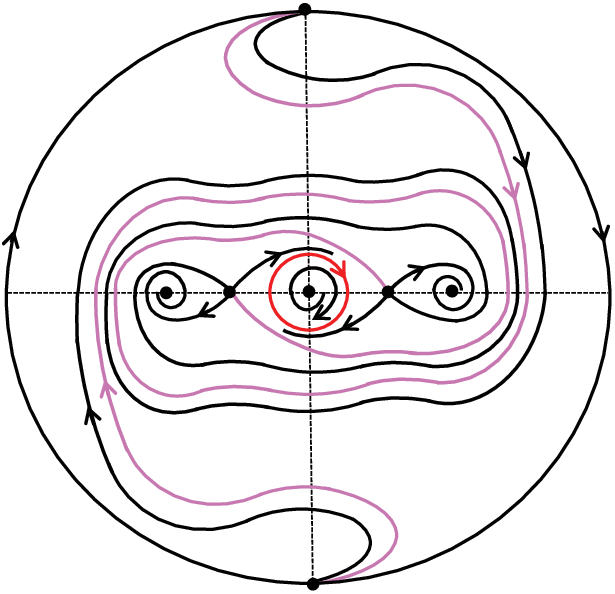}}}
\subfigure[  in $R_{105i}$]{
			\scalebox{0.31}[0.31]{
				\includegraphics{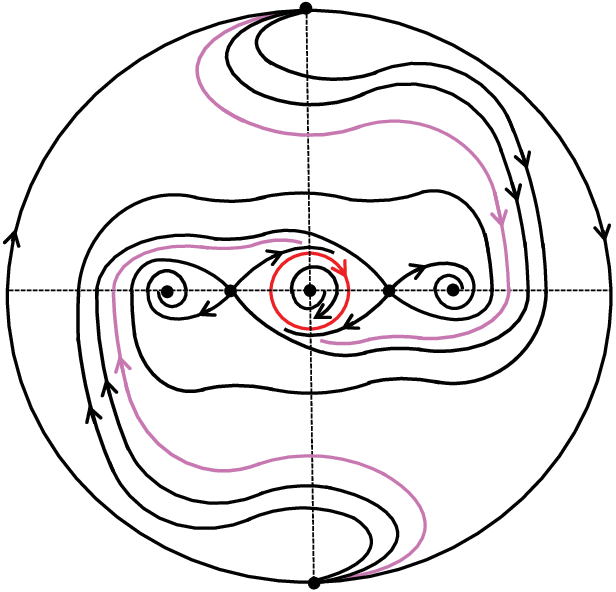}}}
		\subfigure[ in $R_{106i}$]{
			\scalebox{0.31}[0.31]{
					\includegraphics{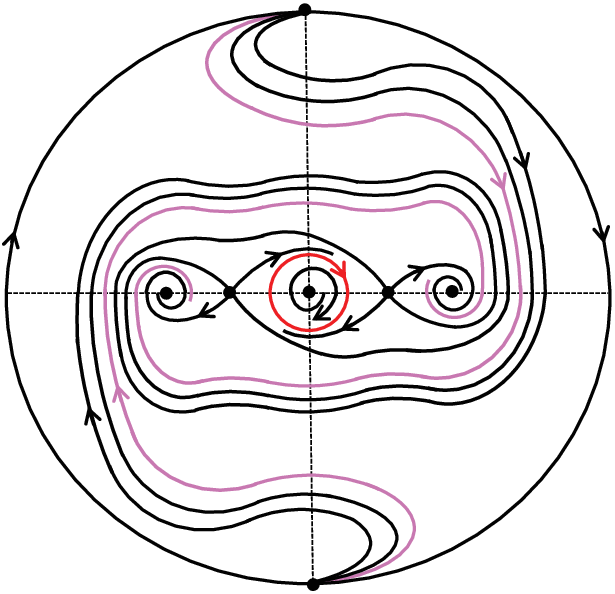}}}
			\subfigure[  in $R_{107i}$]{
			\scalebox{0.31}[0.31]{
			\includegraphics{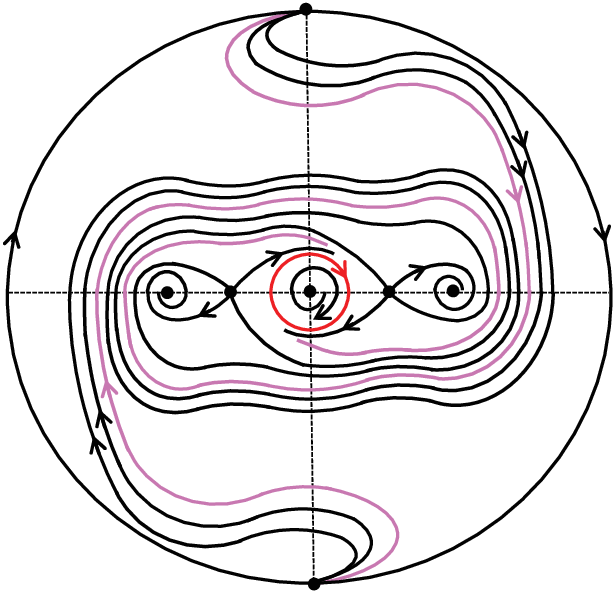}}}
		\subfigure[  in $R_{108i}$]{
			\scalebox{0.31}[0.31]{
				\includegraphics{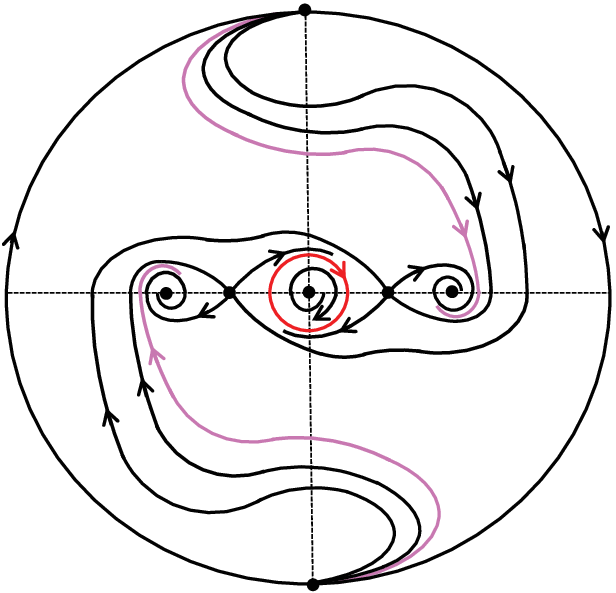}}}
	\caption{Global phase portraits   of    \eqref{initial1}  for $\delta_0\geq2\sqrt{3}$ and $-1<a_1<0$.}
	\label{gpp7}
\end{figure}
\addtocounter{figure}{-1}
\begin{figure}
	\centering
	\subfigure[  in $R_{111i}$]{
			\scalebox{0.31}[0.31]{
			\includegraphics{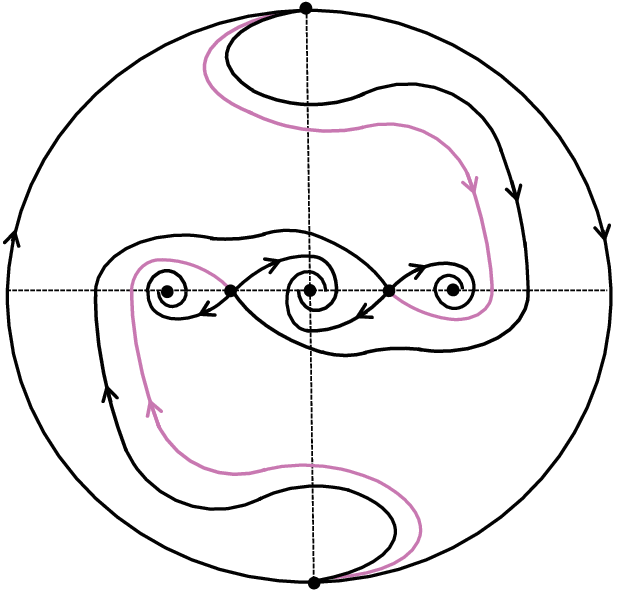}}}
	\subfigure[  in $R_{112i}$]{
			\scalebox{0.31}[0.31]{
			\includegraphics{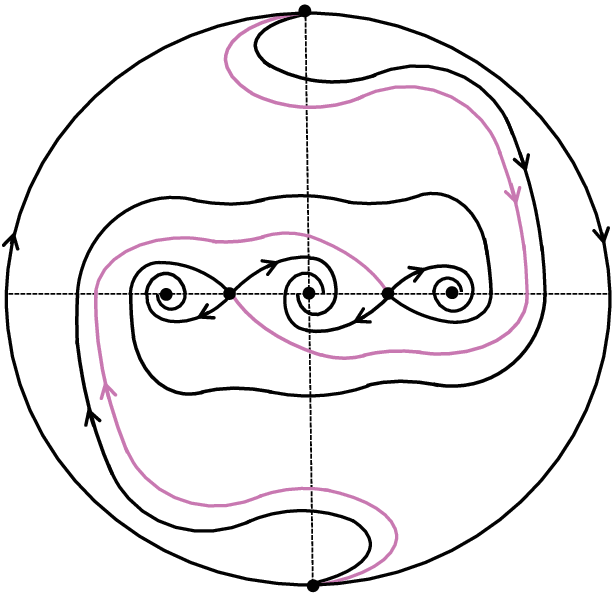}}}
	\subfigure[  in $R_{113i}$]{
			\scalebox{0.31}[0.31]{
			\includegraphics{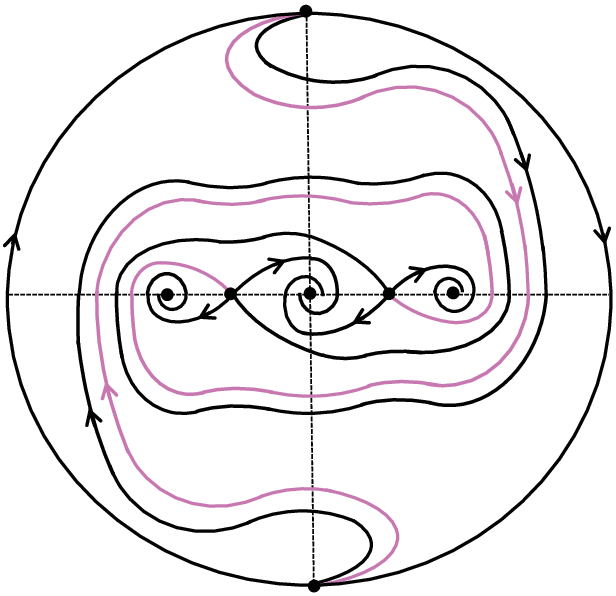}}}
	\subfigure[  in $R_{114i}$]{
			\scalebox{0.31}[0.31]{
			\includegraphics{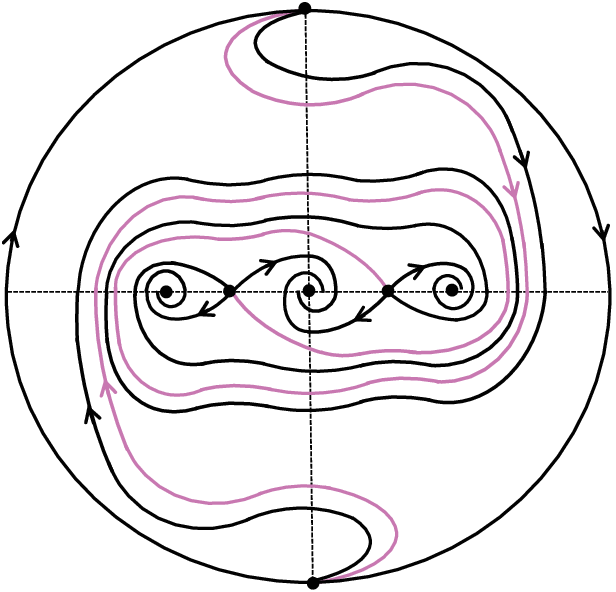}}}
	\subfigure[  in $R_{115i}$]{
			\scalebox{0.31}[0.31]{
			\includegraphics{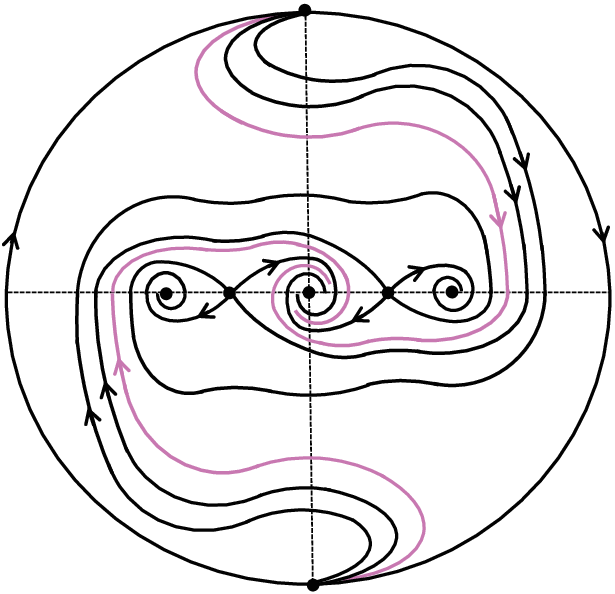}}}
		\subfigure[  in $R_{116i}$]{
			\scalebox{0.31}[0.31]{
				\includegraphics{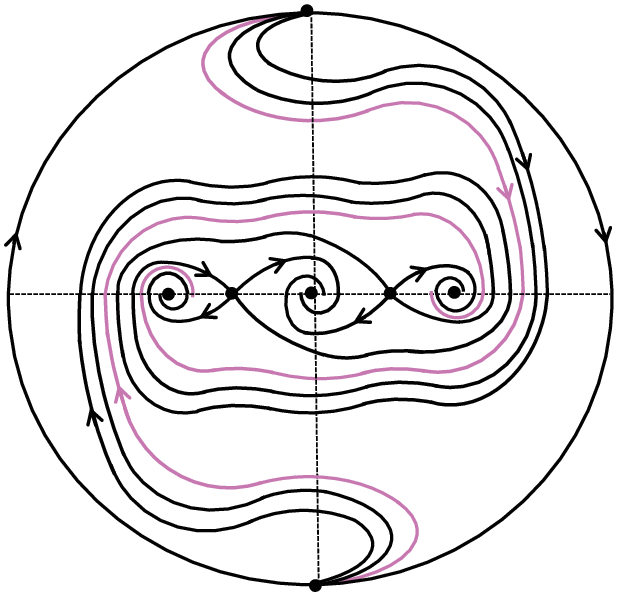}}}
	\subfigure[  in $R_{117i}$]{
			\scalebox{0.31}[0.31]{
			\includegraphics{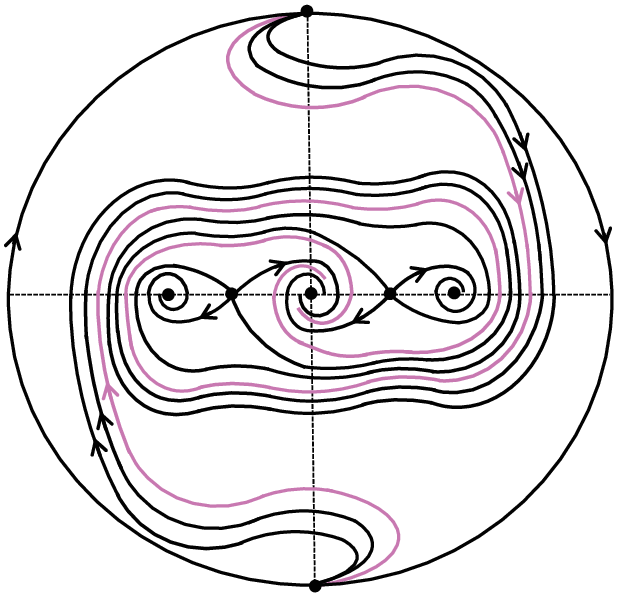}}}
					\subfigure[  in $R_{118i}$]{
			\scalebox{0.31}[0.31]{
			\includegraphics{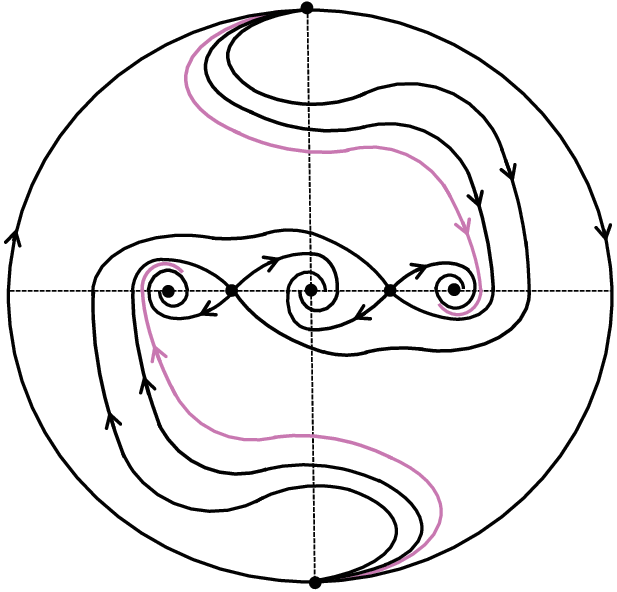}}}
	\caption{Continued. }
	\label{gpp71}
\end{figure}


\begin{theorem}Given  $\delta=\delta_0>0$,
 all global phase portraits   in the Poincar\'e disc  of  system  \eqref{initial1}
	are given in {\rm Figures \ref{gpp2}--\ref{gpp71},
}
	where the parameter regions for characterizing different global phase portraits are presented in Appendix A.
 Moreover,  system  \eqref{initial1} has at most four limit cycles.
	\label{mainresult2}
\end{theorem}

 In Theorem~\ref{mainresult2}, we get an upper bound $4$ of the number of limit cycles for system~\eqref{initial1}.
 However, by the numerical simulations we obtain a unique large limit cycle when there are two small limit cycles,
 and two large limit cycles when there is no small ones. Associated with symmetry of system~\eqref{initial1},
 we conjecture that the maximum number of limit cycles is exactly $3$.

The remainder of this paper is organized as follows.
We study the qualitative properties of equilibria and local bifurcation of system \eqref{initial} in Section 2.
Section 3 is devoted to the research of limit cycles, heteroclinic loops and homoclinic loops of system  \eqref{initial}.
The proofs of our main Theorems \ref{Result1}-\ref{mainresult2} are presented in Section 4 as well
as some numerical examples.


\section{Local bifurcation }

We firstly give the qualitative properties of equilibria of system \eqref{initial} and     please refer to Appendix B for a comprehensive proof.

\begin{lemma}
For any  $(\mu_1,\mu_2,\mu_3,b)\in\mathbb{R}^3\times\mathbb{R}^+$,
 equilibria of system \eqref{initial}  and their properties are
  given in {\rm Table~\ref{lmtable1}}, where $E_0:=(0, 0)$ and
\begin{eqnarray*}
\begin{aligned}
&E_{l2}:=\left(-\sqrt{\frac{-\mu_2+\sqrt{\mu_2^2-4\mu_1}}{2}}, ~0\right),&&~
~E_{l1}:=\left(-\sqrt{\frac{-\mu_2-\sqrt{\mu_2^2-4\mu_1}}{2}}, ~0\right),
\\
&E_{r1}:=\left(\sqrt{\frac{-\mu_2-\sqrt{\mu_2^2-4\mu_1}}{2}}, ~0\right),&&~
~E_{r2}:=\left(\sqrt{\frac{-\mu_2+\sqrt{\mu_2^2-4\mu_1}}{2}}, ~ 0\right).
\end{aligned}
\end{eqnarray*}
\label{fe1}
\end{lemma}

\begin{table}
	\centering \tiny \doublerulesep 0.5pt
	{\renewcommand\baselinestretch{1.8}\selectfont
		\begin{tabular}{||c|c|c|c|l||}
			\hline\hline
			\multicolumn{3}{|c|} {possibilities of $(\mu_1,\mu_2,\mu_3,b)$}& location of equilibria
			& ~~~~~~~~~~~types and stability
			\\
			\cline{1-5}
			& &$\mu_3>0$&  $E_0$ & $E_0$ sink
			\\
			\cline{3-5}
			& $\mu_2>-2\sqrt{\mu_1}$&$\mu_3=0$&  $E_0$ & $E_0$ stable   weak focus of order one
			\\
			\cline{3-5}
			& &$\mu_3<0$&  $E_0$ & $E_0$ source
			\\
			\cline{2-5}
			& & $\mu_3>0$&  $E_{l2}$, $E_0$,  $E_{r2}$ & $E_0$ sink;
			\\& & & &   $E_{l2}$,   $E_{r2}$ saddle-nodes with
			stable nodal sector
			\\
			\cline{3-5}
			&&$\mu_3=0$& $E_{l2}$, $E_0$,  $E_{r2}$ &  $E_0$ stable   weak focus of order one;
			\\& & & &   $E_{l2}$,   $E_{r2}$  saddle-nodes with
			stable nodal sector
			\\
			\cline{3-5}
			$\mu_1>0$& $\mu_2=-2\sqrt{\mu_1}$& $b\mu_2/2<\mu_3<0$&   $E_{l2}$, $E_0$,  $E_{r2}$  & $E_0$ source;  \\& & & &   $E_{l2}$,   $E_{r2}$   saddle-nodes with
			stable nodal sector
			\\
			\cline{3-5}
			&& $\mu_3=b\mu_2/2$&    $E_{l2}$, $E_0$,  $E_{r2}$  & $E_0$ source;     $E_{l2}$,  $E_{r2}$ cusps
			\\
			\cline{3-5}
			&& $\mu_3<b\mu_2/2$&   $E_{l2}$, $E_0$,  $E_{r2}$  & $E_0$ source;
			\\
			& & & &   $E_{l2}$,  $E_{r2}$ saddle-nodes with
			unstable nodal sector
			\\
			\cline{2-5}
			& & $\mu_3>0$&   $E_{l2}$,  $E_{l1}$, $E_0$, $E_{r1}$, $E_{r2}$  & $E_0$, $E_{l2}$, $E_{r2}$ sinks;    $E_{l1}$, $E_{r1}$ saddles
			\\
			\cline{3-5}
			&&$\mu_3=0$&     $E_{l2}$,  $E_{l1}$, $E_0$, $E_{r1}$, $E_{r2}$ & $E_0$ stable   weak focus of order one;
			\\
			&&&&     $E_{l2}$, $E_{r2}$ sinks;  $E_{l1}$, $E_{r1}$ saddles
			\\[3mm]
			\cline{3-5}
			&  $\mu_2<-2\sqrt{\mu_1}$& $\frac{b(\mu_2-\sqrt{\mu_2^2-4\mu_1})}{2}<\mu_3<0$&     $E_{l2}$,  $E_{l1}$, $E_0$, $E_{r1}$, $E_{r2}$  & $E_0$
			source;   $E_{l2}$, $E_{r2}$ sinks; $E_{l1}$, $E_{r1}$ saddles
			\\
			\cline{3-5}
			&& $\mu_3=\frac{b(\mu_2-\sqrt{\mu_2^2-4\mu_1})}{2}$&     $E_{l2}$,  $E_{l1}$, $E_0$, $E_{r1}$, $E_{r2}$ &   $E_0$ source;   $E_{l1}$, $E_{r1}$ saddles;
			\\
			&&&&
			$E_{l2}$,  $E_{r2}$   unstable weak foci of order one
			\\
			\cline{3-5}
			&& $\mu_3<\frac{b(\mu_2-\sqrt{\mu_2^2-4\mu_1})}{2}$&     $E_{l2}$,  $E_{l1}$, $E_0$, $E_{r1}$, $E_{r2}$  &  $E_0$, $E_{l2}$, $E_{r2}$ sources;  $E_{l1}$, $E_{r1}$ saddles
			\\
			\cline{1-5}
			&   &$\mu_3>0$&  $E_0$ & $E_0$ stable degenerate node
			\\
			\cline{3-5}
			&  $\mu_2>0$  &$\mu_3=0$&  $E_0$ & $E_0$   stable focus
			\\
			\cline{3-5}
			$\mu_1=0$&  &$\mu_3<0$&  $E_0$ & $E_0$ unstable degenerate node
			\\
			\cline{2-5}
			&   &$\mu_3>0$&  $E_0$ & $E_0$    stable degenerate node
			\\
			\cline{3-5}
			& $\mu_2=0$  &$\mu_3=0$&  $E_0$ &    stable focus  for  $b\in (0, 2\sqrt{3})$
			\\
			\cline{5-5}
			&&&&  degenerate   node for   $b\in [2\sqrt{3}, +\infty)$
			\\
			\cline{3-5}
			&   &$\mu_3<0$&  $E_0$ & $E_0$    unstable degenerate node
			\\
			\cline{2-5}
			&   &$\mu_3>0$&   $E_{l2}$, $E_0$,  $E_{r2}$  & $E_0$ degenerate  saddle;  $E_{l2}$, $E_{r2}$ sinks
			\\
			\cline{3-5}
			& $\mu_2<0$ &$\mu_3=0$&  $E_{l2}$, $E_0$,  $E_{r2}$   & $E_0$ degenerate saddle;
			\\&&&&
			$E_{l2}$, $E_{r2}$  unstable weak foci of order one
			\\
			\cline{3-5}
			&   &$\mu_3<0$&    $E_{l2}$, $E_0$,  $E_{r2}$   & $E_0$ degenerate saddle;    $E_{l2}$, $E_{r2}$   sources
			\\
			\cline{1-5}
			&&$\mu_3> \frac{b(\mu_2-\sqrt{\mu_2^2-4\mu_1})}{2}$&  $E_{l2}$, $E_0$,  $E_{r2}$  &
			$E_0$ saddle;   $E_{l2}$, $E_{r2}$ sinks
			\\
			\cline{3-5}
			$\mu_1<0$&$\mu_2\in\mathbb{R}$&&  $E_{l2}$, $E_0$,  $E_{r2}$  &$E_0$ saddle;
			\\
			&&
			$\mu_3=\frac{b(\mu_2-\sqrt{\mu_2^2-4\mu_1})}{2}$&&
			$E_{l2}$, $E_{r2}$   unstable weak foci of order one
			\\
			\cline{3-5}
			&&$\mu_3<\frac{b(\mu_2-\sqrt{\mu_2^2-4\mu_1})}{2}$&  $E_{l2}$, $E_0$,  $E_{r2}$  &
			$E_0$ saddle;  $E_{ l2}$, $E_{r2}$  sources
			\\
			\hline\hline
	\end{tabular}}
	\caption{\small Equilibria in finite planes of  \eqref{initial}.}
	\label{lmtable1}
\end{table}

Note that the qualitative properties of equilibria of system \eqref{initial1}   are  easily obtained by  the   scaling transformation  \eqref{sctra}.
We find the following local bifurcations in  system  \eqref{initial}.

\begin{proposition}
 For any  $(\mu_1,\mu_2,\mu_3,b)\in\mathbb{R}^3\times\mathbb{R}^+$, system  \eqref{initial}  includes the following bifurcation surfaces.
\begin{description}
\item[(i)] There are two pitchfork bifurcation surfaces $P_1$ and $P_2$, given by
 \[
 \mu_1=0 ~{\rm and}~\mu_2>0
 \]
 and
  \[
 \mu_1=0 ~{\rm and}~\mu_2<0,
 \]
 respectively.
For $\mu_1<0$, there are three equilibria $ E_{l2}, E_0$ and $E_{r2}$, $E_0$ is a saddle, while $E_{l2}$ and $E_{r2}$
are antisaddles.
For $0<\mu_1<\epsilon$ and $\mu_2>0$ ($\epsilon>0$ is small),  there is a unique equilibrium $E_0$ which is an antisaddle.
For $0<\mu_1<\epsilon$ and $\mu_2<0$, there are five equilibria $ E_{l1}, E_{l2}, E_0, E_{r1}$ and $E_{r2}$,
 $E_0$, $E_{l2}$ and $E_{r2}$ are antisaddles,  while $E_{l1}$ and $E_{r1}$
are saddles.

 \item[(ii)] There is a saddle-node bifurcation surface $SN$, given by
 \[
 \mu_2=-2\sqrt{\mu_1}>0.
 \]
  \item[(iii)] There is  a Hopf bifurcation surface $H_1$ about $E_0$, given by
 \[
 \mu_1>0~{\rm and}~\mu_3=0
 \]
 and
 a Hopf bifurcation surface $H_2$ about $E_{l2}$ and $E_{r2}$, given by
 \[
  \mu_1<{\mu_2^2}/{4} ~{\rm and}~\mu_3={b(\mu_2-\sqrt{\mu_2^2-4\mu_1})}/{2}.
 \]
 \item[(iv)] The intersection of $H_1$ with $P_1$ and $P_2$, respectively, defines with two
 $dBT_1$ and $dBT_2$ of degenerate Bogdanov-Takens bifurcation surfaces, given by
 \[
 \mu_1=\mu_3=0 ~{\rm and}~\mu_2>0
 \]
 and
  \[
 \mu_1=\mu_3=0 ~{\rm and}~\mu_2<0.
 \]
  \item[(v)] The intersection of $H_2$ with $SN$, respectively, defines with one
 $BT$ of  Bogdanov-Takens bifurcation surface, given by
 \[
 \mu_1={{\mu_2}^2}/{4} ~{\rm and}~\mu_2<0,
 \]
 and
 \[
 \mu_3={b\mu_2}/{2}.
 \]
\end{description}
\label{localbi}
\end{proposition}
\begin{proof}
{\bf (i)} It is obvious that the equilibria of system  \eqref{initial} are given by $y=0$
and $\mu_1x+\mu_2x^3+x^5=0$. Then the result can be easily proven.

{\bf (ii)} The equation  $\mu_1x+\mu_2x^3+x^5=0$ has the discriminant $\Delta:=\mu_2^2-4\mu_1$,
and then  the result follows.

{\bf (iii)} The necessary condition that a Hopf bifurcation  occurs at an antisaddle is
the divergence equaling zero.
The divergence at $E_0$ is
\[
{\rm div}(y, -\mu_1x-\mu_2x^3-x^5-\mu_3y-bx^2y)=
-\mu_3-bx^2=-\mu_3
\]
and it is zero for $\mu_3=0$.
Moreover, $E_0$ is an antisaddle for $\mu_1>0$.
Thus,  the Hopf bifurcation  $H_1$ is given. It follows from   Lemma \ref{fe1} that $E_0$  is a stable weak focus of order one when  $\mu_1>0$ and $\mu_3=0$.
Therefore, the Hopf bifurcation $H_1$  is of order one.
The divergence at both $E_{l2}$ and $E_{r2}$ is
\[
{\rm div}(y, -\mu_1x-\mu_2x^3-x^5-\mu_3y-bx^2y)=
-\mu_3+{b(\mu_2-\sqrt{\mu_2^2-4\mu_1})}/{2}
\]
and it is zero for $\mu_3= {b(\mu_2-\sqrt{\mu_2^2-4\mu_1})}/{2}$.
Moreover, both $E_{l2}$ and $E_{r2}$ are  antisaddles for $\mu_1< {\mu_2^2}/{4}$.
Thus, the Hopf bifurcation  $H_2$ is given.
By Lemma \ref{fe1}  again, we obtain that $E_{l2}$ and $E_{r2}$ are unstable weak foci of order one.
Thus, the Hopf bifurcation $H_2$  is of order one.

 {\bf (iv)} Degenerate Bogdanov-Takens bifurcation with symmetry occur
 when the divergence vanishes at a cubic point, that is,
 $\mu_1=\mu_3=0 $ and $\mu_2\neq0$ hold simultaneously.
 When $\mu_1=\mu_3=0 $ and $\mu_2<0$, by \cite[p.259]{CLW},
 the normal form of \eqref{initial} is
$\dot x=y,
\dot y=-\epsilon_1x-\epsilon_2y+x^3-x^2y,
$
where $\epsilon_1$ and $\epsilon_2$ are small real parameters.
 When $\mu_1=\mu_3=0 $ and $\mu_2>0$, by \cite[p.259]{CLW},
 the normal form of \eqref{initial} is
$\dot x=y,
\dot y=-\epsilon_1x-\epsilon_2y-x^3-x^2y,
$
where $\epsilon_1$ and $\epsilon_2$ are small real parameters.

 {\bf (v)}  Bogdanov-Takens bifurcation   occurs
when the divergence vanishes at a quadratic point, that is,
$\mu_1={\mu_2}^2/4, ~\mu_3=b\mu_2/2 $ and $\mu_2<0$ hold simultaneously.
Further,
the normal form of \eqref{initial} is
$\dot x=y,
\dot y=-\epsilon_1-\epsilon_2y+x^2-xy$
by \cite[p.259]{CLW},
where $\epsilon_1$ and $\epsilon_2$ are small real parameters.
\end{proof}

 In order to study the global dynamics of  system  \eqref{initial}, we perform a study of equilibria at infinity.

\begin{proposition}
For any $(\mu_1,\mu_2,\mu_3,b)\in\mathbb{R}^3\times\mathbb{R}^+$,
the dynamics of system \eqref{initial} near infinity  in the Poincar\'e disc  is as sketched
in {\rm Figure \ref{INF}}.
In particular,  the periodic orbit of  system  \eqref{initial}  at infinity  is   repulsive when $0<b<2\sqrt{3}  $.
\label{infty}
\end{proposition}

\begin{figure}[h!]
	\centering
	 \subfigure[ when $0<b<2\sqrt{3}  $  ]{
	\scalebox{0.31}[0.31]{
			\includegraphics{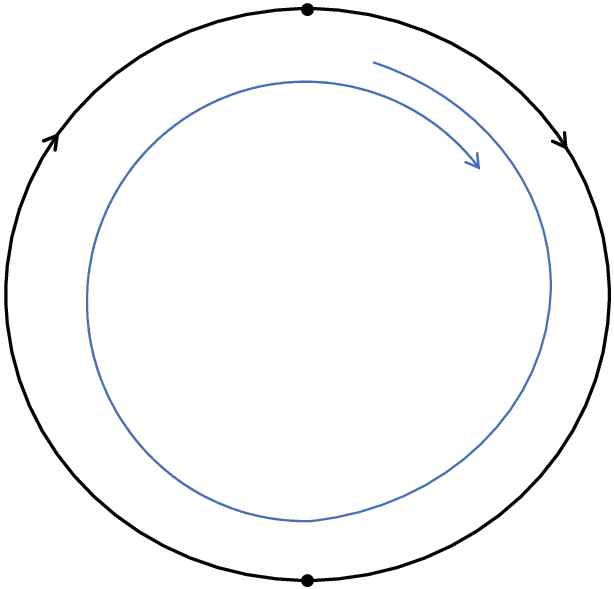}}}
	\subfigure[ when $b\geq2\sqrt{3}  $]{
	\scalebox{0.31}[0.31]{
			\includegraphics{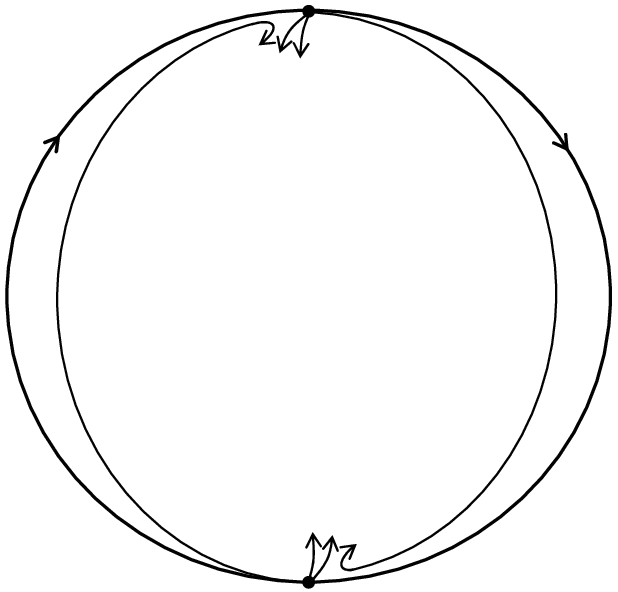}}}
	\caption{Dynamics near infinity  in the Poincar\'e disc  of    \eqref{initial}. }
	\label{INF}
\end{figure}

\begin{proof}
  With a scaling transformation
$(y,~t)\to(by,  t/b),
$
system \eqref{initial} becomes
\begin{eqnarray*}
\begin{array}{ll}
\dot x=y,
\\
\dot y=-\frac{  1}{b^2}(\mu_1x+\mu_2x^3+x^5)-(\frac{ \mu_3}{b} +    x^2)y.
\end{array}
\end{eqnarray*}
By \cite{DH}, we directly obtain the dynamics near infinity of  system  \eqref{initial}, as shown
in {\rm Figure \ref{INF}}.

Besides,  we claim that  the periodic orbit of  system  \eqref{initial}  at infinity  is   repulsive when $0<b<2\sqrt{3}  $.
Using  the transformation
$
(x, y)\rightarrow(x, y -F(x)),
$
system \eqref{initial} can be  written as
\begin{eqnarray}
\begin{array}{ll}
\dot x=y-\mu_3x -  \frac{b}{3}x^3=:y -F(x),
\\
\dot y=-\mu_1x- \mu_2x^3-x^5  =:-g(x).
\end{array}
\label{ini1101}
\end{eqnarray}
Set  a generalized Filippov transformation
$
z(x):=\int_{0}^{x}  h(s)ds.
$
Then, from system \eqref{ini1101}  we get
$
z(x)=
\mu_1 x^2/2- \mu_2 x^4/4+ x^6/6.
$
Denote  $x_1(z)$  and $x_2(z)$  as  the  branches of the inverse of $z(x)$  for $x\geq0$  and  $x<0$, respectively.
Transformation $z = z(x)$
changes  system \eqref{ini1101} into
\begin{eqnarray*}
\frac{dz}{dy}=\frac{{z}'(x)dx}{dy}=\frac{  h(x)dx}{dy} =  F_1(z)-y, ~~~~
\frac{dz}{dy}=\frac{{z}'(x)dx}{dy}=\frac{  h(x)dx}{dy} =F_2(z)-y
\end{eqnarray*}
for  $x\geq0$ and $x<0$ separately,
where
  $  F_1(z):=  F(x_1(z))$ and  $  F_2(z):= F(x_2(z))$.
 By the monotonicity,   it  is clear that  there exists a value $z^*$ such that $F_1(z)> F_2(z)$    for  every  $z\in(z^*,+\infty)$.
Although       system \eqref{initial}  has equilibria at infinity when   $0<b<2\sqrt{3}  $
  and    Proposition 3.3     of \cite{CJT} holds for  a general Li\'enard system    system  satisfying  that   there are no equilibria at infinity,  we can easily show that   Proposition 3.3    of \cite{CJT}   holds for  system \eqref{initial}.
 Therefore,
 the assertion is  proved
 by   Proposition 3.3     of \cite{CJT}.
 The proof is completed.
\end{proof}


\bigskip

\section{Limit cycles, homoclinic loops and heteroclinic loops}

In this section we study the existences of limit cycles,  homoclinic loops and heteroclinic loops of system \eqref{initial}. Moreover, the exact number of limit cycles is obtained
if they exist.
For simplicity, let {\it a large limit cycle} be a limit cycle surrounding more than one  equilibrium
and {\it a small limit cycle} be a limit cycle surrounding a single equilibrium.

\begin{lemma}
 For any  $(\mu_1,\mu_2,\mu_3,b)\in\mathbb{R}^3\times\mathbb{R}^+$,   system \eqref{initial} has no limit cycles  when $
\mu_3\geq0$.
\label{con1}
\end{lemma}
\begin{proof}
When $
\mu_3\geq0$, it is clear that
\[
{\rm div}(y,-\mu_1x-\mu_2x^3-x^5-\mu_3y-bx^2y)=-\mu_3-bx^2<0
\]
for $x\neq0$.
By the Bendixson-Dulac Criterion, system \eqref{initial} has no limit cycles.
\end{proof}

As follows, we study the number of limit cycles of system \eqref{initial} for $\mu_3<0$. For simplicity,
based on the number of equilibria of system \eqref{initial},
we give the following subsections.


\medskip

\subsection{System \eqref{initial} with only one equilibrium}

By Lemma \ref{fe1}, system \eqref{initial} has exactly one equilibrium if and only if
  $(\mu_1,\mu_2)\in\{(\mu_1,\mu_2)\in\mathbb{R}^2:  {\mu_2}^2-4\mu_1<0, ~\mu_2<0 \}\cup\{(\mu_1,\mu_2)\in\mathbb{R}^2: \mu_1\geq0, ~\mu_2\geq0\}=:\mathcal{G}_1$.

\begin{lemma}
When $(\mu_1,\mu_2)\in\mathcal{G}_1$,  $\mu_3<0$ and $b>0$, system \eqref{initial} has a unique limit cycle,
which is stable.
\label{con2}
\end{lemma}
\begin{proof}
With a Li\'enard transformation $(x,y)\to(x, y-F(x))$, system \eqref{initial} can be changed into
\begin{eqnarray}
\begin{array}{ll}
\dot x=y-F(x),
\\
\dot y= -g(x),
\end{array}
\label{Lie}
\end{eqnarray}
where $F(x):=\int_0^xf(s)ds=\mu_3 x+bx^3/3$.
And  $E_0$ of system \eqref{initial}  becomes $\overline{E}_0$  of  system \eqref{Lie}.
When $(\mu_1,\mu_2)\in\mathcal{G}_1$, system \eqref{Lie} has the following properties:
\begin{description}
\item[(i)] $g(x)$ is odd and $xg(x)>0$ for $x\neq0$,
 \item[(ii)] $F(x)$ is odd, $F(x)<0$ for $0<x<\sqrt{-3\mu_3/b}$ and $F(x)>0$ for $x>\sqrt{-3\mu_3/b}$,
  \item[(iii)] $F(+\infty)=\int_0^{+\infty}f(s)ds=+\infty$,
 \item[(iv)] $f$ and $g$ are $C^{\infty}$.
\end{description}
Therefore, all conditions of  \cite[section 4]{LS} or \cite[Theorem 4.1]{Zh} hold, implying
that system \eqref{Lie} or its equivalent system \eqref{initial} has a unique limit cycle,
which is stable.
\end{proof}


\medskip

\subsection{System \eqref{initial} with three equilibria}

By Lemma \ref{fe1}, system \eqref{initial} has exactly three equilibria if and only if either $a_1=-1$ or $a_1\ge 0$ for its equivalent system \eqref{initial1}.
  In the following, we only need to study limit cycles for simplified system \eqref{initial1}
and we  discuss in two subcases: $a_1=-1$ and $a_1 \ge 0$. Moreover,  limit cycles exist only if $a_2<0$ by Lemma \ref{con1}.



\subsubsection{The case: $a_1=-1$}
\begin{lemma}
When $a_1=-1$,   system \eqref{initial1} has a unique  generalized
limit cycle (including singular closed orbit) for $a_2<0$,
which is stable.
There are two continuous functions $p_1(\delta)$ and $p_2(\delta)$
satisfying $-1/3<p_2(\delta)<p_1(\delta)<0$ such that
\begin{description}
\item[(i)]  system \eqref{initial1} has  a unique  limit cycle that is small when $p_1(\delta)<a_2<0${\rm{;}}
 \item[(ii)] system \eqref{initial1} has
 one saddle-node loop when $p_2(\delta_1)\leq a_2\leq p_1(\delta)${\rm{;}}
  \item[(iii)] system \eqref{initial1} has   a unique  limit cycle that is large  when $a_2<p_2(\delta_1)${\rm{.}}
\end{description}
\label{con4}
\end{lemma}

\begin{proof}
By Lemma \ref{fe1}, we know that   $\hat{E}_0$ is an anti-saddle,  and  both $\hat{E}_{l2}$ and $\hat{E}_{r2}$ are saddle-nodes or cusps for system \eqref{initial1}.
Thus,  the index of $\hat{E}_0$ is $+1$ and  the   indices  of $\hat{E}_{l2}$ and $\hat{E}_{r2}$ are $0$   by \cite[Chapter 3]{Zh}.
Then, any limit cycle of  system \eqref{initial1} must   surround $\hat{E}_0$ if it exists.
By the symmetry of system \eqref{initial1} about the origin, $\hat{E}_{l2}$  must lie in the interior of a limit cycle if  $\hat{E}_{r2}$ lies in it.

With a Li\'enard transformation
$
(x, y)\to(x, y-\hat F(x)),
$
system  \eqref{initial1} is changed into
\begin{eqnarray}
\begin{array}{ll}
\dot x=y-\delta(a_2x+\frac{x^3}{3})=:y-\hat F(x),
\\ [2mm]
\dot y=-x(-1+x^2)(a_1+x^2).
\end{array}
\label{initial2}
\end{eqnarray}

We firstly prove that   the two points $(-\sqrt{-3a_2},0)$
 and $(\sqrt{-3a_2},0)$  lie in the interior region surrounded by the limit cycle of   system \eqref{initial2} if it exists for $a_2<0$.
Let
\begin{eqnarray}
\label{Exy}
  E(x,y):=\int_0^x\hat g(s)ds+\frac{y^2}{2}.
\end{eqnarray}
It is obvious that
\begin{eqnarray}
\frac{dE}{dt}\mid_{\eqref{initial2}}=-\hat g(x)\hat F(x).
\label{dEdt}
\end{eqnarray}
When $|x|\leq\sqrt{-3a_2}$, we can obtain
$
\hat g(x)\hat F(x) \leq 0.
$
 Assume that system \eqref{initial2} has a  limit cycle $\gamma$   in the strip  $x\in[-\sqrt{-3a_2},\sqrt{-3a_2}]$.
 Then, we have
 \[
0= \oint_{\gamma}dE= \oint_{\gamma}-\hat g(x)\hat F(x)dt>0.
 \]
 This is a contradiction.
Thus, there is no limit cycles  in the strip $x\in[-\sqrt{-3a_2},\sqrt{-3a_2}]$.
 In other words, if  system \eqref{initial2}  has a  limit cycle, it has to surround the two points $(-\sqrt{-3a_2},0)$ and $(\sqrt{-3a_2},0)$.

 \begin{figure}[h]
\centering
\includegraphics[width=3.5in]{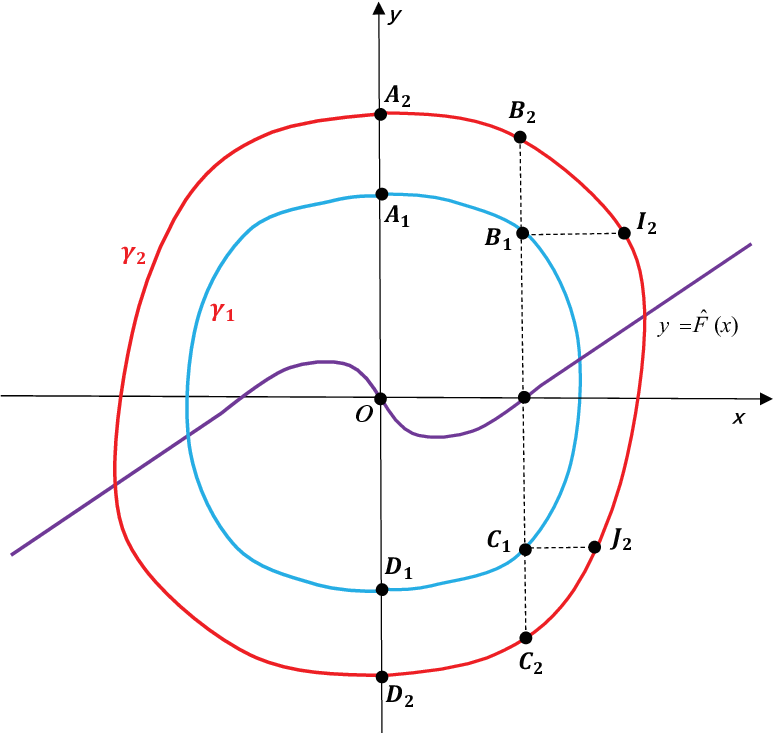}
\caption{Two  closed orbits  $\gamma_1$ and $\gamma_2$ of  \eqref{initial2}.}
\label{2lcs}
\end{figure}

We secondly show that   system \eqref{initial2} has at most one limit cycle.
 Assume that system \eqref{initial2} exhibits at least two closed orbits $\gamma_1$ and $\gamma_2$,
 where $\gamma_1$ lies in the interior region surrounded by $\gamma_2$, as shown in Figure \ref{2lcs}.
 Note that
 \begin{eqnarray}
\oint_{\gamma_1}dE=\oint_{\gamma_2}dE=0.
   \label{=2lc}
 \end{eqnarray}
 By the symmetry of  system \eqref{initial2} about the origin, it follows that
  \begin{eqnarray}
\oint_{\gamma_1}dE=\frac{1}{2}\int_{\widehat{A_1B_1D_1}}dE~{\rm and}~\oint_{\gamma_2}dE=\frac{1}{2}\int_{\widehat{A_2B_2D_2}}dE.
   \label{=2lc1}
 \end{eqnarray}
 Let $y=y_1(x)$ and $y=y_2(x)$ be
  $\widehat{A_1B_1}$ and $\widehat{A_2B_2}$, respectively.
  Then, we have
   \begin{eqnarray}
 \int_{\widehat{A_1B_1}}dE- \int_{\widehat{A_2B_2}}dE&=&-\int_0^{\sqrt{-3a_2}}\frac{\hat g(x)\hat F(x)}{y_1-\hat F(x)}dx
 +\int_0^{\sqrt{-3a_2}}\frac{\hat g(x)\hat F(x)}{y_2-\hat F(x)}dx
 \nonumber\\
 &=&\int_0^{\sqrt{-3a_2}}\frac{\hat g(x)\hat F(x)(y_1-y_2)}{(y_1-\hat F(x))(y_2-\hat F(x))}dx
 \nonumber\\
 &>&0.
   \label{=2lc2}
 \end{eqnarray}
 Similarly, we obtain
    \begin{eqnarray}
 \int_{\widehat{C_1D_1}}dE- \int_{\widehat{C_2D_2}}dE>0.
   \label{=2lc3}
 \end{eqnarray}
  Let $x=x_1(y)$ and $x=x_2(y)$ be
  $\widehat{B_1C_1}$ and $\widehat{I_2J_2}$, respectively.
  It is obvious that $x_1(y)<x_2(y)$ for $y_{C_1}<y<y_{B_1}$.
On the one hand, the function $y=\hat{F}(x)$ is increasing for $x>\sqrt{-3a_2}$.
  Then, we have
   \begin{eqnarray}
 \int_{\widehat{B_1C_1}}dE- \int_{\widehat{I_2J_2}}dE=\int_{y_{B_1}}^{y_{C_1}}\big(\hat F(x_1)-\hat F(x_2) \big)dy>0.
   \label{=2lc4}
 \end{eqnarray}
 On the other hand, the function  $\hat F(x)>0$   for $x>\sqrt{-3a_2}$.
 Then, we obtain
   \begin{eqnarray}
 \int_{y_{B_2}}^{y_{I_2}} \hat F(x)dy<0~{\rm and}~ \int_{y_{J_2}}^{y_{C_2}} \hat F(x)dy<0.
   \label{=2lc15}
 \end{eqnarray}
It follows from (\ref{=2lc1})--(\ref{=2lc15}) that
  \begin{eqnarray*}
\oint_{\gamma_1}dE>\oint_{\gamma_2}dE,
 \end{eqnarray*}
which contradicts  \eqref{=2lc}.
Therefore, either system \eqref{initial2} or its equivalent system \eqref{initial1} has at most one limit cycle.

\begin{figure}
\centering
 \subfigure[for $-1<a_2\leq-1/3$]
{
\scalebox{0.55}[0.55]{
\includegraphics{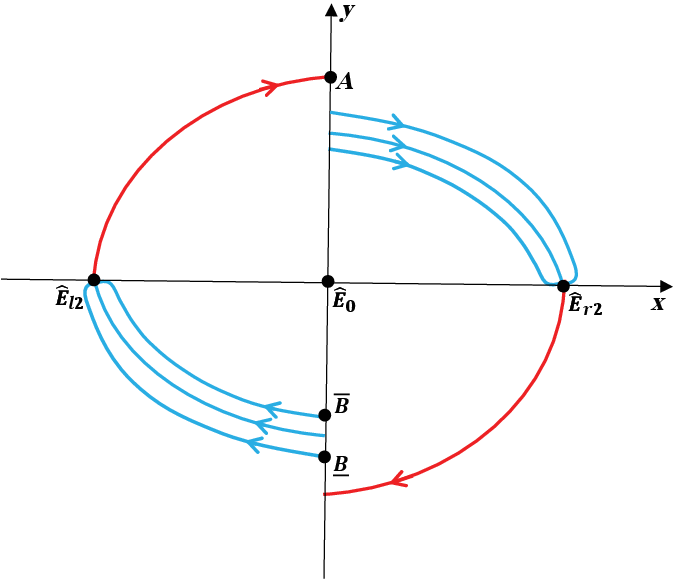}}}

 \subfigure[for $a_2=-1$]
{
\scalebox{0.55}[0.55]{
\includegraphics{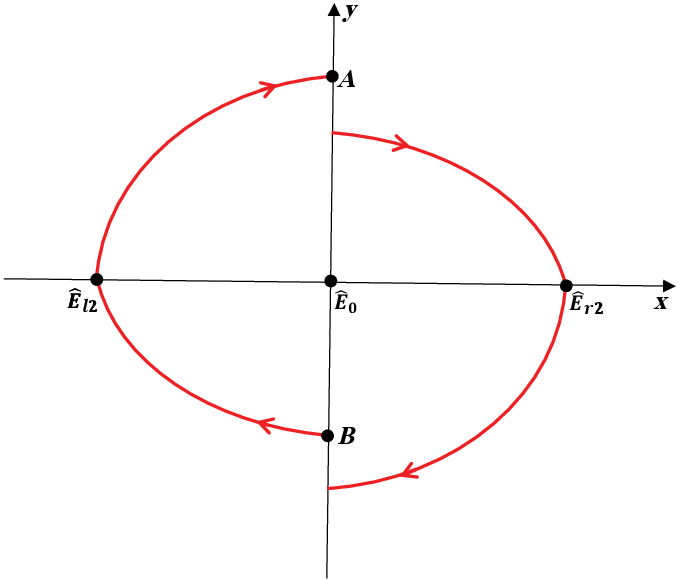}}}
 \subfigure[for $a_2<-1$]
{
\scalebox{0.55}[0.55]{
\includegraphics{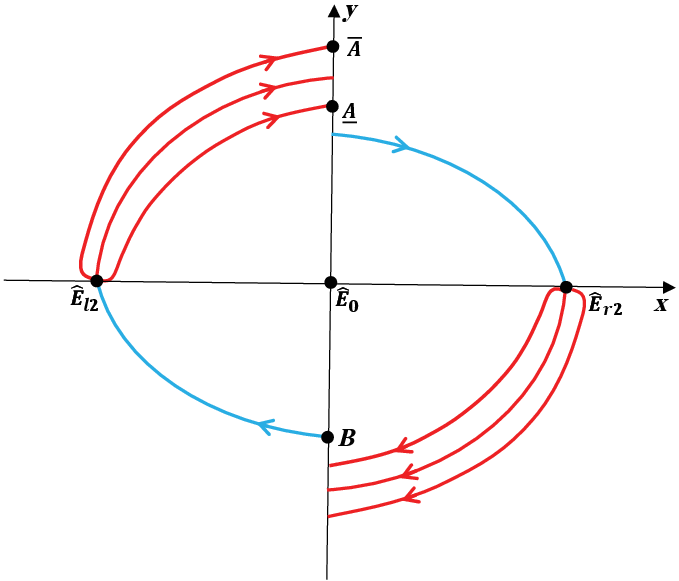}}}~~~
\caption{Relative location of stable and unstable manifolds of  $\hat{E}_{l2}$ and $\hat{E}_{r2}$ of \eqref{initial2}. }
\label{loma}
\end{figure}

We thirdly prove that  system \eqref{initial1} has a unique  limit cycle  that is large  when $a_2\leq-1/3$.
It follows from Lemma \ref{fe1} that
  $\hat{E}_{l2}$ and $\hat{E}_{r2}$ are saddle-nodes with  one  stable nodal  part  
  when $a_2>-1$,
  cusps when $a_2=-1$ and
  saddle-nodes with  one  unstable nodal part when $a_2<-1$.
If   $a_2\leq-1/3$,  then   $\hat{E}_{l2}$ and $\hat{E}_{r2}$ lie in the strip $[-\sqrt{-3a_2},\sqrt{-3a_2}]$. By the aforementioned analysis, we know that
 system \eqref{initial1} has no limit cycles lying in the strip $[-\sqrt{-3a_2},\sqrt{-3a_2}]$.
 Evidently, system \eqref{initial1} has no small limit cycles when  $a_2\leq-1/3$.
We claim that
 the  relative  location of stable and unstable manifolds of  $\hat{E}_{l2}$ and $\hat{E}_{r2}$ is shown in Figure \ref{loma} when  $a_2\leq-1/3$.
 If  the relative   location of stable and unstable manifolds of  $\hat{E}_{l2}$ and $\hat{E}_{r2}$ is not shown in Figure \ref{loma}, we will obtain that system \eqref{initial1} has a heteroclinic loop or a small limit cycle
 by the Poincar\'e-Bendixson Theorem.
  This is a contradiction. The assertion is proven.
 Therefore, system \eqref{initial1} has at least one limit cycle that is large  when   $a_2\leq-1/3$
 by  Proposition \ref{infty} and Poincar\'e-Bendixson Theorem.
 Then,   the uniqueness of  limit cycles  of  system \eqref{initial1} is proven  and   it  is large when $a_2\leq-1/3$.

 We finally    discuss the remainder  case  $-1/3<a_2<0$.
 	Since $\hat{E}_{l2}$ is  a   saddle-node  with one   stable nodal  part    when $a_2=-1/3$,
denote the intersection point of the unstable (resp. stable) manifold of $\hat{E}_{l2}$ and the positive (resp. negative) $y$-axis by $A$ (resp. $B$),
 as shown in Figure \ref{loma}(a).
Let the coordinate of a general point $P$ be $(x_P, y_P)$.
 As proved in Lemma 3.3 of \cite{CCX}, we can similarly prove that  $y_A+y_B$
 is increasing as $\mu_3$ decreases. On the one hand, $y_A+y_{B}>0$ when $a_2=-1/3$.
 	Since $\hat{E}_{l2}$ is a saddle-node with one stable nodal part when $-1<a_2\leq-1/3$, 
 	denote the intersection point of the unstable (resp.the left-most stable;
 	the right-most stable) manifold of $\hat{E}_{l2}$ and the positive (resp. negative) $y$-axis by $A$ (resp. $\underline{B}$; $\bar{B}$).
 On the other hand, we claim that $y_A+y_{\bar{B}}<0$ when $\mu_3=0$.
 By Hopf bifurcation, system \eqref{initial1} occurs a small limit cycle when $-\epsilon<a_2<0$,
 where $\epsilon>0$ is small. It is clear that  $y_A+y_{\bar{B}}=0$ is impossible.
 If  $y_A+y_{\bar{B}}>0$, system \eqref{initial1} has at least one large limit cycle
 by  the Poincar\'e-Bendixson Theorem, which contradicts the uniqueness of closed orbit.
 By the  Intermediate Value Theorem, there exist respectively two values $a_2=p_1(\delta)$ and
 $a_2=p_2(\delta)$ such that
 $y_A+y_{\underline{B}}=0$ and $y_A+y_{\bar{B}}=0$, where  $-\sqrt{\delta}/3<p_2(\delta)<p_1(\delta)<0$.
 Here, system \eqref{initial} has a unique small limit cycle when $p_1(\delta)<a_2<0$,
 one saddle-node heteroclinic loop when $p_2(\delta)\le a_2\le p_1(\delta)$,
 and one large limit cycle when $-1/3<a_2<p_2(\delta)$.
 Then, the proof is finished.
\end{proof}

\subsubsection{The case: $a_1\geq0$.}
 The existence of small limit cycles is  only meaningful  for a small limit cycle surrounding $\hat{E}_{r2}$ or $\hat{E}_{l2}$, because $\hat{E}_{0}$ is a degenerate saddle  for  $a_1=0$ or  a saddle  for  $a_1>0$ by Lemma \ref{fe1}.
The existence of large  limit cycles is  only meaningful  for a large  limit cycle surrounding $E_{0}$,  $\hat{E}_{r2}$ and  $\hat{E}_{l2}$, because  the vector field of    system \eqref{initial1} is symmetric about the origin.

\begin{lemma}
	System \eqref{initial1} has no limit cycles when
	$-1/3\leq a_2<0$.
	\label{nlc}
\end{lemma}
\begin{proof}	
Assume that system \eqref{initial2} exhibits  a limit cycle $\Gamma_0$ 	when $ a_2=-1/3$.
Consider energy function $E(x,y)$ again, as shown in \eqref{Exy}.
		Clearly, $\oint_{\Gamma_0}dE=0$.
		However, from \eqref{dEdt}
	we can show   that    $\oint_{\Gamma_0}dE=\oint_{\Gamma_0}-\delta x^2(a_1+x^2)(-1+x^2)^2dx/3<0$ for $a_2=-1/3$.
	This is a contradiction.
 Therefore, both
	system \eqref{initial2} and system \eqref{initial1} have no limit cycles  when $a_2=-1/3$.

 \begin{figure}[h]
	\centering
	\includegraphics[width=3.5in]{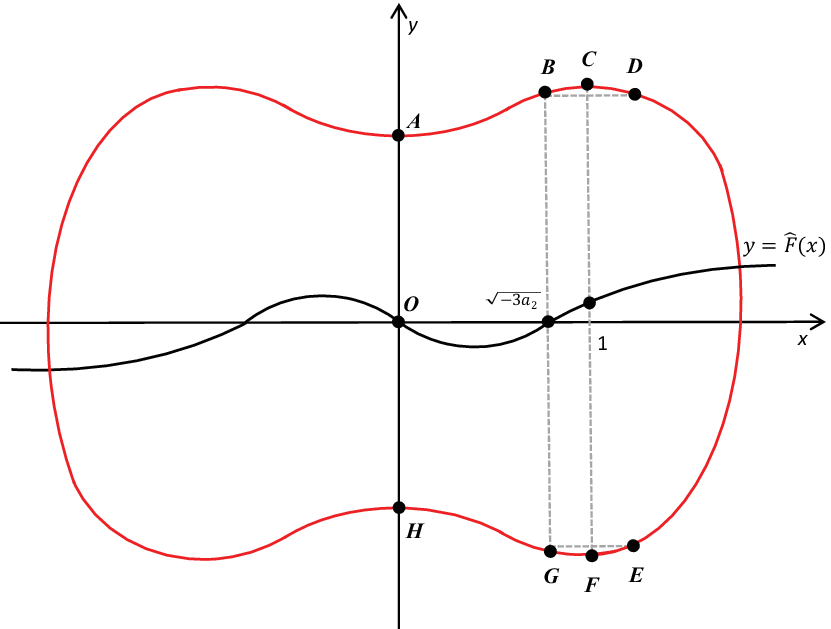}
	\caption{Large limit cycle  of  \eqref{initial2} for $-1/3\leq a_2<0$. }
	\label{nonlic}
\end{figure}	
Assume that system \eqref{initial2} exhibits  a small limit cycle $\Gamma$ surrounding the equilibrium $(1, \hat F(1))$ 	for $a_2\in(-1/3,0)$.
Let
\[
\hat E(x,y)=\int_0^x\hat  g(s)ds+\frac{(y-\hat F(1))^2}{2}.
\]
		Then, we have
\[
\frac{d\hat E(x,y)}{dt}|_{\eqref{initial2}}=-\hat g(x)(\hat F(x)-\hat F(1))=-\frac{\delta}{3}(x+1)(x-1)^2 (x^2+a_1)(x^2+x+1+3a_2)x<0
\]
for $x\in(0,1)\cup(1,+\infty)$.
However,
we can show   that    $\oint_{\Gamma}d\hat E=\oint_{\Gamma}-\hat g(x)(\hat F(x)-\hat F(1))dx=0$ for $a_2=-1/3$.
	This is a contradiction.
Therefore,
system \eqref{initial2} has no small limit cycles  when $-1/3<a_2<0$.  So does system \eqref{initial1}.
Assume that system \eqref{initial2} exhibits  a large limit cycle $\Gamma$ surrounding the point $(1, \hat F(1))$ 	for $a_2=c\in(-1/3,0)$, where  $c$ is a constant.  See Figure \ref{nonlic}, where $A,H$ (resp., $B,G$; $C,F$; $D$; $E$) are intersection points
between $\Gamma$ and the $y$-axis (resp., $x=\sqrt{-3a_2}$; $x=1$; $y=y_B$; $y=y_G$) and $y_B,y_G$ are respectively ordinates of $B,G$.
By the symmetry of 	system \eqref{initial2},  we have $2\int_{\widehat{ABH}}dE=\oint_{\Gamma}dE$ .
Then, we obtain
$$
\int_{\widehat{AB}}dE=\int_{\widehat{AB}}-\hat g(x)\hat F(x)dt<0.
$$
It is similar to prove that
$$
\int_{\widehat{DE}}dE<0  ~~~{\rm and} ~~~\int_{\widehat{GH}}dE<0.
$$
Let $x=x_1(y)$ and $x=x_2(y)$ represent the segment orbits $\widehat{BC}$ and $\widehat{CD}$, respectively.
Since $\hat{F}(x)$ is strictly increasing for $x>\sqrt{-3a_2}$, we have $\hat F(x_1(y))-\hat F(x_2(y))<0$. Further, we obtain
$$
\int_{\widehat{BCD}}dE=\int_{y_B}^{y_C}d(\hat F(x_1(y))-\hat F(x_2(y)))y<0.
$$
Similarly, we can prove
$
\int_{\widehat{EFG}}dE <0.
$
Thus, we have $\oint_{\Gamma}dE<0$, which contradicts $\oint_{\Gamma}dE=0$.
Therefore, neither
system \eqref{initial2} nor system \eqref{initial1} can have  large limit cycles  when $-1/3\le a_2<0$.
\end{proof}


\begin{lemma}
System \eqref{initial1} has exactly  one   limit cycle when
$a_2= -1$, which is stable, hyperbolic and large.
	\label{ulc}
\end{lemma}

\begin{proof}
	Firstly,  we discuss small limit cycles of system \eqref{initial2}, which is equivalent to system \eqref{initial1}.
		When  $a_2=-1$, with transformation
	\begin{eqnarray}
	(x,y)\to(x+1,y+\hat F(1)),
	\label{trans}
\end{eqnarray}
	we move  $(1, \hat  F(1))$ of system  \eqref{initial2} to the origin of the following system
	\begin{eqnarray}
		\begin{array}{ll}
	\dot x=y-\hat F(x+1)+\hat F(1)=:y-\tilde F(x),
	\\
	\dot y= -\hat g(x+1) =:-\tilde g(x),
	\end{array}
		\label{Lie2}
	\end{eqnarray}
 where $\tilde F(x)=\delta (x^3/3 +   x^2)$,
		$\tilde g(x)=x^5+5x^4+(9+a_1)x^3+(7+3a_1)x^2+2(a_1+1)x$
		and $\tilde f(x):=\tilde F'(x)=\delta(x+x^2)$.
	It is clear that $\tilde F(0)=0$, $x\tilde g(x)>0$ for $x\in(-1,0)\cup(0,+\infty)$, $\tilde f(x)<0$ for $-1<x<0$
	and $\tilde f(x)>0$ for $x>0$.
	 Thus, the conditions  {\bf(i)-(iii)} of Proposition 9 of \cite{CLT} hold.
Assume that there exist $x_1$ and $x_2$ such that
	\begin{eqnarray}
\hat{F}(x_1)=\hat{F}(x_2) ~~~{\rm and}~~~\frac{\hat{f}(x_1)}{\hat{g}(x_1)}=\frac{\hat{f}(x_2)}{\hat{g}(x_2)},
	\label{FFF11}
	\end{eqnarray}
	where $0<x_1<1<x_2$.
	By the first equality of \eqref{FFF11}, we have that
	\begin{eqnarray}
	-3+x_1^2+x_1x_2+x_2^2=0.
	\label{FFF2}
	\end{eqnarray}
	It follows from the second equality of \eqref{FFF11}  that
	\begin{eqnarray}
	a_1+x_1^2+x_1x_2+x_2^2=0.
	\label{FFF3}
	\end{eqnarray}
	According to \eqref{FFF2} and \eqref{FFF3}, we obtain $a_1=-3$,
which contradicts to $a_1\geq0$.
	It means that there are no    solutions
	for equations \eqref{FFF11} with $\hat F=\tilde F$, $\hat f=\tilde f$ and $\hat g=\tilde g$,
	where $-1<   x_1<0<x_2$.
 It follows from Corollary 10 of \cite{CLT} that
	 system \eqref{Lie2}  has no limit cycles in the  zone $x>-1$, i.e.,
	system \eqref{initial1} has no small limit cycles.

 Secondly,  we discuss large limit cycles of system \eqref{initial2}.
		Denote   $A := (x_A, 0)$ and $B :=(x_B, 0)$ be respectively the first intersection points of
		the  stable  and  unstable   manifold   of the right-hand side of  $\hat{E}_0$  and  the $x$-axis.
		Let   $D:= (x_D, 0)$ be  the first intersection point   of  an   orbit crossing $C:= (x_C, 0)$   and the $x$-axis, where $x_C\in (x_{\hat{E}_{r2}}, x_{\hat{E}_{r2}}+\varepsilon)$ and $\varepsilon>0$ is sufficiently small.
		According to the stability of   $\hat{E}_{r2}$   and the nonexistence of   small limit cycles,
	it is clear that  $x_A<x_B$   and $x_C<x_D$  when $a_2=-1$, 
see Figure \ref{u3=h1}(a).
	By $x_A<x_B$, Proposition \ref{infty} and   Poincar\'e-Bendixson Theorem,
	system \eqref{initial} has at least one large limit cycle when  $a_2=-1$.
	Assume that $\gamma_1:=\widehat{A_1B_1H_1C_1D_1A_1}$ is a large limit cycle of system  \eqref{Lie},
where $x_{A_1}=x_{D_1}=0$, $x_{B_1}=x_{C_1}=x_{\hat{E}_{r_2}}$ and $y_{H_1}=\hat{F}(x_{H_1})$,  as shown in Figure \ref{u3=h1}(b).
	Next, we will show that
	\begin{eqnarray}
	\oint_{\gamma_1}\hat{f}(x)dt>0.
	\label{ffq00}
	\end{eqnarray}
	By the symmetry, we can obtain
	\begin{eqnarray}
	\int_{\widehat{A_1B_1D_1}} \hat{f}(x)dt=\frac{1}{2}\oint_{\gamma_1} \hat{f}(x)dt.
	\label{ffq0}
	\end{eqnarray}%
	\begin{figure}
		\centering
		\subfigure[Stable and unstable manifolds in system \eqref{initial1}  ]
		{
			\scalebox{0.55}[0.55]{
				\includegraphics{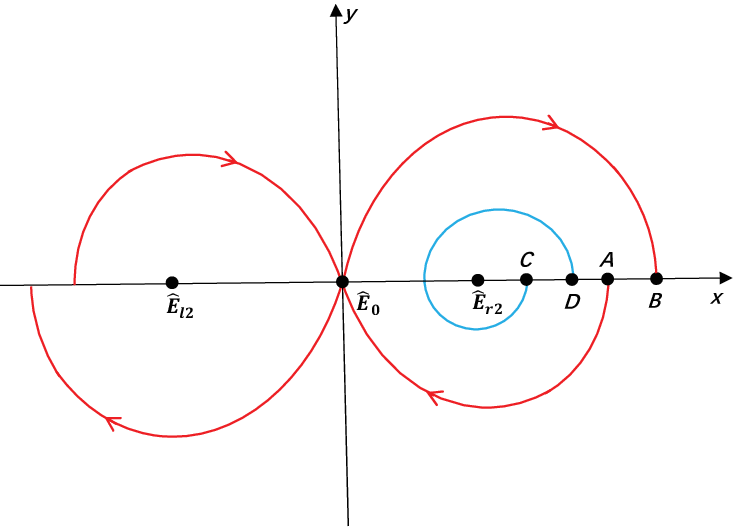}}}
		\subfigure[A large limit cycle  $\gamma_1$  of system \eqref{initial2} ]
		{
			\scalebox{0.57}[0.57]{
				\includegraphics{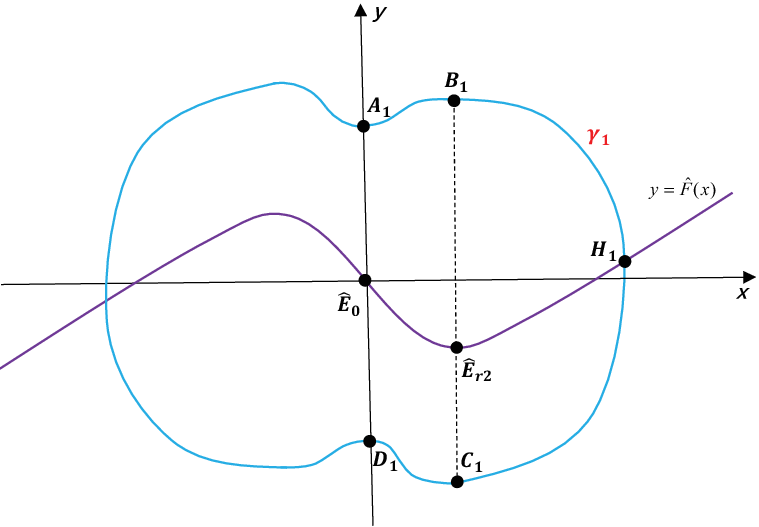}}}
		\subfigure[$w$-$y$ plane]
		{
			\scalebox{0.55}[0.55]{
				\includegraphics{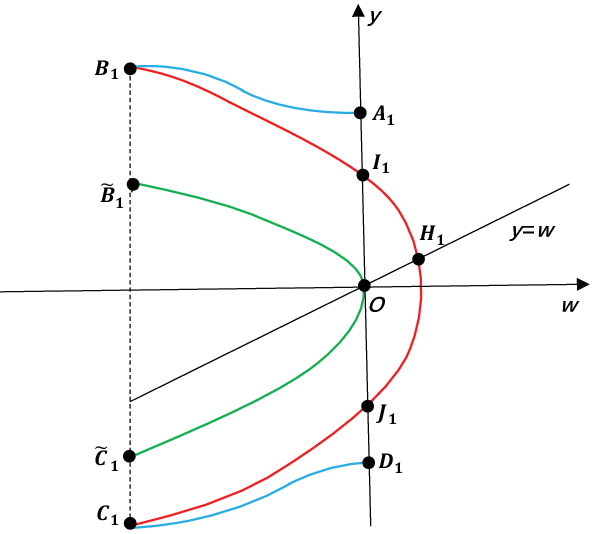}}}
		\caption{Orbits for the case $a_2=-1$. } 
		\label{u3=h1}
	\end{figure}%
	From $\hat{F}(x)=\delta(-x+x^3/3)$, we know that $y=\hat{F}(x)$ has two inverse functions $x_1(w)$ and $x_2(w)$ for $x>0$,
	where $x_1(w)\in(0, 1)$
	and $x_2(w)\in(1, \infty)$ when  $a_2=-1$.
	By changing variable $x$   of system  \eqref{initial2}  to the variable $w=\hat{F}(x)$, we obtain two equations
	\begin{eqnarray}
	\frac{dy}{dw}=\frac{\lambda_i(w)}{w-y},
	\label{tsy1}
	\end{eqnarray}
	where $\lambda_i(w)=\hat{g}(x_i(w))/\hat{f}(x_i(w))$ and $i=1,2$.
	 By \eqref{FFF2} and \eqref{FFF3}, we have $\lambda_2(w)>\lambda_1(w)$.
	Let $y=y_1(w)$, $y=y_2(w)$, $y=z_1(w)$ and $y=z_2(w)$ represent $\widehat{B_1H_1}$,
	$\widehat{B_1A_1}$, $\widehat{C_1H_1}$  and $\widehat{C_1D_1}$, respectively.
	By the Comparison Theorem, it follows that $y_2>y_1$ and $z_2<z_1$.
	We claim that $y_{H_1}>0$. Otherwise, assume that $y_{H_1}\leq0$.
	Let $\mathcal{D}$ be the interior region of $\gamma_1$,   see   Figure \ref{u3=h1}(b).  Set $\mathcal{D}_1$
	be the interior region surrounding by $\widehat{B_1H_1C_1}$ and $w=\hat F(1)$,
	$\mathcal{D}_2$
	be the interior region surrounding by $\widehat{A_1B_1}$, $\widehat{C_1D_1}$, $w=0$ and $w=\hat F(1)$,
 see   Figure \ref{u3=h1}(c).
	It is clear that  $\mathcal{D}_1\subset \mathcal{D}_2$.
	By Green's Formula, it follows that
	\begin{eqnarray*}
		\oint_{\gamma_1}[-\hat{g}(x)dx+(\hat{F}(x)-y)]dy&=&\iint_{ \mathcal{D}}\hat{f}(x)dxdy
		\\
		&=&2\left(\iint_{ \mathcal{D}_1}dwdy-\iint_{ \mathcal{D}_2}dwdy\right)
		\\
		&<&0,
	\end{eqnarray*}
	\textcolor{black}{which contradicts $\oint_{\gamma_1}[-\hat g(x)dx+(\hat F(x)-y) ]dy=0$.}
	This proves the assertion $y_{H_1}>0$.

	When $a_2=-1$, we obtain
	\[
	\frac{(\hat F(x)-\hat F(1))\hat f(x)}{\hat g(x)}=\frac{\delta^2(x^3/3 -x+2/3)}{(x^2+a_1)x}.
	\]
	Then,
	\begin{eqnarray*}
	\frac{d}{dx}\left(\frac{(\hat F(x)-\hat F(1))f(x)}{g(x)}\right)
	=\frac{\delta^2\kappa(x) }{3x^2(x^2+a_1)^2},
	\end{eqnarray*}
where $\kappa(x)= (2a_1+6) x^3-6x^2 -2a_1$.
	Since
 $\kappa'(x)=6(a_1+3) x^2-12x>\kappa'(1)=6+6a_1\geq 0$,
	we have
	$
	\min\kappa(x)=\kappa(1)=0
	$
	for $x>1$.
	Consequently,  $(\hat F(x)-\hat F(1))\hat f(x)/\hat g(x)$ is increasing.
	Let  $y=\tilde y_2(w)$ and $z=\tilde z_2(w)$ represent respectively  $\widehat{\tilde B_1O}$
		and $\widehat{\tilde C_1O}$.
	By the proof of Theorem 2.1 of \cite{DR} or Lemma 4.5 of \cite[Chapter 4]{Zh},
	it follows that
	\begin{eqnarray}
	\int_{\widehat{\tilde B_1O\tilde C_1}} \hat{f}(x)dt-\int_{\widehat{B_1H_1C_1}} \hat{f}(x)dt
	<0.
	\label{ffq1}
	\end{eqnarray}
	On the other hand, we have
	\begin{eqnarray}
	\int_{\widehat{\tilde B_1O}}\hat f(x)dt-\int_{\widehat{B_1A_1}}\hat f(x)dt&=&
	\int_{\hat F(1)}^0\frac{dw}{\tilde y_2-w}-\int_{\hat F(1)}^0\frac{dw}{ y_2-w}
	\nonumber\\
	&=& \int_{\hat F(1)}^0\frac{(y_2-\tilde y_2)}{(y_2-w)(\tilde y_2-w)}dw
	\nonumber\\
	&>&0.
	\label{ffq2}
	\end{eqnarray}
	Similarly, we have
	\begin{eqnarray}
	\int_{\widehat{O\tilde C_1}} \hat{f}(x)dt-\int_{\widehat{D_1C_1}} \hat{f}(x)dt
	>0.
	\label{ffq3}
	\end{eqnarray}
	By   \eqref{ffq0}, \eqref{ffq1},  \eqref{ffq2} and  \eqref{ffq3}, it follows that  $\int_{\widehat{A_1B_1D_1}}\hat f(x)dt>0$ and then
	\eqref{ffq00} holds.
	Therefore, both  system  \eqref{initial2} and system \eqref{initial1} has a unique   large limit cycle  when $a_2=-1$,
	which is hyperbolic.    Combining  the nonexistence of   small limit cycles, we obtain that system \eqref{initial1}  has  a unique    limit cycle when
		$a_2=-1$, where the limit cycle is stable,  large and hyperbolic.
\end{proof}

\begin{proposition}
When  $a_1\geq0$,   there exist two continuous functions  $\varphi_1$ and  $\varphi_5$  satisfying  $-1<\varphi_1(a_1,\delta)<\varphi_5(a_1,\delta)<-1/3$ and  
\begin{description}
\item[(i)] system \eqref{initial1} has no limit cycles when
$\varphi_5(a_1, \delta)< a_2<0${\rm {;}}
 \item[(ii)] system \eqref{initial1} has a unique limit cycle when
$a_2=\varphi_5(a_1,\delta)$, which is semi-stable and large{\rm {;}}
  \item[(iii)] system \eqref{initial1} has two limit cycles when
$\varphi_1(a_1,\delta)<a_2<\varphi_5(a_1,\delta)$, where they are large, the inner one is unstable
and the outer one is stable{\rm {;}}
 \item[(iv)] system \eqref{initial1} has one figure-eight loop and one limit cycle when
$a_2=\varphi_1(a_1,\delta)$, where the figure-eight loop is unstable, and the limit cycle
is stable and large{\rm {;}}
  \item[(v)] system \eqref{initial1} has two small limit cycles and one large limit cycle when
$-1<a_2<\varphi_1(a_1,\delta)$, where the small ones are unstable and the large one
is stable{\rm {;}}
 \item[(vi)] system \eqref{initial1} has  one   limit cycle when
$a_2\leq -1$, where the limit cycle is stable and large{\rm {.}}
\end{description}
\label{c5c6}
\end{proposition}
\begin{proof}
By Lemmas \ref{con1}, \ref{nlc} and \ref{ulc},  we only need to study  the number of  limit cycles of    system \eqref{initial1} for $a_2\in (-\infty, -1)\cup (-1, -1/3)$.

\begin{figure}
\centering
\subfigure[for $a_2<-1$  ]{
\scalebox{0.55}[0.55]{
\includegraphics{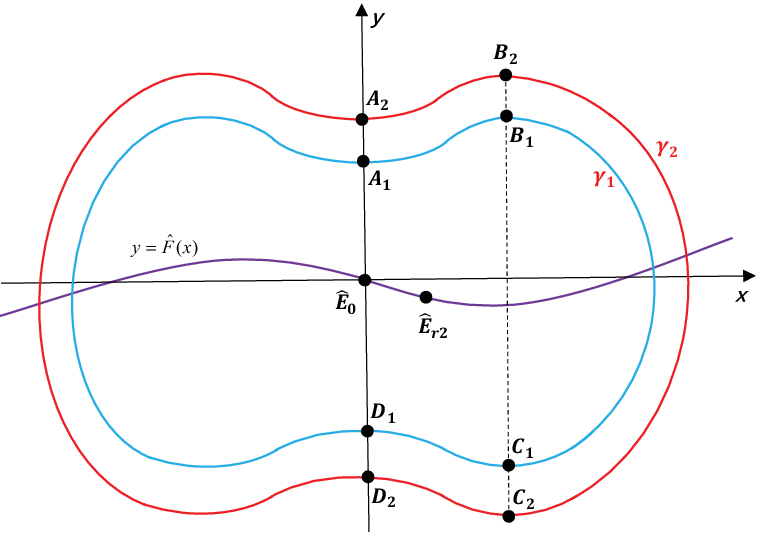}} }
\subfigure[for $-1<a_2<-1/3$]{
\scalebox{0.55}[0.55]{
\includegraphics{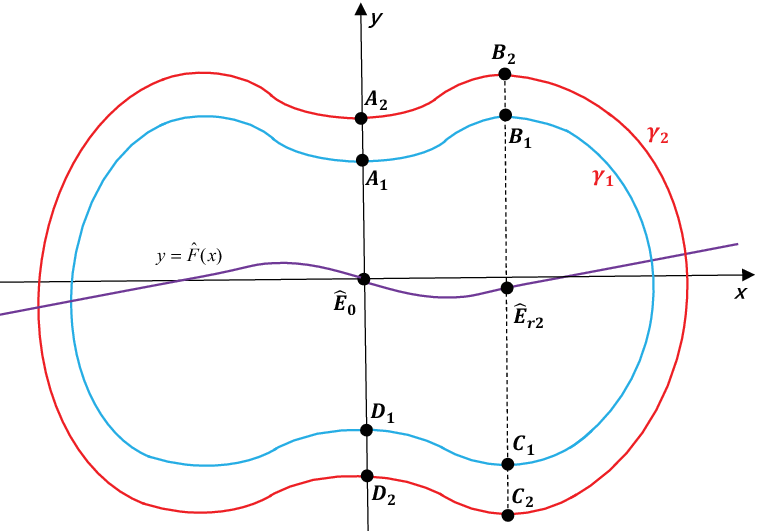}} }
\caption{Two large limit cycles  of  \eqref{initial2}.   }
\label{3eq2lc}
\end{figure}

  For simplicity, we  firstly  discuss large  limit cycles.
Assume that system \eqref{initial2} exhibits at least two large limit cycles $\Gamma_1$ and $\Gamma_2$
for $a_2<-1$ or $-1<a_2<-1/3$,
where  $\Gamma_1$ lies in the interior region surrounded by $\Gamma_2$, as shown in Figure \ref{3eq2lc}.
We aim to  prove that
\begin{eqnarray}
\oint_{\Gamma_1} \hat{f}(x)dt<\oint_{\Gamma_2} \hat{f}(x)dt,
\label{compare}
\end{eqnarray}
 which implies that  system  \eqref{initial2} has at most two      large limit cycles.
 Firstly,  we prove that
 \begin{eqnarray}
\int_{ \widehat{A_1B_1}}\hat f(x)dt<\int_{\widehat{A_2B_2}}\hat f(x)dt
\label{compare2}
\end{eqnarray}
and \begin{eqnarray}
\int_{ \widehat{C_1D_1}}\hat f(x)dt<\int_{\widehat{C_2D_2}}\hat f(x)dt.
\label{compare21}
\end{eqnarray}
 For  each $i=1, 2$ let $y_i$ represent $\widehat{A_iB_i}$, we can calculate that
	\begin{eqnarray*}
		&&\int_{\widehat{A_iB_i}}\hat f(x)dt\\&=&-\int_{0}^{1}\frac{\hat{F}'(x)}{\hat{F}(x)-y_i(x)}dx
       \\&=&  -
		\int_{0}^{1}\frac{  \hat{F}'(x)-y_i'(x)+ y_i'(x)}{\hat F(x)-y_i(x)}dx
		\\
		&=&-
		\int_{0}^{1}\frac{1}{\hat F(x)-y_i(x)}d(\hat F(x)-y_i(x))- \int_{0}^{1}\frac{y_i'(x)}{\hat F(x)-y_i(x)}dx
		\\
		&=&-\ln\left|\frac{\hat F(1)
-y_i(1)}{\hat F(0)-y_i(0)}\right|-\int_{0}^{1}\frac{y_i'(x)}{\hat F(x)-y_i(x)}dx
		\\
		&=&-\ln\left|\frac{\hat F(0)-y_i(1)}{\hat F(0)-y_i(0)}\right|
-\ln\left|\frac{\hat F(1)-y_i(1)}{\hat F(0)-y_i(1)}\right|
		-\int_{0}^{1}\frac{\hat g(x)}{(\hat F(x)-y_i(x))^2}dx
		\\
		&=&-\ln\left|\frac{\hat F(1)-y_i(1)}{\hat F(0)-y_i(1)}\right|+\int_{0}^{1}\frac{y_i'(x)}{\hat F(0)-y_i(x)}dx
		-\int_{0}^{1}\frac{\hat g(x)}{(\hat F(x)-y_i(x))^2}dx
		\\
		&=& -\ln\left|\frac{\hat F(1)-y_i(1)}{\hat F(0)-y_i(1)}\right|
+\int_{0}^{1}\frac{\hat g(x)}{(\hat F(0)-y_i(x))(\hat F(x)-y_i(x))}dx
 -\int_{0}^{1}\frac{\hat g(x)}{(\hat F(x)-y_i(x))^2}dx
		\\	&=&-\ln\left|\frac{y_i(1)-\hat F(1)}{y_i(1)}\right|-
		\int_{0}^{1}\frac{\hat g(x)\hat F(x)}{y_i(x)(\hat F(x)-y_i(x))^2}dx.
\end{eqnarray*}
 Further, we have
	\begin{eqnarray*}
		\int_{\widehat{A_2B_2}}\hat f(x)dt-\int_{\widehat{A_1B_1}}\hat f(x)dt
		=\ln\left|\frac{y_1(1)-\hat F(1)}{y_1(1)}\right|
	-\ln\left|\frac{y_2(1)-\hat F(1)}{y_2(1)}\right|
		+
		\int_{0}^{1}
		H(x)dx,
\end{eqnarray*}%
where
$$
H(x)=\frac{\hat g(x)\hat F(x)}{y_1(x)(\hat F(x)-y_1(x))^2}-\frac{\hat g(x)\hat F(x)}{y_2(x)(\hat F(x)-y_2(x))^2}.
$$
Evidently,
$y_2(x)-\hat F(x)>y_1(x)-\hat F(x)>0, y_2(x)>y_1(x)>0$, $\hat g(x)<0$ and $\hat F(x)<0$ for every $x\in(0, 1)$  because of $a_2<-1$ or $-1<a_2<-1/3$.   It follows that
$$
\frac{y_1(1)
-\hat F(1)}{y_1(1)}>\frac{y_2(1)
-\hat F(1)}{y_2(1)}>1 ~~~{\rm and}~~~ H(x)>0,
$$
which implies that  \eqref{compare2} holds.
 It is similar to prove that
 \eqref{compare21}  holds.
 On the other hand, we can prove that
\begin{eqnarray}
 \int_{\widehat{B_1 C_1}}\hat f(x)dt<\int_{\widehat{B_2C_2}}\hat f(x)dt.
\label{compare6}
\end{eqnarray}
 It is not difficult to compute  that
\begin{eqnarray}\label{compare3}
\frac{[\hat F(x)-\hat F(1)]\hat f(x)}{\hat g(x)}
= \frac{\delta^2}{3}\cdot\frac{ (x^2+x+1+3a_2)}{x^2+x} \cdot\frac{(x^2+a_2) }{ (x^2+a_1)}. \nonumber
\end{eqnarray}%
 When  $a_2<-1$ or $-1<a_2<-1/3$, 
  \begin{eqnarray}
\frac{d[ (x^2+x+1+3a_2)/(x^2+x)]}{dx}= -\frac{(1+3a_2)(1+2x)}{ (x^2+x)^2}>0
\label{compare4}
\end{eqnarray}
and
\begin{eqnarray}
\frac{d[(x^2 +a_2)/(a_1+x^2)]}{dx}= \frac{2(a_1-a_2)x }{(a_1+x^2)^2}>0
\label{compare5}
\end{eqnarray}
for $x>0$.
It follows from (\ref{compare4}) and (\ref{compare5})
 that $[\hat F(x)-\hat F(1)]\hat f(x)/\hat g(x)$ is increasing
for $x\in(1, +\infty)$.
Therefore, according to
  the proof of Theorem 2.1 of \cite{DR} or Lemma 4.5 of \cite[Chapter 4]{Zh},
we obtain  \eqref{compare6}. In conclusion, \eqref{compare} holds by \eqref{compare2}, \eqref{compare21} and \eqref{compare6}.
 Therefore, both system  \eqref{initial2} and system \eqref{initial1} have at most two  large limit cycles  when $a_2<-1$ or $-1<a_2<-1/3$.

 We secondly   discuss small   limit cycles for the case $a_2\in (-\infty, -1)\cup (-1, -1/3)$.   As proved in Lemma 3.3 of \cite{CC18}, $x_A$ decreases continuously and $x_B $ increases continuously as $\beta:=\delta a_2$ decreases.
 So,   we still get   $x_A<x_B$  when $a_2<-1$, which is same as the case  $a_2=-1$  (see Figure \ref{u3=h1}(a)).
 Assume that system \eqref{initial1}  exhibits   small limit cycles surrounding $\hat{E}_{r2}$ when $a_2=\alpha< -1$,  where $\alpha$
is fixed.  Since the vector field of system \eqref{initial1}  is rotated on $\beta$,
there is an annulus region of Poincar\'e-Bendixson for $a_2=-1$,
which contradicts the nonexistence of small limit cycles  surrounding $\hat{E}_{r2}$ in this case.
Therefore, system \eqref{initial1} has  no small limit cycles when $a_2<-1$. 
Combining   the nonexistence of small limit cycles   for
$a_2<-1$
and   the  statement  {\bf (iii)} of  Proposition  \ref{localbi},   system \eqref{initial1} has a unique   small limit cycle surrounding $\hat{E}_{r2}$  for arbitrarily fixed $\delta$   when $-1<a_2<-1+\varepsilon$ and $\varepsilon>0$ is sufficiently small.

 With the help of  the above  analysis results, in the following  we  give the proof    of  the  statements  {\bf (i)-(vi)} of  Proposition  \ref{c5c6}.

 Firstly, we prove statement   {\bf (vi)}.
When $a_2< -1$,
system \eqref{initial1} has at least one large limit cycle  by  $x_A<x_B$, Proposition \ref{infty} and   Poincar\'e-Bendixson Theorem.
We know that system  \eqref{initial1}  has at most two      large limit cycles
and  has  no small limit cycles when $a_2<-1$ by above analysis.  
 We claim that system  \eqref{initial1}  has a   unique limit cycle  in the case  $a_2<-1$,  that  is stable and large.  Otherwise,  this contradicts \eqref{compare}.
 When  $a_2=-1$, system \eqref{initial1}  has also  a unique  limit cycle  that is stable and  large by Lemma \ref{ulc}.  This proves statement   {\bf (vi)}.

Secondly, we prove statement   {\bf (iv)}. System \eqref{initial1} has no small  limit cycles   when $a_2\leq -1$,  see the  statement   {\bf (vi)}. It implies that $x_A<x_B$. Besides, by above analysis system \eqref{initial1} has no limit cycles   as $-1/3\leq a_2<0$,
which means  $x_A>x_B$.  Moreover, as proved in Lemma 3.3 of \cite{CC18}, $x_A$ decreases continuously and $x_B  $ increases continuously as $a_2$ decreases for arbitrarily fixed $\delta$.
Therefore,
 there is a unique    function $\varphi_1(a_1,\delta)$  such that      $x_A-x_B=0$  if  and only if  $a_2=\varphi_1(a_1,\delta)$   for system \eqref{initial1}, which implies the existence of a figure-eight homoclinic loop.
 Regarding the saddle $\hat{E}_0$ of system    \eqref{initial1}, its eigenvalues are denoted by $\lambda_-$ and $\lambda_+$,  where $\lambda_-+\lambda_+=-\beta$.  Since $a_2<0$, the  homoclinic loop of system  \eqref{initial1}   is unstable  by    Theorem 3.3 of \cite{CLW}.
We claim that there exist  no small  limit cycles   in the interior of  the homoclinic loop. Otherwise,
this contradicts the uniqueness    of  small limit cycles for $-1<a_2<-1+\varepsilon$.
 By the stability of the homoclinic loop,  Proposition \ref{infty} and   Poincar\'e-Bendixson Theorem, we  can  prove that
 there is a unique  large limit cycle surrounding the  figure-eight homoclinic  loop, which is stable.
  This proves statement   {\bf (iv)}.

 Thirdly,  statement   {\bf (v)} directly holds from
 statements {\bf (vi)} and   {\bf (iv)} by the rotated properties of the vector field.

 Then, we prove statement   {\bf (ii)}. Combining  the   statements {\bf (v)}   and {\bf (iv)}, we obtain that  system \eqref{initial1} has two       limit cycles that are large,      when $\varphi_1(a_1,\delta)<a_2<\varphi_1(a_1,\delta)+\varepsilon$ and $\varepsilon>0$ is sufficiently small.
We have known that system \eqref{initial1} has no limit cycles   when $-1/3\leq a_2<0$ by above proof.
Since the vector field of system \eqref{initial1} is rotated on $\beta$, there is a function $\varphi_5(a_1,\delta)$  such that  system \eqref{initial1} has a unique limit cycle if and only if
$a_2=\varphi_5(a_1,\delta)$, which is semi-stable and large.    This proves statement   {\bf (ii)}.

Statement   {\bf (iii)} can be obtained by   statements   {\bf (v)},  {\bf (iv)} and {\bf (ii)}.

 At last,  statement   {\bf (i)} follows  from  statement  {\bf (ii)} and  the  nonexistence of   limit cycles   for   $-1/3\leq a_2<0$
 by Lemma \ref{nlc}. The proof   is completed.
\end{proof}

\medskip

\subsection{System \eqref{initial} with five equilibria}

By Lemma \ref{fe1}, system \eqref{initial1} has exactly five equilibria if and only if
$-1<a_1<0$.
 ${E}_{l2}$, ${E}_{l1}$,   ${E}_0$, ${E}_{r1}$ and  ${E}_{r2}$ of system \eqref{initial}      become  ${\hat{E}}_{l2}:=(-1,0)$, ${\hat{E}}_{l1}:=(-\sqrt{-a_1},0)$, ${\hat{E}}_0:=(0, 0)$, ${\hat{E}}_{r1}:=(\sqrt{-a_1},0)$ and  ${\hat{E}}_{r2}:=(1,0)$   of  system \eqref{initial1} respectively.

 By Lemma \ref{fe1},    the  equilibria   ${\hat{E}}_{l1}$  and  ${\hat{E}}_{r1}$   of   system \eqref{initial1} are saddles.  Moreover,   the vector field of    system \eqref{initial1}  is symmetric about the origin.  Therefore,  it    suffices     to give the existence of  small limit cycles of    system \eqref{initial1}   surrounding   ${\hat{E}}_0$  or ${\hat{E}}_{r2}$.
\begin{lemma}
System \eqref{initial1} has no small limit cycles surrounding ${\hat{E}}_0$ when $a_2\leq a_1/3$,
at most one
small limit cycle surrounding ${\hat{E}}_0$ when $a_2>a_1/3$,
no small limit cycles surrounding ${\hat{E}}_{r2}$  when $a_2\leq -1$ or $a_2\geq (a_1-1-\sqrt{-a_1})/3$,
and
at most one
small limit cycle surrounding ${\hat{E}}_{r2}$ when $-1<a_2<(a_1-1-\sqrt{-a_1})/3$.
\label{smc7}
\end{lemma}

\begin{proof}
It suffices to study  small limit cycles    of  system  \eqref{initial2}.
When  $ a_2\leq a_1/3$,  assume that system \eqref{initial2} exhibits a small limit cycle $\Gamma_0$  surrounding
the origin, i.e.,  $\Gamma_0$ lies in the zone $|x|<\sqrt{-a_1}$.
	Let $E(x,y)$ be defined in \eqref{Exy}.
For   $|x|< \sqrt{-a_1}$ and $ a_2\leq a_1/3$, it is clear that
$
\hat g(x)\hat F(x)\le 0.
$
Associated with \eqref{dEdt}, we get $\oint_{\Gamma_0}dE=\oint_{\Gamma_0}-\hat g(x)\hat F(x)dt>0$.
This  contradicts $\oint_{\Gamma_0}d E=0$.
Thus, system \eqref{initial2} has no small limit cycles surrounding the origin when  $ a_2\leq a_1/3$.

When $a_2>a_1/3$,  as proven in Lemma \ref{con4},
 we can obtain that system \eqref{initial2} has at most one
small limit cycle surrounding the origin.

 In the following  paragraphs, we  distinguish four cases  to discuss  the existence of   small limit cycles of    system \eqref{initial1}  surrounding ${\hat{E}}_{r2}$.

Consider $a_2=-1$.
With transformation \eqref{trans},
we move $(1,\hat F(1))$ of system \eqref{initial2} to the origin of system
\eqref{Lie2}.
It is clear that $\tilde F(0)=0$, $x\tilde g(x)>0$ for $x\in(\sqrt{-a_1}-1,0)\cup(0,+\infty)$, $\tilde f(x)<0$ for $\sqrt{-a_1}-1<x<0$
and $\tilde f(x)>0$ for $x>0$.
Therefore, the conditions {\bf(i-iii)} of Proposition 9 of \cite{CLT} hold.
Assume that there exist $x_1$ and $x_2$ such that \eqref{FFF11} holds,
where $ \sqrt{-a_1}-1 <x_1 <0<x_2$.
It follows from the first equality of \eqref{FFF11} that
\begin{eqnarray}
x_1^2+x_2^2+x_1x_2+3x_1+3x_2=0.
\label{FFF21}
\end{eqnarray}%
By the second equality of \eqref{FFF11},
 \begin{eqnarray}
a_1+(x_1+1)^2+(x_1+1)(x_2+1)+(x_2+1)^2=0.
\label{FFF31}
\end{eqnarray}%
By \eqref{FFF21} and \eqref{FFF31}, it follows that $a_1=-3$,
which contradicts $-1<a_1<0$.
By Corollary 10 of \cite{CLT},
system \eqref{Lie2} has no limit cycles in the zone  $x>\sqrt{-a_1}-1$.
In other words, system \eqref{initial1} has no small limit cycles surrounding  ${\hat{E}}_{r2}$  when $a_2=-1$.

 Consider $a_2<-1$.  By Lemma \ref{fe1}, we know that ${\hat{E}}_{r2}$   of system \eqref{initial1} is  a  source when $a_2<-1$ or  an unstable weak foci of order one    when $a_2=-1$.
Assume that system \eqref{initial1} has at least one small limit cycle surrounding ${\hat{E}}_{r2}$   when $a_2=\lambda<-1$ and $\Gamma_1$ is the innermost small limit cycle  surrounding ${\hat{E}}_{r2}$, where $\lambda$
is fixed.  Since the vector field of system \eqref{initial1} is rotated on $a_2$,
there is an annulus region of Poincar\'e-Bendixson for $a_2\in(\lambda,-1]$,
which contradicts the nonexistence of small limit cycles  surrounding ${\hat{E}}_{r2}$  when $a_2=-1$.
 Thus, system \eqref{initial1} has no small limit cycles surrounding ${\hat{E}}_{r2}$ when $a_2<-1$.

 Consider  $a_2\geq(a_1-1-\sqrt{-a_1})/3$.
	Let
\[
\tilde	E(x,y):=\int_1^x\hat g(s)ds+\frac{(y-\hat F(1))^2}{2}.
\]
Then
$$
\frac{ d\tilde E}{dt}|_{ \eqref{initial2}}=-\hat g(x)(\hat F(x)-\hat F(1))\leq0
$$
for $x>\sqrt{-a_1}$. Assume that system \eqref{initial2} exhibits a small limit cycle $\Gamma_2$  surrounding
$(1, \hat F(1))$,
  i.e., $\Gamma_2$ lies in the zone  $ x> \sqrt{-a_1}$.
Then, we have $\oint_{\Gamma_2}d\tilde E=\oint_{\Gamma_2}-\hat g(x)(\hat F(x)-\hat F(1))dt<0$,
which contradicts $\oint_{\Gamma_2}dE=0$.
Thus, system \eqref{initial2} has no small limit cycles surrounding  $(1, \hat F(1))$ when  $a_2\geq   (a_1-1-\sqrt{-a_1})/3$.

 Consider    $-1<a_2< (a_1-1-\sqrt{-a_1})/3$.
Assume that system \eqref{initial1} has  at most two
small limit cycles surrounding ${\hat{E}}_{r2}$
when   $-1<a_2<(a_1-1-\sqrt{-a_1})/3$,
where $\hat\Gamma_1$, $\hat\Gamma_2$ are the such innermost two  limit cycles and
$\hat\Gamma_1$ lies in the interior of $\hat\Gamma_2$.
 As proven in Lemma \ref{con4}, we have
$
 \oint_{\hat\Gamma_1} d\tilde	E<\oint_{\hat\Gamma_2}d\tilde	E,
$
 which contradicts
$\oint_{\hat\Gamma_1}d\tilde	E=\oint_{\hat\Gamma_2}d\tilde	E=0$.
 Thus,
system \eqref{initial1}  has at most one
small limit cycle surrounding ${\hat{E}}_{r2}$  when $-1<a_2<  (a_1-1-\sqrt{-a_1})/3$.
The proof   is completed.
\end{proof}

In the following  three  lemmas, in  order to give the exact number of limit cycles for      system \eqref{initial1},
 we    study   the existence of   large  limit cycles.
 Since the vector field of    system \eqref{initial1}  is symmetric about the origin,
  the existence of large  limit cycles is  only meaningful  for a large  limit cycle surrounding all equilibria.
\begin{lemma}
System \eqref{initial1}  has a unique  limit cycle  when $a_2\leq  -1$,
which is stable and large.
\label{lac72}
\end{lemma}
\begin{proof}
\begin{figure}
\centering
\centering
\subfigure[Invariant manifolds of   \eqref{initial1}  for $a_2\leq -1$]{
\scalebox{0.54}[0.54]{
\includegraphics{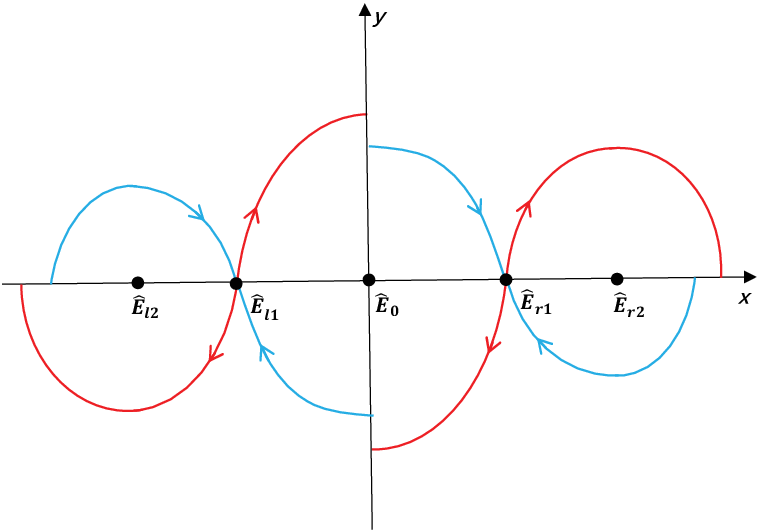}} }
\subfigure[A  large limit cycle    $\gamma_1$ of     \eqref{initial2}      for $a_2=-1$ ]{
\scalebox{0.55}[0.55]{
\includegraphics{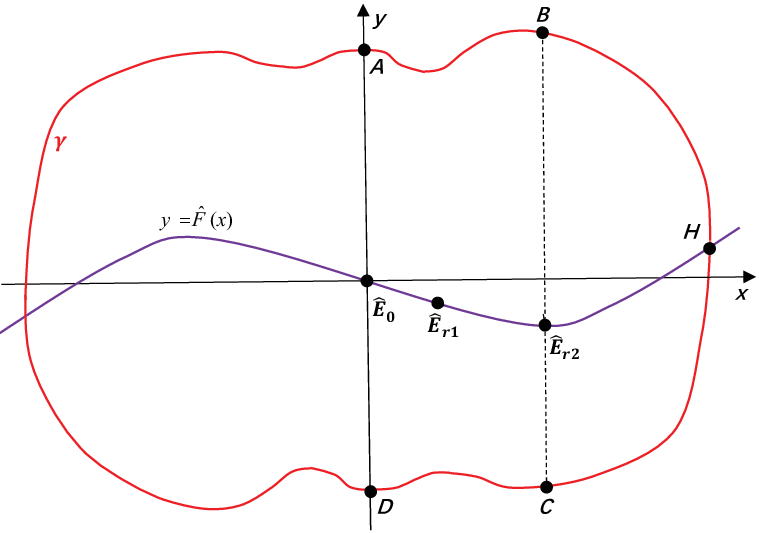}} }
\subfigure[Two  large limit cycle    $\Gamma_1$ and $\Gamma_2$ of       \eqref{initial2}      for $a_2<-1$]{
\scalebox{0.55}[0.55]{
\includegraphics{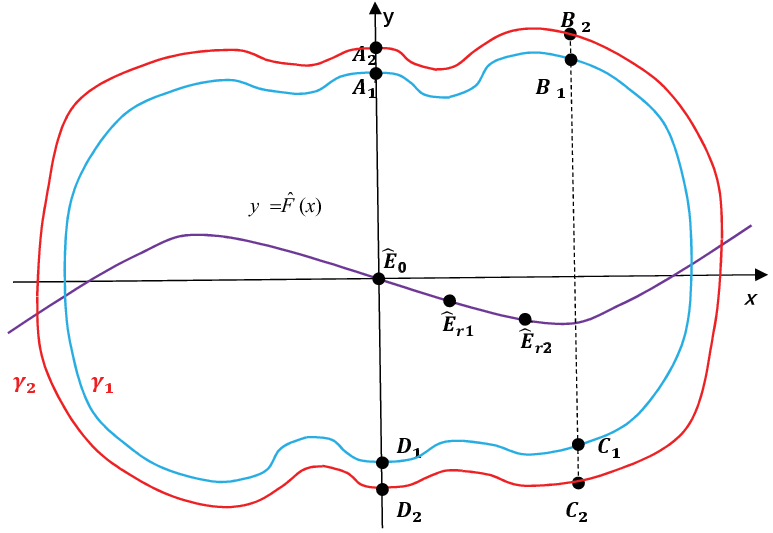}} }
\caption{The discussion of   large limit cycles of     \eqref{initial1}  for $a_2\leq -1$.  }
\label{atllc1}
\end{figure}
By Lemma \ref{fe1},
${\hat{E}}_0$, ${\hat{E}}_{l2}$, ${\hat{E}}_{r2}$  of  system \eqref{initial1} are unstable and ${\hat{E}}_{l1}$, ${\hat{E}}_{r1}$  of  system \eqref{initial1} are saddles when  $a_2\leq-1$.  From Lemma \ref{smc7},
 system \eqref{initial1}  has no small limit cycles when  $a_2\leq-1$.
Therefore, the stable and unstable manifolds of ${\hat{E}}_{l1}$ and ${\hat{E}}_{r1}$ of   system \eqref{initial1} are shown in Figure \ref{atllc1}(a).
Moreover,
since equilibria at infinity  of   system \eqref{initial1} are repelling by Proposition \ref{infty},    system \eqref{initial1}   has at least one large limit cycle
by the Poincar\'e-Bendixson Theorem.  So does  system \eqref{initial2}.

When $a_2=-1$,
assume that system   \eqref{initial2} exhibits a large limit cycle $\gamma_1$,
as shown in Figure \ref{atllc1}(b).
For  the statement   $\oint_{ \gamma_1 } \hat f(x) dt>0	$,
the proof is the same
	as the proof of Proposition~\ref{c5c6}. The details are omitted.  It means that   system \eqref{initial2} has at most one large limit cycle,
which is stable.   Combining the    existence of   large  limit cycles and the nonexistence of  small   limit cycles,    we directly  obtain that   system \eqref{initial2} has a unique  limit cycle  when $a_2= -1$,
which is stable and large.

When $a_2<-1$, suppose that system   \eqref{initial2} exhibits at least two large limit cycles, where $\Gamma_1$,
  $\Gamma_2$ are the two innermost limit cycles, and $\Gamma_1$ lies in the interior of $\Gamma_2$,
 see Figure \ref{atllc1}(c).
  As proven in  Proposition
\ref{c5c6}, we get
\[
\oint_{\Gamma_1}\hat f(x)dt>\oint_{\Gamma_2}\hat f(x)dt.
\]
Moreover,
since the outmost limit cycle is externally stable and the innermost one is internally stable,
system   \eqref{initial2} has also at most one large limit cycle, which is stable.
Based on  the    existence of   large  limit cycles and the nonexistence of  small   limit cycles,  system \eqref{initial2} has a unique  limit cycle  when $a_2<-1$,
which is stable and large. The proof   is completed.
\end{proof}

\begin{lemma}
System \eqref{initial1} has at most two large limit cycles  when $-1<a_2 \leq  {(a_1-1 -\sqrt{-a_1})}/{3}$, and
at most one
large limit cycle   when  $a_1/3 \leq a_2<0$.
\label{lac7}
\end{lemma}
\begin{proof}
\begin{figure}
\centering
\subfigure[Two  large limit cycle    $\Gamma_1$ and $\Gamma_2$ of       \eqref{initial2}   for $-1<a_2\leq  \frac{a_1-1-\sqrt{-a_1}}{3}$]{
\scalebox{0.55}[0.55]{
\includegraphics{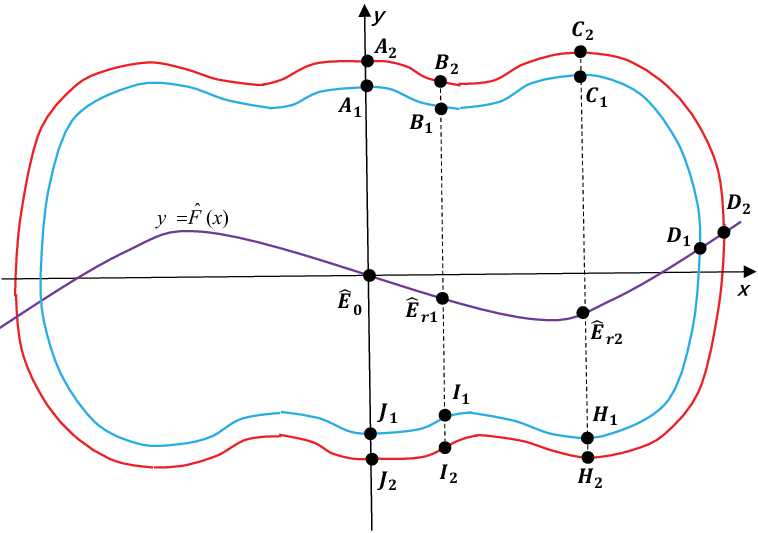}} }
\subfigure[A large limit cycle    $\gamma$  of       \eqref{initial2} for $\frac{a_1}{3}\leq a_2 <0$]{
\scalebox{0.58}[0.58]{
\includegraphics{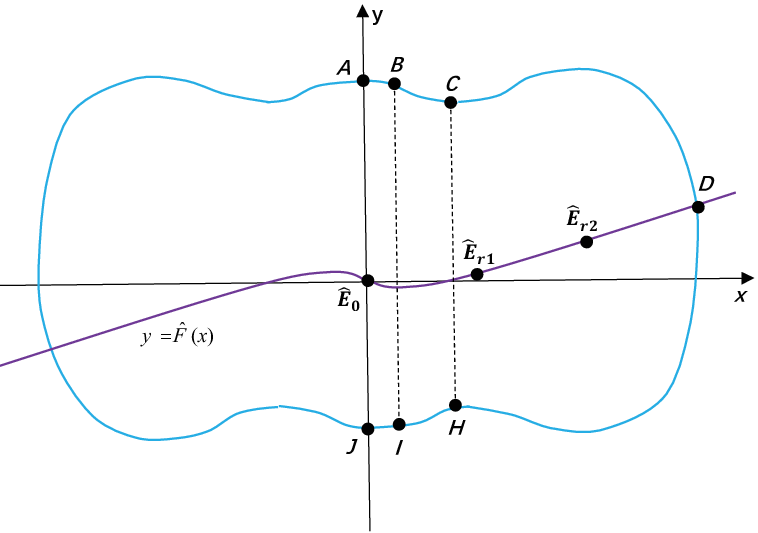}} }
\subfigure[The  large limit cycle    $\gamma$  in   \eqref{tsy1}]{
\scalebox{0.55}[0.55]{
\includegraphics{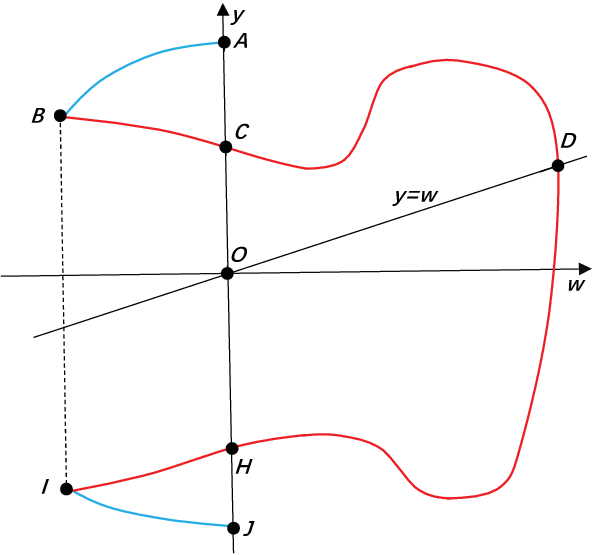}} }
\caption{Discussion of   large limit cycles of      \eqref{initial2}  for $a_2 \in(-1,\frac{a_1-1-\sqrt{-a_1}}{3}]\cup [\frac{a_1}{3}, 0)$.}
\label{atllc}
\end{figure}
When $-1<a_2 \leq  {(a_1-1 -\sqrt{-a_1})}/{3}$,
suppose that system \eqref{initial2}    exhibits at least two large limit cycles $\Gamma_1$ and $\Gamma_2$,
where $\Gamma_1$ lies in the interior of $\Gamma_2$, as shown in Figure \ref{atllc}(a).
Our task now  is to  prove
\begin{eqnarray}
\oint_{\Gamma_1}\hat f(x)dt<\oint_{\Gamma_2}\hat f(x)dt,
\label{compare1}
\end{eqnarray}
which implies that  system \eqref{initial2} has at most two      large limit cycles.
Let $y=y_1(x)$ and $y=y_2(x)$ be respectively $\widehat{A_1B_1}$
and $\widehat{A_2B_2}$.
Then,
  we have
\begin{eqnarray}
\begin{aligned}
&\int_{ \widehat{A_1B_1}} \hat f(x)dt-\int_{\widehat{A_2B_2}}\hat f(x)dt
&&=\int_0^{\sqrt{-a_1}}\frac{\hat f(x)}{y_1-\hat F(x)}dx-\int_0^{\sqrt{-a_1}}\frac{\hat f(x)}{y_2-\hat F(x)}dx
\\
&&&=\int_0^{\sqrt{-a_1}}\frac{ \hat f(x)(y_2-y_1)}{(y_1- \hat F(x))(y_2- \hat F(x))}dx
\\
&&&<0,
\end{aligned}
\label{compare1a}
\end{eqnarray}
 because   $\hat f(x)<0$ for $x\in(0,\sqrt{-\tilde h})$ when $-1<a_2 \leq  {(a_1-1 -\sqrt{-a_1})}/{3}$.
Similarly, we obtain
\begin{eqnarray}
\int_{ \widehat{I_1J_1}}\hat f(x) dt-\int_{\widehat{I_2J_2}} \hat f(x) dt<0.
\label{compare1b}
\end{eqnarray}
As proven in Lemma 3.4 of \cite{CC18}, we get
\begin{eqnarray}
\int_{ \widehat{B_1C_1}} \hat f(x) dt<\int_{\widehat{B_2C_2}}\hat f(x) dt
~{\rm and}~
\int_{ \widehat{H_1I_1}} \hat f(x) dt<\int_{\widehat{H_2I_2}} \hat f(x) dt.
\label{comp31}
\end{eqnarray}
As proven in  Proposition~\ref{c5c6},
$[\hat F(x)-\hat F(\sqrt{-h})]\hat f(x) /\hat g(x)$ is increasing
for $x\in(\sqrt{-h}, +\infty)$.
Then, by the proof of Theorem 2.1 of \cite{DR} or Lemma 4.5 of \cite[Chapter 4]{Zh},
\begin{eqnarray}
 \int_{\widehat{C_1 D_1H_1}}\hat f(x)dt-\int_{\widehat{C_2D_2H_2}}\hat f(x)dt
<0.
\label{comp32}
\end{eqnarray}
By (\ref{compare1a}-\ref{comp32}),
it follows that \eqref{compare1} holds. Therefore,   system \eqref{initial2} has at most two large limit cycles when $-1<a_2 \leq  {(a_1-1 -\sqrt{-a_1})}/{3}$.  So does system  \eqref{initial1}.

When $ a_1/3 \leq a_2<0$,
    assume that system \eqref{initial2}  exhibits a  large limit cycle $\gamma$,
as shown in Figure \ref{atllc}(b).
Similarly,
 $\gamma$ of  system   \eqref{initial2}  in the positive half-plane
is changed into the orbit segments $\widehat{BA}\cup\widehat{BDI}\cup\widehat{IJ}$ of equations
\eqref{tsy1}, where $\lambda_i(w)=\hat g(x_i(w))/\hat f(x_i(w))$ and $i=1,2$, as shown in Figure \ref{atllc}(c).
Letting $y=\hat y_1(w)$ and $y=\hat y_2(w)$ represent respectively the orbit segments
$\widehat{BA}$ and $\widehat{BC}$ in the $wy$-plane, we get $\hat y_1(w)>\hat y_2(w)$.
 On the one hand,
 \begin{eqnarray}
 \begin{aligned}
& \int_{\widehat{BC}}\hat f(x)dt-\int_{\widehat{BA}}\hat f(x)dt&&=
 \int_{\frac{2\delta a_2\sqrt{-a_2}}{3}}^0\frac{dw}{  y_2-w}-\int_{\frac{2\delta a_2\sqrt{-a_2}}{3}}^0\frac{dw}{ y_1-w}
 \\
 &&&= \int_{\frac{2\delta a_2\sqrt{-a_2}}{3}}^0\frac{(\hat y_1- \hat y_2)}{(\hat y_1-w)(\hat  y_2-w)}dw
\\
&&&>0.
\end{aligned}
\label{ic1}
\end{eqnarray}
Similarly, we can obtain
\begin{eqnarray}
 \int_{\widehat{HI}}\hat f(x)dt-\int_{\widehat{JI}}\hat f(x)dt>0.
 \label{ic2}
\end{eqnarray}
On the other hand, it is clear that
\begin{eqnarray}
 \int_{\widehat{CDH}}\hat f(x)dt>0.
 \label{ic3}
\end{eqnarray}
By (\ref{ic1}-\ref{ic3}), $\oint_{\gamma}\hat f(x)dt>0$.
Thus,  system \eqref{initial2}  has at most one  large limit cycle  for  $ {a_1/3} \leq a_2<0$.
So does system  \eqref{initial1}.   The proof is completed.
\end{proof}

\begin{lemma}
	System \eqref{initial1} has
no
	large limit cycles   when  $-1/3\leq a_1<0$ and $-1/3\leq a_2<0$,
 at most one large limit cycle  when $a_1\le a_2< a_1/3$ and $-1<a_1<-1/3$.
	\label{lac8}
\end{lemma}

\begin{figure}
	\centering
	\subfigure[in $xy$-plane ]{
		\scalebox{0.55}[0.55]{
			\includegraphics{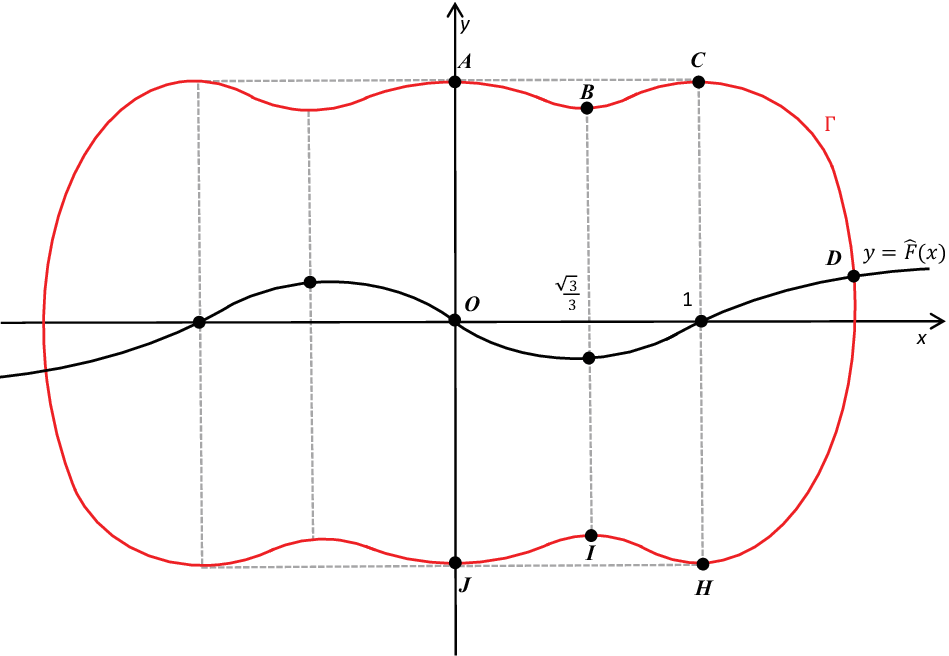}} }
	\subfigure[in $wy$-plane]{
		\scalebox{0.55}[0.55]{
			\includegraphics{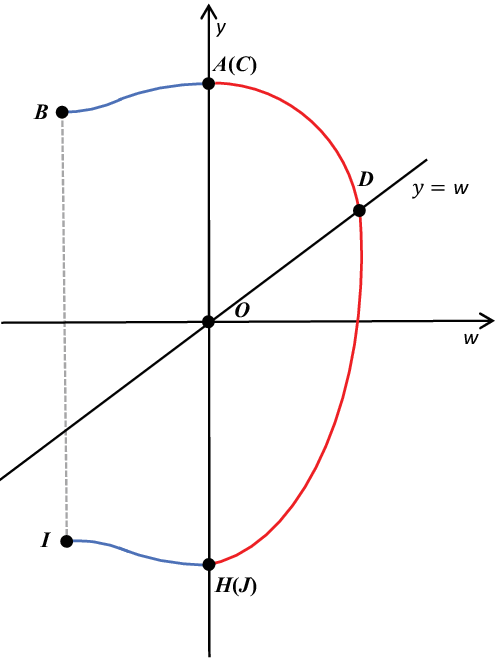}} }
	\caption{A
		large limit cycle $\Gamma$ of   \eqref{initial2}  for $a_1=a_2=-1/3$.}
	\label{atllc11}
\end{figure}
\begin{proof}
	Firstly,
	assume that 	system \eqref{initial1} has
	a
	large limit cycle $\Gamma$
	for $a_1=a_2=-1/3$, as shown in Figure \ref{atllc11}(a).
We claim that $\lambda_1(w)\equiv\lambda_2(w)$ in \eqref{tsy1}, where $w=F(x)$ for $x>0$.
	On the one hand,
when	 $\hat F(x_1)=\hat F(x_2)$,
we have ${x_1}^2+x_1x_2+{x_2}^2=1$,
where $0<x_1<\sqrt{3}/3<x_2<1$.
On the other hand,
when the second equation of \eqref{FFF11} holds, we also have ${x_1}^2+x_1x_2+{x_2}^2=1$.
 Then, the assertion is proven.
 Thus, the two equations  in \eqref{tsy1} are same.
 Letting $y=\hat y_1(w)$ and $y=\hat y_2(w)$ represent respectively the orbit segments
 $\widehat{BA}$ and $\widehat{BC}$ in the $wy$-plane, it follows from \eqref{tsy1}  that $\hat y_1(w)\equiv\hat y_2(w)$.
 In the $wy$-plane, the two orbit segments
 $\widehat{BA}$ and  $\widehat{BC}$ coincide,
 and the two orbit segments
 $\widehat{IH}$ and  $\widehat{IJ}$ coincide, as shown in Figure \ref{atllc11}(b).
 By the Green's formula,
 \begin{eqnarray*}
 	\oint_{\Gamma}(y-\hat F(x))dy+g(x)dx&=&2\iint_{\mathcal{S}} \hat f(x)dxdy
 	\\
 &=&	2\left(\iint_{\mathcal{S}_1} dwdy- 	\iint_{\mathcal{S}_2} dwdy\right)
  	\\
 &=& 2(\Omega(\mathcal{S}_1)-\Omega(\mathcal{S}_2))=2\Omega(\mathcal{S}_3)>0,
 	\end{eqnarray*}
 where $\mathcal{S}$ is the domain between the $y$-axis and the orbit segment $\widehat{ADJ}$ in the $xy$-plane,
 $\mathcal{S}_1$ is the domain between $w=\hat F(\sqrt{3}/3)$ and the orbit segment
 $\widehat{BDI}$ in the $wy$-plane,
$\mathcal{S}_2$ is the domain between the $y$-axis, $w=\hat F(\sqrt{3}/3)$ and the orbit segments $\widehat{BA}$,
$\widehat{IH}$ in the $wy$-plane,
$\mathcal{S}_3$ is the domain between the $y$-axis and the orbit segment  $\widehat{CDH}$ in the $wy$-plane.
This contradicts $\oint_{\Gamma}(y-\hat F(x))dy+g(x)dx=0$.
Thus, system \eqref{initial1} has no large limit cycles
for $a_1=a_2=-1/3$.

\begin{figure}
	\centering
	\subfigure[in $xy$-plane ]{
		\scalebox{0.55}[0.55]{
			\includegraphics{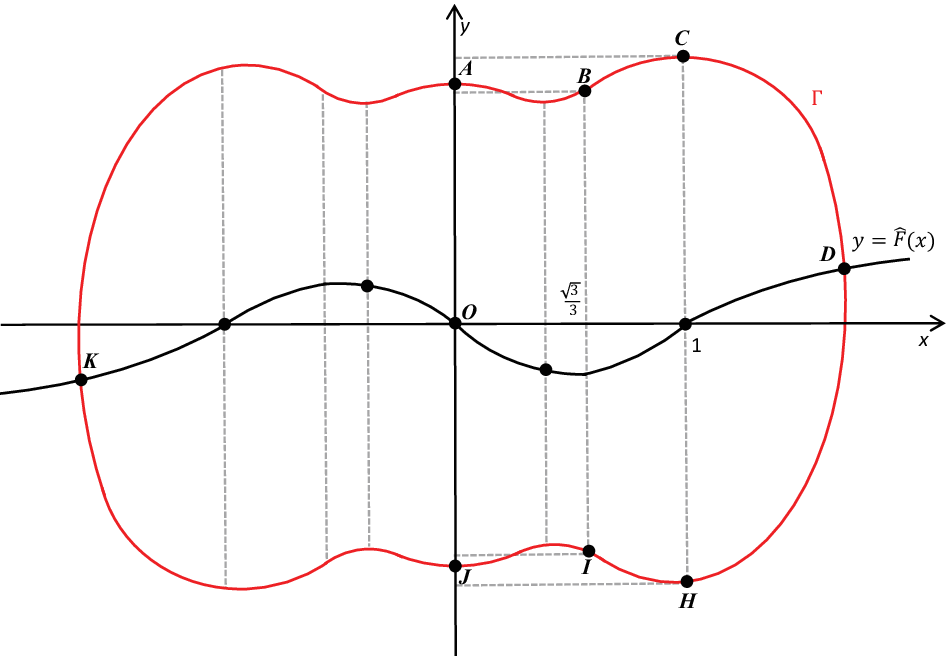}} }
	\subfigure[in $wy$-plane]{
		\scalebox{0.55}[0.55]{
			\includegraphics{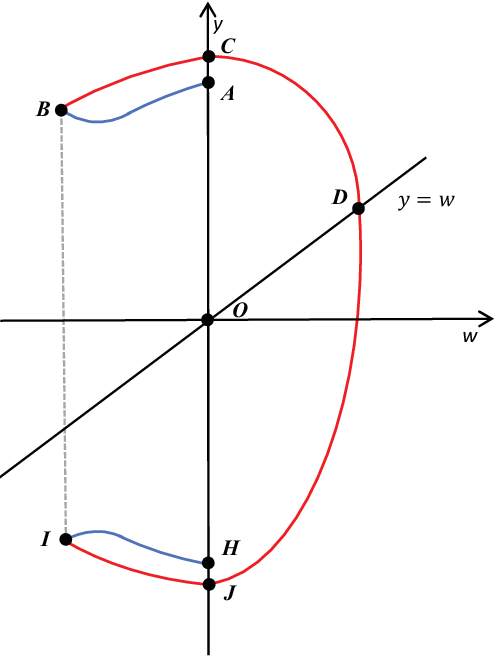}} }
	\caption{A
		large limit cycle $\Gamma$ of   \eqref{initial2}  for $-1/3<a_1<0$ and $a_2=-1/3$.}
	\label{atllc12}
\end{figure}
Secondly,
assume that 	system \eqref{initial1} has
a
large limit cycle $\Gamma$
for $-1/3<a_1<0$ and $a_2=-1/3$, as shown in Figure \ref{atllc12}(a).
When	 $\hat F(x_1)=\hat F(x_2)$,
we have ${x_1}^2+x_1x_2+{x_2}^2=1$, where $0<x_1<\sqrt{3}/3<x_2<1$.
Further, when the second equality of \eqref{FFF11} holds, by ${x_1}^2+x_1x_2+{x_2}^2=1$
we also have $x_1x_2(x_1x_2+1)=0$, which contradicts 	$0<x_1<\sqrt{3}/3<x_2<1$.
Then, we can obtain that $\lambda_1(w)<\lambda_2(w)$ in \eqref{tsy1}.
Letting $y=\hat y_1(w)$ and $y=\hat y_2(w)$ represent respectively the orbit segments
$\widehat{BA}$ and $\widehat{BC}$ in the $wy$-plane, it follows from \eqref{tsy1} and the Comparison Theorem that $\hat y_1(w)<\hat y_2(w)$.
See Figure \ref{atllc12}(b).
Then, $\mathcal{S}_2$ is a subset of $\mathcal{S}_1$.
 By the Green's formula,
\begin{eqnarray*}
	\oint_{\Gamma}(y-\hat F(x))dy+g(x)dx&=&2\iint_{\mathcal{S}} \hat f(x)dxdy
	\\
	&=&	2\left(\iint_{\mathcal{S}_1} dwdy- 	\iint_{\mathcal{S}_2} dwdy\right)
	\\
	&=& 2(\Omega(\mathcal{S}_1)-\Omega(\mathcal{S}_2))>0,
\end{eqnarray*}
which contradicts  $\oint_{\Gamma}(y-\hat F(x))dy+g(x)dx=0$.
Thus, system \eqref{initial1} has
no
large limit cycles
for $-1/3<a_1<0$ and $a_2=-1/3$.
\begin{figure}
	\centering
		\scalebox{0.55}[0.55]{
			\includegraphics{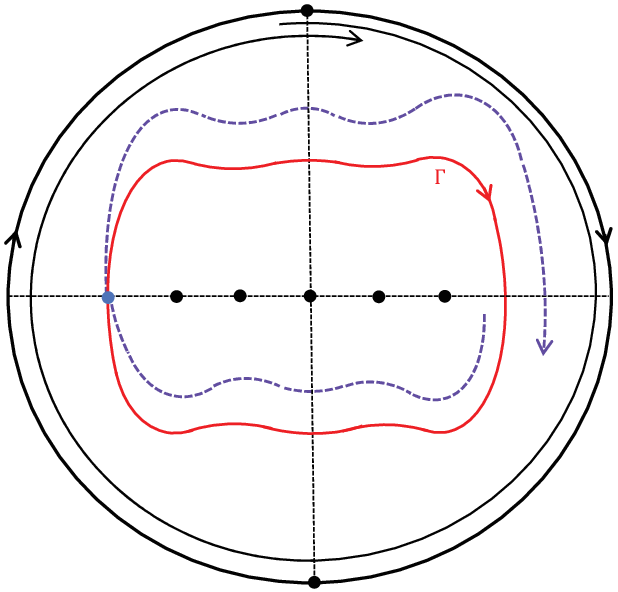}}
	\caption{Discussion on large limit cycles
		  of   \eqref{initial2}  for   $-1/3<a_2<a_1<0$.}
	\label{PB11}
\end{figure}

  Thirdly, assume that 	system \eqref{initial1} has
a
large limit cycle
for $(a_1,a_2)=(a^*,b^*)$,
where $-1/3\le a^*<0$, $-1/3<b^*<0$, and $\Gamma$  is the outermost large limit cycle.
Since the vector field of system \eqref{initial1} is rotated on $a_2$,
 $\Gamma$ is broken when $(a_1,a_2)=(a^*,-1/3)$,
as shown Figure \ref{PB11}.
By the  annulus region of Poincar\'e-Bendixson Theorem, there is a stable limit cycle surrounding $\gamma$ when $(a_1,a_2)=(a^*,-1/3)$,
which contradicts the nonexistence of large limit cycles.

\begin{figure}
	\centering
	\subfigure[in $xy$-plane ]{
		\scalebox{0.55}[0.55]{
			\includegraphics{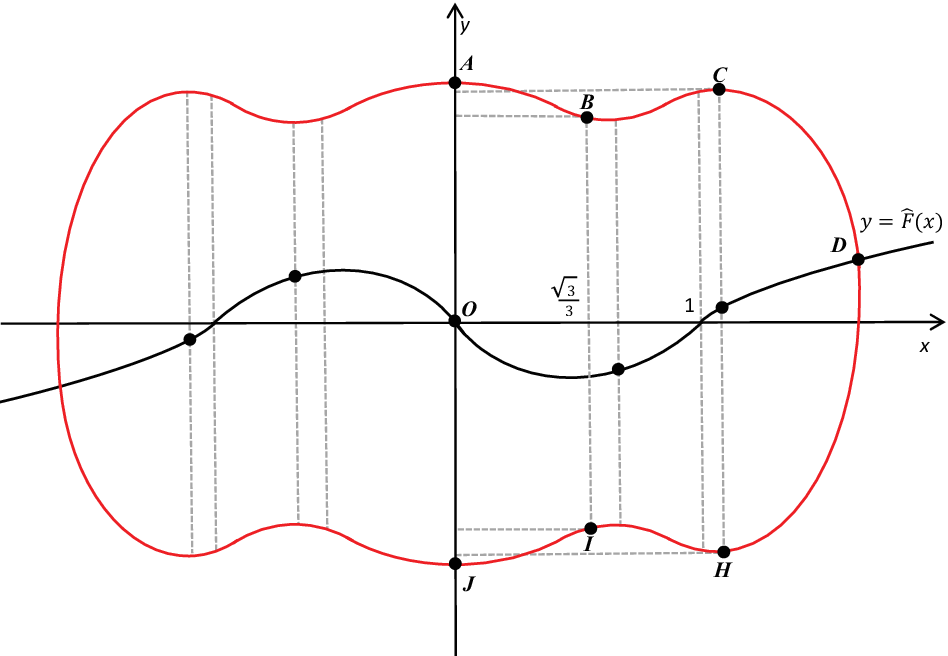}} }
	\subfigure[in $wy$-plane]{
		\scalebox{0.55}[0.55]{
			\includegraphics{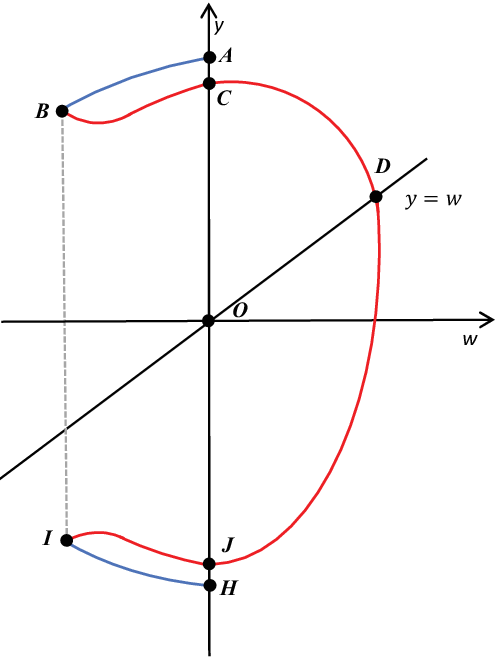}} }
	\caption{A
		large limit cycle $\Gamma$ of   \eqref{initial2}  for $a_1\leq a_2<a_1/3$ and $-1<a_1<-1/3$.}
	\label{atllc13}
\end{figure}

Finally,
assume that 	system \eqref{initial1} has
a
large limit cycle $\Gamma$
for $a_1\leq a_2<a_1/3$ and $-1<a_1<-1/3$, as shown in Figure \ref{atllc13}(a).
We claim that $\lambda_1(w)>\lambda_2(w)$ in \eqref{tsy1}, where $w=F(x)$ for $x>0$.
When \eqref{FFF11} holds,
we have
$${x_1}^2+x_1x_2+{x_2}^2=-3a_2,~~~~~~~~\frac{{x_1}^2{x_2}^2+x_1x_2-a_2(3a_2+1)}{({x_1}^2+a_2)({x_2}^2+a_2)}+\frac{3a_2+1}{a_2-a_1}=0,$$
where $0<x_1<\sqrt{-a_1}<x_2<1$.
 Let
$ \kappa:=x_1+x_2$.
Then,
$ x_1x_2=\kappa^2+3a_2>0$ and $({x_1}^2+a_2)({x_2}^2+a_2)<0$,
 implying that $\kappa\in(\sqrt{-3a_2}, 2\sqrt{-a_2})$.
Define
$h(\sigma):=
(2a_2+1-a_1)\sigma^2+((a_2-a_1)(6a_2+1)+
5a_2(3a_2+1))\sigma+2a_2(3a_2+1)(3a_2-a_1).$
  Notice that
  \[
  -\frac{(a_2-a_1)(6a_2+1)+5a_2(3a_2+1)}{2(2a_2+1-a_1)}+4a_2=\frac{(a_2-a_1)(2a_2-11)-7{a_2}^2-a_2}{2(2a_2+1-a_1)}<0.
  \]
We check that
$
    h(-4a_2)=2a_2(1+a_2)(a_1-a_2)>0,
$
    implying $h(\kappa^2)>0$ for $\kappa\in(\sqrt{-3a_2}, 2\sqrt{-a_2})$.
Then, the assertion is proven.
Letting $y=\hat y_1(w)$ and $y=\hat y_2(w)$ represent respectively the orbit segments
$\widehat{BA}$ and $\widehat{BC}$ in the $wy$-plane, it follows from \eqref{tsy1} and the Comparison Theorem that $\hat y_1(w)>\hat y_2(w)$.
We also claim that $y_D>0$, where $y_D$ is the ordinate of $D$. See Figure \ref{atllc13}(b).
Otherwise, $\mathcal{S}_1$ is a subset of $\mathcal{S}_2$, implying similarly the nonexistence of large limit cycles.
As proven in the case of $a_1/3\leq a_2<0$ of Lemma \ref{lac7}, we can similarly obtain that
  system \eqref{initial1} has
at most one
large limit cycle
for $a_1\le a_2< a_1/3$ and $-1<a_1<-1/3$. Moreover, the limit cycle is stable and hyperbolic if it exists.
\end{proof}

\begin{lemma}
System \eqref{initial1}  has at most four large limit cycles  when  $(a_1-1-\sqrt{-a_1})/3<a_2< \min\{a_1,-1/3\}$ and $-1<a_1<0$.
\label{lac73}
\end{lemma}

\begin{proof}
 Firstly, we consider that $\delta>0$ is sufficiently small. By Appendix C,  system \eqref{initial1}  has at most four large limit cycles.

Secondly, we consider that $\delta>0$ is not small and let $\beta:=\delta a_2$.
	Assume that system \eqref{initial1}  has  at least five large limit cycles
for $(a_1,\beta,\delta)=(\alpha_0,\beta_0,\delta_0)$,
where   $(\alpha_0-1-\sqrt{-\alpha_0})/3<\beta_0< \min\{\alpha_0,-1/3\}$ and $-1<\alpha_0<0$.
Let $\Gamma_1,\ldots,\Gamma_5$ be the outermost five limit cycles in order, where $\Gamma_1$ lies in the interior
regions surrounded by $\Gamma_2$ and $\Gamma_5$ is externally stable.
Notice that
the vector field of system \eqref{initial1}  is rotated with respect to $\beta$ and $\delta$.
Then, we can obtain that stable limit cycles expand and unstable ones contract as one of  $\beta,\delta$ decreases.
Next, we adapt the following steps.
\begin{description}
	\item[(i)] When $a_1:=\alpha_0$ and $\beta:=\beta_0$ are fixed, we lessen $\delta$ and
	claim that 		there exists $\delta_1\in(0,\delta_0)$
	such that either $\Gamma_1$ and $\Gamma_2$ coincide, or $\Gamma_3$ and $\Gamma_4$ coincide, or  $\Gamma_1$ and an interior limit cycle coincide,
	and
	\[
(\alpha_0-1-\sqrt{-\alpha_0})\delta_1/3<\beta_0< \min\{\alpha_0\delta_1,-\delta_1/3\}.
	\]
	Otherwise, this is a contradiction since
	system \eqref{initial1} has at most one large limit cycles   when $\beta_0\geq \min\{\alpha_0\delta_1,-\delta_1/3\}$ by Lemma \ref{lac8}.
	
	\item[(ii)] 	When $a_1:=\alpha_0$ and $\delta=\delta_1$ are fixed, we increase $\beta$ and can find
	a value $\beta=\beta_1\in((\alpha_0-1-\sqrt{-\alpha_0})\delta_1/3, \beta_0)$
	such that either $\Gamma_2$ and $\Gamma_3$ coincide, or  $\Gamma_4$ and $\Gamma_5$ coincide.
	Otherwise, if
	\[
\beta_1\leq (\alpha_0-1-\sqrt{-\alpha_0})\delta_1/3
	\]
	this is a contradiction since
	system  \eqref{initial1} has at most two large limit cycles by Lemma \ref{lac7}.
	
	\item[(iii)] When $a_1:=\alpha_0$ and $\beta:=\beta_1$ are fixed, we lessen $\delta$ and
	claim that there is a value $\delta=\delta_2\in(0,\delta_1)$
	such that either $\Gamma_2$ and $\Gamma_3$ coincide, or  $\Gamma_1$ and an interior limit cycle coincide,
	and
	\[
(\alpha_0-1-\sqrt{-\alpha_0})\delta_2/3<\beta_1< \min\{\alpha_0\delta_2,-\delta_2/3\}.
	\]
	Otherwise, this is a contradiction since
	system  \eqref{initial1} has at most one large limit cycle   when $\beta_1\leq -\alpha_0-\delta_2$ by Lemma \ref{lac72}.
	
	\item[(iv)] Repeating the aforementioned steps, there is an integer $n$ such that $\delta_n$ is small in the $2n+1$-th step.
	Moreover, system  \eqref{initial1} has at least four large limit cycles   in the moment and one of them is semi-stable.
	However, when there is a semi-stable large limit cycle,
system \eqref{initial1} has at most three large limit cycles with sufficiently small $\delta>0$
by Appendix C. It induces a contradiction.
\end{description}
Therefore, system   \eqref{initial1}  cannot  have  five large limit cycles 	for  $(a_1-1-\sqrt{-a_1})/3<a_2< \min\{a_1,-1/3\}$ and $-1<a_1<0$.
\end{proof}

By Lemmas \ref{smc7} and \ref{lac73}, there are at most four large limit cycles and no small limit cycles for   $(a_1-1-\sqrt{-a_1})/3<a_2< \min\{a_1,-1/3\}$ and $-1<a_1<0$.
By the analysis of Appendix C and numerical simulations, we conjecture that there are at most two large limit cycles for   $(a_1-1-\sqrt{-a_1})/3<a_2< \min\{a_1,-1/3\}$ and $-1<a_1<0$.
Based on the conjecture and other results, we can obtain the following proposition.

 \begin{figure}
 	\centering
 	\subfigure[as   $a_2\leq -1$]{
 		\scalebox{0.45}[0.45]{
 			\includegraphics{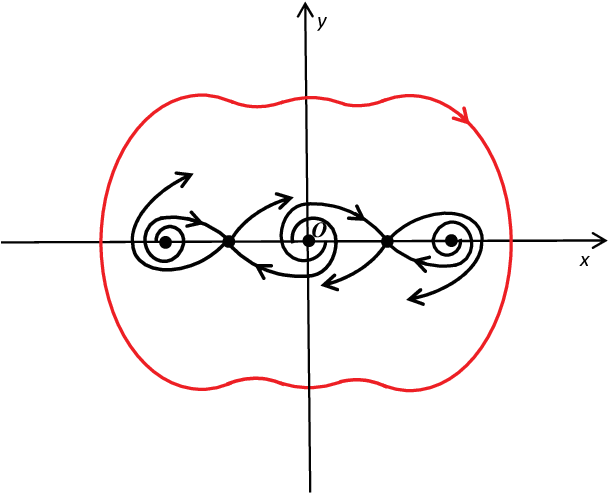}} }
 	\subfigure[as $a_2=\varphi_2(a_1,\delta) $  ]{
 		\scalebox{0.45}[0.45]{
 			\includegraphics{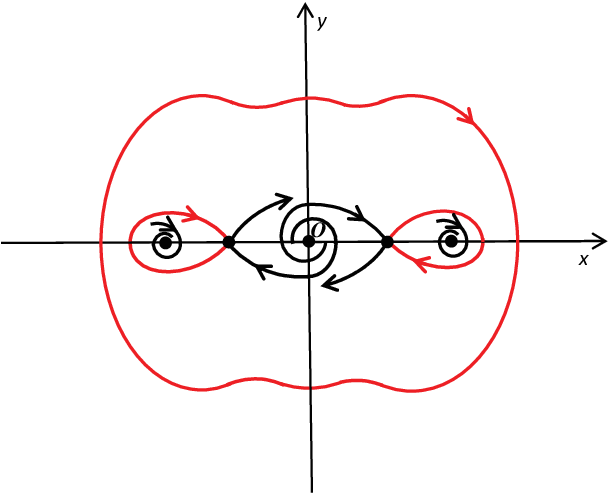}} }
 		 	\subfigure[as $-1<a_2<\varphi_2(a_1,\delta) $ ]{
 			\scalebox{0.45}[0.45]{
 				\includegraphics{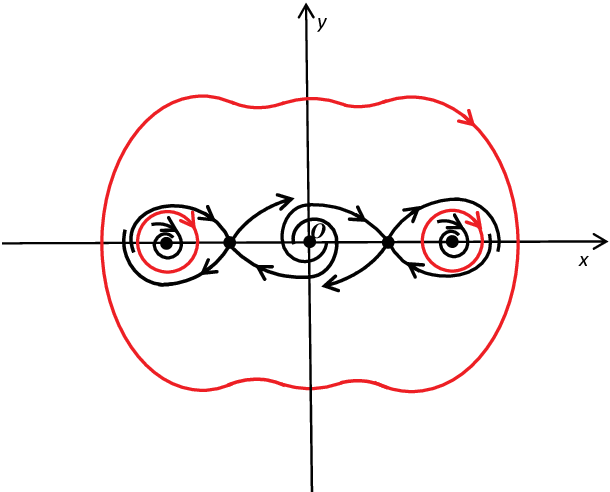}} }
 			 	\subfigure[as   $\varphi_2(a_1,\delta)<a_2< \varphi_3(a_1,\delta)$]{
 				\scalebox{0.45}[0.45]{
 					\includegraphics{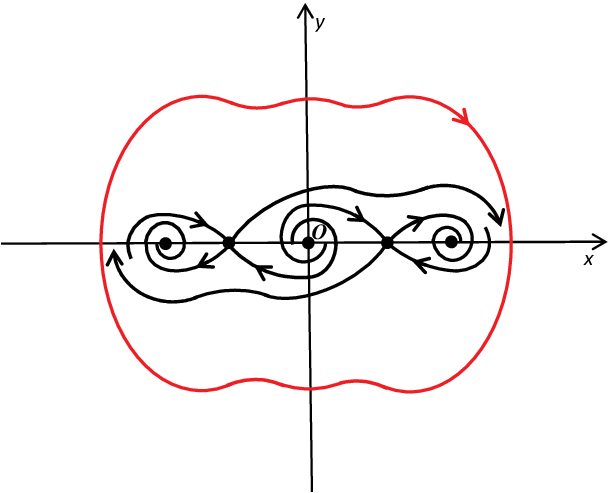}} }
 			\subfigure[as $a_2=\varphi_3(a_1,\delta) $, $-1<a_1\leq a^*$ ]{
 				\scalebox{0.45}[0.45]{
 					\includegraphics{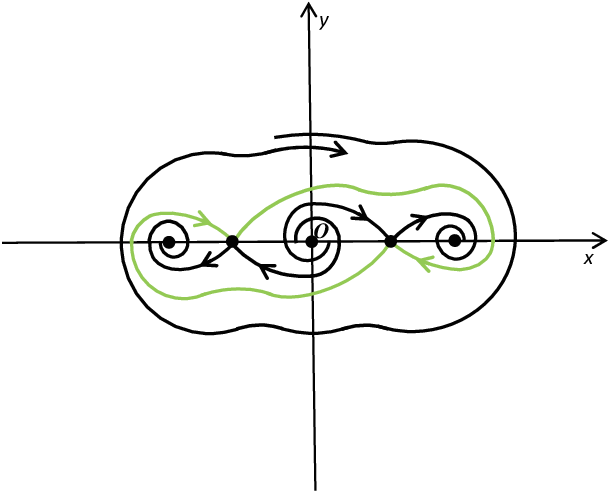}} }
 			\subfigure[as $a_2=\varphi_3(a_1,\delta)$, $a^*<a_1<0$ ]{
 				\scalebox{0.45}[0.45]{
 					\includegraphics{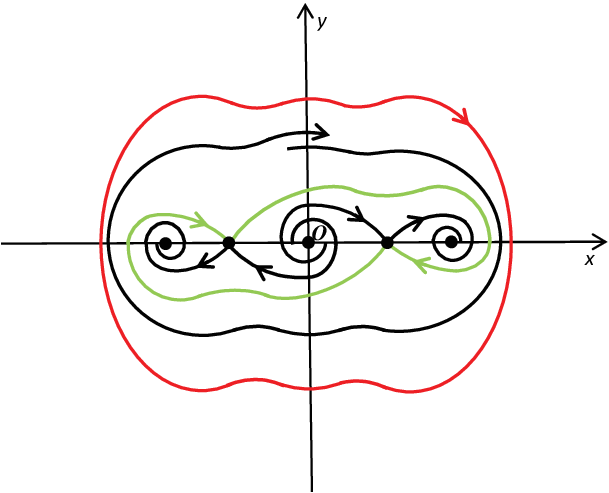}} }
 				 	\subfigure[as   $\varphi_2(a_1,\delta)  <a_2<\varphi_5(a_1,\delta)$, $a^*<a_1<0$]{
 					\scalebox{0.45}[0.45]{
 						\includegraphics{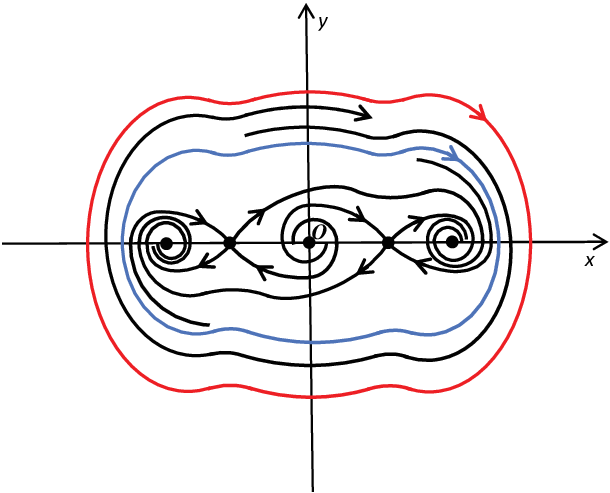}} }
 				\subfigure[as $a_2=\varphi_5(a_1,\delta)$, $a^*<a_1<0$]{
 					\scalebox{0.45}[0.45]{
 						\includegraphics{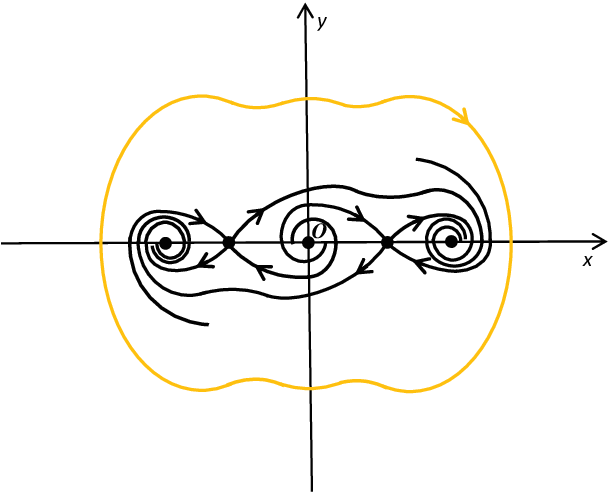}} }
 				\subfigure[as $-1<a_2<\varphi_2(a_1,\delta) $ ]{
 					\scalebox{0.45}[0.45]{
 						\includegraphics{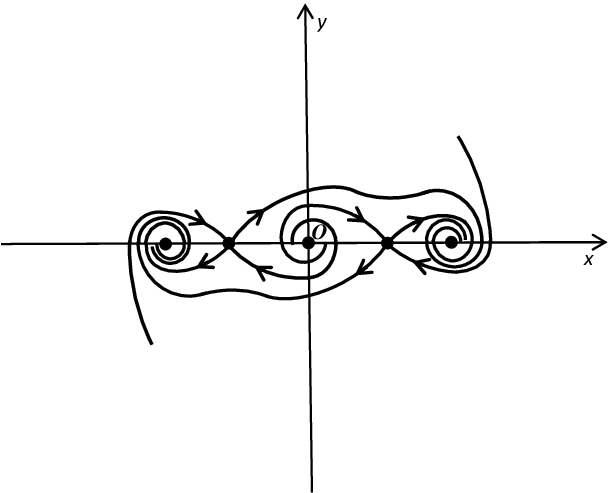}} }
 					 	\subfigure[as   $a_2=\varphi_4(a_1,\delta) $]{
 						\scalebox{0.45}[0.45]{
 							\includegraphics{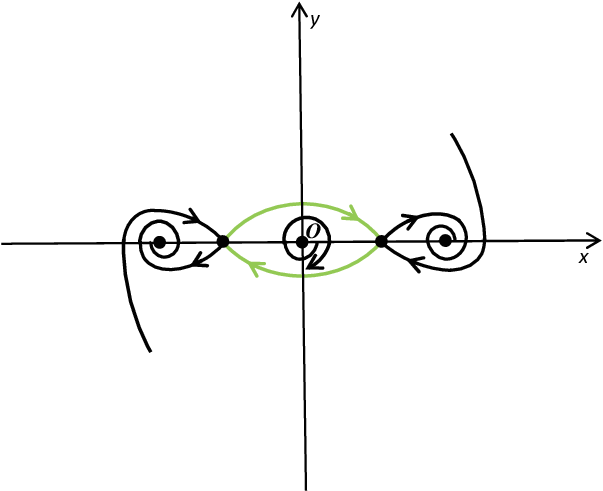}} }
 					\subfigure[as $\varphi_4(a_1,\delta)  <a_2<0$ ]{
 						\scalebox{0.45}[0.45]{
 							\includegraphics{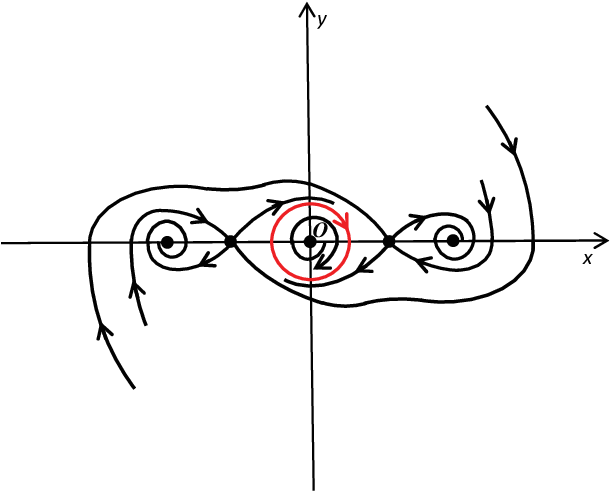}} }
 	\caption{All phase portraits of   \eqref{initial1}  when $-1<a_1<0$ and $a_2<0$.}
 	\label{five}
 \end{figure}
\begin{proposition}
There are four continuous functions $\varphi_2(a_1,\delta), \varphi_3(a_1,\delta), \varphi_4(a_1,\delta),\varphi_5(a_1,\delta)$
such that the following statements hold:
\begin{description}
\item[(i)] system  \eqref{initial1}  has   a unique  limit cycle  when $a_2\leq -1$,
which is stable and large{\rm {;}}
  \item[(ii)] system \eqref{initial1} has two  homoclinic loops    and  one large limit cycle  when
$a_2=\varphi_2(a_1,\delta) $ and $-1<a_1<0$, where the  homoclinic loops   are unstable and   the large limit cycle
is stable{\rm {;}}
 \item[(iii)] system  \eqref{initial1}  has  two small limit cycles and one large limit cycle when
$-1<a_2<\varphi_2(a_1,\delta) $,  where the small ones are unstable and the large one
is stable{\rm {;}}
  \item[(iv)] system \eqref{initial1} has one stable two-saddle  loop surrounding all of ${\hat{E}}_0, \hat E_{l_2}, {\hat{E}}_{r2}$ and no limit cycles  when
$a_2=\varphi_3(a_1,\delta) $ and $-1<a_1\leq a^*${\rm {;}}
  \item[(v)] system \eqref{initial1} has one unstable two-saddle  loop surrounding all of ${\hat{E}}_0, \hat E_{l_2}, {\hat{E}}_{r2}$ and exactly one  stable large  limit cycle when
$a_2=\varphi_3(a_1,\delta) $ and $a^*<a_1<0${\rm {;}}
  \item[(vi)] system \eqref{initial1} has two   large limit cycles  when
$\varphi_2(a_1,\delta)<a_2< \varphi_3(a_1,\delta)$ and $-1<a_1<0$, where the  homoclinic loops   are unstable and   the large limit cycle
is stable{\rm {;}}
  \item[(vii)] system \eqref{initial1} has   two    limit cycle  when
$\varphi_2(a_1,\delta)  <a_2<\varphi_5(a_1,\delta)$   and $a^*<a_1<0$,  where they are large, the inner one is unstable
and the outer one is stable{\rm {;}}
 \item[(viii)] system \eqref{initial1} has  a unique limit cycle   when
$a_2=\varphi_5(a_1,\delta) $  and $a^*<a_1<0$,  which is semi-stable and large{\rm {;}}
 \item[(ix)] system \eqref{initial1} has  no  limit cycles   when
$\varphi_3(a_1,\delta)  <a_2<\varphi_4(a_1,\delta)$   and $-1<a_1\leq a^*$ or $\varphi_5(a_1,\delta)  <a_2<\varphi_4(a_1,\delta)$   and $a^*<a_1<0${\rm {;}}
  \item[(x)] system \eqref{initial1} has  one  two-saddle  loop only surrounding   ${\hat{E}}_0$  if and only if
$a_2=\varphi_4(a_1,\delta) $,  which is  stable {\rm {;}}
  \item[(xi)] system \eqref{initial1} has two limit cycle   when
$\varphi_4(a_1,\delta)  <a_2<0$   and $-1<a_1<0$, where  they are small and   stable{\rm {;}}
\end{description}
where $\varphi_3(a^*,\delta)=a^*$, $-1<\varphi_2(a_1,\delta)  <\varphi_3(a_1,\delta)< \varphi_4(a_1,\delta)<0$ and $\varphi_2(a_1,\delta)  <-1/3< \varphi_4(a_1,\delta)$
for $-1<a_1<0$ and $\varphi_3(a_1,\delta) <\varphi_5(a_1,\delta)  <-1/3$ for $a^*<a_1<0$.
See {\rm Figure \ref{five}}.
\label{p12}
\end{proposition}

\begin{proof}
  Firstly,  by  Lemma \ref{lac72},     the   statement   {\bf (i)} follows, as shown in Figure \ref{five} (a).

Denote   $P := (x_P, 0)$ and $Q:=(x_Q, 0)$ be respectively the first intersection points of
    the  stable  and  unstable   manifold   of the right-hand side of  $\hat{E}_{r1}$  and  the $x$-axis.  By Lemma \ref{fe1} and   the   statement  {\bf (i)},  we know that     ${\hat{E}}_{r2}$ of   system \eqref{initial1}  are  unstable and  system \eqref{initial1} has no small  limit cycles when $ a_2\leq -1$.  It    implies that $x_P<x_Q$  when $a_2\leq-1$ since
     system \eqref{initial1}   has small limit cycles for $x_P>x_Q$ and $a_2\leq-1$ by the annulus region of Poincar\'e-Bendixson Theorem
     and  system \eqref{initial1}   has a homoclinic loop only surrounding ${\hat{E}}_{r2}$ for $x_P=x_Q$.
     By Lemmas   \ref{fe1} and  \ref{smc7},  we know that   $\hat{E}_{r2}$ of   system \eqref{initial1}  are  stable and  system \eqref{initial1}  has   no    small limit cycles surrounding     ${\hat{E}}_{r2}$  when $a_2\geq (a_1-1-\sqrt{-a_1})/3$.  It    implies that $x_P>x_Q$  when $a_2\geq (a_1-1-\sqrt{-a_1})/3$.  Furthermore, as proved in Lemma 3.3 of \cite{CC18}, $x_P$ increases  continuously and $x_B  $ decreases   continuously as $a_2$ increases since the vector field of system \eqref{initial1}  is rotated on $a_2$.
 Therefore,
 there is a unique    function $\varphi_2(a_1,\delta)\in(-1, (a_1-1-\sqrt{-a_1})/3)$  such that      $x_P-x_Q=0$  if  and only if  $a_2=\varphi_2(a_1,\delta)$   
 as shown in Figure \ref{five} (b).
Since the saddle quantities at $\hat{E}_{r1}$ and $\hat{E}_{l1}$ are $\delta(a_1-a_2)>0$ for $a_1>a_2$, the small homoclinic loop only surrounding $\hat{E}_{r1}$
and the small homoclinic loop only surrounding $\hat{E}_{l1}$ are unstable by \cite[Theorem 3.3, Chapter 3]{CLW}.
  In other words,  $x_P<x_Q$  if  and only if  $a_2<\varphi_2(a_1,\delta)$, and  $x_P>x_Q$  if  and only if  $a_2>\varphi_2(a_1,\delta)$.
Then, the   statements   {\bf (ii)} and {\bf (iii)} hold.

Denote   $Q := (0, y_Q)$ and $S:=(0, y_S)$ be respectively the first intersection points of
the  stable  and  unstable   manifold   of the right-hand side of  $\hat{E}_{l1}$  and  the $y$-axis.
On the one hand,
 when $a_2=0$ and $-1<a_1<0$,
since system \eqref{initial1} has no limit cycles and the origin is stable,
we obtain  $y_Q+y_S<0$. Otherwise, system \eqref{initial1} has no small closed orbits surrounding $\hat{E}_0$ for $y_Q+y_S\geq0$.
This is a contradiction.
On the other hand,
when $a_2=a_1/3$ and $-1<a_1<0$,
since system \eqref{initial1} has no limit cycles and the origin is unstable,
we obtain  $y_Q+y_S>0$. Otherwise, system \eqref{initial1} has no small closed orbits surrounding $\hat{E}_0$ for $y_Q+y_S\leq0$.
This is a contradiction.
When $a_1\in(-1,0)$ is fixed, since the vector field of system \eqref{initial1} is rotated on $a_2$,
there is a unique value $a_2=\varphi_4(a_1,\delta)\in(a_1/3,0)$ such that  $y_Q+y_S=0$, i.e.,
 system \eqref{initial1} has a small homoclinic loop surrounding $\hat{E}_0$.
 Since the sum of saddle quantities at $\hat{E}_{r1}$ and $\hat{E}_{l1}$ is $\delta(a_1-a_2)$ and negative for $a_1<a_2$,
 the small homoclinic loop only surrounding $\hat{E}_0$
 is stable by \cite[Theorem 3.3, Chapter 3]{CLW}.
 Then, the   statement   {\bf (x)}  holds.

 Denote   $M := (x_M, 0)$  be   the first intersection points of
the    unstable   manifold   of the right-hand side of  $\hat{E}_{l1}$  and  the $x$-axis.
It is clear that $x_M>x_P$ for $a_2=\varphi_2(a_1,\delta)$.
By Lemmas \ref{fe1} and   \ref{smc7},   $E_{0}$ of   system \eqref{initial}  are  unstable and there is no small  limit cycles surrounding $\hat E_{0}$ when $a_2 \leq -a_1/3$.  It    implies that $x_M<x_N$  when  $a_2 \leq -a_1/3$.
By Lemmas \ref{fe1} and   \ref{con1},  $E_{0}$ of   system \eqref{initial}  are  stable  and  there is no small  limit cycles when $ a_2=0$.  It    implies that $x_M>x_P$  when  $a_2=0$.
Furthermore, as proved in Lemma 3.3 of \cite{CC18}, $x_M$ increases  continuously and $x_P  $ decreases   continuously as $a_2$ increases.
Therefore,
there is a unique    function $\varphi_3(a_1,\delta) $  such that      $x_M-x_P=0$  if  and only if  $a_2=\varphi_3(a_1,\delta)$.
Since the sum of saddle quantities at $\hat{E}_{r1}$ and $\hat{E}_{l1}$ $\delta(a_1-a_2)$ and positive (resp. negative)
for $a_1>a_2$(resp. $a_1<a_2$), a heteroclinic loop is unstable (resp. stable) by \cite[Theorem 3.3, Chapter 3]{CLW}.
In other words,  $x_M<x_P$  if  and only if  $a_2<\varphi_3(a_1,\delta) $, and $x_M>x_P$  if  and only if  $a_2>\varphi_3(a_1,\delta)$.
By Lemma \ref{infty} and the Poincar\'e-Bendixson Theorem,
system \eqref{initial} has at least one large limit cycle for $a_2=\varphi_3(a_1,\delta)$ and $a_1>a_2$.
Associated with that system \eqref{initial} has at most two   large limit cycles,  system \eqref{initial} has exactly one large limit cycle for $a_2=\varphi_3(a_1,\delta) $ and $a_1>a_2$.
Assume that system \eqref{initial} has at least one large limit cycle for $a_2=\varphi_3(a_1,\delta) $ and $a_1\leq a_2$.
 By Lemma \ref{infty} and the stability of heteroclinic loop,  there are at least two large limit cycles
 for $a_2=\varphi_3(a_1,\delta) $ and $a_1\leq a_2$. By the heteroclinic bifurcation,
 there are at least three large limit cycles
 for $a_2=\varphi_3(a_1,\delta) +\epsilon$ and $a_1\leq a_2$, where $0<\epsilon\ll 1$. This is a contradiction.
Then, the   statements   {\bf (iv)} and {\bf (v)} hold.
Furthermore, there is a continuous function $a_2=\varphi_5(a_1,\delta)\in(\varphi_3(a_1,\delta), \min\{a_2, -1/3\})$ such that
 system \eqref{initial} has a semi-stable large limit cycle.
  Then, the   statement   {\bf (viii)}  holds.
Since  the remainder statements can be similarly proved, we omit them.
\end{proof}


\section{Proofs of Theorems \ref{Result1}-\ref{mainresult2} and simulations}

\begin{proof}[Proof of Theorem \ref{Result1}]
By Lemmas \ref{fe1}, \ref{con2}, \ref{con4}, Proposition \ref{infty}, the proof can be obtained directly.
\end{proof}

\begin{proof}[Proof of Theorem \ref{mr1}]
	By  Proposition \ref{localbi}, statements {\bf(i)-(v)} hold.
		By  Proposition \ref{infty}, statement {\bf(vi)} holds.
				By  Proposition \ref{p12}, statements {\bf(vii)-(x)} hold.
	When  		 system \eqref{initial1} has neither large limit cycles nor large singular closed orbits (including large heteroclinic loops
	and figure-eight loops),	the orbit connections between equilibria at finity and equilibria at infinity have many cases
since the vector field of system \eqref{initial1}  is rotated on $a_2$.
Statements  (xi) holds.
Finally, according to those bifurcation sets, we give the complete bifurcation diagram in Figure~\ref{bf1}.
\end{proof}

\begin{proof}[Proof of Theorem \ref{mainresult2}]
 The parameter plane $\delta=\delta_0\in(0,2\sqrt{3})$ (resp. $\delta=\delta_0\in [2\sqrt{3}, +\infty)$)  are divided into the parameter regions $I,II,\ldots,XI$ (resp.,  $R_1, \ldots, R_{11}$) by these bifurcation surfaces.
 Moreover, the corresponding global phase portraits in the Poincar\'e disc can be obtained
 according to lemmas and propositions in Sections 3 and 4.
 By Proposition \ref{p12}, there are at most three limit cycles for  $a_2\in(-\infty, (a_1-1-\sqrt{-a_1})/3]\cup[\min\{a_1,-1/3\},0)$ and $-1<a_1<0$.
By Lemmas \ref{smc7} and  \ref{lac73}, there are at most four large limit cycles and no small limit cycles for   $(a_1-1-\sqrt{-a_1})/3<a_2< \min\{a_1,-1/3\}$ and $-1<a_1<0$.
 Thus we get an upper bound $4$ of the number of limit cycles.
\end{proof}


\begin{figure}[h]
	\centering
	\subfigure[ for $(a_1,a_2)=(0.5,-0.1)$]
	{\scalebox{0.4}[0.4]{
			\includegraphics{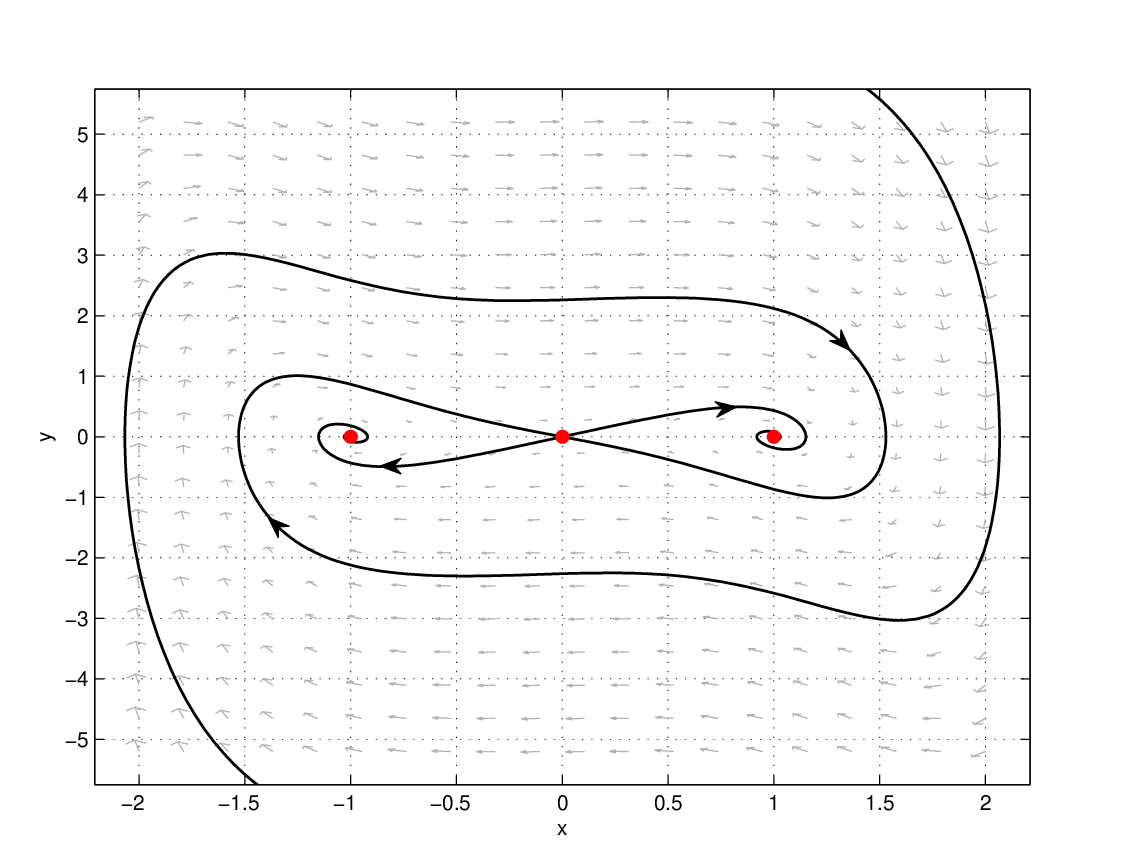}}}
	\subfigure[ for $(a_1,a_2)=(0.5,-0.7)$]
	{\scalebox{0.4}[0.4]{
			\includegraphics{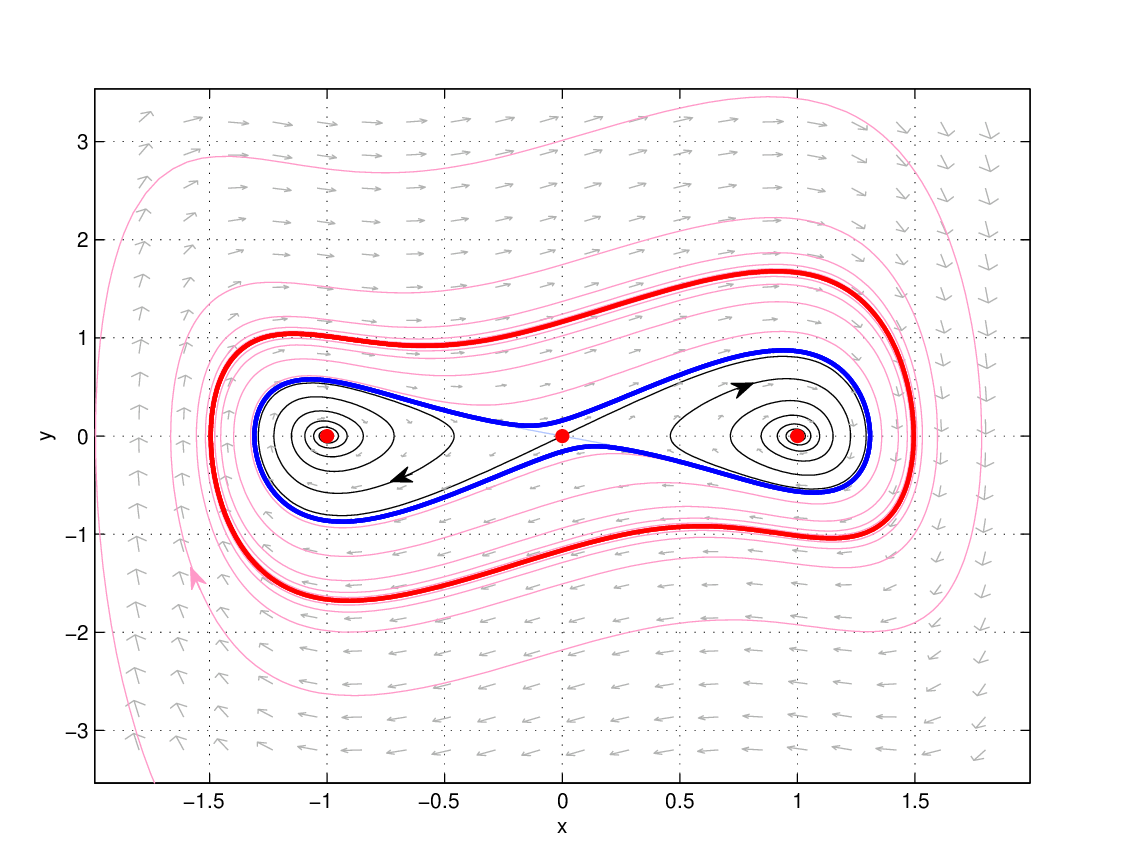}}}
	\subfigure[ for $(a_1,a_2)=(0.5,-0.9)$]
	{\scalebox{0.4}[0.4]{
			\includegraphics{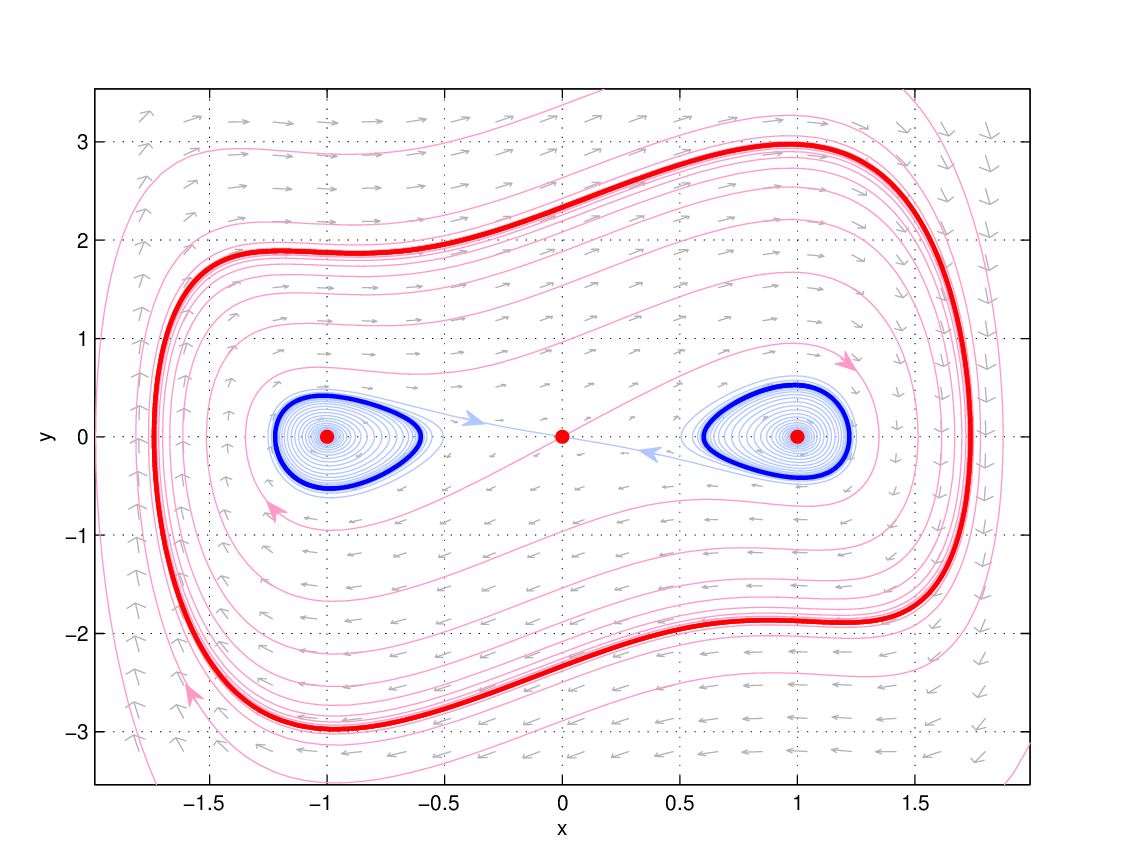}}}
	\subfigure[ for $(a_1,a_2)=(0.5,-1.1)$]
	{\scalebox{0.4}[0.4]{
			\includegraphics{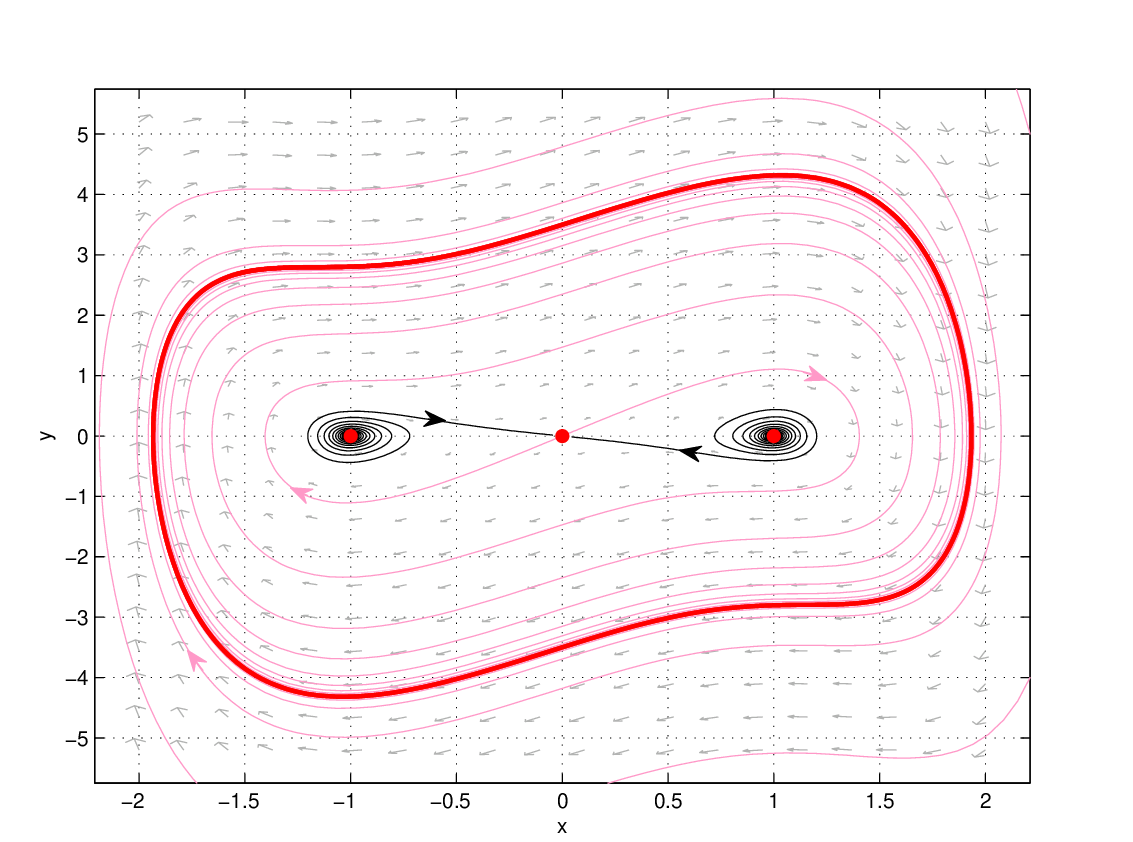}}}
	\subfigure[ for $(a_1,a_2)=(-0.5,-1.1)$]
	{\scalebox{0.4}[0.4]{
			\includegraphics{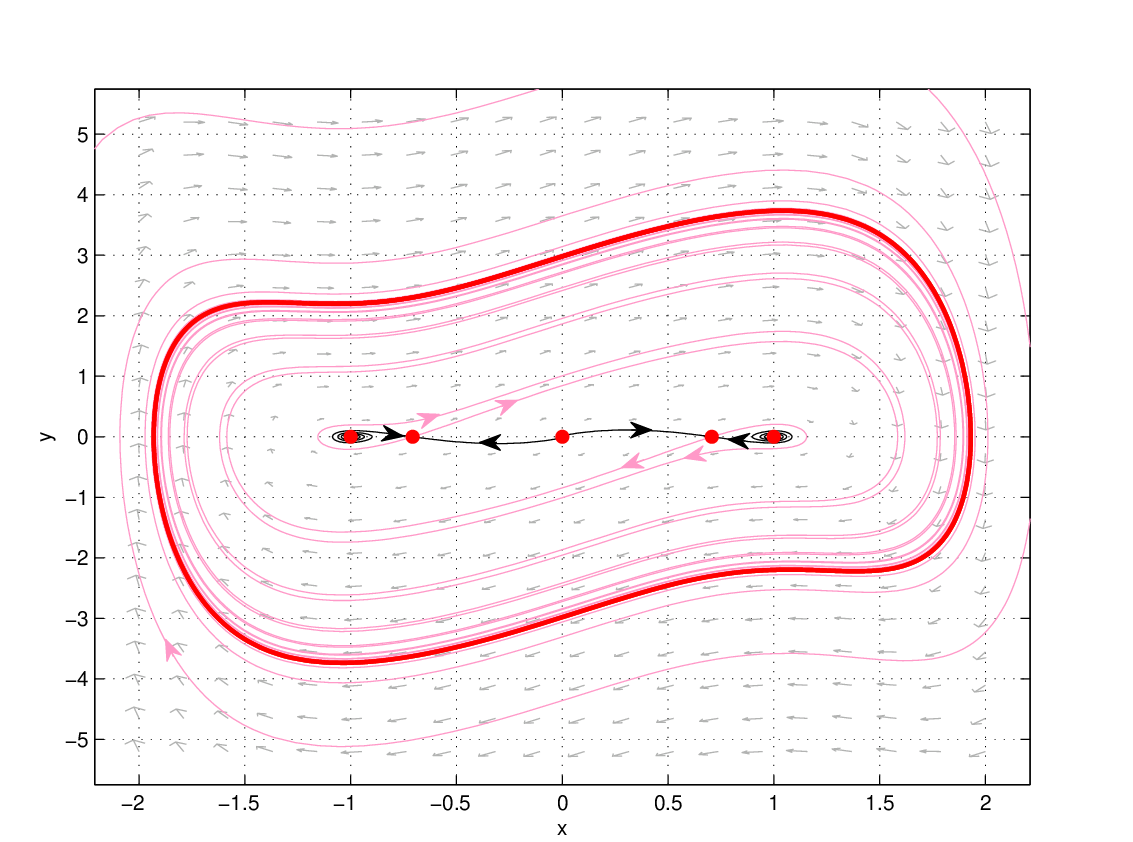}}}
	\subfigure[ for $(a_1,a_2)=(-0.5,-0.95)$]
	{\scalebox{0.4}[0.4]{
			\includegraphics{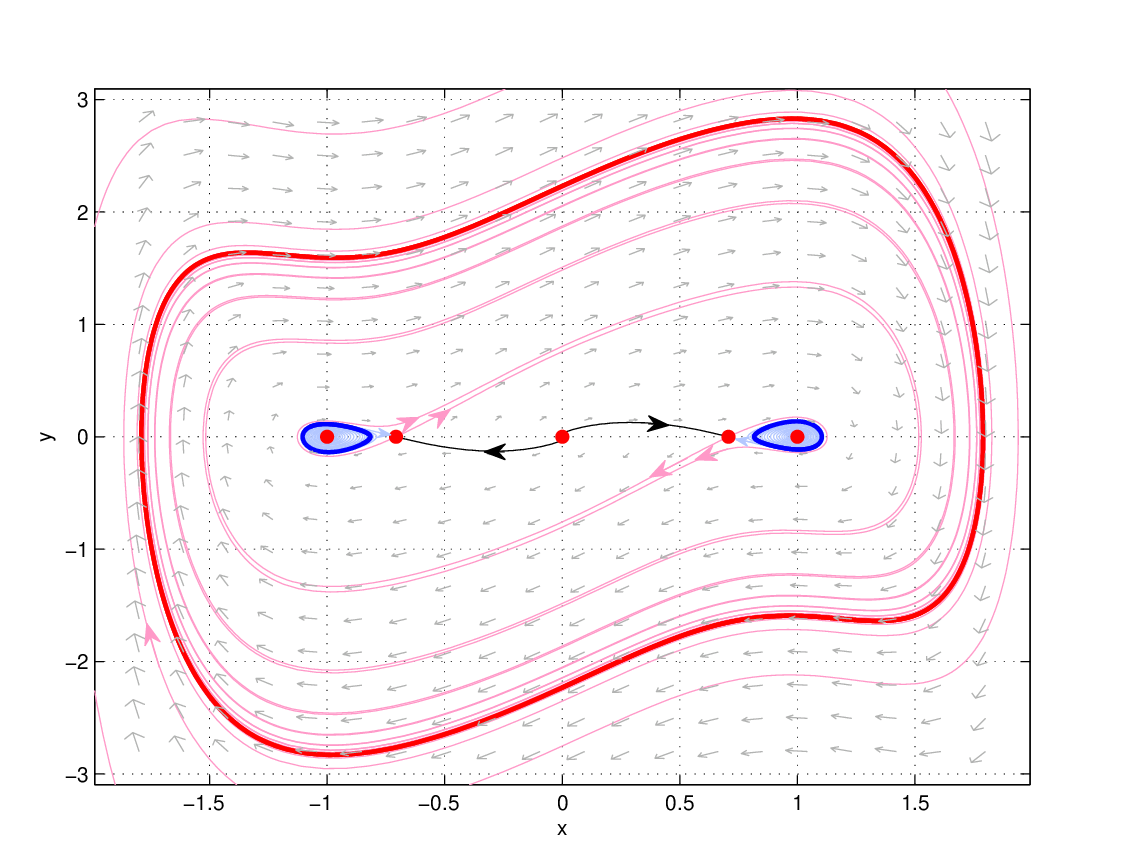}}}
	\caption{Numerical phase portraits  of   \eqref{initial} when $\delta=1$.}
	\label{nm}
\end{figure}

\begin{figure}
	\centering
	\subfigure[ for $(a_1,a_2)=(-0.5,-0.85)$]
	{\scalebox{0.4}[0.4]{
			\includegraphics{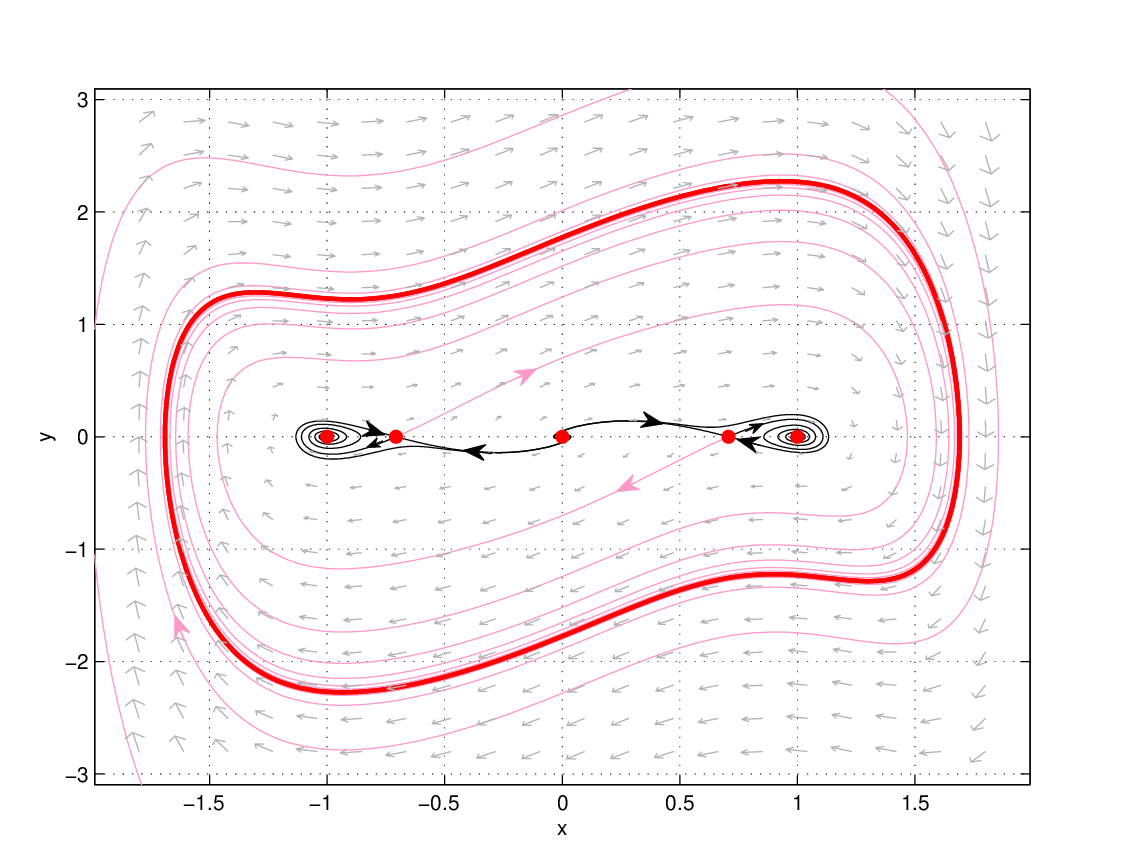}}}
			\subfigure[ for $(a_1,a_2)=(-0.02,-0.63)$]
		{\scalebox{0.4}[0.4]{
				\includegraphics{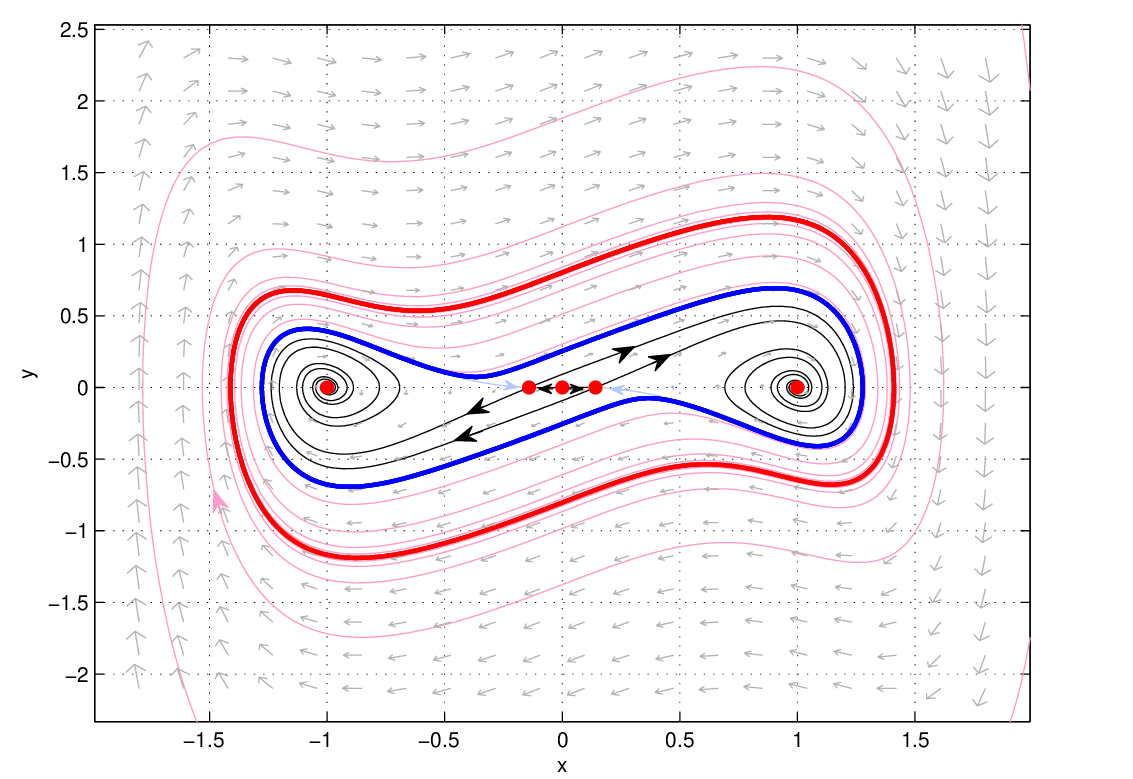}}}
	\subfigure[ for $(a_1,a_2)=(-0.5,-0.1)$]
	{\scalebox{0.4}[0.4]{
			\includegraphics{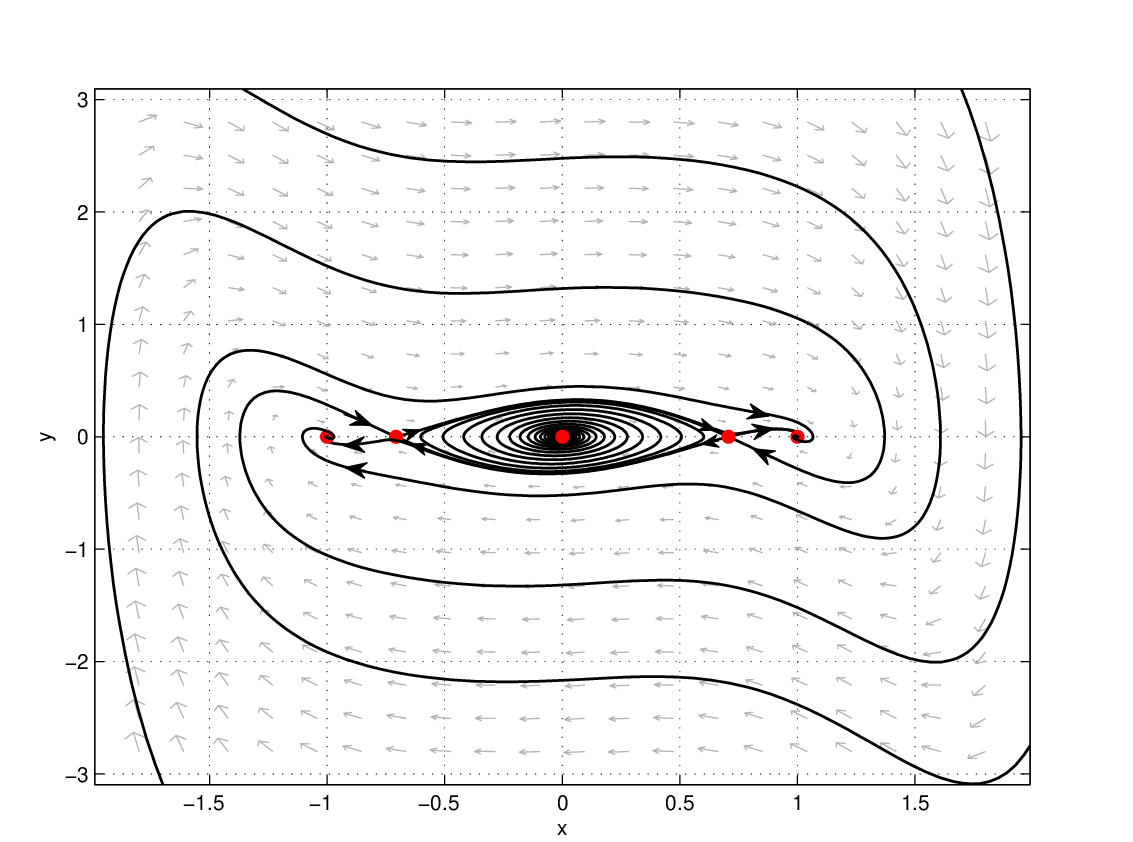}}}
	\subfigure[ for $(a_1,a_2)=(-0.5,-0.05)$]
	{\scalebox{0.4}[0.4]{
			\includegraphics{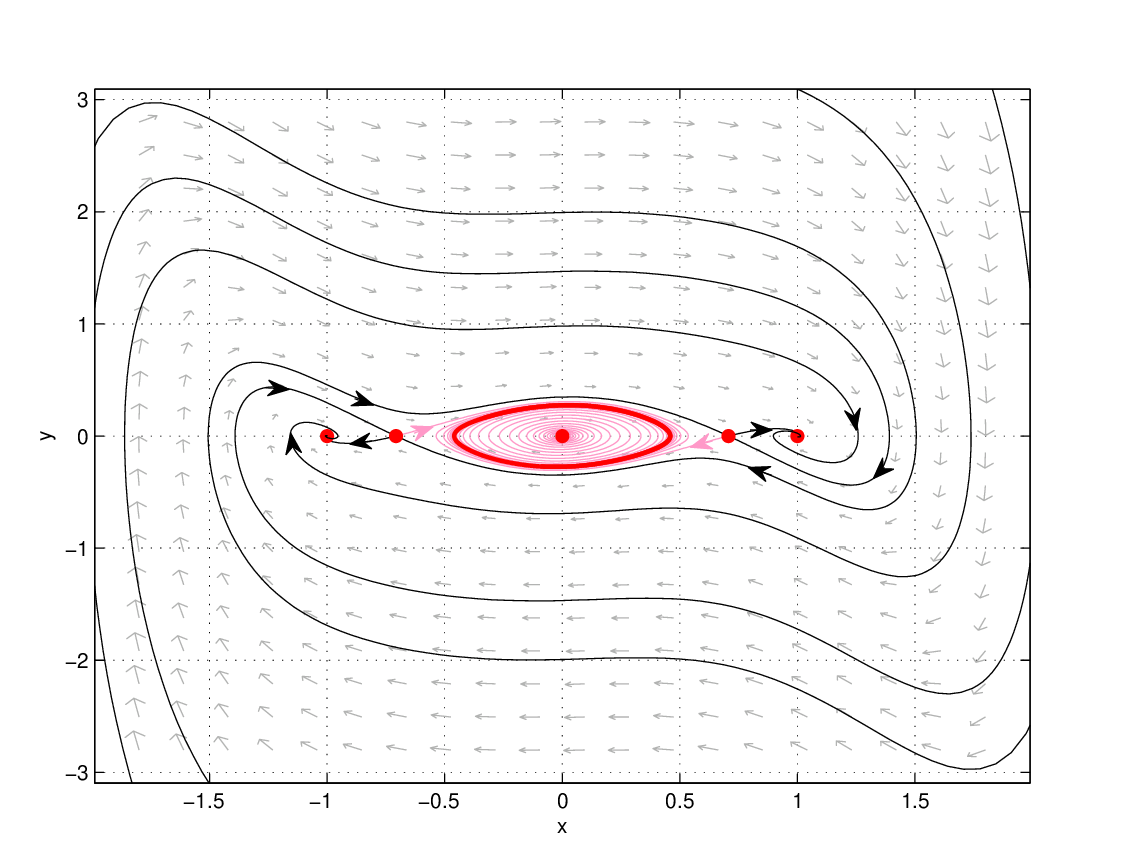}}}
	\subfigure[ for $(a_1,a_2)=(-0.5,0.5)$]
	{\scalebox{0.4}[0.4]{
			\includegraphics{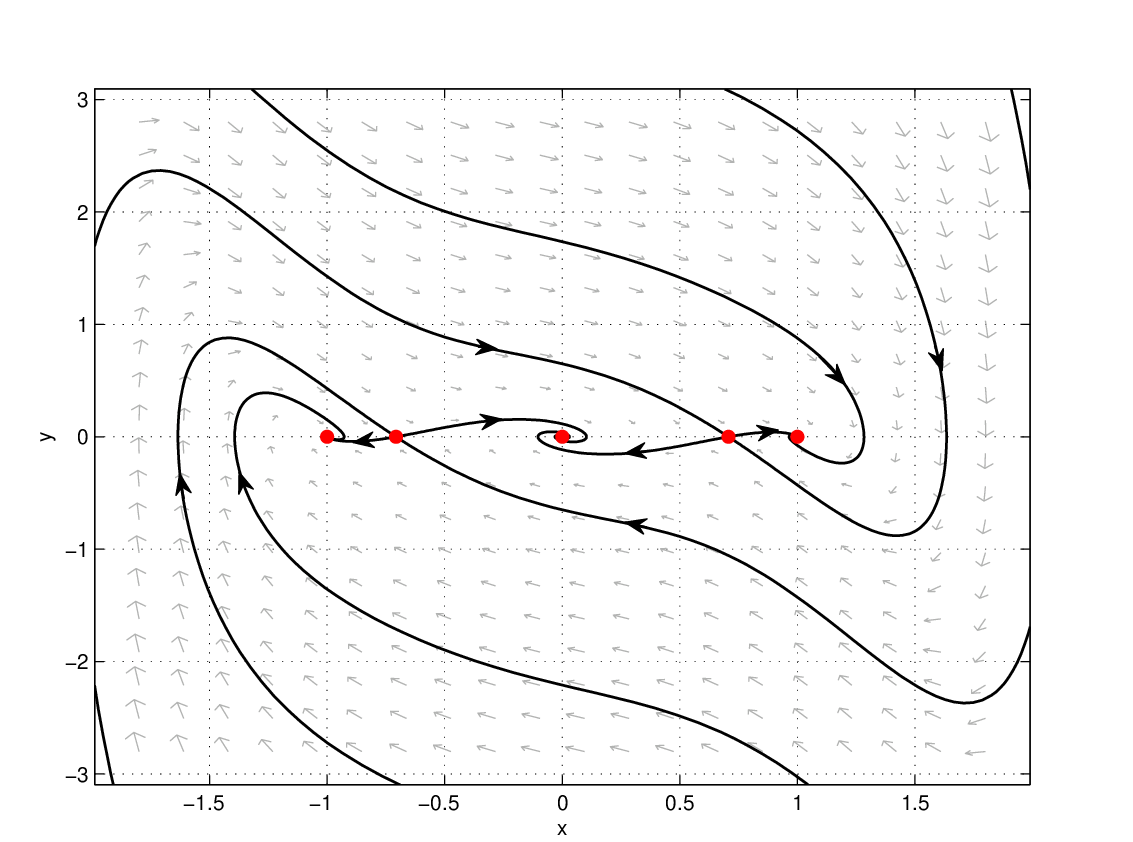}}}
	\caption{Numerical phase portraits  of   \eqref{initial} when $\delta=1$.}
	\label{nm1}
\end{figure}

In the following, we illustrate our theoretical results with some numerical examples.

 {\bf Example 1.} Consider $(a_1,a_2,\delta)=(0.5,-0.1,1)$.
 The numerical phase portrait shows that system \eqref{initial1} has no  limit cycles and three equilibria in the parameter region $I$, as shown in Figure \ref{nm} (a).

  {\bf Example 2.} Consider $(a_1,a_2,\delta)=(0.5,-0.7,1)$.
 The numerical phase portrait shows that system \eqref{initial1} has exactly two large  limit cycles and three equilibria in the parameter region $II$, as shown in Figure \ref{nm} (b).

   {\bf Example 3.} Consider $(a_1,a_2,\delta)=(0.5,-0.9,1)$.
 The numerical phase portrait shows that system \eqref{initial1} has exactly two small  limit cycles, one large  limit cycle and three equilibria in the parameter region $III$, as shown in Figure \ref{nm} (c).

 {\bf Example 4.} Consider $(a_1,a_2,\delta)=(0.5,-1.1,1)$.
 The numerical phase portrait shows that system \eqref{initial1} has exactly  one large  limit cycle and three equilibria in the parameter region $IV$, as shown in Figure \ref{nm} (d).

  {\bf Example 5.} Consider $(a_1,a_2,\delta)=(-0.5,-1.1,1)$.
 The numerical phase portrait shows that system \eqref{initial1} has exactly  one large  limit cycle and five equilibria in the parameter region $V$, as shown in Figure \ref{nm} (e).

    {\bf Example 6.} Consider $(a_1,a_2,\delta)=(-0.5,-0.95,1)$.
 The numerical phase portrait shows that system \eqref{initial1} has exactly two small  limit cycles, one large  limit cycle and five equilibria in the parameter region $VI$, as shown in Figure \ref{nm} (f).

    {\bf Example 7.} Consider $(a_1,a_2,\delta)=(-0.5,-0.85,1)$.
The numerical phase portrait shows that system \eqref{initial1} has exactly   one large  limit cycle and five equilibria in the parameter region $VII$, as shown in Figure \ref{nm1} (a).

    {\bf Example 8.} Consider $(a_1,a_2,\delta)=(-0.02,-0.63,1)$.
The numerical phase portrait shows that system \eqref{initial1} has exactly   two large  limit cycles and five equilibria in the parameter region $VIII$, as shown in Figure \ref{nm1} (b).

    {\bf Example 9.} Consider $(a_1,a_2,\delta)=(-0.5,-0.1 ,1)$.
The numerical phase portrait shows that system \eqref{initial1} has no  limit cycles and five equilibria in the parameter region $IX$, as shown in Figure \ref{nm1} (c).

    {\bf Example 10.} Consider $(a_1,a_2,\delta)=(-0.5,-0.05,1)$.
The numerical phase portrait shows that system \eqref{initial1} has exactly   one small  limit cycle surrounding $\hat{E}_0$ and five equilibria in the parameter region $X$, as shown in Figure \ref{nm1} (d).

    {\bf Example 11.} Consider $(a_1,a_2,\delta)=(-0.5,0.5 ,1)$.
The numerical phase portrait shows that system \eqref{initial1} has no  limit cycles and five equilibria in the parameter region $XI$, as shown in Figure \ref{nm1} (e).
\section*{Acknowledgements}

We sincerely thank Prof. Xiuli Cen for her useful suggestions and valuable comments on Appendix C.

This work is financially supported by the National Key R\&D Program of China (No. 2022YFA 1005900).
The first   author  is  supported by the  National Natural Science Foundation of China (Nos. 12322109, 12171485) and  Science and Technology Innovation Program of Hunan Province (No. 2023RC3040).
The second author is supported by the National Natural Science Foundation of China (No. 12271378) and
Sichuan Science and Technology Program (No. 2024NSFJQ0008).
The third author is   supported  by the China Postdoctoral Science Foundation (No. 2023M743969) and Postdoctoral Fellowship Program of CPSF (No. GZC20233195).
The fourth author is supported by the National Natural Science Foundation of China (Nos. 11931016, 12271355, 12161131001), Science and Technology Innovation Action Plan of Science and Technology Commission of Shanghai Municipality (STCSM, No. 20JC1413200) and Innovation Program of
Shanghai Municipal Education Commission (No. 2021-01-07-00-02-E00087).

 \section*{Appendix A}
	\begin{eqnarray*}
	I&:=&\{(a_1,a_2,\delta)\in \Omega ~|~ \ a_2> \varphi_5(a_1,\delta), a_1\geq0, 0<\delta<2\sqrt{3}\},
	\\
	II&:=&\{(a_1,a_2,\delta) \in \Omega ~|~ \  \varphi_1(a_1,\delta)<a_2< \varphi_5(a_1,\delta), a_1\geq0, 0<\delta<2\sqrt{3}\},
	\\
	III&:=&\{(a_1,a_2,\delta) \in \Omega ~|~ \ -1<a_2< \varphi_1(a_1,\delta), a_1\geq0, 0<\delta<2\sqrt{3}\},
	\\
	IV&:=&\{(a_1,a_2,\delta) \in \Omega ~|~ \  a_2\leq-1, a_1\geq0, 0<\delta<2\sqrt{3}\},
	\\
	V&:=&\{(a_1,a_2,\delta) \in \Omega ~|~ \  a_2\leq-1, -1<a_1<0,  0<\delta<2\sqrt{3}\},
	\\
	VI&:=&\{(a_1,a_2,\delta) \in \Omega ~|~ \ -1<a_2< \varphi_2(a_1,\delta), -1<a_1<0, 0<\delta<2\sqrt{3}\},
	\\
	VII&:=&\{(a_1,a_2,\delta) \in \Omega ~|~ \  \varphi_2(a_1,\delta)<a_2< \varphi_3(a_1,\delta), -1<a_1<0, 0<\delta<2\sqrt{3}\},
	\\
	VIII&:=&\{(a_1,a_2,\delta) \in \Omega ~|~ \  \varphi_3(a_1,\delta)<a_2< \varphi_5(a_1,\delta), a^*<a_1<0, 0<\delta<2\sqrt{3}\},
	\\
	IX&:=&\{(a_1,a_2,\delta) \in \Omega ~|~ \  \varphi_5(a_1,\delta)<a_2< \varphi_4(a_1,\delta), a^*<a_1<0, 0<\delta<2\sqrt{3}\}\cup
	\\
	&&\{(a_1,a_2,\delta) \in \Omega ~|~ \  \varphi_3(a_1,\delta)<a_2< \varphi_4(a_1,\delta), -1<a_1\leq a^*, 0<\delta<2\sqrt{3}\},	
	\\
	X&:=&\{(a_1,a_2,\delta) \in \Omega ~|~ \  \varphi_4(a_1,\delta)<a_2<0, -1<a_1<0, 0<\delta<2\sqrt{3}\},	
	\\
	XI&:=&\{(a_1,a_2,\delta) \in \Omega ~|~ \ a_2\geq 0, -1<a_1<0, 0<\delta<2\sqrt{3}\},	
	\\
	DL_1&:=&\{(a_1,a_2,\delta) \in \Omega ~|~ \ a_2=\varphi_5(a_1,\delta), a_1\geq0, 0<\delta<2\sqrt{3}\},
	\\
	DL_2&:=&\{(a_1,a_2,\delta) \in \Omega ~|~ \ a_2=\varphi_5(a_1,\delta), a^*<a_1<0, 0<\delta<2\sqrt{3}\},
	\\
	HE_{11}&:=&\{(a_1,a_2,\delta) \in \Omega ~|~ \ a_2=\varphi_3(a_1,\delta), -1<a_1<a^*, 0<\delta<2\sqrt{3}\},
\\
	HE_{12}&:=&\{(a_1,a_2,\delta) \in \Omega ~|~ \ a_2=\varphi_3(a_1,\delta), a^*\leq a_1<0,0<\delta<2\sqrt{3}\},
		\\
SL_{1}&:=&\{(a_1,a_2,\delta) \in \Omega ~|~ \ a_2=\varphi_3(-1,\delta),  a_1=-1,\delta>0\},
		\\
SL_{2}&:=&\{(a_1,a_2,\delta) \in \Omega ~|~ \ a_2=\varphi_4(-1,\delta),  a_1=-1, 0<\delta<2\sqrt{3}\},
	\\
	T_1&:=&\{(a_1,a_2,\delta) \in \Omega ~|~ \ a_2<-1, a_1=-1,  \delta>0\},
	\\
	T_2&:=&\{(a_1,a_2,\delta) \in \Omega ~|~ -1<  a_2<\varphi_3(-1,\delta), a_1=-1,  \delta>0\},
	\\
	T_3&:=&\{(a_1,a_2,\delta) \in \Omega ~|~ \ \varphi_3(-1,\delta)< a_2<\varphi_4(-1,\delta), a_1=-1, 0<\delta<2\sqrt{3}\},
	\\
	T_4&:=&\{(a_1,a_2,\delta) \in \Omega ~|~  \ \varphi_4(-1,\delta)<  a_2<0, a_1=-1, 0<\delta<2\sqrt{3}\},
	\\
	T_5&:=&\{(a_1,a_2,\delta) \in \Omega ~|~  \ a_2\geq0, a_1=-1, 0<\delta<2\sqrt{3}\},
	\\
	R_{11i}&:=&\{(a_1,a_2,\delta)\in \Omega ~|~ \ a_2=\phi_{1,i}(a_1,\delta), a_1\geq0, \delta\geq2\sqrt{3}\},
	\\
	R_{12i}&:=&\{(a_1,a_2,\delta)\in \Omega ~|~ \ a_2=\phi_{2,i}(a_1,\delta), a_1\geq0, \delta\geq2\sqrt{3}\},
	\\
	R_{13i}&:=&\{(a_1,a_2,\delta)\in \Omega ~|~ \ \phi_{2,i}(a_1,\delta)<a_2<\phi_{1,i}(a_1,\delta), a_1\geq0, \delta\geq2\sqrt{3}\},
	\\
	R_{14i}&:=&\{(a_1,a_2,\delta)\in \Omega ~|~ \ \phi_{1,i+1}(a_1,\delta)<a_2<\phi_{2,i}(a_1,\delta), a_1\geq0, \delta\geq2\sqrt{3}\},
	\\
	R_2&:=&\{(a_1,a_2,\delta) \in \Omega ~|~ \  \varphi_1(a_1,\delta)<a_2< \varphi_5(a_1,\delta), a_1\geq0, \delta\geq2\sqrt{3}\},
	\\
	R_3&:=&\{(a_1,a_2,\delta) \in \Omega ~|~ \ -1<a_2< \varphi_1(a_1,\delta), a_1\geq0,  \delta\geq2\sqrt{3}\},
	\\
	R_4&:=&\{(a_1,a_2,\delta) \in \Omega ~|~ \  a_2\leq-1, a_1\geq0,  \delta\geq2\sqrt{3}\},
	\\
	R_5&:=&\{(a_1,a_2,\delta) \in \Omega ~|~ \  a_2\leq-1, -1<a_1<0,  \delta\geq2\sqrt{3}\},
	\\
	R_6&:=&\{(a_1,a_2,\delta) \in \Omega ~|~ \ -1<a_2< \varphi_2(a_1,\delta), -1<a_1<0, \delta\geq2\sqrt{3}\},
	\\
	R_7&:=&\{(a_1,a_2,\delta) \in \Omega ~|~ \  \varphi_2(a_1,\delta)<a_2< \varphi_3(a_1,\delta), -1<a_1<0,  \delta\geq2\sqrt{3}\},
	\\
	R_8&:=&\{(a_1,a_2,\delta) \in \Omega ~|~ \  \varphi_3(a_1,\delta)<a_2< \varphi_5(a_1,\delta), a^*<a_1<0,  \delta\geq2\sqrt{3}\},
	\\
	R_{91i}&:=&\{(a_1,a_2,\delta) \in \Omega ~|~ \  \varphi_5(a_1,\delta)<a_2=\phi_{5,i}(a_1,\delta)< \varphi_4(a_1,\delta), a^*<a_1<0,  \delta\geq2\sqrt{3}\}\cup
	\\
	&&\{(a_1,a_2,\delta) \in \Omega ~|~ \  \varphi_3(a_1,\delta)<a_2=\phi_{5,i}(a_1,\delta)< \varphi_4(a_1,\delta), -1<a_1\leq a^*,  \delta\geq2\sqrt{3}\},
	\\
	R_{92i}&:=&\{(a_1,a_2,\delta) \in \Omega ~|~ \  \varphi_5(a_1,\delta)<a_2=\phi_{6,i}(a_1,\delta)< \varphi_4(a_1,\delta), a^*<a_1<0,  \delta\geq2\sqrt{3}\}\cup
	\\
	&&\{(a_1,a_2,\delta) \in \Omega ~|~ \  \varphi_3(a_1,\delta)<a_2=\phi_{6,i}(a_1,\delta)< \varphi_4(a_1,\delta), -1<a_1\leq a^*,  \delta\geq2\sqrt{3}\},	
	\\
	R_{93i}&:=&\{(a_1,a_2,\delta) \in \Omega ~|~ \   \max\{\varphi_5(a_1,\delta),\phi_{6,i}(a_1,\delta)\}<a_2< \min\{\varphi_4(a_1,\delta),\phi_{5,i}(a_1,\delta)\},
	\\
	&&\,\,\
	a^*<a_1<0,  \delta\geq2\sqrt{3}\}\cup
	\{(a_1,a_2,\delta) \in \Omega ~|~ \  \max\{\varphi_3(a_1,\delta),\phi_{6,i}(a_1,\delta)\}<a_2
	\\
	&&\,\,\ < \min\{\varphi_4(a_1,\delta),\phi_{5,i}(a_1,\delta)\},
	-1<a_1\leq a^*,  \delta\geq2\sqrt{3}\},
	\\
	R_{94i}&:=&\{(a_1,a_2,\delta) \in \Omega ~|~ \   \max\{\varphi_5(a_1,\delta),\phi_{5,i+1}(a_1,\delta)\}<a_2< \min\{\varphi_4(a_1,\delta),\phi_{6,i}(a_1,\delta)\},
	\\
	&&
	\,\,\ a^*<a_1<0,  \delta\geq2\sqrt{3}\}\cup
	\{(a_1,a_2,\delta) \in \Omega ~|~ \  \max\{\varphi_3(a_1,\delta),\phi_{5,i+1}(a_1,\delta)\}<a_2
	\\
	&&\,\,\ < \min\{\varphi_4(a_1,\delta),\phi_{6,i}(a_1,\delta)\},
	-1<a_1\leq a^*,  \delta\geq2\sqrt{3}\},
	\\
	R_{101i}&:=&\{(a_1,a_2,\delta) \in \Omega ~|~ \  \varphi_4(a_1,\delta)<a_2=\phi_{3,i}(a_1,\delta)<0, -1<a_1<0,  \delta\geq2\sqrt{3}\},	
	\\
	R_{102i}&:=&\{(a_1,a_2,\delta) \in \Omega ~|~ \  \varphi_4(a_1,\delta)<a_2=\phi_{4,i}(a_1,\delta)<0, -1<a_1<0,  \delta\geq2\sqrt{3}\},	
	\\
	R_{103i}&:=&\{(a_1,a_2,\delta) \in \Omega ~|~ \  \varphi_4(a_1,\delta)<a_2=\phi_{5,i}(a_1,\delta)<0, -1<a_1<0,  \delta\geq2\sqrt{3}\},	
	\\
	R_{104i}&:=&\{(a_1,a_2,\delta) \in \Omega ~|~ \  \varphi_4(a_1,\delta)<a_2=\phi_{6,i}(a_1,\delta)<0, -1<a_1<0,  \delta\geq2\sqrt{3}\},	
	\\
	R_{105i}&:=&\{(a_1,a_2,\delta) \in \Omega ~|~ \
	\max\{\varphi_4(a_1,\delta),\phi_{4,i}(a_1,\delta)\}<a_2 <\min\{0,\phi_{3,i}(a_1,\delta)\},
	\\
	&&
	\,\,\ -1<a_1<0,  \delta\geq2\sqrt{3}\},
	\\
	R_{106i}&:=&\{(a_1,a_2,\delta) \in \Omega ~|~ \
	\max\{\varphi_4(a_1,\delta),\phi_{5,i}(a_1,\delta)\}<a_2 <\min\{0,\phi_{4,i}(a_1,\delta)\},
	\\
	&&
	\,\,\  -1<a_1<0,  \delta\geq2\sqrt{3}\},
	\\
	R_{107i}&:=&\{(a_1,a_2,\delta) \in \Omega ~|~ \
	\max\{\varphi_4(a_1,\delta),\phi_{6,i}(a_1,\delta)\}<a_2 <\min\{0,\phi_{5,i}(a_1,\delta)\},
	\\
	&&
	\,\,\ -1<a_1<0,  \delta\geq2\sqrt{3}\},
	\\
	R_{108i}&:=&\{(a_1,a_2,\delta) \in \Omega ~|~ \
	\max\{\varphi_4(a_1,\delta),\phi_{3,i+1}(a_1,\delta)\}<a_2<\min\{0,\phi_{6,i}(a_1,\delta)\},
	\\
	&&
	\,\,\ -1<a_1<0,  \delta\geq2\sqrt{3}\},
	\\
	R_{111i}&:=&\{(a_1,a_2,\delta) \in \Omega ~|~ \ a_2=\phi_{3,i}(a_1,\delta)>0, -1<a_1<0,  \delta\geq2\sqrt{3}\},	
	\\
	R_{112i}&:=&\{(a_1,a_2,\delta) \in \Omega ~|~ \ a_2=\phi_{4,i}(a_1,\delta)>0, -1<a_1<0,  \delta\geq2\sqrt{3}\},	
	\\
	R_{113i}&:=&\{(a_1,a_2,\delta) \in \Omega ~|~ \ a_2=\phi_{5,i}(a_1,\delta)>0, -1<a_1<0,  \delta\geq2\sqrt{3}\},	
	\\
	R_{114i}&:=&\{(a_1,a_2,\delta) \in \Omega ~|~ \ a_2=\phi_{6,i}(a_1,\delta)>0, -1<a_1<0,  \delta\geq2\sqrt{3}\},	
	\\
	R_{115i}&:=&\{(a_1,a_2,\delta) \in \Omega ~|~ \ \max\{0,\phi_{4,i}(a_1,\delta)\}<a_2<\phi_{3,i}(a_1,\delta), -1<a_1<0,  \delta\geq2\sqrt{3}\},	
	\\
	R_{116i}&:=&\{(a_1,a_2,\delta) \in \Omega ~|~ \ \max\{0,\phi_{5,i}(a_1,\delta)\}<a_2<\phi_{4,i}(a_1,\delta), -1<a_1<0,  \delta\geq2\sqrt{3}\},
	\\
	R_{117i}&:=&\{(a_1,a_2,\delta) \in \Omega ~|~ \ \max\{0,\phi_{6,i}(a_1,\delta)\}<a_2<\phi_{5,i}(a_1,\delta), -1<a_1<0,  \delta\geq2\sqrt{3}\},
	\\
	R_{118i}&:=&\{(a_1,a_2,\delta) \in \Omega ~|~ \ \max\{0,\phi_{3,i+1}(a_1,\delta)\}<a_2<\phi_{6,i}(a_1,\delta), -1<a_1<0,  \delta\geq2\sqrt{3}\},
	\\
	\widehat{DL_1}&:=&\{(a_1,a_2,\delta) \in \Omega ~|~ \ a_2=\varphi_5(a_1,\delta), a_1\geq0,  \delta\geq2\sqrt{3}\},
	\\
	\widehat{DL_2}&:=&\{(a_1,a_2,\delta) \in \Omega ~|~ \ a_2=\varphi_5(a_1,\delta), a^*<a_1<0,  \delta\geq2\sqrt{3}\},
	\\
\widehat{HE}_{11}&:=&\{(a_1,a_2,\delta) \in \Omega ~|~ \ a_2=\varphi_3(a_1,\delta), -1<a_1<a^*,  \delta\geq2\sqrt{3}\},
	\\
	\widehat{HE}_{12}&:=&\{(a_1,a_2,\delta) \in \Omega ~|~ \ a_2=\varphi_3(a_1,\delta), a^*\leq a_1<0, \delta\geq2\sqrt{3}\},
	\\
	HE_{21i}&:=&\{(a_1,a_2,\delta)\in \Omega ~|~\ a_2=\varphi_4(a_1,\delta)=\phi_{5,i}(a_1,\delta),  -1< a_1<0, \delta\geq2\sqrt{3}\},
	\\
	HE_{22i}&:=&\{(a_1,a_2,\delta)\in \Omega ~|~\ a_2=\varphi_4(a_1,\delta)=\phi_{6,i}(a_1,\delta),  -1< a_1<0, \delta\geq2\sqrt{3}\},
	\\
	HE_{23i}&:=&\{(a_1,a_2,\delta)\in \Omega ~|~\ \phi_{6,i}(a_1,\delta)<a_2=\varphi_4(a_1,\delta)<\phi_{5,i}(a_1,\delta),  -1< a_1<0, \delta\geq2\sqrt{3}\},
	\\
	HE_{24i}&:=&\{(a_1,a_2,\delta)\in \Omega ~|~\ \phi_{5,i+1}(a_1,\delta)<a_2=\varphi_4(a_1,\delta)<\phi_{6,i}(a_1,\delta),  -1< a_1<0, \delta\geq2\sqrt{3}\},
 \\
 SL_{21i}&:=&\{(a_1,a_2,\delta) \in \Omega ~|~ \ a_2=\varphi_4(-1,\delta)=\phi_{5,i}(a_1,\delta),  a_1=-1,  \delta\geq2\sqrt{3}\},
  \\
 SL_{22i}&:=&\{(a_1,a_2,\delta) \in \Omega ~|~ \ a_2=\varphi_4(-1,\delta)=\phi_{6,i}(a_1,\delta),  a_1=-1,  \delta\geq2\sqrt{3}\},
   \\
 SL_{23i}&:=&\{(a_1,a_2,\delta) \in \Omega ~|~ \ \phi_{6,i}(a_1,\delta)<a_2=\varphi_4(-1,\delta)=\phi_{5,i}(a_1,\delta),  a_1=-1,  \delta\geq2\sqrt{3}\},
    \\
 SL_{24i}&:=&\{(a_1,a_2,\delta) \in \Omega ~|~ \ \phi_{5,i+1}(a_1,\delta)<a_2=\varphi_4(-1,\delta)=\phi_{6,i}(a_1,\delta),  a_1=-1,  \delta\geq2\sqrt{3}\},
	\\
T_{31i}&:=&\{(a_1,a_2,\delta) \in \Omega ~|~ \ \varphi_3(-1,\delta)<a_2=\phi_{5,i}(a_1,\delta)<\varphi_4(-1,\delta), a_1=-1,  \delta\geq2\sqrt{3}\},
	\\
T_{32i}&:=&\{(a_1,a_2,\delta) \in \Omega ~|~ \ \varphi_3(-1,\delta)<a_2=\phi_{6,i}(a_1,\delta)<\varphi_4(-1,\delta), a_1=-1,  \delta\geq2\sqrt{3}\},
	\\
T_{33i}&:=&\{(a_1,a_2,\delta) \in \Omega ~|~ \   \max\{\varphi_5(-1,\delta),\phi_{6,i}(-1,\delta)\}<a_2< \min\{\varphi_4(-1,\delta),\phi_{5,i}(-1,\delta)\},
	\\
&&
	\,\,\
 a_1=-1,  \delta\geq2\sqrt{3}\},
 	\\
 T_{34i}&:=&\{(a_1,a_2,\delta) \in \Omega ~|~ \   \max\{\varphi_5(-1,\delta),\phi_{5,i+1}(-1,\delta)\}<a_2< \min\{\varphi_4(-1,\delta),\phi_{6,i}(-1,\delta)\},
 \\
 &&
 	\,\,\
 a_1=-1,  \delta\geq2\sqrt{3}\},
 	\\
 T_{41i}&:=&\{(a_1,a_2,\delta) \in \Omega ~|~ \ \varphi_4(-1,\delta)<a_2=\phi_{3,i}(a_1,\delta)<0, a_1=-1,  \delta\geq2\sqrt{3}\},
 \\
 T_{42i}&:=&\{(a_1,a_2,\delta) \in \Omega ~|~\ \varphi_4(-1,\delta)<a_2=\phi_{4,i}(a_1,\delta)<0, a_1=-1,  \delta\geq2\sqrt{3}\},
	\\
T_{43i}&:=&\{(a_1,a_2,\delta) \in \Omega ~|~ \ \varphi_4(-1,\delta)<a_2=\phi_{5,i}(a_1,\delta)<0, a_1=-1,  \delta\geq2\sqrt{3}\},
\\
T_{44i}&:=&\{(a_1,a_2,\delta) \in \Omega ~|~\ \varphi_4(-1,\delta)<a_2=\phi_{6,i}(a_1,\delta)<0, a_1=-1,  \delta\geq2\sqrt{3}\},
\\
T_{45i}&:=&\{(a_1,a_2,\delta) \in \Omega ~|~ \   \max\{\varphi_4(-1,\delta),\phi_{4,i}(-1,\delta)\}<a_2< \min\{0,\phi_{3,i}(-1,\delta)\},
\\
&&
	\,\,\
a_1=-1,  \delta\geq2\sqrt{3}\},
\\
T_{46i}&:=&\{(a_1,a_2,\delta) \in \Omega ~|~ \   \max\{\varphi_4(-1,\delta),\phi_{5,i}(-1,\delta)\}<a_2< \min\{0,\phi_{4,i}(-1,\delta)\},
\\
&&
	\,\,\
a_1=-1,  \delta\geq2\sqrt{3}\},
\\
T_{47i}&:=&\{(a_1,a_2,\delta) \in \Omega ~|~ \   \max\{\varphi_4(-1,\delta),\phi_{6,i}(-1,\delta)\}<a_2< \min\{0,\phi_{5,i}(-1,\delta)\},
\\
&&
\,\,\
a_1=-1,  \delta\geq2\sqrt{3}\},
\\
T_{48i}&:=&\{(a_1,a_2,\delta) \in \Omega ~|~ \   \max\{\varphi_4(-1,\delta),\phi_{3,i+1}(-1,\delta)\}<a_2< \min\{0,\phi_{6,i}(-1,\delta)\},
\\
&&
\,\,\
a_1=-1,  \delta\geq2\sqrt{3}\},
	\\
T_{51i}&:=&\{(a_1,a_2,\delta) \in \Omega ~|~ \ a_2=\phi_{3,i}(a_1,\delta)\geq0, a_1=-1,  \delta\geq2\sqrt{3}\},
\\
T_{52i}&:=&\{(a_1,a_2,\delta) \in \Omega ~|~\ a_2=\phi_{4,i}(a_1,\delta)\geq0, a_1=-1,  \delta\geq2\sqrt{3}\},
	\\
T_{53i}&:=&\{(a_1,a_2,\delta) \in \Omega ~|~ \ a_2=\phi_{5,i}(a_1,\delta)\geq0, a_1=-1,  \delta\geq2\sqrt{3}\},
\\
T_{54i}&:=&\{(a_1,a_2,\delta) \in \Omega ~|~\ a_2=\phi_{6,i}(a_1,\delta)\geq0, a_1=-1,  \delta\geq2\sqrt{3}\},
\\
T_{55i}&:=&\{(a_1,a_2,\delta) \in \Omega ~|~ \ \phi_{4,i}(-1,\delta)<a_2<\phi_{3,i}(-1,\delta), a_2\geq0,
a_1=-1,  \delta\geq2\sqrt{3}\},
\\
T_{56i}&:=&\{(a_1,a_2,\delta) \in \Omega ~|~\ \phi_{5,i}(-1,\delta)<a_2<\phi_{4,i}(-1,\delta), a_2\geq0,
a_1=-1,  \delta\geq2\sqrt{3}\},
\\
T_{57i}&:=&\{(a_1,a_2,\delta) \in \Omega ~|~ \ \phi_{6,i}(-1,\delta)<a_2<\phi_{5,i}(-1,\delta), a_2\geq0,
a_1=-1,  \delta\geq2\sqrt{3}\},
\\
T_{58i}&:=&\{(a_1,a_2,\delta) \in \Omega ~|~\ \phi_{3,i+1}(-1,\delta)<a_2<\phi_{6,i}(-1,\delta), a_2\geq0,
a_1=-1,  \delta\geq2\sqrt{3}\}.
\end{eqnarray*}

\section*{Appendix B}
\begin{proof}[Proof of   Lemma \ref{fe1}]
	It   is easy to check    that
	the number and abscissas of equilibria of system \eqref{initial}
are  determined by the equation $x(x^4 +\mu_2x^2+\mu_1) =0$.
Notice that  the number of
roots of the equation $x(x^4 +\mu_2x^2+\mu_1) =0$
  is determined by the relationship  between   $\Delta:= \mu_2 ^2-4 \mu_1$   and $0$ and   the relationship  between    $\mu_2 $ and $ \sqrt{\Delta}$.
  Thus, we show
  the  number  of  equilibria of system \eqref{initial}    by   the following nine  cases:

{\bf Case (I)}:
System \eqref{initial} has  five equilibria  $\hat{E}_{l2}$, $\hat{E}_{l1}$, $\hat{E}_0$, $\hat{E}_{r1}$, $\hat{E}_{r2}$  when
$\Delta>0$  and  $  \mu_2<- \sqrt{\Delta}$,  that is, $\mu_1>0$ and $\mu_2<-2\sqrt{\mu_1}$.

{\bf Case (II)}: System \eqref{initial} has
three equilibria  $\hat{E}_{l2}$, $\hat{E}_0$,    $\hat{E}_{r2}$  when   $\Delta>0$ and   $\mu_2=-\sqrt{\Delta}$,  that is, $\mu_1=0$ and   $\mu_2<0$.

{\bf Case (III)}: System \eqref{initial} has
three equilibria  $\hat{E}_{l2}$,  $\hat{E}_0$,  $\hat{E}_{r2}$  when  $\Delta>0$ and   $-\sqrt{\Delta}<\mu_2<\sqrt{\Delta}$,  that is,   $\mu_1<0$ and $\mu_2\in\mathbb{R}$.

{\bf Case (IV)}: System \eqref{initial} has
one   equilibrium  $\hat{E}_0$   when  $\Delta>0$ and   $ \mu_2=\sqrt{\Delta}$,  that is,  $\mu_1=0$ and   $\mu_2>0$.

{\bf Case (V)}: System \eqref{initial} has
  one   equilibrium  $\hat{E}_0$   when  $\Delta>0$ and   $ \mu_2>\sqrt{\Delta}$,  that is,    $\mu_1>0$ and   $\mu_2>2\sqrt{\mu_1}$.

{\bf Case (VI)}: System \eqref{initial} has
three equilibria $\hat{E}_{l2}$, $\hat{E}_0$,    $\hat{E}_{r2}$  when  $\Delta=0$ and $ \mu_2 <0$,  that is,  $\mu_1>0$ and   $\mu_2=-2\sqrt{\mu_1}$.

{\bf Case (VII)}: System \eqref{initial} has
  one   equilibrium   $\hat{E}_0$   when $\Delta=0$ and $ \mu_2  =0$,  that is,     $\mu_1=\mu_2=0$.

{\bf Case (VIII)}: System \eqref{initial} has
  one   equilibrium   $\hat{E}_0$  when $\Delta=0$ and $ \mu_2 >0$,  that is,    $\mu_1>0$ and   $\mu_2=2\sqrt{\mu_1}$.

{\bf Case (IX)}: System \eqref{initial} has
one  equilibrium $\hat{E}_0$  when  $\Delta<0$,  that is,  $\mu_1>0$ and    $-2\sqrt{\mu_1}<\mu_2<2\sqrt{\mu_1}$.

  \noindent
As a consequence, we obtain the   location of equilibria  of   system \eqref{initial},      as illustrated in    {\rm Table~\ref{lmtable1}}.
 We now  study the qualitative properties of   $\hat{E}_{l2}$,  $\hat{E}_{l1}$, $\hat{E}_0$, $\hat{E}_{r1}$, $\hat{E}_{r2}$   of system \eqref{initial} in turn.

The   Jacobian   matrix  at  $\hat{E}_0$  is  the following form
\[
	 J_{E_{0}}:= \left(
	\begin{array}{lll}
	&0  \quad   &1
	\\
	&-\mu_1 \quad   & -\mu_3
	\end{array}
	\right).
	\]
According to  ${\rm det}J_{E_{0}}=\mu_1$ and   ${\rm tr}J_{E_{0}}=-\mu_3$,     it follows that
$\hat{E}_0$ is a saddle for $\mu_1<0$,  a  sink for $\mu_1>0$ and  $\mu_3>0$,   and
a    source for $\mu_1>0$ and  $ \mu_3<0$.
When $\mu_1>0$ and $\mu_3=0$,
 with the scaling  transformation   \[(y,~t)\to(\sqrt{\mu_1}y, ~\frac{t }{\sqrt{\mu_1}}),\]
	system \eqref{initial} becomes
	\begin{eqnarray}
		\begin{array}{ll}
	\dot x=y,
	\\
	\dot y=-x-\frac{\mu_2}{\mu_1}x^3-\frac{1}{\mu_1}x^5 -\frac{b}{\sqrt{\mu_1}}x^2y.
	\end{array}
		\label{initia221027}
	\end{eqnarray}
	By \cite[p.211]{CLW},
	we calculate  the first
	focal value for system \eqref{initia221027} and  get that at  the origin
	\[
	g_3=-\frac{b}{8\sqrt{\mu_1}}<0
	\]
because of $b>0$.
Therefore,  the origin of system \eqref{initia221027}  is a  stable weak focus of order one and  so is $\hat{E}_0$. As    $\mu_1= \mu_3=0$,
	system \eqref{initial}  is simplified as
	\begin{eqnarray}
		\begin{array}{ll}
	\dot x=  y,
	\\[2mm]
	\dot y=  - \mu_2x^3-x^5 -  bx^2y.
	\end{array}
		\label{initia1028}
	\end{eqnarray}
It follows  from  \cite[Theorem 7.2 of Chapter 2]{Zh}    that  $\hat{E}_0$  of 	system  \eqref{initia1028} is
a  degenerate    saddle for $\mu_2<0$,   a  degenerate   center or focus  for $\mu_2>0$,   a  degenerate   center or focus  for $\mu_2=0$ and $b\in (0, 2\sqrt{3})$,  and  a degenerate   node for  $\mu_2=0$ and $b\in [2\sqrt{3}, +\infty)$.
Letting
\[
H(x,y)=\frac{\mu_2x^3}{3}+\frac{x^5}{5}+\frac{y^2}{2},
\]
we have
\[
\frac{dH(x,y)}{dt}|_{\eqref{initia1028}}=-bx^2y^2\leq0,
\]
implying that
  a  stable focus  for $\mu_2>0$,   or $\mu_2=0$ and $b\in (0, 2\sqrt{3})$.
Furthermore, for  $\mu_1=0$ and $\mu_3\neq0$,
 using    the translation transformation
	\[
	(x, ~y)\to(x-\frac{y}{\mu_3}, ~y),
	\]
	system \eqref{initial}  is  changed into
	\begin{eqnarray}
		\begin{array}{ll}
	\dot x=  - \frac{\mu_2}{\mu_3}(x-\frac{y}{\mu_3})^3-\frac{1}{\mu_3}(x-\frac{y}{\mu_3})^5-  \frac{b}{\mu_3}(x-\frac{y}{\mu_3})^2y,
	\\[2mm]
	\dot y= -\mu_3y - \mu_2(x-\frac{y}{\mu_3})^3-(x-\frac{y}{\mu_3})^5-  b(x-\frac{y}{\mu_3})^2y.
	\end{array}
		\label{initia10271}
	\end{eqnarray}
 Solving $\dot y=0$ of  system \eqref{initia10271}, we obtain
	\begin{eqnarray}
	y =
	-\frac{\mu_2}{\mu_3}x^3+ \frac{-\mu_3^2+b\mu_2\mu_3-3\mu_2^2}{\mu_3^3}x^5+O(x^6)
	\label{initial311}
	\end{eqnarray}
		by  the Implicit Function Theorem.
 Substituting  \eqref{initial311} into the first  equation of  system \eqref{initia10271}, we    have
\begin{eqnarray*}
\frac{dx}{dt}=  - \frac{\mu_2}{\mu_3}x^3 +\frac{-\mu_3^2+b\mu_2\mu_3-3\mu_2^2}{\mu_3^3}x^5+O(x^6).
\end{eqnarray*}
	According to   \cite[Theorem 7.1 of Chapter 2]{Zh}, it follows     that the origin of system \eqref{initia10271} is
	a stable       degenerate  node    for $\mu_2\geq0 $ and  $\mu_3>0$,
an    unstable         degenerate node     for $\mu_2\geq0$ and  $\mu_3<0$,
     and   a   degenerate saddle     for $\mu_2<0 $ and  $\mu_3\neq0$.   So is $\hat{E}_0$.

The   Jacobian   matrices       at  both  $\hat{E}_{l2}$ and $\hat{E}_{r2}$   are  of  the following form
\[
	 J_{\hat{E}_{l2}}= J_{\hat{E}_{r2}}:= \left(
	\begin{array}{lll}
	&0  \quad   &1
	\\
	&-  \Delta +\mu_2  \sqrt{\Delta}  \quad   & -\mu_3-\frac{b(-\mu_2+\sqrt{\Delta})}{2}
	\end{array}
	\right).
	\]
By  {\bf Cases (I)},  {\bf  (II)},   {\bf  (III)} and  {\bf  (VI)},    we know that  the necessary conditions   of  the existence of   $\hat{E}_{l2}$ and $\hat{E}_{r2}$    are  $\Delta\geq0$  and    $  \mu_2<  \sqrt{\Delta}$.
Therefore, we obtain     ${\rm det}J_{\hat{E}_{l2}}={\rm det}J_{\hat{E}_{r2}} =  \sqrt{\Delta}(\sqrt{\Delta} -  \mu_2)>0  $ for  $\Delta>0$ and  ${\rm det}J_{\hat{E}_{l2}}={\rm det}J_{\hat{E}_{r2}}=0  $ for  $\Delta=0$.
It is easy to check that   ${\rm tr}J_{\hat{E}_{l2}}={\rm tr}J_{\hat{E}_{r2}}= -\mu_3- {b(-\mu_2+\sqrt{\Delta})}/{2}<0$ for $ \mu_3>- {b(-\mu_2+\sqrt{\Delta})}/{2}$
  and
 ${\rm tr}J_{\hat{E}_{l2}}={\rm tr}J_{\hat{E}_{r2}}>0$ for $ \mu_3<- {b(-\mu_2+\sqrt{\Delta})}/{2}$.
 It follows that  both  $\hat{E}_{l2}$ and $\hat{E}_{r2}$    are   sinks for
 $\Delta>0$ and $ \mu_3>- {b(-\mu_2+\sqrt{\Delta})}/{2}$,   and  sources for   $\Delta>0$ and $ \mu_3<- {b(-\mu_2+\sqrt{\Delta})}/{2}$.
When   $\Delta>0$ and $ \mu_3=- {b(-\mu_2+\sqrt{\Delta})}/{2}$,
 by the following translation transformation
 \[
 (x,~y) \rightarrow \left(x+\sqrt{ \frac{ -\mu_2+\sqrt{\Delta}}  {2}},~ y\right),
 \]
system \eqref{initial} becomes
	\begin{eqnarray}
		\begin{array}{ll}
	\dot x=y,
	\\
	\dot y=(-  \Delta +\mu_2  \sqrt{\Delta})x+\sqrt{\frac{-\mu_2+\sqrt{\Delta}}{2}}(2\mu_2-5\sqrt{\Delta})x^2+(4\mu_2-5\sqrt{\Delta})x^3\\
\quad\quad -5\sqrt{\frac{-\mu_2+\sqrt{\Delta}}{2}}x^4
  -x^5
-bx^2y-2b\sqrt{\frac{-\mu_2+\sqrt{\Delta}}{2}}xy,
	\end{array}
		\label{initia221029}
	\end{eqnarray}
  implying  that   $\hat{E}_{r2}$    becomes the  origin of  system \eqref{initia221029}.
Under  the scaling  transformation   	
\[
(y,t)\to \left(\sqrt{\Delta -\mu_2  \sqrt{\Delta} }y, ~\frac{t}{\sqrt{\Delta -\mu_2  \sqrt{\Delta} }}\right),
\]
	system \eqref{initia221029} can be written as
	\begin{eqnarray}
		\begin{array}{ll}
	\dot x=y,
	\\
	\dot y=-x+ \frac{\sqrt{-\mu_2+\sqrt{\Delta}}}{\sqrt{2}\Delta -\mu_2\sqrt{2\Delta} } (2\mu_2-5\sqrt{\Delta})x^2+\frac{4\mu_2-5\sqrt{\Delta}}{\Delta-\mu_2\sqrt{ \Delta}}x^3 -5  \frac{\sqrt{-\mu_2+\sqrt{\Delta}}}{\sqrt{2}\Delta -\mu_2\sqrt{2\Delta}} x^4
	\\
\quad\quad
-\frac{1}{{\Delta-\mu_2\sqrt{ \Delta}}}x^5-\frac{b}{\sqrt{{\Delta-\mu_2\sqrt{ \Delta}}}}x^2y-\sqrt{2}b  \frac{\sqrt{-\mu_2+\sqrt{\Delta}}}{\sqrt{{\Delta-\mu_2\sqrt{ \Delta}}}}  xy.
	\end{array}
		\label{initia2210291}
	\end{eqnarray}
	From \cite[p.211]{CLW},
	we calculate  the first
	focal value for system \eqref{initia2210291} and  get that at  the origin
	
	\[
	g_3=   \frac{ b( 2\sqrt{ \Delta}-    \mu_2)}{4\sqrt{ \Delta}\sqrt{{\Delta-\mu_2\sqrt{ \Delta}}}}>0
	\]
because of  $  \mu_2<  \sqrt{\Delta}$.
Hence   the origin of system   \eqref{initia2210291}  is  an   unstable weak focus of order one  and   so is $\hat{E}_{r2}$.
By the symmetry of system \eqref{initial} about the origin,
  $\hat{E}_{l2}$   is   an unstable weak focus of order one.

 As    $\Delta=0$ and $ \mu_3={ b \mu_2 }/{2}$,
  by the following translation transformation
  \[
  (x,~y) \rightarrow (x+\sqrt{ \frac{ -\mu_2 }  {2}},~ y),
  \]
system \eqref{initial}  becomes
 \begin{eqnarray}
		\begin{array}{ll}
	\dot x=y,
	
	\\
	\dot y= \mu_2 \sqrt{ {-2\mu_2 } }   x^2+ 4\mu_2  x^3
  -5\sqrt{\frac{-\mu_2 }{2}}x^4
  -x^5
- b\sqrt{ {-2\mu_2 } }xy-bx^2y
	\end{array}
		\label{initia221029001}
	\end{eqnarray}
implying   that   $\hat{E}_{r2}$ becomes the origin of  system \eqref{initia221029001}.
It follows  from  \cite[Theorem 7.3 of Chapter 2]{Zh}    that    the origin of  system   \eqref{initia221029001}      is
a     cusp and so is  $\hat{E}_{r2}$.
By the symmetry of system \eqref{initial} about the origin,        $\hat{E}_{l2}$     is
a     cusp for   $\Delta=0$ and $ \mu_3={ b \mu_2 }/{2}$.

 When  $\Delta=0$ and $ \mu_3 \neq{ b \mu_2 }/{2}$,
 using    the translation transformation
 $
 (x, y) \rightarrow (x+\sqrt{ {-\mu_2 }/{2}}, y),
 $
system \eqref{initial}  can be rewritten as
	 \begin{eqnarray}
		\begin{array}{ll}
	\dot x=y,
	\\
	\dot y= \mu_2 \sqrt{ {-2\mu_2 } }   x^2+ 4\mu_2  x^3
  -5\sqrt{\frac{-\mu_2 }{2}}x^4
  -x^5+( -\mu_3+\frac{b \mu_2 }{2})y- b\sqrt{ {-2\mu_2 } }xy
-bx^2y
	\end{array}
		\label{initia221030}
	\end{eqnarray}
  meaning  that   $\hat{E}_{r2}$   becomes  the origin of  system \eqref{initia221030}.  For simplicity, denote $ k=: \mu_3- {b  \mu_2   }/{2}$.
Then,  using    the translation transformation
	$
	(x, y)\to(x-{y}/{k}, y)$,	
	 system \eqref{initia221030}   is changed into
	 \begin{eqnarray}
		\begin{array}{ll}
	\dot x=\frac  {\mu_2 \sqrt{ {-2\mu_2 } }}  {k}  (x-\frac{y}{k})^2+   \frac  {4\mu_2} {k}  (x-\frac{y}{k})^3
  -  \frac{5}{k}  \sqrt{\frac{-\mu_2 }{2}}(x-\frac{y}{k})^4
  - \frac{1}{k} (x-\frac{y}{k})^5
  \\ \quad\quad
 - \frac{b}{k}\sqrt{ {-2\mu_2 } }(x-\frac{y}{k})y - \frac{b}{k}  (x-\frac{y}{k})^2y,
	\\
	\dot y= \mu_2 \sqrt{ {-2\mu_2 } }   (x-\frac{y}{k})^2+ 4\mu_2  (x-\frac{y}{k})^3
  -5\sqrt{\frac{-\mu_2 }{2}}(x-\frac{y}{k})^4
  - (x-\frac{y}{k})^5
  \\ \quad\quad
 - b\sqrt{ {-2\mu_2 } }(x-\frac{y}{k})y -b(x-\frac{y}{k})^2y-ky.
	\end{array}
		\label{initia2210301}
	\end{eqnarray}
 Solving $\dot y=0$ of  system \eqref{initia2210301}, we get
	\begin{eqnarray}
	y =
	 \frac{\mu_2 \sqrt{ {-2\mu_2 } }  }{k}x^2 +O(x^3)
	\label{ini10302}
	\end{eqnarray}
		by  the Implicit Function Theorem.
 Substituting  \eqref{ini10302} into the first  equation of  \eqref{initia2210301}, we    have
\begin{eqnarray*}
\frac{dx}{dt}=   \frac{\mu_2 \sqrt{ {-2\mu_2 } }  }{k}x^2 +O(x^3).
\end{eqnarray*}
  We know     $\mu_2=-2\sqrt{\mu_1}$  by  {\bf Case (VI)}.
 It can  easily be  checked  that     $k>0$ for $ \mu_3>  {b  \mu_2   }/{2}$
  and
 $k<0$ for $ \mu_3< {b  \mu_2   }/{2}$.
Based on  \cite[Theorem 7.1 of Chapter 2]{Zh}, it follows  that the origin of system \eqref{initia2210301}  is a  saddle-node with
one stable nodal sector
for  $ \mu_3>  {b  \mu_2   }/{2}$  or a  saddle-node with unstable nodal sector
for  $ \mu_3< {b  \mu_2   }/{2}$.   So is   $\hat{E}_{r2}$.
By the symmetry of system \eqref{initial} about the origin,       $\hat{E}_{l2}$     is     a  saddle-node with one stable nodal sector
for  $ \mu_3>  {b  \mu_2   }/{2}$  or a  saddle-node with  unstable nodal sector
for  $ \mu_3< {b  \mu_2   }/{2}$  when  $\Delta=0$.

The   Jacobian   matrices        at  both    $\hat{E}_{l1}$ and $\hat{E}_{r1}$   are  the following form
\[
	 J_{\hat{E}_{l1}}= J_{\hat{E}_{r1}}:= \left(
	\begin{array}{lll}
	&0  \quad   &1
	\\
	&- \Delta -\mu_2  \sqrt{\Delta}  \quad   & -\mu_3-\frac{b(-\mu_2-\sqrt{ \Delta})}{2}
	\end{array}
	\right).
	\]
 By  {\bf Case (I)},  we know that  the necessary conditions   of  the existence of   $\hat{E}_{l1}$ and $\hat{E}_{r1}$    are  $\Delta>0$  and    $  \mu_2<- \sqrt{\Delta}$.   Then, we have     ${\rm det}J_{\hat{E}_{l1}}={\rm det}J_{\hat{E}_{r1}} =  \Delta +\mu_2  \sqrt{\Delta}  <0  $.
Therefore, both  $\hat{E}_{l1}$ and $\hat{E}_{r1}$   are saddles.
The proof  of  Lemma \ref{fe1}  is complete.
\end{proof}

\section*{Appendix C}

 \begin{figure}[h]
	\centering
	\includegraphics[width=3.5in]{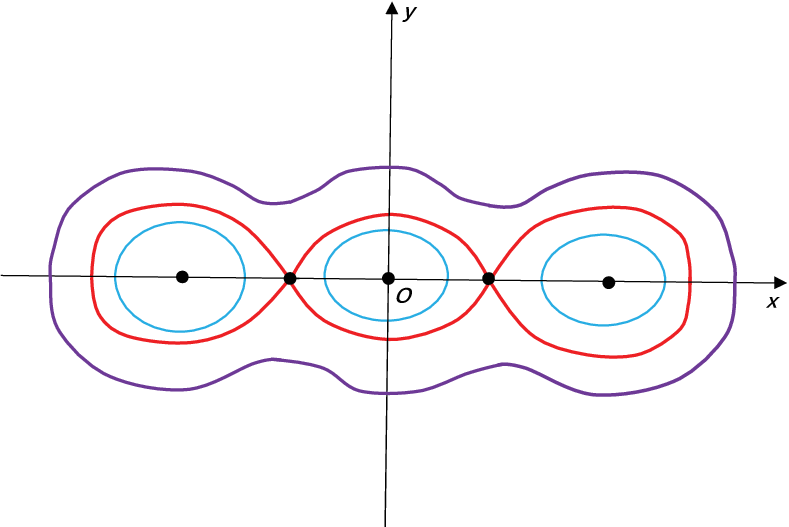}
	\caption{The phase portrait of  \eqref{Hami22}.   }
	\label{Hami1}
\end{figure}

When $\delta=0$,
system \eqref{initial1} is a Hamiltonian system
\begin{eqnarray}
\begin{array}{ll}
\dot x=y,
\\
\dot y=-x(a_1+x^2)(-1+x^2),
\end{array}
\label{Hami22}
\end{eqnarray}
which has the  first integral
\eqref{Exy}.
The level curves
$\{E(x,y)=e,~e\geq e_1\}$
are shown in Figure \ref{Hami1}, where $e_1:=\min\{0, -a_1/4-1/12\}$.
In this proof, we only care these closed orbits surrounding
 five equilibria, i.e., $\{E(x,y)=e,~e> e_2\}$, where $e_2:= {a_1}^2/4+{a_1}^3/12$.
Let $L:=\{(x,y)|x=0,~y>\sqrt{2e_2}\}$.
Then, the large closed level curve intersects $L$ at exactly one point $(0,\alpha(e))$.
Thus, $L$ can be parameterized by $e$.
For every $e\in(e_2,+\infty)$, consider the orbit of system  \eqref{Hami22} passing through $P_e\in L$.
When the orbit goes forward and backward from $P_e$, it has respectively two intersection points $Q_1$ and $Q_2$
with the positive $x$-axis.
Let the piece of orbit from $Q_1$ to $Q_2$ denote by $\gamma(e,\delta,a_2)$.
Clearly, $\gamma(e,\delta,a_2)$ is an orbit if and only if $Q_1=Q_2$.
Moreover, $Q_1=Q_2$ if and only if $E(Q_1)=E(Q_2)$.
It is to notice that
\[
\frac{dE}{dt}\mid_{\eqref{initial}}dt=-\delta(a_2+x^2)y^2dt=-\delta(a_2+x^2)ydx,
\]
implying that
\[
E(Q_2)-E(Q_1)=\int_{t(Q_1)}^{t(Q_2)}\frac{dE}{dt}\mid_{\eqref{initial}}dt
=-\int_{\gamma}\delta(a_2+x^2)ydx.
\]
Thus,  $\gamma(e,\delta,a_2)$ is a closed orbit if and only if $\int_{\gamma}(a_2+x^2)ydx=0$.
Let
\[
F(e,\delta,\zeta)=\int_{\gamma}(a_2+x^2)ydx.
\]
By the same way as Lemma 1.5 of \cite[Chapter 4]{CLW},
we can prove that the function $F(e,\delta,a_2)$ is continuous and $C^{\infty}$ in $\delta$ and $a_2$
on a set
\[
U=\{(e,a_2)| e_2\leq e<+\infty, ~0\leq \delta\leq \delta_0, \zeta_1\leq a_2\leq\zeta_2\},
\]
where $\delta_0$ is positive and $\zeta_1<\zeta_2$ are constants.
Moreover, $F\in C^{\infty}$ in $h$ on the set
\[
V=\{(e,b,\zeta)| e_2< e<+\infty, ~0\leq \delta\leq \delta_0, \zeta_1\leq a_2\leq\zeta_2\}.
\]
By Taylor  expansion, it follows that
\[
F(e,\delta,\zeta)=F(e,0,\zeta)+O(\delta),
\]
where
\[
F(e,0,\zeta)=\oint_{\Gamma_e}(a_2+x^2)ydx=a_2 I_0(e)+I_2(e),
\]
$I_i(e)=\oint_{\Gamma_e}x^iydx$ for $i=0,2$ and $\Gamma_e$ is the large closed orbit of  system  \eqref{Hami22}.
Let
$$P(h)=\frac{I_2(e)}{I_0(e)},$$
where $e\geq e_2$.
By
$$y=\sqrt{2e+a_1x^2 -\frac{(a_1-1)x^4}{2}-\frac{x^6}{3}},$$
it follows that
\begin{equation}\label{dI}
I_i'(e)=\oint_{\Gamma_e}\frac{x^i}{y}\mathrm{d}x,\quad \mbox{for} \ i\in\mathbb{N}.
\end{equation}
On the one hand, it follows from  \eqref{dI} that
\begin{equation}\label{IdI}
I_i(e)=2\int_{-\eta(e)}^{\eta(e)}\frac{x^i y^2}{y}\mathrm{d}x=2eI_i'(e)+a_1I'_{i+2}(e)
-\frac{a_1-1}{2}I'_{i+4}(e)-\frac{1}{3}I'_{i+6}(e).
\end{equation}
On the other hand, by an integration by parts, we have
\begin{equation}\label{Ii}
I_i(e)=2\left(\frac{x^{i+1}y}{i+1}|_{-\eta(e)}^{\eta(e)}
-\frac{1}{i+1}\int_{-\eta(e)}^{\eta(e)}\frac{x^{i+1}(a_1x-(a_1-1)x^3-x^5)}{y}\mathrm{d}x\right).
\end{equation}
By $y(-\eta(e),e)=y(\eta(e),e)=0$ and \eqref{Ii}, it follows that
\begin{equation}\label{IdI2}
I_i(e)=-\frac{a_1I'_{i+2}(e)}{i+1}+\frac{(a_1-1)I'_{i+4}(e)}{i+1}+\frac{I'_{i+6}(e)}{i+1}.
\end{equation}
Taking $i=0,2,4$ in \eqref{IdI2} and solving $I'_6(e), I'_8(e),I'_{10}(e)$, we obtain
\begin{equation}\begin{split}\label{I6810}
I'_6(e)=&I_0(e)+a_1I'_2(e)-(a_1-1)I'_4(e),\\
I'_8(e)=&(1-a_1)I_0(e)+3I_2(e)-(a_1-1)a_1I'_2(e)+(a_1^2-a_1+1)I'_4(e),\\
I'_{10}(e)=&(a_1^2-a_1+1)I_0(e)-3(a_1-1)I_2(e)+5I_4(e)+a_1(a_1^2-a_1+1)I'_2(e)\\
&-(a_1-1)(1+a_1^2)I'_4(e).
\end{split}\end{equation}
Taking $i=0,2,4$ in \eqref{IdI} and using \eqref{I6810}, we obtain that
\begin{equation}\begin{split}\label{I024}
I_0(e)=&\frac{3e}{2}I'_0(e)+\frac{a_1}{2}I'_2(e)+\frac{1-a_1}{8}I'_4(e),\\
I_2(e)=&\frac{(1-a_1)e}{8}I'_0(e)+\frac{a_1-a_1^2+8e}{8}I'_2(e)+\frac{9+14a_1+9a_1^2}{96}I'_4(e),\\
I_{4}(e)=&\frac{3(5+6a_1+5a_1^2)e}{128}I'_0(e)+\frac{3(5a_1+6a_1^2+5a_1^3+8e-8a_1e)}{128}I'_2(e)\\
&+\frac{45+73a_1-73a_1^2-45a_1^3+384e}{512}I'_4(e).
\end{split}\end{equation}
Letting $\mathrm{V}=(I_0(e),I_2(e),I_4(e))^\top$, then by \eqref{I024}, we have
\begin{equation}\label{V}
(12\mathrm{I}e+\mathrm{C})\mathrm{V'}=\mathrm{RV},
\end{equation}
where $\mathrm{I}$ is an unit matrix of order 3, and
\begin{equation}
\mathrm{C}=\left(\begin{array}{ccc}
0 & 4a_1 & 1-a_1 \\
0 & a_1(1-a_1) & (1+a_1)^2 \\
0 & a_1(1+a_1)^2 & 1+2a_1-2a_1^2-a_1^3
\end{array}\right),
\quad
\mathrm{R}=\left(\begin{array}{ccc}
8 & 0 & 0 \\
a_1-1 & 12 & 0 \\
-(1+a_1)^2 & 3(a_1-1) & 16
\end{array}\right).
\end{equation}
Taking
\begin{equation}\label{Z}
Z(e)=\frac{3}{4}(a_1-1)I_2(e)+I_4(e)
\end{equation}
and using \eqref{V}, we have
\begin{equation}\label{D2}
D(e)\left(\begin{array}{c}
I_0''(e) \\
I_2''(e) \\
Z''(e)
\end{array}\right)
=\left(\begin{array}{cc}
a_{11}(e) & a_{12}(e) \\
a_{21}(e) & a_{22}(e) \\
a_{31}(e) & a_{32}(e)
\end{array}
\right)
\left(
\begin{array}{c}
I'_0(e) \\
Z'(e)
\end{array}\right),
\end{equation}
where
\begin{equation}\label{Da}
\begin{split}
D(e)=&3 e (3a_1+12e+1) \left(-a_1^3-3a_1^2+12e\right),\\
a_{11}(e)=&-3e(3+7a_1-7a_1^2-3a_1^3+48e),\\
a_{12}(e)=&10a_1^2+3a_1^3-12e+3a_1(1+4e),\\
a_{21}(e)=&3 e (10 a_1^2 + 3 a_1^3 - 12 e + 3 a_1 (1 + 4 e)),\\
a_{22}(e)=&-12 (1 + a_1)^2 e,\\
a_{31}(e)=&-\frac{9}{4}e (-7 a_1^3 - 3 a_1^4 + 4 e + a_1^2 (7 + 4 e) + a_1 (3 + 56 e)),\\
a_{32}(e)=&-3 (-3 - 7 a_1 + 7 a_1^2 + 3 a_1^3 - 48 e) e.
\end{split}
\end{equation}
Let $M(e)=a_2I_0(e)+I_2(e)$. By \eqref{Z} and \eqref{D2},
\begin{equation}\begin{split}\label{M2e}
M''(e)=&\frac{(a_2a_{11}(e)+a_{21}(e))I'_0(e)+(a_2a_{12}(e)+a_{22}(e))Z'(e)}{D(e)}\\
=&\frac{I'_0(e)}{D(e)}\left(a_2a_{11}(e)+a_{21}(e)+(a_2a_{12}(e)+a_{22}(e))w(e)\right),
\end{split}\end{equation}
where
\begin{equation}
w(e)=\frac{Z'(e)}{I'_0(e)}
\end{equation}
satisfies the differential equation
\begin{equation}\label{ew}
\begin{split}
\dot{e}= & 12e(1+3a_1+12e)(-3a_1^2-a_1^3+12e), \\
\dot{w}=&v_0(e)+v_1(e)w+v_2(e)w^2,
\end{split}
\end{equation}
and
\begin{equation}\begin{split}\label{ve}
v_0(e)=&-9e (3 a_1 + 7 a_1^2 - 7 a_1^3 - 3 a_1^4 + 4 e + 56 a_1 e + 4 a_1^2 e), \\
v_1(e)=&24 e (3 + 7 a_1 - 7 a_1^2 - 3 a_1^3 + 48 e),\\
v_2(e)=&-4(3 a_1 + 10 a_1^2 + 3 a_1^3 - 12 e + 12 a_1 e).
\end{split}\end{equation}
In the following, we will split some cases to study the number of zeros of $M''(e)$ in \eqref{M2e} for $e>e_2$.

First, we consider $a_1=-1/3$. In the case, $e_2=2/81$, and
$$
M''(e)=-\frac{(3a_2+1) (w(e)+9e) I_0'(e)}{e(81e-2)},\quad e>e_2.
$$
Notice that $a_2<\min\{a_1,-1/3\}=-1/3$. Thus, the number of zeros of $M''(e)$ in $(e_2, +\infty)$ equals the number of intersection points of
the curve $\Gamma=\{(e,w)\mid w=w(e), e\in(e_2, \infty)\}$ and the straight line $\mathcal{L}=\{(e,w)\mid w+9e=0, e\in(e_2, \infty)\}$ in the $(e,w)$ plane.

Using \eqref{D2}, we have the asymptotic expansion of $w(e)$ near $e\rightarrow e_2^+$
\begin{equation}\label{we1}
w(e)=-\frac{2}{9}-\frac{2}{3\log(e-e_2)}+o\left(\frac{1}{\log(e-e_2)}\right),
\end{equation}
and the asymptotic expansion of $w(e)$ when $e\rightarrow+\infty$
$$w(e)=c_0e^{\frac{2}{3}}+o\left(e^{\frac{2}{3}}\right),$$
where $c_0$ is a positive constant.
In this case, $w(e)$ satisfies the differential equation
\begin{equation}\label{ew1}
\begin{split}
\dot{e}= & e(81e-2), \\
\dot{w}=& 3(w^2+18ew+2e).
\end{split}
\end{equation}
Notice that the horizontal isocline of this system has two branches
$$
w^+(e)=-9 e+\sqrt{81 e^2-2 e},~~~~
w^-(e)=-9 e-\sqrt{81 e^2-2 e}.
$$
Moreover, $w^+(e)$ has the asymptotic expansion near $e\rightarrow e_2^+$
\begin{equation}\label{W^+}
w^+(e)=-\frac{2}{9}+\sqrt{2} \sqrt{e-e_2}-9 \left(e-e_2\right)+o\left(e-e_2\right).
\end{equation}
Comparing \eqref{we1} with \eqref{W^+}, we have
$$
w(e)>w^+(e), \quad e\rightarrow e_2.
$$
It follows from system \eqref{ew1} that
$$\frac{\mathrm{d}w}{\mathrm{d}e}>0,\quad \mbox{for}\ w>w^+(e),\ e>e_2.$$
Thus, we have $w(e)>w^+(e)$ for all $e\in\left(e_2,+\infty\right)$. Further, we obtain
$$w(e)>w^+(e)>-9e,\quad e\in\left(e_2,+\infty\right).$$
This implies that the curve $\Gamma$ and the straight line $\mathcal{L}$ have no intersection points
in the $(e,w)$ plane.
Hence, the function $M''(e)$ has no zeros in $(e_2,+\infty)$. And we obtain that
$M(e)$ has at most two zeros in $(e_2,+\infty)$.

When $a_1\neq-\frac{1}{3}$, denote
$$
A(e)=a_2a_{11}(e)+a_{21}(e),\quad \mbox{and}\quad B(e)=a_2a_{12}(e)+a_{22}(e),
$$
where ${a_{ij}}^,s$ are given in \eqref{Da}. Obviously, $B(e)$ is a polynomial of $e$ with degree at most $1$. And it is
not identically zero.
Thus, the number of zeros of $M''(e)$ in $(e_2, +\infty)$ equals the number of intersection points of
the curve $\Gamma=\{(e,w)\mid w=w(e), e\in(e_2, \infty)\}$ and the curve $\mathcal{C}=\{(e,w)\mid A(e)+B(e)w=0, e\in(e_2, \infty)\}$ in the $(e,w)$ plane.

Using \eqref{D2}, we have the asymptotic expansion of $w(e)$ near $e\rightarrow e_2^+$
\begin{equation}\label{we12}
w(e)=\frac{1}{4}a_1(3+a_1)+\frac{3a_1(1+a_1)}{\log(e-e_2)}+o\left(\frac{1}{\log(e-e_2)}\right),
\end{equation}
and the asymptotic expansion of $w(e)$ when $e\rightarrow+\infty$
\begin{equation}\label{we22}
w(e)=c_1e^{\frac{2}{3}}+o\left(e^{\frac{2}{3}}\right), \quad c_1>0.
\end{equation}
Note that the horizontal isocline of system \eqref{ew} has two branches
$$
w^+(e)=\frac{-v_1(e)+\sqrt{v_1(e)^2-4 v_0(e)v_2(e)}}{2v_2(e)},~~~~
w^-(e)=\frac{-v_1(e)-\sqrt{v_1(e)^2-4 v_0(e)v_2(e)}}{2v_2(e)},
$$
where $v_i(e), i=0,1,2$ are given in \eqref{ve}, and for $e>e_2$
$$
v_2(e)> v_2(e_2)=-4a_1(1+a_1)^2(3+a_1)>0,$$
$$
v_1^2(e)-4v_0(e)v_2(e)=144 e \left(3a_1^3+7a_1^2-7a_1-64 e-3\right) \left(a_1^3+3a_1^2-12
e\right) (3a_1+12 e+1)>0.
$$
Moreover, $w^+(e)$ has the asymptotic expansion
\begin{equation}\label{W2^+}
w^+(e)=\frac{1}{4} a_1
(a_1+3)+\frac{\sqrt{3} }{2} \sqrt{\frac{7 a_1^2+6
		a_1+3}{a_1+1}} \sqrt{e-e_2}+\frac{3 \left(4 a_1^2+3 a_1+1\right) \left(e-e_2\right)}{a_1 (a_1+1)^2}+o\left(e-e_2\right).
\end{equation}
Comparing \eqref{we12} with \eqref{W2^+}, we have
$
w(e)>w^+(e)$ as $e\rightarrow e_2.
$
It follows from system \eqref{ew} that
$$\frac{\mathrm{d}w}{\mathrm{d}e}>0,\quad \mbox{for}\ w>w^+(e),\ e>e_2.$$
Thus, we have $w(e)>w^+(e)$ for all $e\in\left(e_2,+\infty\right)$.

For the curve $\mathcal{C}$, denote $w_{\mathcal{C}}(e)=-\frac{A(e)}{B(e)}$. Moreover, denote by $l_1$ the curve $a_2=-\frac{(1+a_1)^2}{1-a_1}$. It is easy to verify that $\frac{\mathrm{d}a_2}{\mathrm{d}a_1}=\frac{(-3 + a_1) (1 + a_1)}{(-1 + a_1)^2}<0$ for $-1<a_1<0$,
and $a_2=0$ when $a_1=-1$, $a_2=-\frac{1}{3}$ when $a_1=-\frac{1}{3}$ and $a_2=-1$ when $a_1=0$. That is, the curve $l_1$ passing through $(-1,0), (-\frac{1}{3},-\frac{1}{3})$ and $(0,-1)$ is deceasing on $(-1,0)$. Thus, the curve $l_1$ intersects the line $a_2=a_1$ at $A=(-\frac{1}{3},-\frac{1}{3})$ and is above the line $a_2=a_1$ when $-1<a_1<-\frac{1}{3}$, and the curve $l_1$ intersects the curve $a_2=\frac{a_1-1-\sqrt{-a_1}}{3}$ with $B=(a_1^*,a_2^*)$ with  $a_1^*=\frac{1}{16} (-33-\sqrt{97}+\sqrt{930 + 66 \sqrt{97}})\approx-0.19$,
such that $\frac{a_1-1-\sqrt{-a_1}}{3}<-\frac{(1+a_1)^2}{1-a_1}<-\frac{1}{3}$ when $-\frac{1}{3}<a_1<a_1^*$ and the curve $l_1$ is below the curve $a_2=\frac{a_1-1-\sqrt{-a_1}}{3}$ when $a_1^*<a_1<0$, see Figure \ref{fig1}.
\begin{figure}[!htbp]
	\centering
	\includegraphics[width=0.4\textwidth]{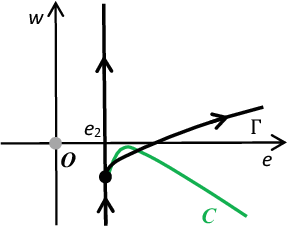}
	\caption{The graphs of the curves of $\mathcal{C}$ and $\Gamma$.}
	\label{fig1}
\end{figure}

When $a_2=-\frac{(1+a_1)^2}{1-a_1}$,
$$w_{\mathcal{C}}(e)=\frac{3e\left(12e-a_1^3-2 a_1^2+2 a_1+1\right)}{a_1 (a_1+1)^2}.$$
Since $w_{\mathcal{C}}(e_2)=\frac{1}{4} a_1(a_1+3)$ and
$$w_{\mathcal{C}}'(e)=\frac{3\left(24e-a_1^3-2 a_1^2+2 a_1+1\right)}{a_1 (a_1+1)^2}
<\frac{3\left(a_1^3+4a_1^2+2 a_1+1\right)}{a_1 (a_1+1)^2}<0,\quad e>e_2.$$
Thus, $w_{\mathcal{C}}(e)<w_{\mathcal{C}}(e_2)<w(e)$, when $e>e_2$. For this case,
the function $M''(e)$ has no zeros in $(e_2,+\infty)$.
When $a_2\neq-\frac{(1+a_1)^2}{1-a_1}$,
$$
w_{\mathcal{C}}(e)=\frac{3(-1+a_1-4a_2)e(e-\bar{e}_1)}{(1+2a_1+a_1^2+ a_2 - a_1 a_2)(e-\bar{e}_2)},~~~
w_{\mathcal{C}}'(e)=\frac{3(-1+a_1-4a_2)(e^2 - 2\bar e_2e + \bar e_1 \bar e_2)}{(1+2a_1+a_1^2+ a_2 - a_1 a_2)(e-\bar{e}_2)^2},
$$
where
$$
\bar{e}_1=  -\frac{(3 + a_1) (1 + 3 a_1) (a_1 - a_2 + a_1 a_2)}{12 (-1 + a_1 - 4 a_2)},~~~~
\bar{e}_2=  \frac{a_1 (3 + a_1) (1 + 3 a_1) a_2}{12 (1 + 2 a_1 + a_1^2 + a_2 - a_1 a_2)}.
$$
Notice that $\frac{a_1-1-\sqrt{-a_1}}{3}<a_2<\min\{a_1,-\frac{1}{3}\}$ with $-1<a_1<0$. Thus, $a_2<a_1$ and $-1+a_1-4a_2>-1-3a_2>0$, implying that
\begin{equation}\begin{split}\label{e1}
\bar{e}_1>&\ 0, \quad \mbox{for}\ -1<a_1<-\frac{1}{3},\\
\bar{e}_1<&\ 0, \quad \mbox{for}\   -\frac{1}{3}<a_1<0,\\
e_2-\bar{e}_1=&\ \frac{(1 + a_1)^2 (3+a_1)(a_1-a_2)}{12(-1+a_1-4a_2)}>0.
\end{split}\end{equation}
Moreover, we have
$$
e_2-\bar{e}_2=\frac{a_1 (1 + a_1)^2 (3 + a_1) (a_1 - a_2)}{12 (1 + 2 a_1 + a_1^2 + a_2 - a_1 a_2)}>0
$$
if $a_2<-\frac{(1+a_1)^2}{1-a_1}$,
and
\begin{equation}\label{e22}
e_2-\bar{e}_2=\frac{a_1 (1 + a_1)^2 (3 + a_1) (a_1 - a_2)}{12 (1 + 2 a_1 + a_1^2 + a_2 - a_1 a_2)}<0
\end{equation}
if $a_2>-\frac{(1+a_1)^2}{1-a_1}$.
It follows that when $a_2<-\frac{(1+a_1)^2}{1-a_1}$, only the right half branch of the curve $w_{\mathcal{C}}(e)$ is located at $e>e_2$ and
$$
w_{\mathcal{C}}(e)=\frac{3(-1+a_1-4a_2)e(e-\bar{e}_1)}{(1+2a_1+a_1^2+ a_2 - a_1 a_2)(e-\bar{e}_2)}<0,
$$
while  when $a_2>-\frac{(1+a_1)^2}{1-a_1}$, the curve $w_{\mathcal{C}}(e)$ has two branches for $e>e_2$ and
$$
w_{\mathcal{C}}(e)=\frac{3(-1+a_1-4a_2)e(e-\bar{e}_1)}{(1+2a_1+a_1^2+ a_2 - a_1 a_2)(e-\bar{e}_2)}<0,\ \mbox{if} \ e_2<e<\bar{e}_2,
$$
and
$$
w_{\mathcal{C}}(e)=\frac{3(-1+a_1-4a_2)e(e-\bar{e}_1)}{(1+2a_1+a_1^2+ a_2 - a_1 a_2)(e-\bar{e}_2)}>0,\ \mbox{if} \ e>\bar{e}_2>e_2.
$$
Moreover, since
$$
\bar e_2^2-\bar e_1 \bar e_2=\frac{
	a_1a_2(1 + a_1)^2 (3 + a_1)^2 (1 + 3 a_1)^2 (a_1 - a_2) (1 + a_2)}{
	144 (-1 + a_1 - 4 a_2) (1 + 2 a_1 + a_1^2 + a_2 - a_1 a_2)^2}>0
$$
when $a_2<-\frac{(1+a_1)^2}{1-a_1}$,
we obtain that the curve $w_{\mathcal{C}}(e)$ is increasing on the interval $(\bar e_2,\bar e_2+\sqrt{\bar e_2^2-\bar e_1 \bar e_2})$ and   decreasing on the interval $(\bar e_2+\sqrt{\bar e_2^2-\bar e_1 \bar e_2}, +\infty)$. When $a_2>-\frac{(1+a_1)^2}{1-a_1}$,
we have that the curve $w_{\mathcal{C}}(e)$ is decreasing on the interval $(\bar e_2-\sqrt{\bar e_2^2-\bar e_1 \bar e_2},\bar e_2)\bigcup(\bar e_2,\bar e_2+\sqrt{\bar e_2^2-\bar e_1 \bar e_2})$ and  increasing on the interval $(\bar e_2+\sqrt{\bar e_2^2-\bar e_1 \bar e_2}, +\infty)$.

When $a_2<-\frac{(1+a_1)^2}{1-a_1}$, compare $e_2$ with $\bar e_2+\sqrt{\bar e_2^2-\bar e_1 \bar e_2}$. A direct computation shows that
\begin{equation}\label{e1e2}
\bar e_2^2-\bar e_1 \bar e_2-(e_2-\bar e_2)^2=\frac{a_1 (1 + a_1)^2 (3 + a_1)^2 (a_1 - a_2) (1 + 3 a_1 +
	4 a_1^2)}{144 (-1 +
	a_1 - 4 a_2) (1 + 2 a_1 + a_1^2 + a_2 - a_1 a_2)}\left(a_2 - \frac{(-1 + a_1) a_1^2}{1 + 3 a_1 + 4 a_1^2}\right).
\end{equation}
Denote by $l_2$ the curve $a_2=\frac{(-1 + a_1) a_1^2}{1 + 3 a_1 + 4 a_1^2}$. It is easy to verify that $\frac{\mathrm{d}a_2}{\mathrm{d}a_1}=\frac{2 a_1 (1 + a_1)^2 (2 a_1-1)}{(1 + 3 a_1 + 4 a_1^2)^2}>0$ for $-1<a_1<0$,
and $a_2=-1$ when $a_1=-1$, $a_2=-\frac{1}{3}$ when $a_1=-\frac{1}{3}$ and $a_2=0$ when $a_1=0$. That is, the curve $l_2$ passing through $(-1,-1), (-\frac{1}{3},-\frac{1}{3})$ and $(0,0)$ is increasing on $(-1,0)$. Owing to
$$\frac{\mathrm{d}a_2}{\mathrm{d}a_1}\Big|_{a_1=-1}=0 \quad\mbox{and}\quad \frac{\mathrm{d}}{\mathrm{d}a_1}\left(\frac{a_1-1-\sqrt{-a_1}}{3}\right)\Big|_{a_1=-1}=\frac{1}{2},$$
we obtain that the curve $l_2$ insects the curve $a_2=\frac{a_1-1-\sqrt{-a_1}}{3}$ at $C=(a_1^{**}, a_2^{**})$ with $a_1^{**}\approx-0.49$ and intersects the line $a_2=a_1$ at $A$, see Figure \ref{fig1}. Two cases are considered.

$(\mathrm{i})$ When $a_2\leq\frac{(-1 + a_1) a_1^2}{1 + 3 a_1 + 4 a_1^2}$,
$\bar e_2^2-\bar e_1 \bar e_2-(e_2-\bar e_2)^2\leq0$ from \eqref{e1e2}, thus we have
$$\bar e_2+\sqrt{\bar e_2^2-\bar e_1 \bar e_2}-e_2=\sqrt{\bar e_2^2-\bar e_1 \bar e_2}-(e_2-\bar e_2)\leq0.$$
Hence, the curve $w_{\mathcal{C}}(e)$ is decreasing on the interval $(e_2, +\infty)$. Combing with
$$w_{\mathcal{C}}(e_2)=\frac{1}{4} a_1(a_1+3),$$
it follows that for $e>e_2$,
$$w_{\mathcal{C}}(e)<w_{\mathcal{C}}(e_2)<w(e),$$
which implies that the function $M''(e)$ has no zeros in $(e_2,+\infty)$ when $(a_1,a_2)\in G_1$. And we obtain that
$M(e)$ has at most two zeros in $(e_2,+\infty)$.

$(\mathrm{ii})$ When $a_2>\frac{(-1 + a_1) a_1^2}{1 + 3 a_1 + 4 a_1^2}$,
$\bar e_2^2-\bar e_1 \bar e_2-(e_2-\bar e_2)^2>0$ from \eqref{e1e2}, thus we have
$$\bar e_2+\sqrt{\bar e_2^2-\bar e_1 \bar e_2}-e_2=\sqrt{\bar e_2^2-\bar e_1 \bar e_2}-(e_2-\bar e_2)>0.$$
Hence, the function $w(e)$ is increasing on the interval $(e_2,\bar e_2+\sqrt{\bar e_2^2-\bar e_1 \bar e_2})$ and is decreasing on the interval $(\bar e_2+\sqrt{\bar e_2^2-\bar e_1 \bar e_2}, +\infty)$.
Notice that near $e\rightarrow e_2^+$
\begin{equation}\label{C1}
w_{\mathcal{C}}(e)=\frac{1}{4} a_1(a_1+3)+\frac{3 ((-1 + a_1) a_1^2 - (1+3 a_1+4 a_1^2) a_2)}{a_1 (1 + a_1)^2 (a_1 - a_2)}(e-e_2)+o(e-e_2),
\end{equation}
and
\begin{equation}\label{C12}
w_{\mathcal{C}}(e)\rightarrow-\infty, \ \mbox{when}\  e\rightarrow+\infty.
\end{equation}
Comparing the results in \eqref{we12}, \eqref{we22}, \eqref{C1} and \eqref{C12}, we have
\begin{equation}\begin{split}\label{wC}
w_{\mathcal{C}}(e)<&\ w(e), \ \mbox{when}\  e\rightarrow e_2^+,\\
w_{\mathcal{C}}(e)<&\ w(e), \ \mbox{when}\  e\rightarrow+\infty.
\end{split}\end{equation}
We claim that the function $M''(e)$ has at most two zeros in $(e_2,+\infty)$ when  $(a_1,a_2)\in G_2$. To prove this,
we show that there exists exactly one point of the curve $\mathcal{C}$ at which the vector field \eqref{ew} is tangent on the curve $\mathcal{C}$. We call it the contact point. A direct computation gives that
\begin{equation}\begin{split}\label{d}
\frac{\mathrm{d}w_{\mathcal{C}}(e)}{\mathrm{d}e}-\frac{\mathrm{d}w(e)}{\mathrm{d}e}\Big{|}_{w=w_{\mathcal{C}}(e)}
=-\frac{3}{4 B^2(e)} \Psi(e),
\end{split}\end{equation}
where
\begin{equation}\begin{split}\label{Pse}
\Psi(e)=&\psi_2e^2+\psi_1e+\psi_0\\
=&-48(a_1-4 a_2-1) \left(5 a_1^2-8 a_1 a_2+6 a_1+8
a_2+5\right) e^2+4 (a_1+3) (3 a_1+1)\\
&\left(6 a_1^3 a_2+3 a_1^3-3
a_1^2 a_2^2+30 a_1^2 a_2+10 a_1^2-74 a_1 a_2^2-30 a_1
a_2+3 a_1-3 a_2^2-6 a_2\right)e\\
&+a_1 (a_1+3)^2 (3 a_1+1)^2 a_2
(3 a_1 a_2+4 a_1-3 a_2).
\end{split}\end{equation}
Denote by $l_3$ the curve $a_2=\frac{5 + 6 a_1 + 5 a_1^2}{8(-1 + a_1)}$. It is easy to verify that $\frac{\mathrm{d}a_2}{\mathrm{d}a_1}=\frac{-11 - 10 a_1 + 5 a_1^2}{8 (-1 + a_1)^2}>0$ for $-1<a_1<\frac{5 - 4 \sqrt5}{5}$, and $\frac{\mathrm{d}a_2}{\mathrm{d}a_1}<0$ for $\frac{5 - 4 \sqrt5}{5}<a_1<0$,
and $a_2=-\frac{1}{4}$ when $a_1=-1$, $a_2=-\frac{1}{3}$ when $a_1=-\frac{1}{3}$ and $a_2=-\frac{5}{8}$ when $a_1=0$. That is, the curve $l_3$ passing through $(-1,-\frac{1}{4}), (-\frac{1}{3},-\frac{1}{3})$ and $(0,-\frac{5}{8})$ is increasing on $(-1,\frac{5 - 4 \sqrt5}{5})$ and is decreasing on $(\frac{5 - 4 \sqrt5}{5},0)$. Thus,
we obtain that the curve $l_3$ intersects the line $a_2=a_1$ at $A$, and insects the curve $a_2=\frac{a_1-1-\sqrt{-a_1}}{3}$ at $D=(a_1^{***}, a_2^{***})$ with $a_1^{***}=\frac{1}{49} \left(-135+36 \sqrt{11}+4 \sqrt{802-240 \sqrt{11}}\right)\approx-0.118$, which locates between the curves $l_1$ and $l_2$ when $-\frac{1}{3}<a_1<0$, see Figure \ref{fig1}. In the region $G_2$,
it follows from $a_2<\frac{5 + 6 a_1 + 5 a_1^2}{8(-1 + a_1)}$ that $\psi_2>0$. Since
\begin{equation}\label{Pe2}
\Psi(e_2)= -\frac{1}{3} a_1 (1 + a_1)^2 (3 + a_1)^2 (a_1 - a_2) (-3 a_1 - 18 a_1^2 + 5 a_1^3 -
9 a_2 - 30 a_1 a_2 - 41 a_1^2 a_2),
\end{equation}
and in the region $G_2$, $-1<a_1<-\frac{1}{3}$ and $a_2>\frac{(-1 + a_1) a_1^2}{1 + 3 a_1 + 4 a_1^2}$,
\begin{equation}\begin{split}\label{Pe21}
-3 a_1 - 18 a_1^2 + 5 a_1^3 -
9 a_2 - 30 a_1 a_2 - 41 a_1^2 a_2
&=  a_1 (-3 - 18 a_1 + 5 a_1^2) + (-9 - 30 a_1 - 41 a_1^2) a_2 \\
&< -\frac{a_1 (1 + a_1) (1 + 3 a_1) (3 + 6 a_1 + 7 a_1^2)}{1 + 3 a_1 + 4 a_1^2}\\
&<0,
\end{split}
\end{equation}
we have $\Psi(e_2)<0$ from \eqref{Pe2} and \eqref{Pe21}. Further, $\Psi(e)$ has a unique zero in $(e_2,+\infty)$ by $\psi_2>0$. This confirms that there are exactly one contact point on the curve $\mathcal{C}$ for $e>e_2$.
It follows from the result in \eqref{wC} that the curve $\mathcal{C}$ and the curve $\Gamma$ has at most two intersection points when $e>e_2$, otherwise, extra contact points will emerge, which results in a contradiction, see Figure \ref{fig2}. Therefore, the function $M(e)$ has at most four zeros on the interval $(e_2,+\infty)$ when $(a_1,a_2)\in G_2$.
\begin{figure}[!htbp]
	\centering
	\includegraphics[width=0.4\textwidth]{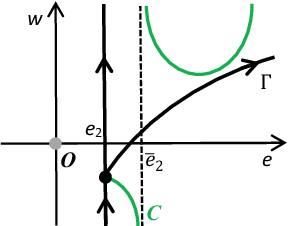}
	\caption{The graphs of the curves of $\mathcal{C}$ and $\Gamma$.}
	\label{fig2}
\end{figure}

Finally, we study the number of zeros of $M(e)$ on the interval $(e_2,+\infty)$ when $(a_1,a_2)$ locates in the regions $G_3$  and $G_4$ and their intersection curve $l_3$. Notice that in these regions $-\frac{1}{3}<a_1<0$ and $a_2>-\frac{(1+a_1)^2}{1-a_1}$. For this case, $e_2<\bar e_2$ and $\bar e_2-\sqrt{\bar e_2^2-\bar e_1 \bar e_2}<0<e_2$ by \eqref{e1}
and \eqref{e22}.
Thus, the function $w_{\mathcal{C}}(e)$ is decreasing on the interval $(e_2,\bar e_2)\bigcup(\bar e_2,\bar e_2+\sqrt{\bar e_2^2-\bar e_1 \bar e_2})$ and is increasing on the interval $(\bar e_2+\sqrt{\bar e_2^2-\bar e_1 \bar e_2}, +\infty)$.

Notice that near $e\rightarrow e_2^+$, $w_{\mathcal{C}}(e)$ has the asymptotic expansion \eqref{C1},
and
\begin{equation}\begin{split}
w_{\mathcal{C}}(e)\rightarrow-\infty,  \ \mbox{when}\  e\rightarrow\bar e_2^-,\\
w_{\mathcal{C}}(e)\rightarrow+\infty \ \mbox{when}\  e\rightarrow\bar e_2^+,
\end{split}\label{C14}
\end{equation}
and when $e\rightarrow+\infty$
\begin{equation}\label{C15}
w_{\mathcal{C}}(e)=\frac{3 (-1 + a_1 - 4 a_2)}{(1 + 2 a_1 + a_1^2 + a_2 - a_1 a_2)}e+o(e).
\end{equation}
Comparing the results in \eqref{we12}, \eqref{we22}, \eqref{C1}, \eqref{C14} and \eqref{C15}, we have
\begin{equation}\begin{split}\label{wC1}
w_{\mathcal{C}}(e)<&\ w(e), \ \mbox{when}\  e\rightarrow e_2^+,\\
w_{\mathcal{C}}(e)<&\ w(e), \ \mbox{when}\  e\rightarrow\bar e_2^-,\\
w_{\mathcal{C}}(e)>&\ w(e), \ \mbox{when}\  e\rightarrow\bar e_2^+,\\
w_{\mathcal{C}}(e)>&\ w(e), \ \mbox{when}\  e\rightarrow+\infty.
\end{split}\end{equation}
By the result in \eqref{wC1} and the monotonicity of the function $w_{\mathcal{C}}(e)$ in $(e_2,\bar e_2)$ and the function $w(e)$, we know that  the curves $\mathcal{C}$ and
$\Gamma$ do not have intersection points when $e\in(e_2,\bar e_2)$. In the following, we just need to consider $e>\bar e_2$.

For $(a_1,a_2)\in l_3=\partial G_3\bigcup\partial G_4$, $-\frac{1}{3}<a_1\leq a_1^{***}$ and $a_2=\frac{5 + 6 a_1 + 5 a_1^2}{8(-1 + a_1)}$. One has $\psi_2=0$ and $\Psi(e)$ has a unique zero
$$\hat e=\frac{(1 - a_1) a_1 (3 + a_1) (1 + 3 a_1) (5 + 6 a_1 + 5 a_1^2)}{
	4 (11 - 76 a_1 - 126 a_1^2 - 76 a_1^3 + 11 a_1^4)}.$$
Since
$$\hat e-e_2=-\frac{(-5 + a_1) a_1 (1 + a_1)^2 (3 + a_1) (3 + 2 a_1 + 11 a_1^2)}{
	12 (11 - 76 a_1 - 126 a_1^2 - 76 a_1^3 + 11 a_1^4)}<0,$$
we have there is no contact point on the curve $\mathcal{C}$ for $e>e_2$ from \eqref{d}. Thus, by the result in \eqref{wC1}, the curve $\mathcal{C}$ and the curve $\Gamma$ has no  intersection points when $e>e_2$, otherwise, an extra contact point will emerge, which results in a contradiction. Therefore, the function $M(e)$ has at most two zeros on the interval $(e_2,+\infty)$ for $(a_1,a_2)\in l_3=\partial G_3\bigcup\partial G_4$.

In the region $G_3$, $-\frac{1}{3}<a_1\leq a_1^{***}$ and $a_2<\frac{5 + 6 a_1 + 5 a_1^2}{8(-1 + a_1)}<-\frac{1}{3}$, one has $\psi_2>0$ and
$$\psi_0=a_1a_2 (3+a_1)^2 (1 + 3 a_1)^2 (4 a_1 - 3 a_2 + 3 a_1 a_2)>a_1 a_2(3+a_1)^2 (1 + 3 a_1)^3>0.$$
A direct computation shows that
\begin{eqnarray*}\begin{split}
\psi_1\Big{|}_{a_2=\frac{5 + 6 a_1 + 5 a_1^2}{8(-1 + a_1)}} =& \frac{5 (3 + a_1)^2 (1 + 3 a_1)^2 (11 - 76 a_1 - 126 a_1^2 - 76 a_1^3 +
	11 a_1^4)}{16 (-1 + a_1)^2} >0\\
\psi_1\Big{|}_{a_2=-\frac{(1+a_1)^2}{1-a_1}}=&\frac{4 (3 + a_1)^2 (1 + 3 a_1)^2 (1 - 19 a_1 - 44 a_1^2 - 19 a_1^3 +
	a_1^4)}{(-1 + a_1)^2}>0\\
\psi_1\Big{|}_{a_2=\frac{a_1-1-\sqrt{-a_1}}{3}}=&\frac{4}{9}(3 + a_1) (1 + 3 a_1) (15 + 12 \sqrt{-a_1} + 34 a_1 -
52 \sqrt{-a_1} a_1 + 126 a_1^2\\
&+ 52 \sqrt{-a_1} a_1^2 + 34 a_1^3 -
12 \sqrt{-a_1} a_1^3 + 15 a_1^4)>0.
\end{split}
\end{eqnarray*}
It is also easy to verify that $\psi_1>0$ by considering the intersection points of the curve $\psi_1=0$ and $l_1, l_3$ and $a_2=\frac{a_1-1-\sqrt{-a_1}}{3}$ for $a_1\in(-\frac{1}{3},0)$. Thus, the function $\Psi(e)$ has no zero in $(\bar e_2,+\infty)$. In other words, there does not exist the contact point on the curve $\mathcal{C}$ for $e>\bar e_2$ from \eqref{d} when $(a_1,a_2)\in G_3$. Similarly, by the result in \eqref{wC1}
the curve $\mathcal{C}$ and the curve $\Gamma$ has no  intersection points when $e>\bar e_2$. Otherwise, an extra contact point will emerge, which results in a contradiction. Therefore, the function $M(e)$ has at most two zeros on the interval $(e_2,+\infty)$ for $(a_1,a_2)\in G_3$.

\begin{figure}[!htbp]
	\centering
	\includegraphics[width=0.4\textwidth]{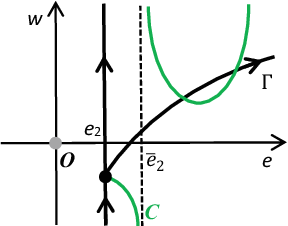}
	\caption{The graphs of the curves of $\mathcal{C}$ and $\Gamma$.}
	\label{fig3}
\end{figure}

In the region $G_4$, $-\frac{1}{3}<a_1<0$ and $\frac{5 + 6 a_1 + 5 a_1^2}{8(-1 + a_1)}<a_2<-\frac{1}{3}$, one has $\psi_2<0$ and
$$\psi_0=a_1a_2 (3+a_1)^2 (1 + 3 a_1)^2 (4 a_1 - 3 a_2 + 3 a_1 a_2)>a_1 a_2(3+a_1)^2 (1 + 3 a_1)^3>0.$$
A direct computation shows that
\begin{equation}\label{Pe3}
\Psi(\bar e_2)= \frac{a_1 (1 + a_1)^2 (3 + a_1)^2 (1 + 3 a_1)^2 (a_1 - a_2) a_2 (1 + a_2) (15 +
	34 a_1 + 15 a_1^2 + 12 a_2 - 12 a_1 a_2)}{3 (1 + 2 a_1 + a_1^2 + a_2 -
	a_1 a_2)^2},
\end{equation}
and
\begin{equation}\label{Pe22}
15 + 34 a_1 + 15 a_1^2 + 12 a_2 - 12 a_1 a_2>\frac{5}{2}(3 + a_1) (1 + 3 a_1)>0.
\end{equation}
Hence, we have $\Psi(\bar e_2)>0$ from \eqref{Pe3} and \eqref{Pe22}. Further, $\Psi(e)$ in \eqref{Pse} has a unique zero in $(\bar e_2,+\infty)$ by $\psi_2<0$. This confirms that there is exactly one contact point on the curve $\mathcal{C}$ for $e>\bar e_2$.
It follows from the result in \eqref{wC1} that the curve $\mathcal{C}$ and the curve $\Gamma$ has at most two intersection points when $e>\bar e_2$. Otherwise, extra contact points will emerge, which results in a contradiction, see Figure \ref{fig3}. Therefore, the function $M(e)$ has at most four zeros on the interval $(e_2,+\infty)$ when $(a_1,a_2)\in G_4$.

\end{document}